\documentclass[11pt]{article}
\pdfoutput=1 
\usepackage[letterpaper]{geometry}
\RequirePackage{amsthm,amsmath,amsfonts,amssymb}
\RequirePackage[authoryear]{natbib}%% uncomment this for author-year citations
\setcitestyle{authoryear,open={(},close={)}} 
\RequirePackage[colorlinks,citecolor=blue,urlcolor=blue]{hyperref}
\RequirePackage{graphicx,authblk}

\usepackage{titletoc,mdframed}
\usepackage{amsmath,amssymb,amsthm}
\usepackage{graphicx,wrapfig}
\usepackage{pbox}
\usepackage{bm}
\usepackage{amsthm,todonotes}
\usepackage{mathrsfs,mathabx}
\usepackage{mathtools}%textcomp
\usepackage{float,xcolor}
\usepackage{indentfirst,color}
\usepackage{tcolorbox}
\usepackage{fancyhdr}
\usepackage{epstopdf}
\usepackage{amsopn}
\usepackage{latexsym} 
\usepackage{algorithm}
\usepackage{algorithmic}
\usepackage{caption}
\usepackage{subcaption}
\usepackage{tabularx} 
\newtheorem{thm}{Theorem}[section]

\newtheorem{prop}{Proposition}[section]

\theoremstyle{definition}
\newtheorem{defn}{Definition}[section]
\newtheorem{lemma}{Lemma}[section]
\newtheorem{rmk}{Remark}[section]
\newtheorem{example}{Example}[section]

\newtheorem{asp}{Assumption}[subsection]

\DeclareMathOperator{\oraclerisk}{R_{\lambda}^{\textsf{OR}}}
\DeclareMathOperator{\prisk}{R_{\lambda}}
\DeclareMathOperator{\oraclerisktilde}{\tilde{R}_{\lambda}^{\textsf{OR}}}
\DeclareMathOperator{\prisktilde}{\tilde{R}_{\lambda}}

\newcommand{\specialcell}[2][c]{%
\begin{tabular}[#1]{@{}c@{}}#2\end{tabular}} 
\allowdisplaybreaks
\newcommand{\kb}[1]{{\color{magenta}\bf[KB: #1]}}
\usepackage{xcolor} 
\usepackage{marginnote}

\newcommand{\todol}[2][]{{%
 \let\marginpar\marginnote
 \reversemarginpar
 \renewcommand{\baselinestretch}{0.8}%
 \todo[#1]{#2}}}

\allowdisplaybreaks

\begin{document}

\title{\bf Meta-Learning with Generalized Ridge Regression:\\ High-dimensional Asymptotics, Optimality and Hyper-covariance Estimation}
%\runtitle{Meta-Learning with Generalized Ridge Regression}

\author[1]{Yanhao Jin}
\author[1]{Krishnakumar Balasubramanian}
\author[1,2]{Debashis Paul}
\affil[1]{Department of Statistics, University of California, Davis.}
\affil[2]{Applied Statistics Unit,
Indian Statistical Institute, Kolkata.}
\affil[1]{\texttt{\{yahjin,kbala,depaul\}}@ucdavis.edu}
\date{}

\maketitle

\begin{abstract}
Meta-learning involves training models on a variety of \emph{training tasks} in a way that enables them to generalize well on new, unseen \emph{test tasks}. In this work, we consider meta-learning within the framework of high-dimensional multivariate random-effects linear models and study generalized ridge-regression based predictions. The statistical intuition of using generalized ridge regression in this setting is that the covariance structure of the random regression coefficients could be leveraged to make better predictions on new tasks. Accordingly, we first characterize the precise asymptotic behavior of the predictive risk for a new \emph{test task} when the data dimension grows proportionally to the number of samples per task. We next show that this predictive risk is optimal when the weight matrix in generalized ridge regression is chosen to be the inverse of the covariance matrix of random coefficients. Finally, we propose and analyze an estimator of the inverse covariance matrix of random regression coefficients based on data from the \emph{training tasks}. As opposed to intractable MLE-type estimators, the proposed estimators could be computed efficiently as they could be obtained by solving (global) geodesically-convex optimization problems. Our analysis and methodology use tools from random matrix theory and Riemannian optimization. Simulation results demonstrate the improved generalization performance of the proposed method on new unseen \emph{test tasks} within the considered framework.  
\end{abstract}

%\tableofcontents
\section{Introduction} 

Classical statistical machine learning involves models that are trained on a specific task using a given dataset, and their predictive performance is evaluated on the same task. In contrast, meta-learning \citep{baxter2000model} aims to learn models that generalize well in new, unseen but related tasks. This is facilitated by training a model on a distribution over tasks so that it can efficiently adapt to new tasks with limited data. In this sense, meta-learning can be viewed as a form of higher-level learning that leverages knowledge gained from multiple \emph{training tasks} to improve performance on new \emph{test task}. In this paper, we analyze meta-learning under a natural multivariate random effects model to model the relationship between different tasks. When dealing with multiple tasks in meta-learning, each task can be considered a random effect. The random effects model may then help in modeling the variability between tasks, allowing development of efficient meta-learning algorithms that leverage this shared information across task to obtain better performance on new tasks.

Specifically, we consider the \emph{training tasks} from the following linear model, 
\begin{align*}
y^{(\ell)}=X^{(\ell)}\bar{\beta}^{(\ell)}+\varepsilon^{(\ell)}, \quad\text{for}\quad \ell=1,\ldots, L,
\end{align*}
where $X^{(\ell)}$ is the $n_{\ell}\times p$ data matrix and $y^{(\ell)}$ is $n_{\ell}$-dimensional vector, and  $n_{\ell}$ represents the number of observations in each task. The rows denoted by $x_{j}^{(\ell)}$ are random samples with covariance matrix $\Sigma$ and they are independent across $j$. Besides, the noise $\varepsilon^{(\ell)}$ corresponding to each task is a random vector which has mean zero and covariance matrix $\sigma^{2}I_{n_{\ell}}$. In order to encode task similarity, we consider a multivariate random-effects model where it is assumed that the true coefficients $\bar{\beta}^{(1)},\dots,\bar{\beta}^{(L)}$ are sampled from a common distribution with $\mathbb{E}[\bar{\beta}^{(j)}]=0$ and $\operatorname{Var}\big(\bar{\beta}^{(j)}\big)=p^{-1}\Omega$. The assumption that the regression coefficients are centered is made mostly for convenience. If they are uncentered (but share a common hyper-expectation), it is possible to learn the common hyper-expectation relatively easily based on techniques introduced later in later Section~\ref{MTL_cov_est_section}. The parameter $\Omega$ is the \emph{hyper-covariance matrix} encoding the similarity among the different tasks. Note in particular that we do not make any parametric assumptions on the task distribution, as is commonly made in the literature on random effects model. A similar model was considered by~\cite{balasubramanian2013high} in the context of multi-task learning. In contrast to that work, here we consider the meta-learning problem where the objective is to do well on unseen tasks. Specifically, given a new \emph{test task}, with index ${L+1}$, and $n_{L+1}$ observations from model, $y^{(L+1)}=X^{(L+1)}\bar{\beta}^{(L+1)}+\varepsilon^{(L+1)}$, where the rows $x_{j}^{(L+1)}$ of data matrix $X^{(L+1)}$ is random sample with covariance $\Sigma^{(L+1)}$ and independent across $j$, and $\bar{\beta}^{(L+1)}$ is a sample from a distribution with the same shared hyper-covariance matrix $\Omega$ as that of the training tasks, our goal is do well in terms of predictive performance on this new \emph{test task}. 

To accomplish this goal, we consider prediction using generalized ridge regression estimators~\citep{strawderman1978minimax,casella1980minimax} of the form
\begin{equation}\label{ridge_type_estimator}
    \beta_{\lambda}^{(\ell)}(A)=\big(X^{(\ell)\top}X^{(\ell)}+n_{\ell}\lambda A^{-1}\big)^{-1}X^{(\ell)\top}y^{(\ell)},\quad\text{for}\quad \ell=1,\ldots, L+1,
\end{equation}
where $A$ is a given positive definite matrix. Note that the above estimator is the solution of the regularized regression problem $$\tilde{\beta}_{\lambda}^{(\ell)}=\arg \min _{\beta \in \mathbb{R}^p}\Big\{\frac{1}{n_{\ell}}\|y^{(\ell)}-X^{(\ell)} \beta\|_2^2+\lambda\beta^{(\ell)^\top}A^{-1}\beta^{(\ell)}\Big\}.$$
In particular, when $A=\Omega$, the true common hyper-covariance matrix of the regression coefficients,
%task distribution, 
we denote the corresponding distinguished estimator as the oracle estimator given by
\begin{equation}\label{oracle_estimator_generalized_ridge}
    \tilde{\beta}_{\lambda}^{(\ell)}\coloneqq \beta_{\lambda}^{(\ell)}(\Omega),
\end{equation}
However, in practice, the hyper-covariance matrix 
%structure for the random feature 
$\Omega$ is unknown. One natural idea is to estimate $\Omega$ by $\hat{\Omega}$ based on previous tasks. Then the true coefficient in $\ell$-th task could be estimated by  
\begin{align}
    \hat{\beta}_{\lambda}^{(\ell)}\coloneqq \beta_{\lambda}^{(\ell)}(\hat{\Omega}).\label{estimator_generalized_ridge}
\end{align}
Under the stated model, we study the generalization performance of the above generalized ridge regression based predictors on a new task.

The statistical intuition of using generalized ridge regression under the meta-learning framework is that the hyper-covariance structure of the random regression coefficients could be leveraged to make better predictions on new tasks. To elaborate, under the meta-learning framework, we can construct an estimator $\hat \Omega$ of the shared hyperparameter $\Omega$ from the $L$ \emph{training tasks}, which in turn could be used in the form of the estimator in~\eqref{ridge_type_estimator} in the context of prediction in the \emph{test task} $L+1$. If the estimator $\hat{\Omega}$ is accurate in some appropriate sense, then this procedure should help in obtaining predictive accuracy in the new task $\bar{\beta}^{(L+1)}$. We provide a rigorous justification to the above intuition in this work. We do so by analyzing the prediction performance of the above approach on the new $(L+1)$-th \emph{test task}. For any matrix $A\in\mathbb{R}^{p}$, the predictive risk using $\beta_{\lambda}^{(L+1)}(A)$ is defined by
\begin{align*}
\prisk(A\mid X^{(L+1)})\coloneqq\mathbb{E}\Big[(x_{\text{new}}^{(L+1)\top}\beta_{\lambda}^{(L+1)}(A)-y_{\text{new}}^{(L+1)})^{2}\mid X^{(L+1)}\Big], %\label{convention_prisk}
\end{align*} 
and in particular, the predictive risk using oracle ridge estimator $\beta_{\lambda}^{(L+1)}(\Omega)$ is defined as
\begin{align}
    \oraclerisk(\Omega\mid X^{(L+1)})\coloneqq \mathbb{E}\Big[(x_{\text{new}}^{(L+1)\top}\beta_{\lambda}^{(L+1)}(\Omega)-y_{\text{new}}^{(L+1)})^{2}\mid X^{(L+1)}\Big],
\end{align}
where $(x_{\text{new}}^{(L+1)},y_{\text{new}}^{(L+1)})$ is an independent test example from the same distribution as the training data. %in new task. 
Here, $\prisk(A\mid X^{(L+1)})$ could be treated as a function of positive definite matrix $A$. In Section~\ref{MTL_risk_section}, we derive the high-dimensional limit of predictive risk, $\prisk(A\mid X^{(L+1)})$, when $p,n_{L+1} \to \infty$ such that $p/n_{L+1} \to \gamma_{L+1} \in (1,\infty)$ using tools from random matrix theory; see Theorem~\ref{asymptotic_behavior_of_predictive_risk} and Theorem~\ref{thm_consistency_L}. In Section~\ref{sec:ood}, we also briefly discuss the consequence for out-of-distribution prediction risk, i.e., when the test task distribution is different from the training tasks' distribution. In Section \ref{Advantage_section}, we then show using tools from Riemannian analysis that the function $\prisk(A\mid X^{(L+1)})$ on the space of positive definite matrix is minimized when %$\prisk(A\mid X^{(L+1)})$ at 
$A=\Omega$, under mild regularity conditions (see Theorem~\ref{statistical_advantage_of_Omega}). These two results provide a strong justification of the proposed approach for meta-learning under the assumed random effects framework. 

Motivated by this, we next consider the problem of hyper-covariance estimation, i.e., estimating $\Omega$ from the \emph{training tasks}. Traditionally, in the random effects model literature, maximum likelihood based approaches are considered under parametric distributional assumptions on the task and noise distributions. However, such approaches %turn out to 
lead to non-convex optimization problems which are computationally harder to solve efficiently. In contrast, we extend the approach initiated in~\cite{balasubramanian2013high} and propose a novel method-of-moments based approach for estimating the hyper-covariance matrix. Our estimators are constructed as solutions to geodesically convex optimization problems which can be efficiently solved using Riemannian optimization techniques.

In Section \ref{Mom_estimator_Omega_section}, we prove consistency and rates of convergence results of the proposed estimators under sub-Gaussian assumptions on the random coefficient $\bar{\beta}^{(\ell)}$ and noise $\varepsilon^{(\ell)}$, as the number of training tasks $L$ grows. In particular, we show  consistency as $n_\ell,L,p\rightarrow\infty$ under appropriate scaling, as long as there is a non-vanishing  fraction of the tasks for which $p/n_{\ell} \leq 1-\delta$, for some $\delta >0$. This shows the benefit of meta-learning, i.e., as long as there is a small fraction of the tasks with adequate data, it is possible to estimate the hyper-parameter $\Omega$ consistently. The associated scaling however requires that $p^2/L \to 0$ which is not practical when $p$ is very large. To overcome this limitation, it is necessary to enforce further structural assumptions on the hyper-covariance $\Omega$.  In Section \ref{MTL_sparse_section}, we assume that the true hyper-covariance $\Omega$ satisfies certain sparsity conditions (indexed by a sparsity parameter $s$) and study $L_1$ regularized approaches for estimation. We prove consistency and rates of convergence of the resulting estimators under improved scalings on $s, p$ and $L$ under various assumptions; see Table~\ref{tab:maintable} for a full overview of the developed results. 

\subsection{Related Work}

Our work lies at the intersection of random effects models, multitask and meta-learning. Traditional approaches for estimation in random effects model include (restricted) maximum likelihood to estimate variance components in the linear mixed models literature (e.g. \cite{thompson1962problem}, \cite{corbeil1976restricted} and \cite{harville1977maximum}). However, these methods are mainly studied in low-dimensional settings. High-dimensional analysis of a similar multivariate random effects model was considered by~\cite{sun2021towards} and~\cite{huang2022provable}. However, the question of precise asymptotics and optimality are not studied in these works. Our work is mainly motivated by the work by~\cite{balasubramanian2013high}, where a similar random effects model was analyzed in the context of multi-task learning, mainly from a methodological perspective. 

Subspace-based meta-learning, where the multiple tasks share a common set of low-dimensional features, was used and analyzed in the works by~\cite{tripuraneni2021provable,du2021fewshot} and~\cite{duan2023adaptive}. A mixed linear regression models for meta-learning, where the prior over the tasks corresponds to a discrete distribution, was analyzed by~\cite{kong2020meta}. General statistical learning theory results for multi-task and meta-learning are examined by many authors. While it is not possible to cover the extensive literature on this topic, here, we list a few recent representative works by~\cite{argyriou2008convex,lounici2009taking,maurer2016benefit,amit2018meta,finn2019online,khodak2019adaptive, lucas2021theoretical,farid2021generalization,chen2021generalization,chen2022bayesian} and~\cite{li2023provable}. The above works are mainly focused on non-asymptotic \emph{bounds} and do not consider the question of optimality and deriving precise asymptotics, which is the main focus of our work.

%Instead of treating this as random effect model, earlier work such as \cite{lounici2009taking}, \cite{zhang2010probabilistic}, \cite{lozano2012multi} and \cite{balasubramanian2013high} 

 %When the number of nonzero entries of covariance matrix $\Omega$ is bounded by some sparse parameter $s$, one could penalize the loss function with a lasso penalty on the entries of the covariance matrix. The idea of proving consistency of $\hat{\Omega}$ mainly comes from \cite{rothman2008sparse}. \cite{rothman2008sparse} proposes a method for constructing a sparse estimator for the inverse covariance matrix in high-dimensional settings by minimizing penalized normal likelihood function. And this work also proved that a correlation based version of the estimator shares better rates in operator norm. For leveraging these ideas, we refer to \cite{rothman2008sparse} and \cite{bien2011sparse}. 
%\cite{tulino2004random}.

Our methodology is based on generalized ridge regression. The methodological idea behind the formulation \eqref{ridge_type_estimator}, in the context of single-task ridge regression, was studied by \cite{strawderman1978minimax}, \cite{casella1980minimax} and \cite{maruyama2005new} under the setting of $n\gg p$. High-dimensional asymptotics of ridge regression (i.e., \eqref{ridge_type_estimator} with $A=I$) in the single-task setting has been studied extensively in the last decade. \cite{karoui2013asymptotic} studied the asymptotic behavior of ridge estimators under the scenario $\Sigma=\Omega=I_{p}$ when $p/n$ tends to a finite non-zero limit. \cite{dicker2016ridge} studied asymptotic minimax problems for estimating a regression parameter over growing dimension $p$ such that $p/n\rightarrow\rho$ when the sample $x_{i}$ are i.i.d. Gaussian random vector.  However, these results only focus on the estimation error. The behavior of prediction error of single-task ridge regression has been studied, for example, in \cite{hsu2012random}, \cite{dobriban2018high}, \cite{wu2020optimal} and~\cite{richards2021asymptotics}. In particular, \cite{hsu2012random} studied finite-sample concentration inequalities on the out-of-sample prediction error of random-design ridge regression. \cite{dobriban2018high} later provided an explicit formula of prediction error when $\Omega=I_{p}$ under high-dimensional asymptotics $p,n\rightarrow\infty$ and $p/n\rightarrow \gamma$. This result has been extended by \cite{richards2021asymptotics}, and \cite{wu2020optimal}. \cite{richards2021asymptotics} studied the asymptotic behavior of prediction error when $\Omega$ could be expressed by some source function of $\Sigma$. \cite{wu2020optimal} extended previous works on the asymptotic behavior of prediction error of generalized ridge regression when arbitrary weight matrix is used in ridge estimator $\hat{\beta}$. Explicit formula of limiting risk is provided in \cite{wu2020optimal} under several different choices of weight matrix, and an expression of optimal regularization parameter $\lambda$ based on the limiting prediction error is also provided. None of the above works focuses on the meta-learning setup that we focus on, and more importantly, none of the above works focuses on estimating the shared hyper-covariance matrix.

Finally, as mentioned above, we use tools from Random matrix theory (RMT) and Riemannian geometry/optimization for our methodology and analysis. RMT has been widely used for high-dimensional analysis of statistical problems. We refer to \cite{yao2015sample} and \cite{couillet2022random} for the fundamental of RMT and high-dimensional statistics. For our analysis, we specifically use the work by \cite{ledoit2011eigenvectors}. We also refer to the books by~\cite{tu2011manifolds} and~\cite{boumal2020introduction} for an overview of Riemannian manifolds and optimization over Riemannian manifolds respectively.

\subsection{Notation}
Here, we list several commonly used notations in the rest of the paper.
\begin{itemize}
    \item The parameter $\lambda$ refers to the ridge regularization parameter in \eqref{ridge_type_estimator}, and $\tilde{\lambda}$ to be regularized parameter in other proposed methods.
    \item  For a vector $v\in\mathbb{R}^{p}$, $\|v\|_{k}$ denotes the $L_{k}$ norm of the vector. 
    \item For a matrix $A\in\mathbb{R}^{p\times p}$, denote $\lambda_{i}(A)$ to be the $i$-th eigenvalues, $\lambda_{\min}(A)$ and $\lambda_{\max}(A)$ to be the smallest and largest eigenvalue of matrix $A$ respectively. Furthermore, $\|A\|_{F}$ and $\|A\|$ denotes the Frobenius and operator norms respectively.
    \item For any matrix $M=\big[m_{i j}\big]$, we write $M^{+}=\operatorname{diag}(M)$ for a diagonal matrix with the same diagonal as $M$, and $M^{-}=M-M^{+}$.
    \item We also write $|\cdot|_1$ for the $l_1$ norm of a vector or a (vectorized) matrix, i.e., for a matrix $|M|_1=\sum_{i, j}|m_{i j}|$.
    \item For a random variable $Z$ we denote $\|Z\|_{\psi_1}$ and $\|Z\|_{\psi_{2}}$ to be the $\psi_1$ and $\psi_2$ norm whose precise definition is given in \eqref{psi_norm_def}.
    \item For a sequence of random variable $X_n$ and $X$, $X_n\stackrel{p}{\rightarrow}X$, $X_n\Rightarrow X$ and $X_n\stackrel{a.s.}{\rightarrow}X$ denotes convergence in probability, convergence in distribution and almost sure convergence respectively. 
    %, i.e. for all $\varepsilon>0$, $\lim _{n \rightarrow \infty} \mathbb{P}\left(\left|X_n-X\right|>\varepsilon\right)=0$. Denote  to be convergence in distribution, i.e. $\lim _{n \rightarrow \infty} F_n(x)=F(x)$ for every number $x \in \mathbb{R}$ at which $F$ is continuous where $F_n$ and $F$ are the CDF of $X_n$ and $X$. Denote  to be convergence almost surely, i.e. $\mathbb{P}\left(\lim _{n \rightarrow \infty} X_n=X\right)=1$.
    \item We say a random variable $X_n=\mathcal{O}_{P}(1)$ as $n \rightarrow \infty$ means that $\sup _n \mathbb{P}\left(\left|X_n\right|>K\right) \rightarrow 0$ as $K \rightarrow \infty$. And $X_n=O_p\left(b_n\right)$ means that $X_n / b_n=O_p(1)$ as $n \rightarrow \infty$.
    \item $\mathbb{S}_{p}^{+}$ and $\mathbb{S}_{p}$  denotes the space of $p\times p$ positive definite matrices and symmetric matrices respectively.
    \item The sample covariance matrix for $\ell$-th task is denoted by $\hat{\Sigma}^{(\ell)}=\frac{1}{n_{\ell}}X^{(\ell)\top}X^{(\ell)}$.
    \item $\oraclerisk(\Omega\mid X^{(L+1)})$ denotes the risk function of our generalized ridge regression using oracle estimators $\beta_{\lambda}^{(L+1)}(\Omega)$ and $\prisk(\hat{\Omega}\mid X^{(L+1)})$ to be the true risk using $\beta_{\lambda}^{(L+1)}(\hat{\Omega})$ for any estimator $\hat{\Omega}$ in the new task. Besides, we denote $r(\lambda,\gamma_{L+1})$ to be the limiting risk for generalized ridge regression investigated in Section \ref{MTL_risk_section}.
    \item Throughout the paper, we use $C$ to represent some absolute constant which does not depend on important problem parameters, like the dimension $p$, sample size $n_{\ell}$ and number of tasks $L$. Here, $C$ may chance from instance to instance. 
\end{itemize}
\section{Characterizing the predictive performance on a new task} 
\subsection{Assumptions} 
We start by introducing the assumptions we require for our analysis.  
\begin{asp}[Data generation]\label{asp1} 
For the $\ell^{\text{th}}$ task (for $\ell=1,\dots,L,L+1$), the data matrix $X^{(\ell)} \in \mathbb{R}^{n_{\ell} \times p}$ is generated as 
$$X^{(\ell)}=Z^{(\ell)} \Sigma^{(\ell)1 / 2},$$ 
for an $n_{\ell} \times p$ matrix $Z^{(\ell)}$ with i.i.d. entries satisfying $\mathbb{E}\big[Z_{i j}^{(\ell)}\big]=0$, $\operatorname{Var}\big[Z_{i j}^{(\ell)}\big]=1$ and $\mathbb{E}\big[(Z_{ij}^{(\ell)}\big)^{12}]\leq c^{(\ell)}_{\text{m}}$ for any $p$. $\Sigma^{(\ell)}$ is a deterministic $p \times p$ positive definite covariance matrix such that $\|\Sigma^{(\ell)}\|\leq \bar{c}^{(\ell)}$ for any $p$. Furthermore, there are %uniform 
constants $c_\text{m}$ and $\bar{c}_{\text{op}}$ such that $\sup_{\ell \in \mathbb{N}}c^{(\ell)}_{\text{m}} =c_\text{m} <\infty$ and $\sup_{\ell \in \mathbb{N} }\bar{c}^{(\ell)} =\bar{c}_{\text{op}}<\infty$.
\end{asp}

The above conditions on the data matrix corresponding to the individual tasks are rather mild, and are made frequently in the literature on random matrix theory based analysis of statistical models~\citep{dobriban2018high}. We emphasize in particular that the covariance matrices of the data across the task are allowed to change arbitrarily, allowing for a flexible statistical modeling. The bounded $12^{th}$ moment condition is a technical condition which could be further relaxed using more sophisticated random matrix theory tools. However, we do not pursue such an extension in this work. On a more technical note, the condition $\|\Sigma^{(L+1)}\|\leq \bar{c}^{(\ell)}$ ensures the existence of the limits for terms $(\mathsf{I})$ and $(\mathsf{II})$ appearing in Theorem \ref{thm_predictive_risk} that follows. It also allows us to express these limits in terms of derivatives of the limit of term $(\mathsf{III})$. Moreover, in the proof of the theorem presented in Section \ref{MTL_cov_est_section}, explicit expressions of the quantities dependent on the constants $\bar{c}^{(\ell)}$ are not stated, as we mainly focus on their asymptotic orders as $p,n_{\ell}\rightarrow \infty$.

\begin{asp}[Random regression coefficients]\label{asp2}
The true coefficients $\bar{\beta}^{(\ell)}$ for the training tasks are i.i.d. random vectors from a common distribution with mean $\mathbb{E}\bar{\beta}=0$ and hyper-covariance $\mathbb{E}\bar{\beta}\bar{\beta}^{\top}=\frac{1}{p}\Omega$.
\end{asp}
The above assumption models the task similarity by positing that the tasks share a common distribution parametrized by the hyper-covariance matrix $\Omega \in \mathbb{R}^{p\times p}$. 
Such a learning setup is called a multivariate random effects model in classical statistics~\citep{jiang2007linear}, and forms a special case of the meta-learning setup~\citep{baxter2000model}. Compared to classical random effects models, we emphasize here that we do not make any parametric assumption on the task distribution. In the context of estimating the hyper-covariance matrix, we later enforce additional sub-Gaussian type conditions on the task distribution to obtain high-probability error bounds.

\begin{asp}[High dimensional asymptotics]\label{asp3} 
The predictor dimension $p$ and the number of samples in each task $n_{\ell}$ satisfy the following condition:
\begin{equation*}
    \frac{p}{n_{\ell}}\rightarrow \gamma_{\ell}, \quad\text{for}\quad \ell=1,\dots,L,L+1,
\end{equation*}
as $p$ and $n_{\ell}$ go to infinity. Besides, the limiting ratios satisfy $\gamma_{\ell} \in (1,\infty)$, for $\ell=1, \ldots, L$. 
\end{asp}
The above assumption characterizes the high-dimensional setup that we are interested in. When $\gamma_{\ell} \in (0,1]$, i.e., the proportional but low-dimensional setting, there is a rich literature on understanding covariance matrix estimation (see, for example,~\cite{paul2007asymptotics,tao2012random,pillai2014universality}) which could also be leveraged in the context of ridge regression analysis. 
%Motivated by modern large-dimensional data sets, in this work, our work is mainly in the above-mentioned high-dimensional setup. 
Furthermore, we emphasize that the above condition will be relaxed in a delicate manner in the context of estimating the hyper-covariance matrix in the later sections. In particular, we require that $\gamma_{\ell} \in (0,1)$ for a proportion of the \emph{training tasks} to ensure consistency in estimating the hyper-covariance matrix. 

Before we proceed further, we require the following additional definitions. 

\begin{defn}[Empirical and limiting spectral distribution (ESD)]\label{ESD_Definition}
For any symmetric matrix $A$, the empirical spectral distribution (ESD) function of $A$ is the empirical distribution of its eigenvalues: 
$$
F_A(x)=\frac{1}{p} \sum_{i=1}^p \mathrm{1}\big(\lambda_i(A) \leq x\big).
$$  
Given a sequence of matrices $A_p\in\mathbb{R}^{p\times p}$, with corresponding empirical spectral distributions $F_{A_p}$, if $\{F_{A_p}\}$ converges weakly (as $p$ tends to infinity), either almost surely or in probability, to some probability distribution, then the latter distribution is called the Limiting Spectral Distribution (LSD) of the sequence $\left\{A_p\right\}$.
\end{defn}

Now, consider a \emph{test task} $(L+1)$ with coefficient drawn from \emph{any} zero-mean distribution with the same hyper-covariance matrix $\Omega$ as the training tasks. Note in particular that the testing task distribution is fully characterized by the covariance matrix for the linear models we consider. Under this setting, the marginal distribution of $y^{(L+1)}\mid X^{(L+1)}$ will be a centered distribution whose (scaled) variance given by $$
\frac{1}{n_{L+1}}\mathrm{Var}\big(y^{(L+1)}|X^{(L+1)}\big)= \frac{1}{n_{L+1}} X^{(L+1)}\Omega X^{(L+1)\top}+\frac{\sigma^2}{n_{L+1}} I.
$$ 
The term $X^{(L+1)}\Omega X^{(L+1)\top}/n_{L+1}$ plays an important role in our analysis. The expectation 
of this matrix, and those of its spectral moments, depend on $\Sigma^{(L+1)\frac{1}{2}} \Omega \Sigma^{(L+1)\frac{1}{2}}$ where $\Sigma^{(L+1)\frac{1}{2}}$ is a positive semidefinite square root of $\Sigma^{(L+1)}$. The predictive risk for the $(L+1)$-th task, depends on the spectrum of the matrix $X^{(L+1)}\Omega X^{(L+1)\top}/n_{L+1}$. Hence, the asymptotic behavior of predictive risk depends on the limiting spectral distribution of $\Sigma^{(L+1)\frac{1}{2}} \Omega \Sigma^{(L+1)\frac{1}{2}}$, or equivalently, the LSD of
\begin{align}\label{eq:importantmatrix}
\Lambda^{(L+1)}\coloneqq\Omega^{\frac{1}{2}} \Sigma^{(L+1)} \Omega^{\frac{1}{2}}.
\end{align}
%$\Omega^{\frac{1}{2}} \Sigma^{(L+1)} \Omega^{\frac{1}{2}}$. 
In order to precisely characterize the asymptotic behavior of predictive risk, we also make the following assumption. 

\begin{asp}[Spectral structure]\label{asp4}
There is a limiting spectral distribution $H_{\Lambda^{(L+1)}}$ such that as $n_{L+1}$ and $p$ goes to infinity, the empirical spectral distribution of $\Lambda^{(L+1)}$ converges in distribution to the limiting spectral distribution $F_{\Lambda^{(L+1)}} \Rightarrow H_{\Lambda^{(L+1)}}$, and the support of $H_{\Lambda^{(L+1)}}$ is contained in a compact interval bounded away from $0$ and $\infty$.
\end{asp}

If $\Sigma^{(L+1)}$ and $\Omega$ share the same eigenvectors, i.e. they commute, then Assumption \ref{asp4} could be guaranteed if the empirical spectral distribution of two individual matrices converge in distribution to their limiting spectral distribution and the support of limiting spectral distribution is contained in a compact interval bounded away from $0$ and $\infty$. Assumption \ref{asp4} is more general and allows for matrices $\Sigma^{(L+1)}$ and $\Omega$ that do not necessarily commute. Below, we explicitly provide a numerical example to illustrate Assumption \ref{asp4}, which will also be used in our numerical experiments.

%If $\Sigma^{(L+1)}$ and $\Omega$ do not commute, following example illustrates that the  is still satisfied.
%\kb{To Yanhao: Can you generalize below example to general example where $\Sigma$ and $\Omega$ are both tri-diagonal matrices?}

\begin{example}
Consider the following choice of the matrices $\Omega$ and $\Sigma^{(L+1)}$, %\todo{can you have non-zero off-diagonal terms in $\Omega$? like the $b$ in $\Sigma$},
\begin{align}\label{eq:examplematrix} 
\Omega=\begin{bmatrix}
a & b & 0 & \dots & 0 \\
b & a & b & \dots & 0 \\
0 & b & a & \dots & 0 \\
\vdots &\ddots & \ddots & \ddots & \vdots\\
0 & \dots & b & a & b\\
0 & \dots & 0 & b & a
\end{bmatrix}\quad\quad 
\Sigma^{(L+1)}=\begin{bmatrix}
c &    &    &        &   \\
   & d &    &        &   \\
   &    & d &        &   \\
   &    &    & \ddots &   \\
   &    &    &        & d
\end{bmatrix} 
\end{align}
where $a,b,c,d>0$ and $a>b$. Then
\begin{align*} 
\Sigma^{(L+1)\frac{1}{2}}\Omega\Sigma^{(L+1)\frac{1}{2}}=\begin{bmatrix}
ac  & b\sqrt{cd} & 0   & \dots & 0 \\
b\sqrt{cd}  & ad & bd & \dots & 0 \\
0    & bd & ad & \dots & 0 \\
\vdots &\ddots & \ddots & \ddots & \vdots\\
0  & \dots & bd & ad & bd\\
0  & \dots & 0   & bd & ad
\end{bmatrix}.
\end{align*}

The eigenvalues of $\Omega$ are given by $\lambda_{i}=a+b\cos\frac{k\pi}{p+1}$. The ESD of $\Omega$ converges to the distribution of the random variable $a+b\cos U$ for $U\sim \mathrm{Uniform}(-\pi,\pi]$. The ESD of $\Sigma^{(L+1)}$ converges to the distribution of the degenerate probability distribution with probability 1 at $d$. We now characterize the limiting spectral distribution of $\Sigma^{(L+1)\frac{1}{2}}\Omega\Sigma^{(L+1)\frac{1}{2}}$. Note that 
\begin{align*}
\Omega^{\frac{1}{2}}\Sigma^{(L+1)}\Omega^{\frac{1}{2}}=\LaTeXunderbrace{\begin{bmatrix}
ad  & bd & 0   & \dots & 0 \\
bd  & ad & bd & \dots & 0 \\
0    & bd & ad & \dots & 0 \\
\vdots &\ddots & \ddots & \ddots & \vdots\\
0  & \dots & bd & ad & bd\\
0  & \dots & 0   & bd & ad
\end{bmatrix}}_{\mathsf{N}}+\LaTeXunderbrace{\begin{bmatrix}
a(c-d)  & b\sqrt{d}(\sqrt{c}-\sqrt{d})  & 0   & \dots & 0 \\
b\sqrt{d}(\sqrt{c}-\sqrt{d})   &   0 & 0    & \dots & 0 \\
0    &  0   &  0   & \dots & 0 \\
\vdots &\ddots & \ddots & \ddots & \vdots\\
0  & \dots & 0    &   0   &  0 \\
0  & \dots & 0   &    0  &   0 
\end{bmatrix}}_{\mathsf{R}}.
\end{align*}
Denote $\mu_1\geq\dots\geq\mu_{p}$ to be ordered eigenvalues of $\Sigma^{(L+1)\frac{1}{2}}\Omega\Sigma^{(L+1)\frac{1}{2}}$, $\nu_1\geq\dots\geq\nu_{p}$ to be ordered eigenvalues of $\mathsf{N}$ and $\rho_1\geq\dots\geq\rho_{p}$ to be ordered eigenvalues of $\mathsf{R}$. Then, one can calculate $\nu_{k}=ad+bd\cos\frac{k\pi}{p+1}$ for $k=1,\dots,p$ and 
\begin{align*}
\rho_{1}&=\frac{a(c-d) +(\sqrt{c}-\sqrt{d}) \sqrt{(\sqrt{c}+\sqrt{d})^2 a^2+4 b^2 d}}{2}>0,\\
\rho_{2}&=\dots=\rho_{p-1}=0,\\
\rho_{p}&=\frac{a(c-d) -(\sqrt{c}-\sqrt{d}) \sqrt{(\sqrt{c}+\sqrt{d})^2 a^2+4 b^2 d}}{2}<0.    
\end{align*}
By Weyl's inequality, one has $\nu_{k+1}=\nu_{k+1}+\rho_{p-1}\leq  \mu_k \leq \nu_{k-1}+\rho_2=\nu_{k-1}$ for $k=2,\dots,p-1$. Now, for any fixed $x\in\mathbb{R}$, then one has
$$\Big|\frac{1}{p} \sum_{i=1}^1 1\big\{\mu_i \leq x\big\}-\frac{1}{p} \sum_{i=1}^p 1\big\{v_i \leq x\big\}\Big|\leq \frac{2}{p}\rightarrow 0~~~\text{when }p \to \infty.$$
Therefore, the ESD of $\Sigma^{(L+1)\frac{1}{2}}\Omega\Sigma^{(L+1)\frac{1}{2}}$ and ESD of $\mathsf{N}$ will have the same limiting probability distribution. Since the support of limiting distribution of $\mathsf{N}$ is defined on the compact interval $[ad-bd,ad+bd]$, the limiting distribution of $\Sigma^{(L+1)\frac{1}{2}}\Omega\Sigma^{(L+1)\frac{1}{2}}$ will be the same. Finally, since $\Sigma^{(L+1)\frac{1}{2}}\Omega\Sigma^{(L+1)\frac{1}{2}}$ and $\Omega^{\frac{1}{2}}\Sigma^{(L+1)}\Omega^{\frac{1}{2}}$ share the same eigenvalues, the limiting distribution of $\Omega^{\frac{1}{2}}\Sigma^{(L+1)}\Omega^{\frac{1}{2}}$ will also be the same. By a simple computation, the form of $\Sigma^{(L+1)}$ could be generalized to $\operatorname{diag}\{c,\dots,c,d,\dots,d\}$ where the proportion of $c$ goes to zero as dimension $p\rightarrow\infty$. Furthermore, in this example, $\Sigma^{(L+1)}$ could also be allowed to be a block-diagonal matrix with less variability within each block. In this case too, a similar calculation holds, albeit being more tedious to carryout.  
\end{example}

%\textcolor{red}{use wide hat/check/tilde only for matrix Lamda, and normal hat/check/tilde for other matrix}
%\textcolor{red}{In other place, the notation $\Lambda$ should be changed into other symbol, e.g. used in eigenvalue decomposition from experiment section}

%\textcolor{red}{Yanhao, add details as required.}

\subsection{Precise high-dimensional asymptotics of the predictive risk}\label{MTL_risk_section}

We now investigate the predictive performance of the generalized ridge regression estimator on a new task in the high-dimensional setting when $p,n_{L+1}\rightarrow\infty$ such that ${p}/{n_{L+1}}\rightarrow\gamma_{L+1}$. Our result about the predictive risk of generalized ridge regression is stated in terms of the expected predictive risk on a new task ${L+1}$:, denoted as $\prisk(\hat{\Omega}\mid X^{(L+1)})$. We first characterize the asymptotic behavior of predictive risk $\oraclerisk(\Omega\mid X^{(L+1)})$ using oracle estimator $\beta_{\lambda}^{(L+1)}(\Omega)$ when $n_{L+1}$ and $p$ goes to infinity. The quantity $\oraclerisk(\Omega\mid X^{(L+1)})$ is the benchmark that we can compare with and is the optimal risk that we can achieve under mild assumptions; see Section~\ref{Advantage_section}. When using an arbitrary estimator $\hat{\Omega}$ of $\Omega$, the corresponding predictive risk is denoted as $\prisk(\hat{\Omega}\mid X^{(L+1)})$. Assuming the estimator $\hat\Omega$ is consistent in appropriate sense, we then show that the asymptotic behavior of $\prisk(\hat{\Omega}\mid X^{(L+1)})$ is the same as the 
asymptotic risk using true $\Omega$. 

We begin with the following result that provides explicit expressions for the predictive risk of generalized ridge regression on the new task using the oracle estimator \eqref{oracle_estimator_generalized_ridge} and estimator $\hat{\beta}^{(L+1)}$ from \eqref{estimator_generalized_ridge}. 
Before we present our results, we also introduce \emph{sample versions} of $\Lambda^{(L+1)}$ defined in~\eqref{eq:importantmatrix}: 
\begin{align} \label{eq:importantmatrixest}
\begin{aligned}
\widetilde{\Lambda}^{(L+1)}\coloneqq \Omega^{\frac{1}{2}}\hat{\Sigma}^{(L+1)}\Omega^{\frac{1}{2}},\\    
\widecheck{\Lambda}^{(L+1)}\coloneqq\hat{\Omega}^{\frac{1}{2}}\Sigma^{(L+1)}\hat{\Omega}^{\frac{1}{2}},\\
\widehat{\Lambda}^{(L+1)}\coloneqq\hat{\Omega}^{\frac{1}{2}}\hat{\Sigma}^{(L+1)}\hat{\Omega}^{\frac{1}{2}}.
%\hat{\Omega}^{\frac{1}{2}} \Sigma^{(L+1)} \hat{\Omega}^{\frac{1}{2}}\\
%\Omega^{\frac{1}{2}} \hat{\Sigma}^{(L+1)} \Omega^{\frac{1}{2}}
\end{aligned}
\end{align}
Recall that $\Lambda^{(L+1)}$ is composed to both the data covariance ($\Sigma$) and the hyper-covariance ($\Omega$). The above matrices are essentially the sample version of $\Lambda^{(L+1)}$ when $\Omega$ is known but $\Sigma^{(L+1)}$ is estimated,  $\Sigma^{(L+1)}$ is known and $\Omega$ is estimated, and both $\Sigma^{(L+1)}$ and $\Omega$ are estimated, respectively.

\begin{thm}\label{thm_predictive_risk}
The predictive risk of generalized ridge regression on the new task indexed by ${L+1}$, using oracle estimator $\tilde{\beta}_{\lambda}^{(L+1)}$ and using estimator $\hat{\beta}_{\lambda}^{(L+1)}$ from \eqref{estimator_generalized_ridge}, 
are given by  
\begin{align*}%\label{oraclerisk}
\begin{aligned}
&\oraclerisk\big(\Omega\mid X^{(L+1)}\big)=\sigma^2+\LaTeXunderbrace{\frac{\lambda^2}{p}\operatorname{tr}\big(\Lambda^{(L+1)}\big(\widetilde{\Lambda}^{(L+1)}+\lambda I\big)^{-2}\big)}_{(\mathsf{I})}   \\
     &\quad\LaTeXunderbrace{-\frac{\lambda\sigma^2}{n_{L+1}}  \operatorname{tr}\big(\Lambda^{(L+1)}\big(\widetilde{\Lambda}^{(L+1)}+\lambda I\big)^{-2}\big)}_{(\mathsf{II})}+\LaTeXunderbrace{\frac{\sigma^2}{n_{L+1}}\operatorname{tr}\big(\Lambda^{(L+1)}\big(\widetilde{\Lambda}^{(L+1)}+\lambda I\big)^{-1}\big)}_{(\mathsf{III})},
\end{aligned}
\end{align*} 
and
\begin{align}\label{pred_risk}
\begin{aligned} 
&\prisk\big(\hat{\Omega}\mid X^{(L+1)}\big) =\sigma^2+\LaTeXunderbrace{\frac{\lambda^2}{p} \operatorname{tr}\big((\hat{\Omega}^{-\frac{1}{2}}\Omega\hat{\Omega}^{-\frac{1}{2}})\widecheck{\Lambda}^{(L+1)}(\widehat{\Lambda}^{(L+1)}+\lambda I)^{-2}\big)}_{(\mathsf{I}')}\\
    &\quad\LaTeXunderbrace{-\frac{\lambda\sigma^2}{n_{L+1}} \operatorname{tr}\big(\widecheck{\Lambda}^{(L+1)}(\widehat{\Lambda}^{(L+1)}+\lambda I)^{-2}\big)}_{(\mathsf{II}')}+\LaTeXunderbrace{\frac{\sigma^2}{n_{L+1}} \operatorname{tr}\big(\widecheck{\Lambda}^{(L+1)}(\widehat{\Lambda}^{(L+1)}+\lambda I)^{-1}\big)}_{(\mathsf{III}')} \\
\end{aligned}    
\end{align}
respectively.
\end{thm}  
\iffalse
Recall that in Theorem  \ref{thm_predictive_risk}, we calculate the expression of the $\oraclerisk(\Omega\mid X^{(L+1)})$ and $\prisk(\hat{\Omega}\mid X^{(L+1)})$. The expression of $\oraclerisk(\Omega\mid X^{(L+1)})$ mainly contain three parts, named as $(I)$, $(II)$ and $(III)$ whose expression are given below
\begin{align*}
    (I) &= \frac{\lambda^2}{p} \operatorname{tr}\Big(\Lambda^{(L+1)}\big(\widetilde{\Lambda}^{(L+1)}+\lambda I\big)^{-2} \Big)\\
    (II) &=-\frac{\lambda\sigma^2}{n_{L+1}}\operatorname{tr}\Big(\Lambda^{(L+1)}\big(\widetilde{\Lambda}^{(L+1)}+\lambda I\big)^{-2} \Big)\\
    (III) &=\frac{\sigma^2}{n_{L+1}} \operatorname{tr}\Big(\Lambda^{(L+1)}\big(\widetilde{\Lambda}^{(L+1)}+\lambda I\big)^{-1} \Big)
\end{align*}
and $\prisk(\hat{\Omega}\mid X^{(L+1)})$ could be decomposed into three terms in the similar way, named as $(I')$, $(II')$ and $(III')$ whose expression are given below:
\begin{align} 
\begin{aligned}
    (I')&=\frac{\lambda^2}{p} \operatorname{tr}\big(\Omega \hat{\Omega}^{-\frac{1}{2}}\big(\widehat{\Lambda}^{(L+1)}+\lambda I\big)^{-1}\widecheck{\Lambda}^{(L+1)}\big(\widehat{\Lambda}^{(L+1)}+\lambda I\big)^{-1} \hat{\Omega}^{-\frac{1}{2}}\big)\\
    (II')&=-\frac{\lambda\sigma^2}{n_{L+1}} \operatorname{tr}\big(\widecheck{\Lambda}^{(L+1)}\big(\widehat{\Lambda}^{(L+1)}+\lambda I\big)^{-2}\big)\\
    (III')&=\frac{\sigma^2}{n_{L+1}} \operatorname{tr}\big(\widecheck{\Lambda}^{(L+1)}\big(\widehat{\Lambda}^{(L+1)}+\lambda I\big)^{-1}\big) 
\end{aligned}
\end{align} 
\fi
From the expression above, we see that term $(\mathsf{I}')$ consist the bias part which is independent of the noise level $\sigma^2$. And terms $(\mathsf{II}')$, $(\mathsf{III}')$ consist of the variance part involving $\sigma^2$ but do not depend on the true $\Omega$. As the estimator $\hat{\Omega}$ only depends on the observation and responses from first $L$, we have $\hat{\Omega}$ is independent to $X^{(L+1)}$.

\begin{rmk}
The oracle risk $\oraclerisk(\Omega\mid X^{(L+1)})$ is still a random quantity because it depends on the samples. This predictive predictive risk using oracle estimator depends on the spectrum of $\Omega^{\frac{1}{2}}\Sigma^{(L+1)}\Omega^{\frac{1}{2}}$. If we assume some stabilizing behavior of the spectrum that can be stated in terms of the limiting spectral distribution of this matrix, this random quantity $\prisk(\Omega\mid X^{(L+1)})$ converges a.s. to some deterministic function of $\lambda,\gamma_{L+1},\sigma^{2}$ and the LSD of $\Omega^{\frac{1}{2}}\Sigma^{(L+1)}\Omega^{\frac{1}{2}}$.
\end{rmk}

Based on the explicit expressions obtained above, we now examine the asymptotic behavior of predictive risk function in~\eqref{pred_risk}, as $p,n_{L+1}\rightarrow\infty$ such that $p/n_{L+1}\rightarrow\gamma_{L+1}$. We introduce a few basic random matrix theory tools that are required to present our subsequent results. 

\begin{defn}[Stieltjes transform]\label{RMT4MLdef3}For a real probability measure $\mu$ with support $\operatorname{supp}(\mu)$, the Stieltjes transform $s(z)$ is defined, for all $z \in \mathbb{C} \backslash \operatorname{supp}(\mu)$, as
\begin{align*}%\label{notationequ1}
s(z) \equiv \int \frac{1}{t-z} \mu(d t).
\end{align*}
\end{defn}

\begin{thm}[\cite{marvcenko1967distribution}]
Let 
$$
s_{p,L+1}(z)=\frac{1}{p}\sum_{i=1}^{p}\big(\lambda_i(\widetilde{\Lambda}^{(L+1)})-z\big)^{-1}=\frac{1}{p}\operatorname{tr}\big((\widetilde{\Lambda}^{(L+1)}-zI)^{-1}\big)
$$ 
be the Stieltjes transform of the matrix $\widetilde{\Lambda}^{(L+1)}$. Under Assumptions \ref{asp1} to \ref{asp4}, one has that for all $z \in \mathbb{C}^{+}$, $\lim_{p \rightarrow \infty} s_{p,L+1}(z)=s_{L+1}(z)$ a.s. where
\begin{align*}
    \forall z \in \mathbb{C}^{+}, \quad s_{L+1}(z)=\int_{-\infty}^{+\infty}\left\{\tau\left[1-\gamma_{L+1}-\gamma_{L+1} z s_{L+1}(z)\right]-z\right\}^{-1} d H_{\Lambda^{(L+1)}}(\tau).
\end{align*}
Furthermore, the E.S.D. of the matrix $\widetilde{\Lambda}^{(L+1)}$ given by $F_{\widetilde{\Lambda}}(t)=\frac{1}{p} \sum_{i=1}^{p} 1\big(\lambda_i(\widetilde{\Lambda}^{(L+1)})\leq t\big)$ converges a.s. to a limiting distribution supported on $[0, \infty)$.
\end{thm}

We also define the companion Stieltjes transform $v_{L+1}(z)$, which is the Stieltjes transform of the limiting spectral distribution of the matrix $\underline{\widetilde{\Lambda}}^{(L+1)}=n_{L+1}^{-1}X^{(L+1)}\Omega X^{(L+1)\top}$. This is related to $s_{L+1}(z)$ by the following identities:
\begin{align*}
&\gamma_{L+1}\big(s_{L+1}(z)+z^{-1}\big)=v_{L+1}(z)+z^{-1}, \\
&\gamma_{L+1}\big(s_{L+1}^{\prime}(z)-z^{-2}\big)=v_{L+1}^{\prime}(z)-z^{-2}.
\end{align*} 
\cite{ledoit2011eigenvectors} proved that the following quantity that appears in the risk of ridge regression will converge almost surely to $\kappa(\lambda)$, a function of Stieltjies transform $v(z)$ under suitable moment condition as $n_{(L+1)},p \rightarrow \infty$ and $\frac{p}{n_{L+1}}\rightarrow\gamma_{L+1}$, i.e., 
\begin{align*}
\frac{1}{p} \operatorname{tr}\big(\Lambda^{(L+1)}\big(\widetilde{\Lambda}^{(L+1)}+\lambda I_{p \times p}\big)^{-1}\big) \stackrel{a.s.}{\rightarrow} \frac{1}{\gamma_{L+1}}\Big(\frac{1}{\lambda v(-\lambda)}-1\Big)\stackrel{\Delta}{=}\kappa(\lambda). 
\end{align*}
Besides, \cite{dobriban2018high} proved that if we further assume that the operator norm of $\Sigma^{(L+1)}$ is bounded by some absolute constant $C$ for any $p$, the other quantities appearing in the predictive risk in ridge regression will converge almost surely to the negative derivative of $\kappa(\lambda)$, i.e.,
\begin{align*}
p^{-1} \operatorname{tr}\big(\Lambda^{(L+1)}\big(\widetilde{\Lambda}^{(L+1)}+\lambda I_{p \times p}\big)^{-2}\big)\stackrel{a.s.}{\rightarrow} -\kappa^{\prime}(\lambda).
\end{align*}

\begin{thm}\label{asymptotic_behavior_of_predictive_risk}
For any $\lambda>0$, the oracle predictive risk $\oraclerisk\big(\Omega\mid X^{(L+1)}\big)$ converges almost surely to the limiting predictive risk $r(\lambda,\gamma_{L+1})$ as $p,n_{L+1}\rightarrow\infty$ such that $p/n_{L+1}\rightarrow\gamma_{L+1}$, where
\begin{align}\label{oracle_limiting_risk}
\begin{aligned}
    &r\big(\lambda,\gamma_{L+1}\big)\\
    =&\frac{1}{\lambda \gamma_{L+1} s_{L+1}(-\lambda)+(1-\gamma_{L+1})}\Big[\sigma^2+\big(\frac{\lambda}{\gamma_{L+1}}-\sigma^2\big) \frac{\lambda^2 \gamma_{L+1} s_{L+1}^{\prime}(-\lambda)+(1-\gamma_{L+1})}{\gamma_{L+1} \lambda s_{L+1}(-\lambda)+(1-\gamma_{L+1})}\Big].
\end{aligned}
\end{align}
where $s_{L+1}$ is the Stieltjes transform of limiting spectral distribution of $\widetilde{\Lambda}^{(L+1)}$. In particular, the choice $\lambda^*=\gamma_{L+1}\sigma^2$ minimizes the limiting risk. 
\end{thm}

\begin{rmk}
Analytically computing the limiting risk and obtaining a more tangible expression is still a non-trivial task. In certain special cases, it is possible to obtain more interpretable expressions. When $\Sigma^{(L+1)}=\varrho\Omega^{-1}$, for some $\varrho >0$, we can have a closed form expression, given by
\begin{align}
r(\lambda, \gamma_{L+1})=\sigma^2+\gamma_{L+1}\sigma^{2} m_{\varrho I}(-\lambda ; \gamma_{L+1})+\lambda(\lambda -\gamma_{L+1}\sigma^2) m_{\varrho I}^{\prime}(-\lambda ; \gamma_{L+1}),\label{expression_risk_MPLaw}    
\end{align} 
where
\begin{align*}
m_{\varrho I}(-\lambda ; \gamma_{L+1})=\frac{-(\varrho-\varrho\gamma_{L+1}+\lambda)+\sqrt{(\varrho-\varrho\gamma_{L+1}+\lambda)^2+4\varrho\gamma_{L+1} \lambda}}{2 \varrho\gamma_{L+1} \lambda}.    
\end{align*} 
A plot of the above limiting risk is provided in Figure~\ref{Risk_plot_MPLaw_multiple_gamma} when $\Omega$, as in~\eqref{eq:examplematrix} with $a=16$ and $b=5$. In general, when $\gamma_{L+1}$ and $\varrho$ are fixed, the risk function decreases rapidly before hitting $\lambda=\lambda^*=\frac{p\sigma^2}{n_{L+1}}$. After attain the minimum risk at $\lambda^*$, the risk function increases slowly as $\lambda$ increases. Larger value of $a$ means larger eigenvalues for $\Lambda$, which will leads to larger risk due to \eqref{expression_risk_MPLaw} when $\gamma_{L+1}>1$ and $\lambda,\gamma_{L+1}$ are fixed. Besides, the values of $a$ does not affect the value of $\lambda^*$ which could also be seen from \eqref{expression_risk_MPLaw}. Also, as $\gamma_{L+1}$ increases, the minimum of the risk function $\lambda^{*}$ increases, as long as $\varrho$ is kept fixed. This agrees with the conventional wisdom that one should emphasize the effect of regularizer more when the dimension $p$ is much larger comparing to number of samples $n_{L+1}$. Furthermore, in this case, the optimal limiting risk can also be calculated when $\lambda=\gamma_{L+1}\sigma^{2}$ and $\Sigma^{(L+1)}=\varrho\Omega^{-1}$, and it is is given by
\begin{align*}
\Big(1-\frac{1}{2 \varrho}\Big) \sigma^2+\frac{\gamma_{L+1}-1}{2 \gamma_{L+1}}+\frac{1}{2} \sqrt{\Big(\frac{\sigma^2}{\varrho}-\frac{\gamma_{L+1}-1}{\gamma_{L+1}}\Big)^2+\frac{4 \sigma^2}{\varrho}} .   
\end{align*} 
Another example generalizing the above setting is provided in Section~\ref{sec:additionalriskeg} for illustration.  
% \textcolor{red}{add a plot for equ 13 for few values of $\gamma_{L+1}$, x axis is lamda and y is r(lambda), gamma=1,2,5,10,... multiple lines in the same figure. $\Sigma^{(L+1)}=\Omega^{-k}$. results in terms of eigenvalues of $\Omega$}
\end{rmk}

\begin{figure}[t] 
\begin{subfigure}[b]{0.3\textwidth}
         \centering
         \includegraphics[width=\textwidth]{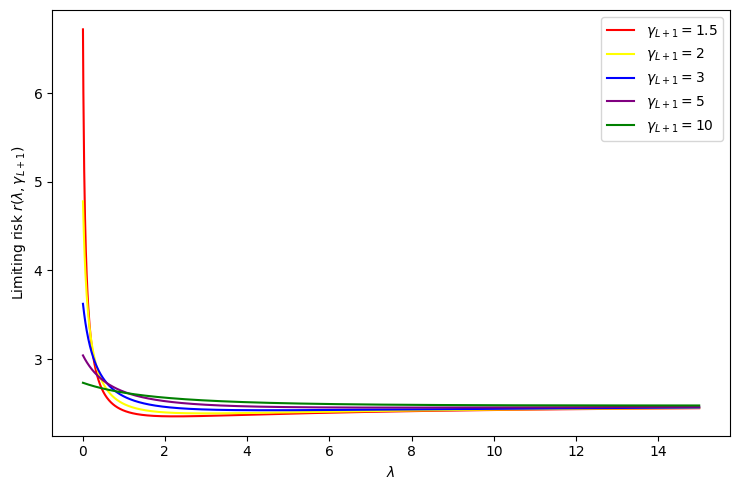}
	\caption{$\varrho=2$}\label{Risk_plot_MPLaw_multiple_gamma_a=2}
     \end{subfigure}
 \begin{subfigure}[b]{0.3\textwidth}
         \centering
         \includegraphics[width=\textwidth]{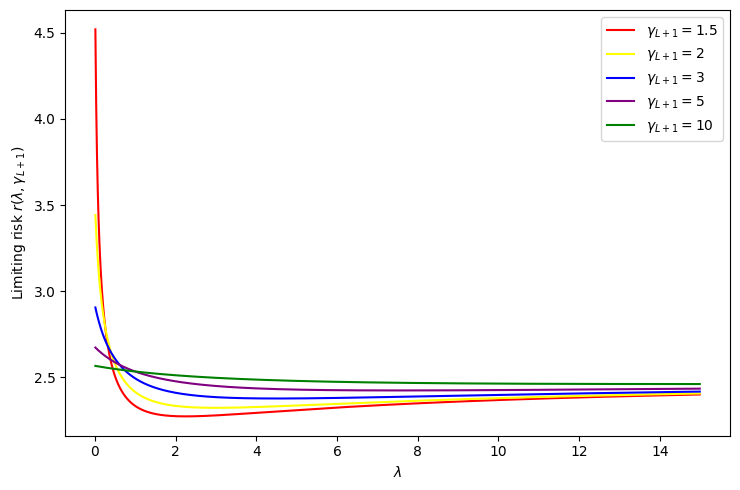}
         \caption{$\varrho=1$}\label{Risk_plot_MPLaw_multiple_gamma_a=1}
     \end{subfigure}
     \begin{subfigure}[b]{0.3\textwidth}
         \centering
         \includegraphics[width=\textwidth]{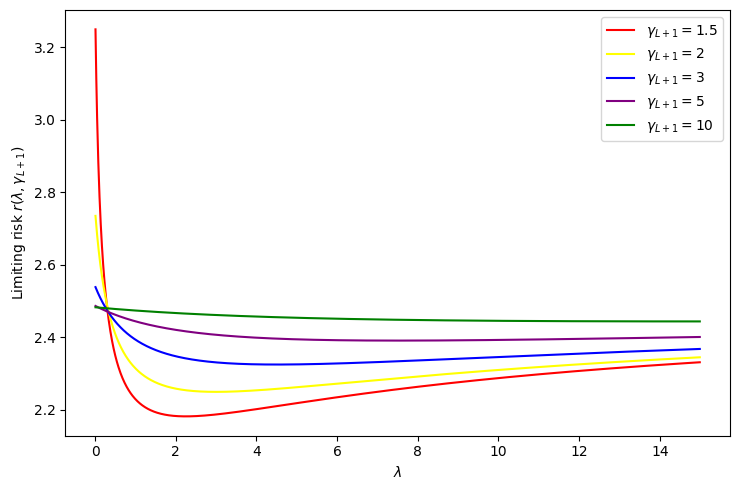}
	\caption{$\varrho={1}/{2}$}\label{Risk_plot_MPLaw_multiple_gamma_a=1/2}
     \end{subfigure}
\caption{Plot of limiting risk in \eqref{expression_risk_MPLaw}. Here, $\Sigma^{(L+1)}=\varrho\Omega^{-1}$ (for various choices of $a$), $\sigma^2=1.5$  and $\gamma_{L+1}$ takes values in the set $\{1.5,2,3,5,10\}$.}  
\label{Risk_plot_MPLaw_multiple_gamma}
\end{figure}

\begin{rmk}
Focusing on the single-task setting, and with general weight matrix,~\cite{wu2020optimal} studied the precise high-dimensional asymptotics of generalized ridge regression estimator in~\eqref{ridge_type_estimator}. Their focus is on explaining the double descent phenomenon, which depends on the alignment between the data covariance matrix and the weight matrix in~\cite{wu2020optimal}. Compared their work, our proof is more elementary and is directly suited for the meta-learning setup we consider in this work.  A more elaborate comparison between their proof technique and ours is provided in Section~\ref{sec:comparisontowuandxu}.
\end{rmk}

Note that the predictive risk of the generalized ridge regression estimator is derived in terms of a generic estimator $\hat{\Omega}$ of the true $\Omega$. Hence, we need to analyze the difference between the oracle risk $\oraclerisk(\Omega\mid X^{(L+1)})$ and actual predictive risk. In particular, $\prisk(\hat{\Omega}\mid X^{(L+1)})$ is naturally determined by how good an estimator $\hat{\Omega}$ is. Therefore, conditions (as stated in Assumption \ref{asp_Omegahat} below) are made in terms of consistency of the estimator $\hat{\Omega}$ that guarantee the consistency of $\hat{\Omega}^{-1}$ and $\hat{\Omega}^{-1}\Omega$.
\begin{asp}\label{asp_Omegahat}
We make the following conditions on $\hat\Omega$ and $\Omega$:
\begin{itemize}
\item[(i)] The estimator $\hat{\Omega}$ is consistent, i.e. $\|\hat{\Omega}-\Omega\|\stackrel{p}{\rightarrow} 0$ as $L,n_{\ell},p\rightarrow\infty$ (the specific rate depends on the choice of $\hat{\Omega}$);
\item[(ii)] $\|\Omega\|$ is bounded away from $0$ and infinity by some absolute constant for any $p$;
\item[(iii)] The condition number $\varsigma(\Omega)=\|\Omega^{-1}\|\|\Omega\|$ is bounded by universal constant $c_{\Omega}$ for any $p$.
\end{itemize}
\end{asp}
The next result shows that under the above conditions, we have consistency of $\hat{\Omega}^{-1}$ and $\hat{\Omega}^{-1}\Omega$.

\begin{lemma}\label{lemma1:for:consistency:of:Omega:inverse}
%Assume that (i) $\|\Omega-\hat{\Omega}\|\rightarrow 0$; (ii) $\|\Omega\|$ is bounded away from $0$ and infinity for any $p$ and (iii) the condition number $\varsigma(\Omega)=\|\Omega^{-1}\|\|\Omega\|$ is bounded
Suppose the conditions in Assumption~\ref{asp_Omegahat} hold. Then it holds that:
$$\|\hat{\Omega}^{-1}-\Omega^{-1}\|\stackrel{p}{\rightarrow} 0\quad\text{and}\quad\|\hat{\Omega}^{-1}\Omega-I\|\stackrel{p}{\rightarrow}0.$$
\end{lemma}
In particular, the consistency of $\hat{\Omega}^{-1}\Omega$ is used to control term $(\mathsf{I}')$ in~\eqref{pred_risk} and the consistency of $\hat{\Omega}^{-1}$ is used to control terms $(\mathsf{II}')$ and $(\mathsf{III}')$ in~\eqref{pred_risk}, in terms of closeness to their respective oracle versions. Our next result characterizes the asymptotic behavior of the $\oraclerisk(\Omega\mid X^{(L+1)})$ and $\prisk(\hat{\Omega}\mid X^{(L+1)})$ as $p,n_{L+1}$ goes to infinity.

\begin{asp}\label{asp5}
The distribution $H_{\Lambda^{(L+1)}}$ in Assumption \ref{asp4} converges in distribution to the limiting distribution $H_{\Lambda}$ as $L$ goes to infinity, and the support of $H_{\Lambda}$ is contained in a compact interval bounded uniformly (over $L$) away from $0$ and $\infty$.
\end{asp}
\begin{thm}\label{thm_consistency_L}
For any fixed $p$, $L$ and $n_{\ell}$, the difference between term $(\mathsf{III})$ and $(\mathsf{III}')$ is given by
\begin{align*}
    \frac{-\lambda\sigma^2}{n_{L+1}} \operatorname{tr}\Big(\Omega^{\frac{1}{2}}\Sigma^{(L+1)}\hat{\Omega}^{\frac{1}{2}}\big(\widehat{\Lambda}^{(L+1)}+\lambda I\big)^{-1}\hat{\Omega}^{\frac{1}{2}}\big(\hat{\Omega}^{-1} - \Omega^{-1}\big)\hat{\Omega}^{\frac{1}{2}}\big(\widehat{\Lambda}^{(L+1)}+\lambda I\big)^{-1}\Big).
\end{align*} 
Furthermore, under Assumptions \ref{asp_Omegahat} and \ref{asp5}, as $L,n_{L+1},p\rightarrow\infty$ such that for each fixed $L$, $p/n_{L+1}\rightarrow\gamma_{L+1}$, while
$\lim_{L\to \infty} \gamma_{L+1} =\gamma_* \in (1,\infty)$, we have  %difference between the predictive risk $\prisk(\hat{\Omega}\mid X^{(L+1)})$ and the limiting oracle risk \eqref{oracle_limiting_risk} converges to zero in probability, i.e. 
%$$
%\big|\prisk\big(\hat{\Omega}\mid X^{(L+1)}\big) - r(\lambda,\gamma_*)\big|
% \stackrel{p}{\rightarrow} 0.
%$$}
\begin{align*}
    \prisk\big(\hat{\Omega} \mid X^{(L+1)}\big)
   \stackrel{p}{\rightarrow} \,\,\frac{1}{\lambda \gamma_{*} s(-\lambda)+(1-\gamma_{*})}\Big[\sigma^2+\Big(\frac{\lambda}{\gamma_{*}}-\sigma^2\Big) \frac{\lambda^2 \gamma_{*} s^{\prime}(-\lambda)+(1-\gamma_{*})}{\gamma_{*} \lambda s(-\lambda)+(1-\gamma_{*})}\Big],
\end{align*}
where $s(z)$ is the solution to the following equation 
\begin{align*}
    s(z)=\int_{-\infty}^{+\infty}\left\{\tau\left[1-\gamma_{*}-\gamma_{*} z s(z)\right]-z\right\}^{-1} d H_{\Lambda}(\tau).
\end{align*}
\end{thm}

\subsection{Out-of-distribution Prediction Risk}\label{sec:ood} 

In this section, we briefly discuss the consequence of our results for the case when the new \emph{testing} task $\bar{\beta}^{(L+1)}$ covariance matrix is different from that of the training tasks. Such a setting is called as out-of-distribution prediction in the literature; in particular note that for linear models the task distribution (which is assumed to have zero mean) is completely characterized by the covariance matrix.  Specifically, we assume that $\operatorname{Var}\left(\bar{\beta}^{(L+1)}\right)=\frac{1}{p} \Upsilon$ for the $(L+1)$-th task. In order to model the relationship between the training and the test task covariance, we assume that 
\begin{align}\label{eq:vartheta}
\|\Upsilon-\Omega\|= \vartheta,
\end{align}
for any $p$. The corresponding distinguished oracle estimator (in comparison to~\eqref{oracle_estimator_generalized_ridge}) becomes 
\begin{equation}\label{risk2}
    \tilde{\beta}_{\lambda}^{(\ell)}\coloneqq \beta_{\lambda}^{(\ell)}(\Upsilon),
\end{equation}
However, in reality we are still making our predictions for the test task based on the estimator $\hat{\beta}^{(L+1)}=\beta_{\lambda}^{(L+1)}(\hat{\Omega})$. Hence, in our next result, we compare the risk of this estimator with that of the oracle estimator in~\eqref{risk2}.

\begin{prop}\label{Prop_Out-of-distribution_Prediction_Risk}
    Under Assumption \ref{asp_Omegahat} and suppose that the condition number of $\Upsilon$ is bounded by some universal constant $c_{\Upsilon}$ for any $p$ and let~\eqref{eq:vartheta} hold. When the coefficient in the new task is from some distribution with covariance $\frac{1}{p}\Upsilon$, as $L,n_{L+1},p\rightarrow\infty$ such that for each fixed $L$, $p/n_{L+1}\rightarrow\gamma_{L+1}$, while
    $\lim_{L\to \infty} \gamma_{L+1} =\gamma_* \in (1,\infty)$, it holds that 
\begin{align*}
   \big|\prisk(\hat{\Omega}\mid X^{(L+1)})-\oraclerisktilde(\Upsilon\mid X^{(L+1)})\big|&\rightarrow M(\vartheta,\lambda),
\end{align*}
where, in particular, we have
\begin{align}\label{eq:reminder}
%M(\vartheta)\leq \frac{\gamma_{*}\sigma^2}{\lambda}\bar{c}_{\text{op}}c_{\Omega}c_{\Upsilon}\vartheta+\Big(1+\frac{\gamma_{*}\sigma^2}{\lambda}\Big)\bar{c}_{\text{op}}(1+c_{\Omega}+c_{\Upsilon})c_{\Omega}c_{\Upsilon}\vartheta 
|M(\vartheta,\lambda)|\leq \frac{\gamma_{*}\sigma^2}{\lambda} \bar{c}_{\text{op}}c_{\Omega}c_{\Upsilon}\vartheta  
\left(2 +  c_{\Omega}+c_{\Upsilon}\right) 
+\bar{c}_{\text{op}}(1+c_{\Omega}+c_{\Upsilon})c_{\Omega}c_{\Upsilon}\vartheta.   
\end{align} 
\end{prop}
Note that the first term in~\eqref{eq:reminder} could be controlled by both $\vartheta$ and $\lambda$. However, the second term is purely controllable by $\vartheta$ demonstrating the unavoidable error incurred due to the train-test model mismatch.

\subsection{Statistical Optimality of Using $\Omega^{-1}$ as the Weight Matrix}\label{Advantage_section}

We now demonstrate the statistical advantage of using $\Omega^{-1}$ as generalized ridge regression comparing to using identity matrix in \cite{dobriban2018high}, using Riemannian optimization and analysis. Recall that our benchmark is the oracle risk $\oraclerisk(\Omega\mid X^{(L+1)})$ computed based on $\beta_{\lambda}^{(L+1)}(\Omega)$. 

We first introduce some basics of Riemannian geometric analysis that allow us to characterize the minimizer of functions defined on a manifold $\mathcal{M}$.  In this work, the domain of these functions will be the space of positive definite symmetric matrices, i.e., 
\begin{align*}
    \mathcal{M}=\mathbb{S}_{p}^{+}=\big\{Q\in\mathbb{R}^{p\times p}:Q=Q^\top;v^\top Qv > 0, \forall v\in\mathbb{R}^{p}\big\}.
\end{align*}
For any matrix $A\in\mathbb{S}_{p}^{+}$, the tangent space $\mathcal{T}_{A}$ could be identified with the space of symmetric matrices $\mathbb{S}_{p}=\big\{Q\in\mathbb{R}^{p\times p}:Q=Q^\top \big\}$ since the tangent space $T_a V$ to a vector space $V$ (in this case $V=\mathbb{S}_{p}$) can be identified with the vector space itself (via the isomorphism which takes an element $v \in V$ to the directional derivative $\left.D_v\right|_a$). Moreover, the tangent space to an open subset of a manifold is isomorphic to the tangent space of the manifold itself. Hence,
$$\mathcal{T}_A \mathbb{S}_{p}^{+} \cong \mathcal{T}_A \mathbb{S}_{p} \cong \mathbb{S}_{p}=\big\{Q\in\mathbb{R}^{p\times p}:Q=Q^\top\big\}.$$

A differentiable manifold $\mathcal{M}$ is a Riemannian manifold if it is equipped with an inner product (called Riemannian metric) on the tangent space, $\langle \cdot, \cdot \rangle _x : \mathcal{T}_x\mathcal{M} \times \mathcal{T}_x\mathcal{M} \rightarrow \mathbb{R}$, that varies smoothly on $\mathcal{M}$. The norm of a tangent vector is defined as $\|\xi\|_x\coloneqq\sqrt{\langle \xi, \xi\rangle _x}$. We drop the subscript $x$ and simply write $\langle \cdot, \cdot \rangle$ (and $\|\xi\|$) if $\mathcal{M}$ is an embedded submanifold with Euclidean metric. Here we use the notion of the tangent space $\mathcal{T}_x\mathcal{M}$ of a differentiable manifold $\mathcal{M}$, whose precise definition can be found in \citet[Chapter 8]{tu2011manifolds}. We now introduce the concept of a Riemannian gradient.
\begin{defn}[Riemannian Gradient]\label{def_riemann_grad}
    Suppose $f$ is a smooth function on Riemannian manifold $\mathcal{M}$. The Riemannian gradient $\operatorname{grad} f(x)$ is a vector in $\mathcal{T}_x\mathcal{M}$ satisfying $\left.\frac{d(f(\gamma(t)))}{d t}\right|_{t=0}=\langle v, \operatorname{grad} f(x)\rangle_{x}$ for any $v\in \mathcal{T}_x\mathcal{M}$, where $\gamma(t)$ is a curve satisfying $\gamma(0)=x$ and $\gamma'(0)=v$.
\end{defn}
The Riemannian gradient on Riemannian manifold $\mathcal{M}$ could be conveniently computed using the retraction on the $\mathcal{M}$ defined formally below.
\begin{defn}[Retraction]
A retraction on a manifold $\mathcal{M}$ is a smooth map
$$P: \mathcal{T}_x\mathcal{M} \rightarrow \mathcal{M}:(x, v) \mapsto P_x(v),$$
such that each curve $c(t)=\mathrm{R}_x(t v)$ satisfies $c(0)=x$ and $c^{\prime}(0)=v$.
\end{defn}
Now let $f: \mathcal{M} \rightarrow \mathbb{R}$ be a smooth function on a Riemannian manifold $\mathcal{M}$ equipped with a retraction $\mathrm{R}$. According to \citet[Proposition 3.59]{boumal2020introduction}, the Riemannian gradient $\operatorname{grad}f(x)$ could be computed as,  
\begin{equation}\label{Bou20equ3.35}
\operatorname{grad} f(x)=\operatorname{grad}\left(f \circ \mathrm{R}_x\right)(0), \quad \forall x \in \mathcal{M},
\end{equation}
where $f \circ \mathrm{R}_x: \mathrm{T}_x \mathcal{M} \rightarrow \mathbb{R}$ is defined on the tangent space $\mathcal{T}_x\mathcal{M}$ equipped with inner product $\langle\cdot,\cdot\rangle_{x}$. $\mathcal{T}_{x}\mathcal{M}$ equipped with inner product is just a Euclidean space, hence the right hand side of \eqref{Bou20equ3.35} is a ``classical'' gradient.

In general, checking whether a point $x$ on $\mathcal{M}$ is a local minimizer for $f: \mathcal{M} \rightarrow \mathbb{R}$ is not easy. However, we can identify the necessary conditions for a point $x$ to be a local minimizer. Following proposition in \cite{boumal2020introduction} shows that critical points of a function defined on the manifold $\mathcal{M}$ are exactly those points where the Riemannian gradient vanishes.
\begin{prop}[\cite{boumal2020introduction}]
Let $f: \mathcal{M} \rightarrow \mathbb{R}$ be smooth on a Riemannian manifold $\mathcal{M}$. Then, $x$ is a critical point of $f$ if and only if $\operatorname{grad} f(x)=0$.
\end{prop}

Now, suppose that $Q$ is any symmetric positive definite matrix, then the predictive risk of the generalized ridge estimator $\beta_{\lambda}^{(L+1)}(Q)$ is given by
\begin{align*}
   \prisk\big(Q\mid X^{(L+1)}\big) =&\sigma^2+\frac{\lambda^2}{p} \operatorname{tr}\Big(\Omega Q^{-1}\big(\hat{\Sigma}^{(L+1)}+\lambda Q^{-1}\big)^{-1} \Sigma^{(L+1)}\big(\hat{\Sigma}^{(L+1)}+\lambda Q^{-1}\big)^{-1} Q^{-1}\Big)\\
   &-\frac{\lambda\sigma^2}{n_{L+1}} \operatorname{tr}\Big(\big(\hat{\Sigma}^{(L+1)}+\lambda Q^{-1}\big)^{-1} \Sigma^{(L+1)}\big(\hat{\Sigma}^{(L+1)}+\lambda Q^{-1}\big)^{-1} Q^{-1}\Big)\\
   &+\frac{\sigma^2}{n_{L+1}} \operatorname{tr}\Big(\Sigma^{(L+1)}\big(\hat{\Sigma}^{(L+1)}+\lambda Q^{-1}\big)^{-1}\Big). 
\end{align*}
Given the predictive risk in this form, we show that under some restriction on $\lambda$, the predictive risk $\prisk(Q\mid X^{(L+1)})$ is minimized at $Q=\Omega$, the true covariance matrix, for any finite $p,n_{L+1}$.

To do so, we derive the Riemannian Gradient and check the optimality condition of predictive risk on $\mathbb{S}_{p}^{+}$, i.e. the manifold of $p \times p$ symmetric positive definite matrices. $\mathbb{S}_{p}^{+}$ becomes a Riemannian manifold when it is equipped with with the affine-invariant metric $g_{Q}(A_{Q},B_{Q})$ given by
\begin{align*}
    g_{Q}(A_{Q},B_{Q})=\operatorname{tr}\big(A_Q Q^{-1} B_Q Q^{-1}\big).
\end{align*}
See, for example, \cite{pennec2006riemannian}, \cite{sra2015conic} for additional details. In order to optimize the predictive risk, or check the optimality condition, one needs to find a proper retraction on $\mathbb{S}_{p}^{+}$. Following notation from \cite{boumal2020introduction}, possible choices of proper retraction on $\mathbb{S}_{p}^{+}$ are given by
\begin{align*}
    P_{Q}(A_{Q})=Q^{\frac{1}{2}} \exp \big(Q^{-\frac{1}{2}} A_Q Q^{-\frac{1}{2}}\big) Q^{\frac{1}{2}},
\end{align*}
or
\begin{align*}
    P_{Q}(A_{Q})=Q+A_Q+\frac{1}{2} A_Q Q^{-1} A_Q.
\end{align*}
The latter one is preferred since it is more computationally cheap. This is indeed a proper retraction since for any $Q\in\mathbb{S}_{p}^{+}, A_Q\in \mathrm{T}\mathbb{S}_{p}^{+}$ and any vector $0\not=v\in\mathbb{R}^{p}$
\begin{align*}
v^{\top} P_Q(A_{Q}) v & =\frac{1}{2} v^\top\big(Q+2 A_Q+A_Q Q^{-1} A_Q\big) v+\frac{1}{2} v^\top Q v \\
& =\frac{1}{2} v^\top\big(Q^{\frac{1}{2}}+A_Q Q^{-\frac{1}{2}}\big)\big(Q^{\frac{1}{2}}+A_Q Q^{-\frac{1}{2}}\big)^\top v+\frac{1}{2} v^\top Q v>0 .
\end{align*}
Hence, $P_{Q}(A)=Q+A_Q+\frac{1}{2} A_Q Q^{-1} A_Q$ remains symmetric positive definite for all $Q\in\mathbb{S}_{p}^{+}$ and $A_Q\in \mathrm{T}\mathbb{S}_{p}^{+}$. The Riemannian gradient of the predictive risk $\prisk(Q^{-1}\mid X^{(L+1)})$ w.r.t. $Q^{-1}\in\mathbb{S}_{p}^{+}$ is then given by
$$\operatorname{grad} \prisk(Q^{-1}\mid X^{(L+1)})=\operatorname{grad}\big(\prisk \circ P_{Q^{-1}}\big)(0)=\left.\operatorname{grad} \prisk(P_{Q^{-1}}(\Xi)\mid X^{(L+1)})\right|_{\Xi=0}.$$
Note that $\operatorname{grad}\big(\prisk \circ P_{Q^{-1}}\big)$ is defined on a Euclidean space (linear space $\mathrm{T}_{Q^{-1}}\mathbb{S}_{p}^{+}$ with inner product $g_{Q^{-1}}(A_{Q^{-1}},B_{Q^{-1}})=A_{Q^{-1}}QB_{Q^{-1}}Q$). Hence, $\operatorname{grad}\big(\prisk \circ P_{Q^{-1}}\big)$ is the classical gradient. We now have the following result showing that as the tuning parameter $\lambda$ is appropriately chosen, the predictive risk is minimized at $Q=\Omega$.
\begin{thm}\label{statistical_advantage_of_Omega}
If $\lambda= c\frac{p\sigma^{2}}{n_{L+1}}$ for any $c>0$, then $\prisk(Q\mid X^{(L+1)})$ is minimized at $Q^{*}=c\Omega$. The optimal risk is given by
\begin{align*}
  \prisk(Q^*=c\Omega\mid X^{(L+1)})&=\sigma^2+\frac{\sigma^2}{n_{L+1}} \operatorname{tr}\Big(\Sigma^{(L+1)}\big(\hat{\Sigma}^{(L+1)}+\frac{p \sigma^2}{n_{L+1}} \Omega^{-1}\big)^{-1}\Big)\\
  &=  \oraclerisk(\Omega\mid X^{(L+1)})\bigg|_{\lambda=\frac{p\sigma^2}{n_{L+1}}}.
\end{align*}
In particular, if $c=1$, then $\prisk(Q\mid X^{(L+1)})$ is minimized at $Q^{*}=\Omega$. Furthermore, the risk $\prisk$ at $\Omega$ is exactly the same as the oracle risk when $\lambda=\frac{p\sigma^2}{n_{L+1}}$.
\end{thm} 

\begin{rmk}
    ~\cite{wu2020optimal} also studied the question of the optimal weight matrix in generalized ridge regression. However, their work required a stringent assumption that the matrices $\Omega$ and $\Sigma^{(L)}$ commute, which makes the proof straightforward. Our result above is more generally applicable without the aforementioned restriction, which is enabled by our proof technique based on Riemannian optimization and analysis.
\end{rmk}
 
\section{Estimation of the hyper-covariance matrix $\Omega$}\label{MTL_cov_est_section}

The main message from the previous section is that the optimal matrix $A$ to consider in the generalized ridge regression estimator~\eqref{ridge_type_estimator} is the \emph{unknown} hyper-covariance matrix $\Omega$. In this section, we propose a number of approaches to estimate the hyper-covariance matrix from the training tasks. A natural approach is to perform maximum likelihood estimation of $\Omega$. However, as we discuss next, such an approach suffers from the following drawbacks: (i) it requires explicit parametric assumption on the task distribution; (ii) it is computationally intensive as it requires inversion of large matrices. Furthermore, as we show next, the negative log-likelihood function is not necessarily globally geodesically convex. To overcome these issues, we propose a \emph{method-of-moments} based estimation procedure that involves minimizing a globally geodesically convex objective function. Hence, the proposed method can be implemented efficiently using off-the-shelf Riemannian optimization techniques. An overview of the proposed estimators and their rates of consistency is provided in Table~\ref{tab:maintable}. 

%In section \ref{Mom_estimator_Omega_section}, we propose a framework to estimate the covariance matrix following the ideas from \cite{balasubramanian2013high}. Then in section \ref{consistency_of_Omegahat_general}, we show that if both $\sqrt{p}\bar{\beta}^{(\ell)}$ and $\varepsilon^{(\ell)}$ are i.i.d. sub-Gaussian random vector with parameter $\tau_{\beta}$ and $\tau_{\varepsilon}$ respectively, under some conditions on $\Sigma^{(\ell)}$ and $\gamma_{\ell}$, our proposed estimator $\hat{\Omega}$ is consistent as $p,L,n_{\ell}\rightarrow\infty$ such that $\frac{p^2}{L}\rightarrow 0$. 

\begin{table}[t]
        \centering
\resizebox{\textwidth}{!}{\begin{tabular}{|l|c|c|c|c|}\hline
         Method & Unregularized \eqref{optimization_for_Omega}  & \pbox{8cm}{\hspace{1mm}\\$L_1$ regularized\\estimator \eqref{L1regularized_estimator_Omega}\\
         \vspace{1mm}} & \pbox{10cm}{\hspace{1mm}\\$L_1$ regularized\\
         estimator \eqref{full_rank_correlation_estimator_Omega}\\
         \vspace{1mm}} & \pbox{8cm}{\hspace{1mm}\\$L_1$ regularized\\
         estimator \eqref{correlation_based_estimator_Omega}\\
         \vspace{1mm}} \\
         \hline
         \pbox{8cm}{\hspace{1mm}\\
         Assumption\\ 
         on $\Omega$\\
         \hspace{1mm}} & $C_0<\lambda_{\min}(\Omega)\leq \lambda_{\max}(\Omega)<C_1$ & \multicolumn{3}{c|}{\pbox{7cm}{$\quad C_0<\lambda_{\min}(\Omega)\leq \lambda_{\max}(\Omega)<C_1$\\
$S=\{(i,j):\Omega_{ij}\neq 0,i\neq j\}$, 
$\operatorname{card}(S)\leq s$}} \\
\hline
\pbox{8cm}{Assumption\\
on $\gamma_{\ell}$} & \pbox{18cm}{ \hspace{0.01mm}\\
$\exists L_0$ s.t. $L_0/L\rightarrow c>0$\\
$\max_{1\leq \ell\leq L_0}\gamma_{\ell}\leq 1-\delta$,\\
$\forall \ell, 0<\underline{c}\leq \gamma_{\ell}\leq \overline{c}<\infty$\\
\hspace{-0.01mm}} & \multicolumn{3}{c|}{$\forall \ell, 0<\underline{c}\leq \gamma_{\ell}\leq \overline{c}<\infty$} \\
\hline
\pbox{8cm}{\hspace{1mm}\\
Assumption\\ 
on $\Sigma^{(\ell)}$\\
\hspace{1mm}} & \multicolumn{4}{c|}{$\underline{c}^{(\ell)}\leq \lambda_{\min}(\Sigma^{(\ell)})\leq \lambda_{\max}(\Sigma^{(\ell)})\leq \bar{c}_{(\ell)},\sup_{\ell\in\mathbb{N}}\bar{c}^{(\ell)}\leq \bar{c}_{\text{op}},\inf_{\ell\in\mathbb{N}}\underline{c}^{(\ell)}>\underline{c}_{\text{op}}$}  \\
\hline
\pbox{8cm}{Assumption\\ 
on $X^{(\ell)}$} & \pbox{8cm}{$\{x_{i}^{(\ell)}\}_{i=1}^{n_{\ell}}$ are i.i.d.} & \pbox{8cm}{$\{x_{i}^{(\ell)}\}_{i=1}^{n_{\ell}}\stackrel{\text{i.i.d.}}{\sim}\text{subG}(\tau_{x})$\\ 
for all $\ell$} & \pbox{18cm}{$\{x_{i}^{(\ell)}\}_{i=1}^{n_{\ell}}\stackrel{\text{i.i.d.}}{\sim}\text{subG}(\tau_{x})$\\
$\operatorname{rank}(X^{(\ell)})=p$ for all $\ell$} & \pbox{18cm}{\hspace{0.01mm}\\
$\exists L_0$ such that\\ $L_0/L\rightarrow c>0$\\
and $\operatorname{rank}(X^{(\ell)})=p$\\ 
for $1\leq \ell \leq L_0$.\\
$\{x_{i}^{(\ell)}\}_{i=1}^{n_{\ell}}\stackrel{\text{i.i.d.}}{\sim}\text{subG}(\tau_{x})$\\
\hspace{0.01mm}}\\
\hline
\pbox{8cm}{\hspace{1mm}\\Assumption\\
on $\varepsilon^{(\ell)}$\\
\hspace{1mm}} & \multicolumn{2}{c|}{$\sigma^2>0$ (noisy setting)} & \multicolumn{2}{c|}{$\sigma^2=0$ (noiseless setting)}\\
\hline
         \pbox{2cm}{\hspace{1mm}\\
         Rate for \\
         $L$ and $L_0$\\ 
         } & \pbox{3cm}{\;\;\;\;$\sqrt{\frac{p^2}{L}}$\\
         (Theorem \ref{consistency_of_Omegahat_Lpn})} & \pbox{3cm}{$\sqrt{\frac{(p+s)\log p}{L}}$\\ 
         (Theorem \ref{thm1_sparse_cov_est})} & \pbox{3cm}{$\sqrt{\frac{(s+1)\log p}{L}}$\\
         (Theorem \ref{thm2_sparse_cov_est})} & \pbox{10cm}{$\sqrt{\frac{s\log p}{L_0}}$\\
         (Theorem \ref{thm3_sparse_cov_est})}\\
\hline
\end{tabular}}
\caption{Summary of estimation methods for $\Omega$: In all approach, sub-Gaussian assumption is proposed on $\sqrt{p}\bar{\beta}^{(\ell)}$ and $\varepsilon^{(\ell)}$. $L_{0}$ is the number of special tasks that have special properties on $X^{(\ell)}$ or $\gamma_{\ell}$.}
\label{tab:maintable}
\end{table}

\textbf{Non-convexity of MLE.} Following classical works in the literature on random effects models, suppose that $\bar{\beta}^{(\ell)}\stackrel{i.i.d.}{\sim}N(0,\frac{1}{p}\Omega)$ and $\varepsilon^{(\ell)}\stackrel{i.i.d.}{\sim}N(0,\sigma^{2}I)$. Then, the log-likelihood function is given by
\begin{align}\label{loglikelihood_for_Omega}
\begin{aligned}
        &l(\Omega,\sigma^{2})\\
        =&c-\frac{1}{2} \sum_{\ell=1}^L \log \operatorname{det}\Big(\sigma^2 I+\frac{1}{p} X^{(\ell)} \Omega X^{(\ell)^{\top}}\Big) -\frac{1}{2} \sum_{\ell=1}^L y^{(\ell)^{\top}}\Big(\sigma^2 I+\frac{1}{p} X^{(\ell)} \Omega X^{(\ell)^{\top}}\Big)^{-1} y^{(\ell)},
\end{aligned}
\end{align}
for some constant $c$. Maximizing this log-likelihood yields the MLE for $\hat{\Omega}$. However, the negative log-likelihood function is not necessarily globally (geodesically) convex, according to the Definition \ref{geodesic_convex_def1} below.

\begin{defn}\label{geodesic_convex_def1}
A function $f: \mathcal{M} \rightarrow \mathbb{R}$ defined on a Riemannian manifold is said to be geodesically convex if for any $x, y \in \mathcal{M}$, a geodesic $\gamma$ such that $\gamma(0)=x$ and $\gamma(1)=y$, and $t \in[0,1]$, it holds that
$$f(\gamma(t)) \leq(1-t) f(x)+t f(y).$$
\end{defn}

To see that the negative log-likelihood in (\ref{loglikelihood_for_Omega}) is not globally geodesically convex, we first note that, equipped with the natural Riemannian metric over space of positive definite matrices, the geodesic path \citep{lim2013convex} between any $A,B\in\mathbb{S}_{p}^{+}$ becomes $\gamma_{A,B}(t)=A^{\frac{1}{2}}\big(A^{-\frac{1}{2}} B A^{-\frac{1}{2}}\big)^t A^{\frac{1}{2}}$. By Definition \ref{geodesic_convex_def1}, the function $l(\Omega)$ is geodesically convex if and only if the composition $l(\gamma_{\Omega_{1},\Omega_{2}}(t)):[0,1]\rightarrow \mathbb{R}$ is convex in usual sense for any $\Omega_{1},\Omega_{2}\in\mathbb{S}_{p}^{+}$. In this case, $$\gamma_{\Omega_{1},\Omega_{2}}(t):=\Omega_1^{\frac{1}{2}}\big(\Omega_1^{-\frac{1}{2}} \Omega_2 \Omega_1^{-\frac{1}{2}}\big)^t \Omega_1^{\frac{1}{2}}.$$
Fix $\Omega_{1}\in\mathbb{S}_{p}^{+}$ and $\Omega_{2}=k\Omega_{1}$ where $k>0$, and denote $X^{(\ell)}\Omega_{1}X^{(\ell)\top}=\sum_{i=1}^{p}\lambda_{i}^{(\ell)}e_{i}^{(\ell)}e_{i}^{(\ell)\top}$ to be the eigenvalue decomposition of $X^{(\ell)}\Omega_{1}X^{(\ell)\top}$. Then, 
\begin{align*}
    l(\gamma(t))&=c-\frac{1}{2}\sum_{\ell=1}^{L}\sum_{i=1}^{p}\log\Big(\sigma^2+\frac{k^t}{p}\lambda_{i}^{(\ell)}\Big)-\frac{1}{2}\sum_{\ell=1}^{L}\sum_{i=1}^{p}\Big(\sigma^2+\frac{k^t}{p}\lambda_{i}^{(\ell)}\Big)^{-1}\Big(e_{i}^{(\ell)\top}y^{(\ell)}\Big)^2.
\end{align*}
The first and second derivatives with respect to $t$ are respectively given by
\begin{align*}
 -\frac{1}{2} \sum_{\ell=1}^L \sum_{i=1}^p\Big(\sigma^2+\frac{k^t}{p} \lambda_i^{(\ell)}\Big)^{-1} \frac{\lambda_i^{(\ell)}}{p} k^t \ln k+\frac{1}{2} \sum_{l=1}^L \sum_{i=1}^p\Big(\sigma^2+\frac{k^t}{p} \lambda_i^{(\ell)}\Big)^{-2} \frac{\lambda_i^{(\ell)}}{p} k^t \ln k,
\end{align*}
and 
\begin{align*}
    %\frac{d^2}{dt^2}l(\gamma(t)) &=
    - \sum_{\ell=1}^L \sum_{i=1}^p\frac{(\ln k)^2\lambda_i^{(\ell)}k^t}{2p} \Big(\sigma^2+\frac{k^t}{p} \lambda_i^{(\ell)}\Big)^{-2} \LaTeXunderbrace{\Big[\sigma^2+\big(e_i^{(\ell) \top} y^{(\ell)}\big)^2\Big(\sigma^2+\frac{k^t}{p} \lambda_i^{(\ell)}\Big)^{-1}\Big(\frac{\lambda_i^{(\ell)}}{p} k^t-\sigma^2\Big)\Big]}_{\textsf{Ind}}.
\end{align*}
The presence of the term \textsf{Ind},  %$\Big[\sigma^2+\big(e_i^{(\ell) \top} y^{(\ell)}\big)^2\Big(\sigma^2+\frac{k^t}{p} \lambda_i^{(\ell)}\Big)^{-1}\Big(\frac{\lambda_i^{(\ell)}}{p} k^t-\sigma^2\Big)\Big]$ 
makes the second derivative to be indefinite, depending on the sample configurations $X^{(\ell)}$ and $y^{(\ell)}$. Thus the negative log-likelihood function might not be globally geodesically convex. % when some of $\Big[\sigma^2+\big(e_i^{(\ell) \top} y^{(\ell)}\big)^2\Big(\sigma^2+\frac{k^t}{p} \lambda_i^{(\ell)}\Big)^{-1}\Big(\frac{\lambda_i^{(\ell)}}{p} k^t-\sigma^2\Big)\Big]$ are negative. 

To sum up, MLE has many limitations: (i) it relies on distributional assumption on $\bar{\beta}^{(\ell)}$ and $\varepsilon^{(\ell)}$, (ii) evaluating the objective function in~\eqref{loglikelihood_for_Omega} requires inverting large matrices which can be computationally expensive when $p$ is large, and (iii) the negative log-likelihood function is not geodesically convex. Numerical approach to calculate MLE, such as Newton-Raphson method, might be sensitive to initial values and can be inefficient when the dimension of the solution is relatively high.

\subsection{Estimation without sparsity assumptions}\label{Mom_estimator_Omega_section} 

Given the limitation listed above for MLE, a new approach to estimating $\Omega$ is proposed below, motivated by the procedure indicated by~\cite{balasubramanian2013high}. Note that as $\mathbb{E}\bar{\beta}^{(\ell)}\bar{\beta}^{(l)\top}=\frac{1}{p}\Omega$ and $y^{(\ell)}=X^{(\ell)} \bar{\beta}^{(\ell)}+\varepsilon^{(\ell)}$, it holds that 
$$\mathbb{E}_{\bar{\beta}^{(\ell)}, \varepsilon^{(\ell)}} y^{(\ell)} y^{(\ell) \top}=\frac{1}{p}X^{(\ell)} \Omega X^{(\ell) \top}+\sigma^2 I_{n^{(\ell)} \times n^{(\ell)}}.$$ 
This suggests the following estimator of $\Omega$:
\begin{equation}\label{optimization_for_Omega}
    \hat{\Omega}\coloneqq\arg\min_{\tilde{\Omega} \in \mathbb{S}_{p}^{+}}\Bigg[f(\tilde{\Omega})=\frac{1}{L}\sum_{\ell=1}^{L} \Big\| y^{(\ell)} y^{(\ell) \top}-\frac{1}{p}X^{(\ell)} \tilde{\Omega} X^{(\ell) \top}-\sigma^{2}I\Big\|_F^2 \Bigg].
\end{equation} 
Problem \eqref{optimization_for_Omega} is an optimization problem on the manifold of positive definite matrices. By definition of Frobenius norm, the objective function, denoted as $f(\tilde{\Omega})$, could be equivalently written as
\begin{align}\label{eq:reformulation}
    f(\tilde{\Omega})=\frac{1}{L} \sum_{\ell=1}^L \operatorname{tr}\Big[\big(y^{(\ell)} y^{(\ell)\top}-\frac{1}{p} X^{(\ell)} \tilde{\Omega} X^{(\ell)^{\top}}-\sigma^2 I\big)^{\top}\big(y^{(\ell)} y^{(\ell)\top}-\frac{1}{p} X^{(\ell)} \tilde{\Omega} X^{(\ell)\top}-\sigma^{2} I\big)\Big].
\end{align}
The minimizer of $f(\tilde{\Omega})$ could be characterized by setting the Riemannian gradient (see Definition~\ref{def_riemann_grad}) to zero. Using the retraction 
\begin{align*}%\label{retraction}
P_{\tilde{\Omega}}(\Xi)=\tilde{\Omega}+\Xi+\frac{1}{2} \Xi \tilde{\Omega}^{-1} \Xi,  \quad\text{for}\quad \tilde{\Omega}\in\mathbb{S}_{p}^{+},\, \Xi\in\mathrm{T}\mathbb{S}_{p}^{+},
\end{align*} 
and the reformulation in~\eqref{eq:reformulation}, it is easy to see that the Riemannian gradient of $f(\tilde{\Omega})$ is given by
\begin{equation}\label{Riemannian_grad}
    \operatorname{grad} f(\tilde{\Omega})=-\frac{4}{p L} \sum_{\ell=1}^L X^{(\ell)^{\top}}\Big(y^{(\ell)} y^{(\ell)\top}-\frac{1}{p} X^{(\ell)} \tilde{\Omega} X^{(\ell)^{\top}}-\sigma^2 I\Big) X^{(\ell)},
\end{equation}
and $\hat{\Omega}$ is characterized by $\operatorname{grad}f(\tilde{\Omega})=0$. Our next result shows that the problem \eqref{optimization_for_Omega} is (globally) geodesically convex.

\begin{prop}\label{geod_convex_of_f}
The objective function \eqref{optimization_for_Omega}, when conditioned on all the random quantities involved and treated as a deterministic function, is (globally) geodesically convex.
\end{prop}
Our framework based on \eqref{optimization_for_Omega} is hence free of stringent distributional assumptions for random coefficient and noise. Also, it does not rely on computing the inverse of large matrix. And finally, since this problem is geodesically convex, numerical approaches such as Riemannian gradient descent will efficiently converge to the global minimum.

\begin{rmk}
In practice, one should also estimate the parameter $\sigma^{2}$. \cite{dicker2014variance} proposed a good approach to estimate $\sigma^{2}$; see also~\cite{hu2022misspecification}. Within our meta-learning framework, one could estimate $\sigma^2$ using their approach:
\begin{align*}
    \hat{\sigma}^2(\hat{\Sigma})&=\frac{p+n_{\ell}+1}{n_{\ell}(n_{\ell}+1)}\|y^{(\ell)}\|^2 -\frac{1}{n_{\ell}(n_{\ell}+1)}\big\|\hat{\Sigma}^{-1 / 2} X^{(\ell)\top} y^{(\ell)}\big\|^2,%\label{estiamtor_sigma2}
\end{align*}
where $\hat{\Sigma}$ is a norm-consistent estimator for $\Sigma$ as $p,n_{\ell}\rightarrow\infty$ such that $\frac{p}{n_{\ell}}\rightarrow\gamma_{\ell}$. In general, we could use one of the tasks to estimate $\sigma^2$ and remaining tasks to estimate $\Omega$. Having different noise variance is a more challenging problem, and is left as future work.    
\end{rmk}

\subsubsection{Consistency and rates when $p$ and $L$ go to infinity}\label{consistency_of_Omegahat_general}

In this section, we show that the estimator $\hat{\Omega}$ given by \eqref{optimization_for_Omega} is consistent as $p,L\rightarrow\infty$ under sub-Gaussian assumptions on $\bar{\beta}^{(\ell)}$ and $\varepsilon^{(\ell)}$.

\begin{defn}[\citep{vershynin2010introduction}]
A random vector $x \in \mathbb{R}^p$ is sub-gaussian $x \in S G_p\big(\tau\big)$ with parameter $\tau$ if for all $v \in \mathcal{S}^{p-1}$, we have $\mathbb{E}\big[\exp \big(\lambda v^{\top}(x-\mu)\big)\big] \leq \exp \big(\lambda^2 \tau^2 / 2\big).$
\end{defn}

Our main result below establishes the consistency of $\hat{\Omega}$ based on \eqref{optimization_for_Omega} under some assumptions on $\hat{\Sigma}^{(\ell)}$ and $\gamma_{\ell}$. The main idea of proving consistency of $\hat{\Omega}$ is to provide an upper bound on $\|\hat{\Omega}-\Omega\|_{F}$ in terms of $\|\operatorname{grad}f(\Omega)\|_{F}$, and to show that $\|\operatorname{grad}f(\Omega)\|_{F}\stackrel{p}{\rightarrow}0$ as $p,L\rightarrow\infty$. In particular, the first assumption in Theorem \ref{consistency_of_Omegahat_Lpn} is mainly used to find a lower bound on $\langle\operatorname{grad}f(\Omega)-\operatorname{grad}f(\hat{\Omega}),\Omega-\hat{\Omega}\rangle$ in terms of $\|\Omega-\hat{\Omega}\|_{F}^{2}$ so that one can upper bound $\big\|\hat{\Omega}-\Omega\big\|_{F}$ in terms of $\big\|\operatorname{grad}f(\Omega)\big\|_{F}$ by using the inequality 
$$
\big|\langle \operatorname{grad}f(\Omega)-\operatorname{grad}f(\hat{\Omega}),\Omega-\hat{\Omega}\rangle\big|=\big|\langle \operatorname{grad}f(\Omega),\Omega-\hat{\Omega}\rangle\big|\leq \big\|\operatorname{grad}f(\Omega)\big\|_{F}\big\|\hat{\Omega}-\Omega\big\|_{F}.
$$
For this purpose, we slightly modify Assumption \ref{asp3} such that there is a significant proportion of tasks whose limiting dimension-to-sample-size ratio $\gamma_{\ell}$ is strictly less than $1$.
\begin{asp}\label{nonsparse_asp}
We have that:
%\hspace{1cm}
\begin{itemize}
    \item[(a)] For any $L$, there exists $L_{0}$ such that $\lim_{L\rightarrow\infty}\frac{L_{0}}{L}= c> 0$ and $\max_{1\leq \ell\leq L_{0}}\gamma_{\ell}\leq 1-\delta$ for some $\delta>0$
    \item[(b)] For any $\ell$, $0<\underline{c}\leq \gamma_{\ell}\leq \overline{c}<\infty$ and for any dimension $p$, $0<\underline{c}^{(\ell)}\leq \lambda_{\min}(\Sigma^{(\ell)})\leq \lambda_{\max}(\Sigma^{(\ell)})\leq \bar{c}^{(\ell)}<\infty, \sup_{\ell\in\mathbb{N}}\bar{c}^{(\ell)}\leq \bar{c}_{\text{op}},\inf_{\ell\in\mathbb{N}}\underline{c}^{(\ell)}>\underline{c}_{\text{op}}$  
    \item[(c)] $\sqrt{p}\bar{\beta}^{(\ell)}$'s are independent zero mean and sub-Gaussian with parameter $\tau_{\beta}$; $\varepsilon^{(\ell)}$'s are independent zero mean and sub-Gaussian with parameter $\tau_{\varepsilon}$.
\end{itemize}
\end{asp}
\begin{thm}\label{consistency_of_Omegahat_Lpn}
Under assumption \ref{nonsparse_asp}, for the estimator~\eqref{optimization_for_Omega}, we have 
\begin{align*}
   \|\hat{\Omega}-\Omega\|=\mathcal{O}_{P}\Bigg(\sqrt{\frac{p^2}{L}}\Bigg).
\end{align*}
Hence, $\|\hat{\Omega}-\Omega\|\stackrel{p}{\rightarrow}0$ when $L,p,n_{\ell}\rightarrow\infty$ such that $\frac{p^{2}}{L} \rightarrow 0$ and $\frac{p}{n_{\ell}}\rightarrow\gamma_{\ell}$. In addition,
\begin{itemize}
\item[(i)] If $L_0/L\rightarrow 0$, the condition $\frac{p^2}{L}\rightarrow 0$ needs to be replaced by $\frac{L}{L_0}\frac{p^2}{L}\rightarrow 0$ to guarantee $\|\hat{\Omega}-\Omega\|\stackrel{p}{\rightarrow}0$. 
\item[(ii)] If all $\gamma_{\ell}=\gamma$ for $\ell=1,\dots, L$, then $\|\Omega-\hat{\Omega}\|=\mathcal{O}_{P}\Big(\frac{(1+\sqrt{\gamma})^2\gamma^2}{(1-\sqrt{\gamma})^2}\sqrt{\frac{p^2}{L}}\Big)$.
\end{itemize}
\end{thm}

\begin{rmk}
    Condition (a) above shows the benefit of structure-sharing between the training tasks in terms of estimating the common hyper-covariance matrix. In particular, as long as there is a non-trivial number $L_0$ of tasks for which there are more observations that the dimensions, it suffices to have consistency in hyper-covariance estimation under otherwise high-dimensional setting.  
\end{rmk}

\subsection{Estimation under sparsity assumptions}\label{MTL_sparse_section}
In Theorem \ref{consistency_of_Omegahat_Lpn}, we show that $\hat{\Omega}$ is consistent when $p,n_{\ell},L\rightarrow\infty$ such that $p^2/L\rightarrow 0$. This means if we want to estimate $\Omega$ well by \eqref{optimization_for_Omega}, it requires $L$ to be order of $p^2$.  The result in Theorem~\ref{consistency_of_Omegahat_Lpn} has the drawback that 
the aforementioned scaling of the dimension with respect to the number of training tasks is not favourable. In this section, we show that this scaling could be further improved under an additional structural assumptions on $\Omega$, namely sparsity. We then propose a $L_1$ regularized version of \eqref{optimization_for_Omega} for estimating $\Omega$ as follows:
\begin{align}
\hat{\Omega}=\arg\min_{\tilde{\Omega} \in \mathbb{S}_{p}^{+}}\Bigg[\frac{1}{L}\sum_{\ell=1}^{L} \Big\| y^{(\ell)} y^{(\ell) \top}-\frac{1}{p}X^{(\ell)} \tilde{\Omega} X^{(\ell) \top}-\sigma^2 I\Big\|_F^2+\tilde{\lambda}\sum_{i\not=j}|\tilde{\Omega}_{ij}|\Bigg]\label{L1regularized_estimator_Omega}
\end{align} 
The outline of the rest of this section is that we first prove the consistency of $\hat{\Omega}$ as $p,n_{\ell}\rightarrow\infty$ under fixed design of $X^{(\ell)}$ for $\ell=1,\dots,L$ in section \ref{estimator_sparse_fixed_design_section}. Next, in section \ref{estimator_sparse_noiseless_setting_section}, we discuss some potential improvement on the convergence rate under the noiseless setting $y^{(\ell)}=X^{(\ell)}\bar{\beta}^{(\ell)}$. This approach is motivated the work of~\cite{rothman2008sparse}. The main idea of estimating $\Omega$ is that we first estimate the diagonal part of $\Omega$ using some of tasks whose data matrix $X^{(\ell)}$ has full column rank. Next, the remaining tasks are used to estimate the correlation matrix. Specifically, if the data matrix $X^{(\ell)}$ of $L_0$ tasks has full column rank, define the left inverse of $X^{(\ell)}$ to be $\big(X^{(\ell)}\big)^{-1}_{\mathsf{left}}=\big(X^{(\ell)\top}X^{(\ell)}\big)^{-1}X^{(\ell)\top}$ and also $z^{(\ell)}=\big(X^{(\ell)}\big)^{-1}_{\mathsf{left}}y^{(\ell)}$. One can first get an estimator $\hat{W}$ of the diagonal entries of $\Omega$ based on 
\begin{align*}
    \hat{W}_{ii}&=\Big[\frac{p}{L_{0}}\sum_{\ell=1}^{L_0}z^{(\ell)}z^{(\ell)\top}\Big]_{ii}=\Big[\frac{p}{L_{0}}\sum_{\ell=1}^{L_0}\bar{\beta}^{(\ell)}\bar{\beta}^{(\ell)\top}\Big]_{ii}. %\label{estimator_of_diagonal_entry}
\end{align*} 
Then one could estimate $\Omega$ based on some modified correlation-based estimator 
\begin{align}\label{correlation_based_estimator_Omega}
    \hat{\Omega}_{w}=\hat{W}^{\frac{1}{2}}\hat{\Theta}_{\lambda}\hat{W}^{\frac{1}{2}}..
\end{align}
where $\hat{\Theta}_{\lambda}$ is an estimator of the correlation matrix $\Theta=W^{-\frac{1}{2}}\Omega W^{-\frac{1}{2}}$ by solving problem
\begin{align*}%\label{estimator_of_corr_mat}
\hat{\Theta}_{\lambda}=\arg\min_{\tilde{\Theta}\in\Gamma_{+}^{p}}\Bigg[\frac{1}{L-L_0}\sum_{\ell=L_{0}+1}^{L}\Big\|y^{(\ell)}y^{(\ell)\top}-\frac{1}{p}X^{(\ell)}\hat{W}^{\frac{1}{2}}\tilde{\Theta}\hat{W}^{\frac{1}{2}}X^{(\ell)\top}\Big\|_{F}^{2}+\tilde{\lambda}\sum_{i\not=j}|\tilde{\Theta}_{ij}|\Bigg],
\end{align*}
where $\Gamma_{+}^{p}$ is a sub-manifold defined to be
$\Gamma_{+}^{p}=\{A\in \mathbb{R}^{p\times p}:A\in\mathbb{S}_{p}^{+}, \operatorname{diag}A=I_{p}\}$.

%%%todo Finally, we extend the consistency of $\hat{\Omega}$ under fixed design into random design further assuming that the samples $x_{i}^{(\ell)}\in\mathbb{R}^{p}$ in each task are i.i.d. sub-Gaussian random vector.

\subsubsection{Fixed design case}\label{estimator_sparse_fixed_design_section}
In this section, we prove that the estimator given by \eqref{L1regularized_estimator_Omega} under fixed design matrix $X^{(\ell)}$ for $\ell=1,\dots,L$ is consistent when $p,n_{\ell},L$ goes to infinity under some specific rate of $L$ in term of $p$. %However, if the covariance matrix $\Omega$ is sparse such that the non-zero entry of $\Omega$ is no larger than the sparse parameter $s$, then we need $L$ to be less order of $p^2$. 
Similar to the assumptions proposed in Theorem \ref{consistency_of_Omegahat_Lpn}, following assumption are imposed in this section.
\begin{asp}\label{fixed_design_asp1}
Suppose that conditions (b) and (c) in Assumption~\ref{nonsparse_asp} hold and in addition,
\begin{itemize}
 %   \item[(1)] For any $\ell$, $0<\underline{c}\leq \gamma_{\ell}\leq \overline{c}<\infty$. And for any dimension $p$, $0<\underline{c}^{(\ell)}\leq \lambda_{\min}(\Sigma^{(\ell)})\leq \lambda_{\max}(\Sigma^{(\ell)})\leq \bar{c}^{(\ell)}<\infty, \sup_{\ell\in\mathbb{N}}\bar{c}^{(\ell)}\leq \bar{c}_{\text{op}},\inf_{\ell\in\mathbb{N}}\underline{c}^{(\ell)}>\underline{c}_{\text{op}}$
 %   \item[(2)] $\sqrt{p}\bar{\beta}^{(\ell)}$'s are independent zero mean and sub-Gaussian with parameter $\tau_{\beta}$; $\varepsilon^{(\ell)}$'s are independent zero mean and sub-Gaussian with parameter $\tau_{\varepsilon}$.
    \item[(d)] Let the set $S=\{(i, j): \Omega_{i j} \neq 0, i \neq j\}$. Then $\operatorname{card}(S) \leq s$.
%\end{asp} 
%\begin{asp}\label{fixed_design_asp2}
\item [(e)] There exists some absolute constant $\kappa_{0}$ such that matrix $X^{(\ell)}\otimes X^{(\ell)}$ satisfies the property
\begin{align}
    \frac{1}{p^2}\big\|( X^{(\ell)}\otimes X^{(\ell)}) \operatorname{vec}\big(\Delta\big)\big\|_{2}^{2}\geq \kappa_{0}^{(\ell)}\|\Delta \|_{F}^{2}\label{RE_equation}
\end{align}
for any symmetric matrix $\Delta\in\mathbb{R}^{p\times p}$ and $\kappa_{0}^{(\ell)}$ is uniformly bounded below for all $\ell$.
\end{itemize}
\end{asp}

Condition (e) above is an analog of Condition (a) listed in Theorem~\ref{consistency_of_Omegahat_Lpn} motivated by our structural sparsity assumption. We now provide our consistency result.

\begin{thm}\label{thm1_sparse_cov_est}
    Let $\hat{\Omega}$ be the minimizer defined by \eqref{L1regularized_estimator_Omega}, under Assumption \ref{fixed_design_asp1}, if we set $\tilde{\lambda}\asymp \sqrt{\frac{\log p}{L}}$, then we have that
    \begin{align*}
        \big\|\hat{\Omega}-\Omega\big\|_F=\mathcal{O}_P\Bigg(\sqrt{\frac{(p+s) \log p}{L}}\Bigg).
    \end{align*}
\end{thm}     

Theorem \ref{thm1_sparse_cov_est} indicates that under fixed design case, the estimator based on \eqref{L1regularized_estimator_Omega} is consistent as $p,n_{\ell},L\rightarrow\infty$ such that $(p+s) \log p/L\rightarrow 0$. The factor $\sqrt{p \log p / L}$ in particular comes from having to estimate the diagonal entries of the $\Omega$. Hence, in sparse case, one could get an consistent estimator using $L$-1 regularized approach that requires $L$ to be less order of $p$ comparing to order of $p^2$ in Theorem \ref{consistency_of_Omegahat_Lpn}.
\subsubsection{Improved rates in the noiseless setting}\label{estimator_sparse_noiseless_setting_section}
In this section, we further try to improve the rates by estimating the correlation matrix instead of estimating the covariance matrix directly, as discussed previously. To show the improvement, we start with the simplest case when all $X^{(\ell)}$'s are full column rank and show the convergence rate is given by $\|\hat{\Omega}-\Omega\|_{F}=\mathcal{O}_{P}\big(\sqrt{(s+1)\log p/L}\big)$. Suppose that $X^{(\ell)}$ in all tasks are full rank, then one can estimate $\Omega$ in following ways
\begin{align}
    \hat{\Theta}_{\lambda}&=\arg\min_{\tilde{\Theta}\in\Gamma_{+}^{p}}\Big[\frac{1}{L}\sum_{\ell=1}^{L} \Big\|\hat{W}^{-\frac{1}{2}}z^{(\ell)}z^{(\ell)\top}\hat{W}^{-\frac{1}{2}}-\frac{1}{p}\tilde{\Theta}\Big\|_{F}^{2}+\tilde{\lambda}\sum_{i\not=j}|\tilde{\Theta}_{ij}|\Big],\nonumber \\
    \hat{W}_{ii}&=\Big[\frac{p}{L}\sum_{\ell=1}^{L}z^{(\ell)}z^{(\ell)\top}\Big]_{ii}=\Big[\frac{p}{L}\sum_{\ell=1}^{L}\bar{\beta}^{(\ell)}\bar{\beta}^{(\ell)\top}\Big]_{ii},\nonumber \\
    \hat{\Omega}_{w}&=\hat{W}^{\frac{1}{2}}\hat{\Theta}_{\lambda}\hat{W}^{\frac{1}{2}}.\label{full_rank_correlation_estimator_Omega}
\end{align}
\begin{thm}\label{thm2_sparse_cov_est}
For $\ell$-th task, suppose we observe $y^{(\ell)}$ and $X^{(\ell)}$ under noiseless setting, let $\hat{\Omega}_{w}$ be the minimizer defined by \eqref{full_rank_correlation_estimator_Omega}. Under Assumption \ref{fixed_design_asp1}, if the data matrices of all these $L$ tasks have full column rank structure and $\tilde{\lambda}\asymp \sqrt{\frac{\log p}{L}}$, 
    \begin{align*}
        \|\hat{\Omega}_{w}-\Omega\|_F=\mathcal{O}_{P}\bigg(\sqrt{\frac{(s+1)\log p}{L}}\bigg).
    \end{align*}
\end{thm}     

Then we relax this stringent assumption into the case when only a proportion of $X^{(\ell)}$'s are of full column rank. In this case, we show that the convergence rate in operator norm is given by $\|\hat{\Omega}_{w}-\Omega\|\leq \mathcal{O}_{P}\big(\sqrt{s\log p/(L-L_0)}+\sqrt{s\log p/L_0}+\sqrt{s(\log p)^2/L_0(L-L_0)}\big)$. Following theorem shows that with appropriate choice of $\tilde{\lambda}$ the convergence rate of $\hat{\Omega}_{w}$ could be improved compared to that of $\hat{\Omega}$ given by (\ref{L1regularized_estimator_Omega})
\begin{thm}\label{thm3_sparse_cov_est}
    Under Assumptions \ref{fixed_design_asp1}, let $\hat{\Omega}_{w}$ be the estimator based on \eqref{correlation_based_estimator_Omega} in the noiseless setting. With $\tilde{\lambda}=2C_{1}\Big(\sqrt{\frac{\log p}{L-L_0}}+\sqrt{\frac{\log p}{L_{0}}}+\sqrt{\frac{\log p}{L-L_0}}\sqrt{\frac{\log p}{L_0}}\Big)$, it holds that
    \begin{align*}
        \big\|\hat{\Omega}_{w}-\Omega\big\|\leq \mathcal{O}_{P} \Bigg(\sqrt{\frac{s\log p}{L-L_0}}+\sqrt{\frac{s\log p}{L_0}}+\sqrt{\frac{s(\log p)^2}{L_0(L-L_0)}}\Bigg) 
    \end{align*}
\end{thm}     
Therefore, Theorem \ref{thm3_sparse_cov_est} states that $\|\hat{\Omega}_{w}-\Omega\|\stackrel{p}{\rightarrow}0$ as $p,L,L_0,n_{\ell}\rightarrow\infty$ as long as $s\log p/L\rightarrow 0$ and $\frac{L_0}{L}\rightarrow c>0$ under noiseless setting. Therefore, with appropriate choice of $\tilde{\lambda}$, the convergence rate could be improved based on \eqref{correlation_based_estimator_Omega} comparing to \eqref{L1regularized_estimator_Omega}.

%There exists some absolute constant $\kappa_0$, under appropriate condition on $n_{\ell},p$ and $L$, it holds that $$\frac{1}{p^{2}}\big\|(X^{(\ell)}\otimes X^{(\ell)}) v\big\|_{2}^{2}\geq \kappa_{0}^{(\ell)}\|v\|_{2}^{2}$$ with high probability for $v=\operatorname{vec}(\Delta)$. 

To extend previous results to random design case, we need to prove the condition (e) in Assumption \ref{fixed_design_asp1} holds with high probability. Theorem \ref{sparsethm_lowerbound} shows that when rows of $X^{(\ell)}\in\mathbb{R}^{n_{\ell}\times p}$ are i.i.d. sub-Gaussian random vector, the condition (e) in Assumption \ref{fixed_design_asp1} holds with high probability. 
\begin{thm}\label{sparsethm_lowerbound}
    Suppose that the rows of $X^{(\ell)}\in\mathbb{R}^{n_{\ell}\times p}$ are i.i.d. sub-Gaussian random vector with parameter $\tau_{x}^{(\ell)}$ and for all $\ell$ $\lambda_{\min}\big(\Sigma^{(\ell)}\big)\geq \underline{c}^{(\ell)}$ for some absolute constant $\underline{c}^{(\ell)}>0$, then for any $q\geq 2$ with probability at least $1-C_{q}p^{-\frac{q}{4}}$, \eqref{RE_equation} holds for any symmetric matrix $\Delta$. The constant $C_{q}$ does not depends on $p$ and $n_{\ell}$.
\end{thm}     
%Theorem \ref{sparsethm_lowerbound} shows that the condition (e) in Assumption \ref{fixed_design_asp1} holds with high probability when the samples $x^{(\ell)}$'s from $\ell$-th task are i.i.d. sub-Gaussian random vector with parameter $\tau_{x}^{(\ell)}$. 
With the above result in hand, the results in Theorem \ref{thm3_sparse_cov_est} extend to random design case with sub-Gaussian assumption on the samples $x^{(\ell)}$ by applying $\Delta=W^{\frac{1}{2}}\Delta W^{\frac{1}{2}}$. Besides, same quantity $\operatorname{tr}(X^{(\ell) \top} X^{(\ell)} \Delta  X^{(\ell) \top} X^{(\ell)} \Delta )$ also appears in the proof of Theorem \ref{thm1_sparse_cov_est}, in which we need to find a lower bound on this quantity. Hence, the convergence results in Theorem \ref{thm1_sparse_cov_est} could also be extended into random design case under sub-Gaussian assumption.  
\begin{rmk}
  For the approaches in Section \ref{MTL_sparse_section}, when all $\gamma_{\ell}=\gamma$, the order in Theorem \ref{thm1_sparse_cov_est}, \ref{thm2_sparse_cov_est} and \ref{thm3_sparse_cov_est} becomes   
\begin{align*}
&\mathcal{O}_P\Bigg(\zeta(\gamma)\sqrt{\frac{(p+s) \log p}{L}}\Bigg), \mathcal{O}_{P}\big(\zeta(\gamma)\sqrt{(s+1)\log p/L}\big), \quad\text{and}\quad\\
&\mathcal{O}_{P} \Bigg(\zeta(\gamma)\Bigg[\sqrt{\frac{s\log p}{L-L_0}}+\sqrt{\frac{s\log p}{L_0}}+\sqrt{\frac{s(\log p)^2}{L_0(L-L_0)}}\Bigg]\Bigg) 
\end{align*}
respectively, where $\zeta(\gamma)=\mathcal{O}\Big((1+\sqrt{\gamma})^2\big((1+\gamma)^2+\gamma^2(1+\sqrt{\gamma})^2\big)\Big)$. See Remark~\ref{rmk_track_gamma} for a justification. 
\end{rmk}

%%% To be cleared and organzied for experiment section
\section{Numerical Experiments}\label{MTL_experiment_section}

We now provide numerical simulation illustrating the proposed approach. The codes for all experiments could be found at
\begin{center}
    \href{https://github.com/yanhaojin/Generalized-Ridge-Regression-for-Meta-Learning}{https://github.com/yanhaojin/Generalized-Ridge-Regression-for-Meta-Learning.}
\end{center}
For the simulation in this section, Algorithm \ref{alg_simulation_metalearning} is performed for every choice of dimension $p$, number of samples $n_{\ell}$ in each task, number of samples in the new task, total number of tasks $L$.

\begin{algorithm}[t]
\caption{Simulation for Meta-learning}\label{alg_simulation_metalearning}
\begin{algorithmic}
\FOR{each run from 1 to 50}
\STATE Generate data matrix $Z^{(\ell)}$ for $l$-th task ($l=1,\dots,L$) whose entries are i.i.d. sampled from Gaussian $N(0,1)$.
\STATE Compute $X^{(\ell)}=Z^{(\ell)}\Sigma^{(\ell)\frac{1}{2}}$ for $l=1,\dots,L$
\STATE Generate the coefficient $\bar{\beta}^{(\ell)}$ from $N(0,\frac{1}{p}\Omega)$ and $\varepsilon^{(\ell)}$ from $N(0,\sigma^2 I)$ with $\sigma^2=1$.
\STATE Generate $y^{(\ell)}$ based on $y^{(\ell)}=X^{(\ell)}\bar{\beta}^{(\ell)}+\varepsilon^{(\ell)}$ for $\ell=1,\dots,L$
\STATE To compute the matrix $\hat{\Omega}$ by running RGD Algorithm \ref{riemanniangd_opt} depending on:

\quad\quad\textbf{If} unregularized estimator is used, \textbf{then} $f(x)$ is given by \eqref{optimization_for_Omega}. 

\quad\quad\textbf{If} $L$-1 regularized estimator is used, \textbf{then} the $f(x)$ is given by \eqref{L1regularized_estimator_Omega}

%\textcolor{red}{use if arguement/ for regularized or unregularized}
\STATE For the new task ${L+1}$, generate the training data $X^{(L+1)},y^{(L+1)}$ in the same way as previous tasks.  
\STATE Compute the estimator of $\bar{\beta}_{\lambda}^{(L+1)}$ by \eqref{estimator_generalized_ridge} and the predictive risk on the test data from new task. 
\ENDFOR
\end{algorithmic}
\end{algorithm}

For our initial experiments, the hyper-covariance matrix of the coefficients, $\Omega$, as in~\eqref{eq:examplematrix} with $a=16$ and $b=5$, 
\iffalse
is given by 
\begin{align}\label{choice1_of_Omega}
\Omega=\begin{bmatrix}
16 & 5 & 0 & \dots & 0 \\
5  & 16& 5 & \dots & 0 \\
0  & 5 & 16& \dots & 0 \\
\vdots &\ddots & \ddots & \ddots & \vdots\\
0  & \dots & 5 & 16 & 5\\
0  & \dots & 0 & 5 & 16
\end{bmatrix},  
\end{align} 
\fi
and $\Sigma^{(\ell)}=I$ for all $\ell=1,\dots,L,L+1$. According to \cite{elliott1953characteristic}, the eigenvalues of this $p\times p$ matrix is given by $\lambda_k=16+10 \cos \frac{k \pi}{p+1}\in [6,26]$ . Notably, the conditions in Assumption \ref{asp_Omegahat} are verified for this setting.  In our experimental setup, problems \eqref{optimization_for_Omega} or \eqref{L1regularized_estimator_Omega} demands numerical methods. To tackle \eqref{optimization_for_Omega}, we implement Riemannian gradient descent utilizing the \textsf{Pymanopt} package by \cite{townsend2016pymanopt}, as detailed in Algorithm \ref{riemanniangd_opt}. For \eqref{L1regularized_estimator_Omega}, we adopt a Riemannian proximal gradient method. To do so, note that \eqref{L1regularized_estimator_Omega} has the structure $h(\tilde{\Omega})=f(\tilde{\Omega})+\psi(\tilde{\Omega})$ where $f(\tilde{\Omega})=\frac{1}{L}\sum_{\ell=1}^{L}\big\|y^{(\ell)}y^{(\ell)\top}-\frac{1}{p}X^{(\ell)}\tilde{\Omega}X^{(\ell)\top}-\sigma^2 I\big\|_{F}^2$ is differentiable part and $\psi(\tilde{\Omega})=\tilde{\lambda}\sum_{i\not=j}|\tilde{\Omega}_{ij}|$ is non-smooth part. Hence, Riemannian proximal methods are immediately applicable \citep{huang2022riemannian}. %In ,  
Let 
\begin{align}\label{eq:lfunction}
L_{\Omega_k}(\eta)=\langle\operatorname{grad} f\left(\Omega_k\right), \eta\rangle_{\Omega_k}+\frac{\tilde{L}}{2}\|\eta\|_{\Omega_k}^2+\psi\big(P_{\Omega_k}(\eta)\big),
\end{align} 
where $\tilde{L}>L$ serves as a constant larger than the smooth parameter $L$ of $f(\tilde{\Omega})$. This allows us to employ proximal Riemannian gradient descent, which is employed in Algorithm \ref{riemanniangd_opt}.

\begin{algorithm}[t]
\caption{(Proximal) Riemannian Gradient Descent}\label{riemanniangd_opt}
\begin{algorithmic}
\STATE Given the retraction $P_{\tilde{\Omega}}(\Xi)$, the Riemannian gradient descent (RGD) iterates
\REQUIRE $\Omega_{0}\in\mathbb{S}_{p}^{+}$
\FOR{For $k=0,1,2,\dots$,}
\STATE \textbf{If} simple Riemannian gradient descent is used, \textbf{then} pick a step-size $\alpha>0$, and update: $$\Omega_{k+1}=P_{\Omega_{k}}\big(-\alpha \operatorname{grad}f(\Omega_{k})\big)$$ 
\STATE \textbf{If} proximal Riemannian gradient descent is used, \textbf{then} $\Omega_{k+1}=P_{\Omega_k}\big(\eta_{\Omega_k}^*\big)$ where $\eta_{\Omega_k}^*$ is a stationary point of $L_{\Omega_k}(\eta)$ on $\mathcal{T}_{\Omega_k} \mathbb{S}_{p}^{+}$ and $L_{\Omega_k}(0) \geq L_{\Omega_k}(\eta_{\Omega_k}^*)$, where the function $L$ is as in~\eqref{eq:lfunction}.
\ENDFOR
\STATE where $\operatorname{grad}f(x)$ is the Riemannian gradient defined in \eqref{Riemannian_grad}.
\end{algorithmic}
\end{algorithm}

All results reported in our experiments are averaged over $50$ random runs. In each experiment, the predictive risk using identity matrix $\prisk(I\mid X^{(L+1)})$, the predictive risk $\prisk(\hat{\Omega}\mid X^{(L+1)})$ using $\hat{\Omega}$ and the limiting risk $r(\lambda,\gamma_{L+1})$ are reported. In addition, the $\|\hat{\Omega}-\Omega\|_F$ is reported for the experiment in section \ref{Exp_change_L}. In each random run, the predictive risk is approximated by averaging the squared $l_2$ norm of predicted and true value of $y$ over 200 independent new samples in new task. Besides, the limiting risk $r(\lambda,\gamma_{L+1})$ is approximated in following way: For each choice of $p$ and $n_{L+1}$, we choose a surrogate version of $p$ and $n_{L+1}$, denoted by $\tilde{p}$ and $\tilde{n}_{L+1}$, such that $\frac{\tilde{p}}{\tilde{n}_{L+1}}=\frac{p}{n_{L+1}}$. Then the surrogate covariance matrix $\tilde{\Sigma}^{(L+1)},\tilde{\Omega}\in\mathbb{R}^{\tilde{p}\times \tilde{p}}$ is generated and the surrogate data $\tilde{X}^{(L+1)}\in\mathbb{R}^{\tilde{n}_{L+1}\times \tilde{p}},\tilde{y}^{(L+1)}\in\mathbb{R}^{\tilde{n}_{L+1}}$ is generated based on $\tilde{\Sigma}$ and $\tilde{\Omega}$. In \eqref{oracle_limiting_risk}, the limiting risk $r(\lambda,\gamma_{L+1})$ mainly depends on the Stieltjes transform $s$ and its derivative $s'$. $s(-\lambda)$ and $s'(-\lambda)$ could be approximated by
\begin{align*}
    \hat{s}(-\lambda)&=\frac{1}{\tilde{n}_{L+1}}\operatorname{tr}\Big(\big(\frac{1}{\tilde{n}_{L+1}}\tilde{\Omega}^{\frac{1}{2}} \tilde{X}^{(L+1)\top}\tilde{X}^{(L+1)}\tilde{\Omega}^{\frac{1}{2}}+\lambda I_{\tilde{n}_{L+1}}\big)^{-1}\Big)\\
    \hat{s}'(-\lambda)&=\frac{1}{\tilde{n}_{L+1}}\operatorname{tr}\Big(\big(\frac{1}{\tilde{n}_{L+1}}\tilde{\Omega}^{\frac{1}{2}} \tilde{X}^{(L+1)\top}\tilde{X}^{(L+1)}\tilde{\Omega}^{\frac{1}{2}}+\lambda I_{\tilde{n}_{L+1}}\big)^{-2}\Big),
\end{align*}
and $r(\lambda,\gamma_{L+1})$ could be approximated by 
\begin{align*}%\label{oracle_limiting_risk_est}
\frac{1}{\lambda \gamma_{L+1} \hat{s}(-\lambda)+(1-\gamma_{L+1})}\Big[\sigma^2+\big(\frac{\lambda}{\gamma_{L+1}}-\sigma^2\big) \frac{\lambda^2 \gamma_{L+1} \hat{s}^{\prime}(-\lambda)+(1-\gamma_{L+1})}{\gamma_{L+1} \lambda \hat{s}(-\lambda)+(1-\gamma_{L+1})}\Big].
\end{align*}  
Finally, the difference percentage of the risk is computed by
\begin{equation*}
    \frac{R_{\lambda}(\hat{\Omega}\mid X^{(L+1)})-r(\lambda,\gamma_{L+1})}{r(\lambda,\gamma_{L+1})}\times 100\%.
\end{equation*}
%\kb{Yanhao: Change the word ExcessRisk in the table. Add a definition of that here.}

\subsection{Unregularized Setting}

\subsubsection{Estimation of\,\,$\Omega$ changing the number of tasks $L$}\label{Exp_change_L}
In the first part of simulation, we investigate how the error of estimator $\hat{\Omega}$ changes as the number of tasks $L$ increases, when the number of samples $n_{\ell}$ for first $L$ tasks are less than dimension $p$. In this part, we fix dimension $p=128$, the number of samples in previous $L$ tasks $n_{\ell}=100$. The total number of tasks $L$ varies from $L=100,500,1000,5000,10000$. The results are given in Table \ref{tab1_experiment_low_dim_change_L}. In scenarios where the number of tasks is limited, the estimator exhibits a substantial error $\|\hat{\Omega}-\Omega\|_F$ in terms of the Frobenius norm. Additionally, the predictive risk incurred by the estimator $\hat{\Omega}$ turns to be inferior to that using the identity matrix, which totally ignores estimating the hyper-covariance matrix modeling the task similarity. However, as the number of tasks $L$ increases, the error between $\hat{\Omega}$ and $\Omega$ diminishes, leading to a significant reduction in \emph{difference percentage} showing the benefit of incorporating estimating the hyper-covariance matrix explicitly for prediction.

\begin{table}[t]
    \centering
    \begin{tabular}{|c|c|c|c|c|c|}
    \hline
       $L$ & $\|\hat{\Omega}-\Omega\|_{F}$ & $R(I\mid X)$ & $R(\hat{\Omega}\mid X)$ &  $r(\lambda,\gamma_{L+1})$ & Difference Percentage\\ 
       \hline\hline
        $100$   & 366.53 & 9.93  & 16.93  & 5.85 & 189.22\% \\\hline
        $500$   & 184.19 & 9.8045  & 13.24  & 5.85 & 126.21\% \\\hline
        $1000$  & 136.59 & 9.71  & 9.01   & 5.85 & 53.93\%\\\hline
        $5000$  & 63.05  & 9.80  & 6.23   & 5.85 & 6.52\%\\\hline
        $10000$ & 42.04  & 9.87  & 5.93   & 5.85 & 1.45\%\\\hline
    \end{tabular}
    \caption{Frobenius norm of $\hat{\Omega}-\Omega$ and prediction risk on new task, with $p=128$, $n_{\ell}=100$, (for $\ell=1,\dots,L,L+1$), for  $L=100,500,1000,5000,10000$.}
\label{tab1_experiment_low_dim_change_L}
\end{table}

\subsubsection{Behavior of predictive risk based on \eqref{optimization_for_Omega} when changing $n_{L+1}$}

In the second part of the experiment, we fixed the number of task $L=10000$ to guarantee a good approximation for $\Omega$ and we consider the high dimensional case. In this case, the dimension $p$ fixed to be $128$, the number of previous tasks $L=10000$ and set all $n_{\ell}=50$ ($\ell=1,\dots,L$) to be same and vary $n_{L+1}$ from 25, 50, 75, 100, 125, 150. In this part, the initialization of optimization process \eqref{optimization_for_Omega} is given by five different matrices (identity matrix and four different randomly generated positive definite matrices). These results are given in Table \ref{Tab1_of_MoM_four_initialization}. Given an adequate number of training tasks, the predictive risk associated with the estimator $\hat{\Omega}$ demonstrates superior performance under various choices of $n_{L+1}$ compared to the risk incurred using the identity matrix. Furthermore, the predictive risk using $\hat{\Omega}$ consistently approaches the limiting risk with relatively small \emph{difference percentage}. Notably, the results exhibit similarity across different initializations of the optimization problem \eqref{optimization_for_Omega}, affirming benefit of geodesic convexity of \eqref{optimization_for_Omega} and its insensitivity to initialization.

\begin{table}[t]
    \begin{subtable}{1\textwidth}
    \centering
    \begin{tabular}{|c|c|c|c|c|}
    \hline 
       \specialcell{$n_{L+1}$} & \specialcell{$R(I\mid X)$} & \specialcell{$R(\hat{\Omega}\mid X)$} &\specialcell{$r(\lambda,\gamma_{L+1})$} &\specialcell{Difference Percentage} \\
        \hline \hline 
        $25$     & 14.11  & 13.89  & 13.58 & 2.25\%\\\hline 
        $50$     & 12.34  & 11.24  & 10.58 & 6.22\%\\\hline 
        $75$     & 11.60  & 8.41  & 7.94  & 5.91\%\\\hline 
        $100$    & 9.76  & 6.53  & 5.85  & 11.71\%\\\hline 
        $125$    & 8.12  & 4.80  & 4.32  & 11.13\%\\\hline  
        $150$    & 7.25  & 3.57  & 3.34  & 7.02\%\\\hline 
        \end{tabular} 
        \caption{Initialization: Identity matrix}
    \end{subtable}
    \\
    \begin{subtable}{1\textwidth}
    \centering
    \begin{tabular}{|c|c|c|c|c|}
    \hline 
       \specialcell{$n_{L+1}$} & \specialcell{$R(I\mid X)$} & \specialcell{$R(\hat{\Omega}\mid X)$} &\specialcell{$r(\lambda,\gamma_{L+1})$} &\specialcell{Difference Percentage} \\
        \hline \hline 
        $25$     & 14.11  & 13.91  & 13.58 & 2.42\%\\\hline 
        $50$     & 12.34  & 11.40  & 10.58 & 7.73\%\\\hline 
        $75$     & 11.60  & 8.62  & 7.94  & 8.56\%\\\hline 
        $100$    & 9.76  & 6.77  & 5.85  & 15.67\%\\\hline 
        $125$    & 8.12  & 4.91  & 4.32  & 13.72\%\\\hline 
        $150$    & 7.25  & 3.71  & 3.34  & 10.99\%\\\hline 
        \end{tabular} 
        \caption{Initialization: First randomly generated SPD matrix}
    \end{subtable}
    \\ 
    \begin{subtable}{1\textwidth}
    \centering
    \begin{tabular}{|c|c|c|c|c|}
    \hline 
       \specialcell{$n_{L+1}$} & \specialcell{$R(I\mid X)$} & \specialcell{$R(\hat{\Omega}\mid X)$} &\specialcell{$r(\lambda,\gamma_{L+1})$} &\specialcell{Difference Percentage} \\
        \hline \hline 
        $25$     & 14.11  & 14.00  & 13.58 &3.07\% \\\hline 
        $50$     & 12.34 & 11.89  & 10.58 & 12.37\%\\\hline 
        $75$     & 11.60  & 9.10  & 7.94  & 14.51\%\\\hline 
        $100$    & 9.76  & 7.01  & 5.85  & 19.83\%\\\hline 
        $125$    & 8.12  & 5.11  & 4.32  & 18.25\%\\\hline 
        $150$    & 7.25  & 3.92  & 3.34  & 17.25\%\\\hline 
        \end{tabular} 
        \caption{Initialization: Second randomly generated SPD matrix}
    \end{subtable}
    \\`
    \begin{subtable}{1\textwidth}
    \centering
    \begin{tabular}{|c|c|c|c|c|}
    \hline 
       \specialcell{$n_{L+1}$} & \specialcell{$R(I\mid X)$} & \specialcell{$R(\hat{\Omega}\mid X)$} &\specialcell{$r(\lambda,\gamma_{L+1})$} &\specialcell{Difference Percentage} \\
        \hline \hline 
        $25$     & 14.11  & 13.90  & 13.58 &2.33\% \\\hline 
        $50$     & 12.34 & 11.36  & 10.58 & 7.40\%\\\hline 
        $75$     & 11.60  & 8.53  & 7.94  & 7.45\%\\\hline 
        $100$    & 9.76  & 6.69   & 5.85  & 14.42\%\\\hline 
        $125$    & 8.12  & 4.86  & 4.32  & 12.47\%\\\hline 
        $150$    & 7.25  & 3.64 & 3.34  & 8.95\%\\\hline 
        \end{tabular} 
        \caption{Initialization: Third randomly generated SPD matrix}
    \end{subtable}
\\     
  \begin{subtable}{1\textwidth}
    \centering
    \begin{tabular}{|c|c|c|c|c|}
    \hline 
       \specialcell{$n_{L+1}$} & \specialcell{$R(I\mid X)$} & \specialcell{$R(\hat{\Omega}\mid X)$} &\specialcell{$r(\lambda,\gamma_{L+1})$} &\specialcell{Difference Percentage} \\
        \hline \hline 
        $25$     & 14.11  & 13.97  & 13.58 &2.89\% \\\hline 
        $50$     & 12.34  & 11.56  & 10.58 & 9.26\%\\\hline 
        $75$     & 11.60  & 8.91  & 7.94  & 12.16\%\\\hline 
        $100$    & 9.76  & 6.81  & 5.85  & 16.37\%\\\hline 
        $125$    & 8.12  & 5.01  & 4.32  & 16.09\%\\\hline 
        $150$    & 7.25  & 3.87  & 3.34  & 15.80\%\\\hline 
        \end{tabular} 
        \caption{Initialization: Fourth randomly generated SPD matrix}
    \end{subtable}
\caption{Prediction risk when $\hat{\Omega}$ is estimated based on \eqref{optimization_for_Omega}, with 5 different initialization. The max running time for optimizing \eqref{optimization_for_Omega} is 360 minutes.}
\label{Tab1_of_MoM_four_initialization}
\end{table} 

\subsubsection{Behavior of predictive risk based on MLE when changing $n_{L+1}$}

In the third part of the experiment in this section, we consider the estimator of the covariance matrix $\hat{\Omega}$ given by MLE approach. The initialization of the optimization is given by identity matrix, four different randomly generated positive definite matrices same as previous case, and the estimator given by \eqref{optimization_for_Omega} with identity as initialization. These results given in Table \ref{Tab1_of_MLE_four_initialization}. The predictive risk results obtained using the Maximum Likelihood estimator (MLE)  exhibits significant variability based on different choices of initializations. Specifically, the performance of the predictive risk using the MLE is notably poor, characterized by a large \emph{difference percentage}, when employing four randomly generated symmetric positive definite matrices as initialization. This undesirable behavior arises due to the lack of global geodesic convexity in the optimization problem aimed at minimizing the negative log-likelihood function. Diverse initialization choices may lead the solution to converge to local minima during Riemannian gradient descent. In contrast, performing MLE with an initialization given by the identity matrix, or the output obtained from \eqref{optimization_for_Omega}, yields favorable results. This is attributed to the initialization's proximity to the global minimum of the negative log-likelihood function, resulting in good predictive performance with minimal \emph{difference percentage}.

\begin{table} [t]
\begin{subtable}{1\textwidth}
    \centering
    \begin{tabular}{|c|c|c|c|c|}
    \hline
       \specialcell{$n_{L+1}$} & \specialcell{$R(I\mid X)$} & \specialcell{$R(\hat{\Omega}\mid X)$} &\specialcell{$r(\lambda,\gamma_{L+1})$} &\specialcell{Difference Percentage} \\
        \hline \hline
        $25$     & 14.11  & 13.84  & 13.58 & 1.91\%\\
        \hline
        $50$     & 12.34  & 11.14  & 10.58 & 5.33\%\\
        \hline
        $75$     & 11.60  & 8.26  & 7.94  & 4.02\%\\
        \hline
        $100$    & 9.76  & 6.21  & 5.85  & 6.14\%\\
        \hline
        $125$    & 8.12  & 4.52  & 4.32  & 4.62\%\\
        \hline
        $150$    & 7.25  & 3.49  & 3.34  & 4.41\%\\
        \hline
        \end{tabular}
        \caption{Initialization: identity matrix}
    \end{subtable}  
    \\
    \begin{subtable}{1\textwidth}
    \centering
    \begin{tabular}{|c|c|c|c|c|}
    \hline 
       \specialcell{$n_{L+1}$} & \specialcell{$R(I\mid X)$} & \specialcell{$R(\hat{\Omega}\mid X)$} &\specialcell{$r(\lambda,\gamma_{L+1})$} &\specialcell{Difference Percentage} \\
        \hline \hline 
        $25$     & 14.11  & 18.25  & 13.58 & 34.36\%\\
        \hline 
        $50$     & 12.34  & 15.91  & 10.58 & 50.35\%\\
        \hline 
        $75$     & 11.60  & 14.08  & 7.94  & 77.22\%\\
        \hline 
        $100$    & 9.76  & 10.65  & 5.85  & 82.07\%\\
        \hline 
        $125$    & 8.12  & 6.93  & 4.32  & 60.30\%\\
        \hline 
        $150$    & 7.25 & 6.64 & 3.34  & 98.57\%\\
        \hline 
        \end{tabular}
        \caption{Initialization: First randomly generated SPD matrix}
    \end{subtable}
    \\
    \begin{subtable}{1\textwidth}
    \centering
    \begin{tabular}{|c|c|c|c|c|}
    \hline 
    \specialcell{$n_{L+1}$} & \specialcell{$R(I\mid X)$} & \specialcell{$R(\hat{\Omega}\mid X)$} &\specialcell{$r(\lambda,\gamma_{L+1})$} &\specialcell{Difference Percentage} \\
        \hline \hline 
        $25$     & 14.11  & 19.59 & 13.58 & 44.27\%\\
        \hline 
        $50$     & 12.34  & 17.03  & 10.58 & 61.00\%\\
        \hline 
        $75$     & 11.60  & 15.84  & 7.94  & 99.40\%\\
        \hline 
        $100$    & 9.76  & 11.19  & 5.85  & 91.26\%\\
        \hline 
        $125$    & 8.12  & 7.32  & 4.32  & 69.30\%\\
        \hline 
        $150$    & 7.25 & 6.81 & 3.34  & 103.65\%\\
        \hline 
        \end{tabular}
        \caption{Initialization: Second randomly generated SPD matrix}
    \end{subtable}  
\caption{Prediction risk when $\hat\Omega$ is the MLE in~\eqref{loglikelihood_for_Omega}, with 5 different initializations. The max running time for MLE iteration is 120 minutes and the max running time for optimizing \eqref{optimization_for_Omega} is 360 minutes. \texttt{(continued in next page)}}
\label{Tab1_of_MLE_four_initialization}  
\end{table}
%%%%%%%%
\begin{table}
\ContinuedFloat
    \begin{subtable}{1\textwidth}
    \centering
    \begin{tabular}{|c|c|c|c|c|}
    \hline 
       \specialcell{$n_{L+1}$} & \specialcell{$R(I\mid X)$} & \specialcell{$R(\hat{\Omega}\mid X)$} &\specialcell{$r(\lambda,\gamma_{L+1})$} &\specialcell{Difference Percentage} \\
        \hline \hline 
        $25$     & 14.11  & 16.59  & 13.58 & 22.19\%\\\hline 
        $50$     & 12.34  & 15.07  & 10.58 & 42.43\%\\\hline 
        $75$     & 11.60  & 12.89  & 7.94  & 62.24\%\\\hline 
        $100$    & 9.76  & 9.92  & 5.85  & 69.54\%\\\hline 
        $125$    & 8.12  & 7.03  & 4.32  & 62.66\%\\\hline 
        $150$    & 7.25 & 6.08 & 3.34  & 81.92\%\\\hline 
        \end{tabular}
        \caption{Initialization: Third randomly generated SPD matrix}
    \end{subtable}
    \\
    \begin{subtable}{1\textwidth}
    \centering
    \begin{tabular}{|c|c|c|c|c|}
    \hline 
       \specialcell{$n_{L+1}$} & \specialcell{$R(I\mid X)$} & \specialcell{$R(\hat{\Omega}\mid X)$} &\specialcell{$r(\lambda,\gamma_{L+1})$} &\specialcell{Difference Percentage} \\
        \hline \hline 
        $25$     & 14.11 & 18.29  & 13.58 & 34.68\%\\\hline 
        $50$     & 12.34  & 16.37  & 10.58 & 54.69\%\\\hline 
        $75$     & 11.60  & 15.07  & 7.94  & 89.69\%\\\hline 
        $100$    & 9.76  & 10.96  & 5.85  & 87.23\%\\\hline 
        $125$    & 8.12  & 7.45   & 4.32  & 72.34\%\\\hline 
        $150$    & 7.25  & 6.71  & 3.34  & 100.83\%\\\hline 
        \end{tabular} 
        \caption{Initialization: Fourth randomly generated SPD matrix}
    \end{subtable}
    \\
    \begin{subtable}{1\textwidth}
    \centering
    \begin{tabular}{|c|c|c|c|c|}
    \hline 
       \specialcell{$n_{L+1}$} & \specialcell{$R(I\mid X)$} & \specialcell{$R(\hat{\Omega}\mid X)$} &\specialcell{$r(\lambda,\gamma_{L+1})$} &\specialcell{Difference Percentage} \\
        \hline \hline 
        $25$     & 14.11  & 13.75  & 13.58 & 1.22\%\\\hline 
        $50$     & 12.34  &  11.01 & 10.58 & 4.08\%\\\hline 
        $75$     & 11.60  & 8.16  & 7.94  & 2.79\%\\\hline 
        $100$    & 9.76  & 6.12  & 5.85  & 4.70\%\\\hline 
        $125$    & 8.12  & 4.50  & 4.32  & 4.1910\%\\\hline 
        $150$    & 7.25 & 3.36 & 3.34  & 0.69\%\\\hline 
        \end{tabular}
        \caption{Initialization: Output given by problem \eqref{optimization_for_Omega}}
    \end{subtable} 
\caption{\texttt{(Continuation from previous page}) Prediction risk when $\hat\Omega$ is the MLE in~\eqref{loglikelihood_for_Omega}, with 5 different initializations. The max running time for MLE iteration is 120 minutes and the max running time for optimizing \eqref{optimization_for_Omega} is 360 minutes.} 
\end{table}

\subsection{$L_1$ Regularized Setting}

In our next set of experiments, we estimate $\Omega$ by $L_1$ regularization using \eqref{L1regularized_estimator_Omega}.  Algorithm \ref{alg_simulation_metalearning} is perform based on Riemannian optimization for problem \eqref{L1regularized_estimator_Omega}. In this experiment, settings for dimension $p$, choice of $\Omega$ and $\Sigma^{(\ell)}$ and $n_{L+1}$ are the same as the general setting at the beginning of Section \ref{MTL_experiment_section}. The main difference in the experimental settings compared to the previous case lies in the number of samples within the tasks and the total number of tasks. In this experiment, we have reduced the number of tasks $L$ to $1000$, a significantly smaller quantity than in the prior scenario. Regarding the number of samples for the tasks, we considered two cases:
\begin{itemize}
    \item Equal Sample Size: all tasks ($\ell=1,\dots,L$) have an identical sample size, specifically set to $n_{\ell}=50$.
    \item Variable Sample Sizes: we adopted a varied approach. For tasks $\ell=1,\dots,200$, we set the sample size to $n_{\ell}=150$, whereas for tasks $\ell=201,\dots,1000$, the sample size was $n_{\ell}=50$.
\end{itemize}
The results for these two cases are given in Table \ref{tab1_L1regularized_experiment_high_dim_change_n_L+1} and \ref{tab2_L1regularized_experiment_high_dim_change_n_L+1}. The results indicates that we could achieve comparative results on the predictive risk using much less number of tasks based on \eqref{L1regularized_estimator_Omega} than that based on \eqref{optimization_for_Omega}. Besides, if we have sufficient number of samples in a proportion of tasks, the behavior of predictive risk $\prisk(\hat{\Omega}\mid X^{(L+1)})$ is slightly better than that when all tasks have same number of samples $n_{\ell}=50$.

We conclude this section by highlighting that in Section~\ref{sec:addexp}, we provide additional experiments specifically for the cases when the assumptions required for the theoretical results are violated. Specifically, we consider the case when the covariance matrices have eigenvalues that decay to zero as the dimension goes to infinity. We note from our results that the proposed approach performs well even in such cases. 

\subsubsection*{Acknowledgement}
YJ and KB were partially supported by the National Science Foundation (NSF) Grant DMS-2053918. DP was partially supported by NSF grant DMS-1915894.

%\todo{add a sentence of proximal stuff and give a citation. no need to add the full algorithm}\textcolor{red}{write which steps changes in the algorithm and cite equations or packages} 

\begin{table} [t]
\begin{subtable}{1\textwidth}
    \centering
    \begin{tabular}{|c|c|c|c|c|}
    \hline
       $n_{L+1}$ &  $R(I\mid X)$ & $R(\hat{\Omega}\mid X)$ &  $r(\lambda,\gamma_{L+1})$ & Difference Percentage \\
       \hline\hline
        $25$     & 14.95  & 14.03  & 13.58 & 3.70\%\\\hline
        $50$     & 12.57  & 10.74  & 10.58 & 3.92\%\\\hline
        $75$     & 10.99  & 8.51   & 7.94  & 7.32\%\\\hline
        $100$    & 9.70    & 6.65   & 5.85  & 14.37\%\\\hline
        $125$    & 8.63   & 5.04   & 4.32  & 17.23\%\\\hline
        $150$    & 7.26   & 3.65  &  3.34  & 9.27\%\\\hline
    \end{tabular}
    \caption{$n_{\ell}=50$ for all $\ell=1,\dots, L$}
    \label{tab1_L1regularized_experiment_high_dim_change_n_L+1}
    \end{subtable} 
    \medskip
    \begin{subtable}{1\textwidth}
    \centering
    \begin{tabular}{|c|c|c|c|c|}
    \hline
       $n_{L+1}$ &  $R(I\mid X)$ & $R(\hat{\Omega}\mid X)$ &  $r(\lambda,\gamma_{L+1})$ & Difference Percentage \\
       \hline\hline
        $25$     & 14.95  & 13.87 & 13.58 &  2.15\%\\\hline
        $50$     & 12.57  & 10.52   & 10.58 & -0.58\%\\\hline
        $75$     & 10.99  & 8.43  & 7.94  & 6.17\% \\\hline
        $100$    & 9.70   & 6.63   & 5.85  & 13.26\% \\\hline
        $125$    & 8.63   & 5.01   & 4.32  & 16.09\% \\\hline
        $150$    & 7.26   & 3.65   & 3.34  & 9.20\% \\\hline
    \end{tabular}
    \caption{$n_{\ell}=150$ for $\ell\leq 200$ and $n_{\ell}=50$ for $\ell>200$.}
    \label{tab2_L1regularized_experiment_high_dim_change_n_L+1}  
    \end{subtable}
    \caption{Prediction risk when $\hat\Omega$ is estimated based on~\eqref{L1regularized_estimator_Omega}, with $p=128$, $L=1000$. The regularization parameter is set as $\lambda=0.0004$ (The initial point is $I$).}\label{Tab_of_l1_regularization}
\end{table}

%\section{Discussion}

%In this work, 

\iffalse
\begin{figure}[t]
	\centering
	\includegraphics[width=12cm]{R_Omegahat_different_c.png}
\end{figure}
\fi
%% coel = c(0.8,0.85,0.9,0.95,1,1.05,1.1,1.15,1.2)
%% R_hatOmega1 =  15.7867,15.1924,14.7452,14.1936,14.0378,14.0392,14.2649,14.5953,15.0089 
%% R_hatOmega2 =  12.2743,11.7829,11.1274,10.8982,10.7492,10.7089,10.7692,11.0841,11.4653 
%% R_hatOmega3 =  9.7536,9.1285,8.7092,8.4028,8.5146,8.7902,9.0217,9.3891,9.6728 
%% R_hatOmega4 =  7.6956,7.1957,6.8673,6.5992,6.6527,6.8981,7.1659,7.3280,7.6376
%% R_hatOmega5 = 6.2353,5.8902,5.5789,5.2391,5.0481,5.1485,5.4873,5.7992,6.1028
%% R_hatOmega6 = 4.1676,3.9863,3.7538,3.6794,3.6540,3.6937,3.7521,3.8706,4.0278

\bibliographystyle{abbrvnat}
\bibliography{mybib.bib}

\clearpage
\appendix 
\section{Proofs for Section \ref{MTL_risk_section}}
In order to make the manipulations more easily readable, in our proofs we will explicitly write $\Lambda^{(L+1)}$, $\widehat{\Lambda}^{(L+1)}$, $\widecheck{\Lambda}^{(L+1)}$ and $\widetilde{\Lambda}^{(L+1)}$ from~\eqref{eq:importantmatrix} and~\eqref{eq:importantmatrixest}. 
\subsection{Derivation of Predictive Risk}
\begin{proof}[Proof of Theorem \ref{thm_predictive_risk}]
We first calculate the predictive risk using oracle estimator $\tilde{\beta}_{\lambda}^{(L+1)}$ in \eqref{oracle_estimator_generalized_ridge}. Let $(x,y)$ be the new test sample whose distribution is the same as training data in $(L+1)$-th task. Note that, we then have
\begin{align*}
    y-\tilde{y} &=x^{\top} \bar{\beta}^{(L+1)}+\varepsilon^{(L+1)}-x^{\top} \tilde{\beta}_{\lambda}^{(L+1)} =x^{\top}\big(\bar{\beta}^{(L+1)}-\tilde{\beta}_{\lambda}^{(L+1)}\big)+\varepsilon^{(L+1)}.
\end{align*}
Therefore,
\begin{align*}
    &\mathbb{E}\Big[\big(x^{\top}\big(\bar{\beta}^{(L+1)}-\tilde{\beta}_{\lambda}^{(L+1)}\big)+\varepsilon^{(L+1)}\big)^2 \mid X^{(L+1)}\Big]\\
    =&\mathbb{E}\Big[\big(\varepsilon^{(L+1)}\big)^2 \mid X^{(L+1)}\Big]+\mathbb{E}\Big[\big(x^{\top}\big(\bar{\beta}^{(L+1)}-\tilde{\beta}_{\lambda}^{(L+1)}\big)\big)^2\mid X^{(L+1)}\Big]\\
    =&\sigma^2+\mathbb{E}\Big[\big(\bar{\beta}^{(L+1)}-\tilde{\beta}_{\lambda}^{(L+1)}\big)^{\top} xx^{\top}\big(\bar{\beta}^{(L+1)}-\tilde{\beta}_{\lambda}^{(L+1)}\big) \mid X^{(L+1)}\Big]\\
    =&\sigma^2+\mathbb{E}\Big[\big(\bar{\beta}^{(L+1)}-\tilde{\beta}_{\lambda}^{(L+1)}\big)^{\top} \Sigma\big(\bar{\beta}^{(L+1)}-\tilde{\beta}_{\lambda}^{(L+1)}\big)\mid X^{(L+1)}\Big].
\end{align*}
By plugging in the expression of $\tilde{\beta}_{\lambda}^{(L+1)}$, it then holds that
\begin{align}
    \bar{\beta}^{(L+1)}-\tilde{\beta}_{\lambda}^{(L+1)}
    &= \lambda \big(\hat{\Sigma}^{(L+1)}+\lambda \Omega^{-1}\big)^{-1} \Omega^{-1} \bar{\beta}^{(L+1)}\nonumber\\
    &\quad\quad-\frac{1}{n_{L+1}}\big(\hat{\Sigma}^{(L+1)}+\lambda \Omega^{-1}\big)^{-1} X^{(L+1)^{\top}} \varepsilon^{(L+1)}. \label{expression1_beta} 
\end{align}
The oracle risk is hence given by
\begin{align*}
\oraclerisk\big(\Omega\mid X^{(L+1)}\big)=&\sigma^2+\mathbb{E}\Big[\big(\bar{\beta}^{(L+1)}-\tilde{\beta}^{(L+1)}\big)^{\top} \Sigma^{(L+1)}\big(\bar{\beta}^{(L+1)}-\tilde{\beta}^{(L+1)}\big)\mid X^{(L+1)}\Big]\\
=&\sigma^2+\lambda^2 \mathbb{E}\Big[\bar{\beta}^{(L+1)^{\top}} \Omega^{-1}\big(\hat{\Sigma}^{(L+1)}+\lambda \Omega^{-1}\big)^{-1}\Sigma^{(L+1)}\\
&\quad\quad\quad\quad\quad\quad\big(\hat{\Sigma}^{(L+1)}+\lambda \Omega^{-1}\big)^{-1} \Omega^{-1} \bar{\beta}^{(L+1)}\mid X^{(L+1)}\Big]\\
&+\frac{1}{n_{L+1}^2}\mathbb{E}\Big[\varepsilon^{(L+1)^\top} X^{(L+1)}\big(\hat{\Sigma}^{(L+1)}+\lambda \Omega^{-1}\big)^{-1} \Sigma^{(L+1)}\\
&\quad\quad\quad\quad\quad\quad\big(\hat{\Sigma}^{(L+1)}+\lambda \Omega^{-1}\big)^{-1} X^{(L+1)^{\top}} \varepsilon^{(L+1)}\mid X^{(L+1)}\Big].
\end{align*}
Using the decomposition
\begin{align*}%\label{decomposition_equ1}
\begin{aligned}
    &\frac{1}{n_{L+1}}\Sigma^{(L+1)}\big(\hat{\Sigma}^{(L+1)}+\lambda \Omega^{-1}\big)^{-1} \hat{\Sigma}^{(L+1)}\big(\hat{\Sigma}^{(L+1)}+\lambda \Omega^{-1}\big)^{-1}\\
    =&\frac{1}{n_{L+1}}\Sigma^{(L+1)}\big(\hat{\Sigma}^{(L+1)}+ \lambda \Omega^{-1}\big)^{-1}\\
    &\quad\quad- \frac{\lambda}{n_{L+1}} \Sigma^{(L+1)}\big(\hat{\Sigma}^{(L+1)}+ \lambda \Omega^{-1}\big)^{-1} \Omega^{-1}\big(\hat{\Sigma}^{(L+1)}+\lambda\Omega^{-1}\big)^{-1},
\end{aligned}    
\end{align*}
and the trace trick, we finally obtain
\begin{align*}
    &\oraclerisk\big(\Omega\mid X^{(L+1)}\big)\\
    =&\sigma^2+\mathbb{E}\Big[\big(\bar{\beta}^{(L+1)}-\tilde{\beta}^{(L+1)}\big)^{\top} \Sigma^{(L+1)}\big(\bar{\beta}^{(L+1)}-\tilde{\beta}^{(L+1)}\big)\mid X^{(L+1)}\Big]\\
    =&\sigma^2+ \frac{\lambda^2}{p}\operatorname{tr}\Big(\Sigma^{(L+1)}\big(\hat{\Sigma}^{(L+1)}+\lambda\Omega^{-1}\big)^{-1}\Omega^{-1}\big(\hat{\Sigma}^{(L+1)}+\lambda\Omega^{-1}\big)^{-1}\Big)\\
    &\quad\quad -\frac{\lambda\sigma^2}{n_{L+1}}  \operatorname{tr}\Big(\Sigma^{(L+1)}\big(\hat{\Sigma}^{(L+1)}+\lambda\Omega^{-1}\big)^{-1}\Omega^{-1}\big(\hat{\Sigma}^{(L+1)}+\lambda\Omega^{-1}\big)^{-1}\Big) \\
    &\quad\quad +\frac{\sigma^2}{n_{L+1}} \operatorname{tr}\big(\Sigma^{(L+1)}\big(\hat{\Sigma}^{(L+1)}+\lambda\Omega^{-1}\big)^{-1}\big)\\
    =&\sigma^2+(\mathsf{I})+(\mathsf{II})+(\mathsf{III}),
\end{align*} 
where these three terms could also be expressed as below
\begin{align*}
    (\mathsf{I}) &= \frac{\lambda^2}{p} \operatorname{tr}\Big(\Omega^{\frac{1}{2}} \Sigma^{(L+1)} \Omega^{\frac{1}{2}}\big(\Omega^{\frac{1}{2}} \hat{\Sigma}^{(L+1)} \Omega^{\frac{1}{2}}+\lambda I\big)^{-2} \Big)\\
    (\mathsf{II}) &=-\frac{\lambda\sigma^2}{n_{L+1}}\operatorname{tr}\Big(\Omega^{\frac{1}{2}} \Sigma^{(L+1)} \Omega^{\frac{1}{2}}\big(\Omega^{\frac{1}{2}} \hat{\Sigma}^{(L+1)} \Omega^{\frac{1}{2}}+\lambda I\big)^{-2} \Big)\\
    (\mathsf{III}) &=\frac{\sigma^2}{n_{L+1}} \operatorname{tr}\Big(\Omega^{\frac{1}{2}}\Sigma^{(L+1)} \Omega^{\frac{1}{2}}\big(\Omega^{\frac{1}{2}} \hat{\Sigma}^{(L+1)} \Omega^{\frac{1}{2}}+\lambda I\big)^{-1} \Big).
\end{align*}
Similar to \eqref{expression1_beta}, it holds that
\begin{align*}
    \bar{\beta}^{(L+1)}-\hat{\beta}_{\lambda}^{(L+1)}  &= \lambda \big(\hat{\Sigma}^{(L+1)}+\lambda \Omega^{-1}\big)^{-1} \hat{\Omega}^{-1}\bar{\beta}^{(L+1)}-\frac{1}{n_{L+1}} \big(\hat{\Sigma}^{(L+1)}+\lambda \Omega^{-1}\big)^{-1} X^{(L+1)^{\top}} \varepsilon^{(L+1)}.
\end{align*}
Therefore, again using the trace trick, we get
\begin{align*}
    &\prisk\big(\hat{\Omega}^{(L+1)}\mid X^{(L+1)}\big)\\
    =&\sigma^2+\mathbb{E}\Big[\big(\bar{\beta}^{(L+1)}-\hat{\beta}_{\lambda}^{(L+1)}\big)^{\top} \Sigma^{(L+1)}\big(\bar{\beta}^{(L+1)}-\hat{\beta}_{\lambda}^{(L+1)}\big)\mid X^{(L+1)}\Big]\\
    =&\sigma^2+ \frac{\lambda^2 }{p} \operatorname{tr}\big(\Omega \hat{\Omega}^{-1} \big(\hat{\Sigma}^{(L+1)}+\lambda \hat{\Omega}^{-1}\big)^{-1} \Sigma^{(L+1)} \big(\hat{\Sigma}^{(L+1)}+\lambda \hat{\Omega}^{-1}\big)^{-1}\hat{\Omega}^{-1}\big) \\
    &\quad\quad+\frac{\sigma^2}{n_{L+1}^{2}} \operatorname{tr}\big(\Sigma^{(L+1)} \big(\hat{\Sigma}^{(L+1)}+\lambda \hat{\Omega}^{-1}\big)^{-1}X^{(L+1)^{\top}} X^{(L+1)} \big(\hat{\Sigma}^{(L+1)}+\lambda \hat{\Omega}^{-1}\big)^{-1}\big).
\end{align*}
Now, the third term could be further decomposed as
\begin{align*}
    &\frac{1}{n_{L+1}}\Sigma^{(L+1)}\big(\hat{\Sigma}^{(L+1)}+\lambda \hat{\Omega}^{-1}\big)^{-1} \hat{\Sigma}^{(L+1)}\big(\hat{\Sigma}^{(L+1)}+\lambda \hat{\Omega}^{-1}\big)^{-1}\\
    =&\frac{1}{n_{L+1}}\Sigma^{(L+1)}\big(\hat{\Sigma}^{(L+1)}+ \lambda \hat{\Omega}^{-1}\big)^{-1}\\
    &\quad- \frac{\lambda}{n_{L+1}} \Sigma^{(L+1)}\big(\hat{\Sigma}^{(L+1)}+ \lambda \hat{\Omega}^{-1}\big)^{-1} \hat{\Omega}^{-1}\big(\hat{\Sigma}^{(L+1)}+\lambda \hat{\Omega}^{-1}\big)^{-1}.
\end{align*}
Therefore, the risk $\prisk(\hat{\Omega}\mid X^{(L+1)})$ could be simplified to
\begin{align*}
   &\prisk\big(\hat{\Omega}\mid X^{(L+1)}\big) \\
   =&\sigma^2+\frac{\lambda^2}{p} \operatorname{tr}\big(\Omega \hat{\Omega}^{-1}\big(\hat{\Sigma}^{(L+1)}+\lambda \hat{\Omega}^{-1}\big)^{-1} \Sigma^{(L+1)}\big(\hat{\Sigma}^{(L+1)}+\lambda \hat{\Omega}^{-1}\big)^{-1} \hat{\Omega}^{-1}\big)\\
   &\quad\quad-\frac{\lambda\sigma^2}{n_{L+1}} \operatorname{tr}\big(\big(\hat{\Sigma}^{(L+1)}+\lambda \hat{\Omega}^{-1}\big)^{-1}\Sigma^{(L+1)}\big(\hat{\Sigma}^{(L+1)}+\lambda \hat{\Omega}^{-1}\big)^{-1}\hat{\Omega}^{-1}\big)\\
   &\quad\quad+\frac{\sigma^2}{n_{L+1}} \operatorname{tr}\big(\Sigma^{(L+1)}\big(\hat{\Sigma}^{(L+1)}+\lambda \hat{\Omega}^{-1}\big)^{-1}\big)\\
   &=\sigma^2+(\mathsf{I}')+(\mathsf{II}')+(\mathsf{III}'),
\end{align*}
where these three terms could also be expressed as below
\begin{align*}
    (\mathsf{I}')&=\frac{\lambda^2}{p} \operatorname{tr}\big(\Omega \hat{\Omega}^{-\frac{1}{2}}\big(\hat{\Omega}^{\frac{1}{2}}\hat{\Sigma}^{(L+1)}\hat{\Omega}^{\frac{1}{2}}+\lambda I\big)^{-1}\hat{\Omega}^{\frac{1}{2}}\Sigma^{(L+1)}\hat{\Omega}^{\frac{1}{2}}\big(\hat{\Omega}^{\frac{1}{2}}\hat{\Sigma}^{(L+1)}\hat{\Omega}^{\frac{1}{2}}+\lambda I\big)^{-1} \hat{\Omega}^{-\frac{1}{2}}\big),\\
    (\mathsf{II}')&=-\frac{\lambda\sigma^2}{n_{L+1}} \operatorname{tr}\big(\hat{\Omega}^{\frac{1}{2}}  \Sigma^{(L+1)} \hat{\Omega}^{\frac{1}{2}}\big(\hat{\Omega}^{\frac{1}{2}}\hat{\Sigma}^{(L+1)}\hat{\Omega}^{\frac{1}{2}}+\lambda I\big)^{-2}\big),\\
    (\mathsf{III}')&=\frac{\sigma^2}{n_{L+1}} \operatorname{tr}\big(\hat{\Omega}^{\frac{1}{2}}\Sigma^{(L+1)}\hat{\Omega}^{\frac{1}{2}}\big(\hat{\Omega}^{\frac{1}{2}}\hat{\Sigma}^{(L+1)}\hat{\Omega}^{\frac{1}{2}}+\lambda I\big)^{-1}\big). 
\end{align*} 
\end{proof}
\subsection{Asymptotic Behavior of Predictive Risk}  

\begin{proof}[Proof of Theorem \ref{asymptotic_behavior_of_predictive_risk}]
We first consider the asymptotic behavior of oracle risk $\oraclerisk\big(\Omega\mid X^{(L+1)}\big)$ as $p,n_{L+1}\rightarrow\infty$ such that $p/n_{L+1}\rightarrow\gamma_{L+1}$. The terms $(\mathsf{I})$ and $(\mathsf{II})$ could be combined together, and hence we have
\begin{align*}
    (\mathsf{I})+(\mathsf{II}) &=\big(\lambda^2 -\lambda \frac{p\sigma^2}{n_{L+1}}\big) \frac{1}{p}\operatorname{tr}\Big(\Omega^{\frac{1}{2}}\Sigma^{(L+1)}\Omega^{\frac{1}{2}}\big(\Omega^{\frac{1}{2}}\hat{\Sigma}^{(L+1)}\Omega^{\frac{1}{2}}+\lambda I\big)^{-2} \Big)\\
    (\mathsf{III})&=\frac{p\sigma^2}{n_{L+1}} \frac{1}{p}\operatorname{tr}\Big(\Omega^{\frac{1}{2}}\Sigma^{(L+1)}\Omega^{\frac{1}{2}}\big(\Omega^{\frac{1}{2}}\hat{\Sigma}^{(L+1)}\Omega^{\frac{1}{2}}+\lambda I\big)^{-1} \Big).
\end{align*}
Define $\tilde{X}^{(L+1)}=X^{(L+1)}\Omega^{\frac{1}{2}}$. Let $v_{L+1}$ be the Stieltjes transform of limiting spectral distribution of $\widetilde{\underline{\Lambda}}_{*}^{(L+1)}=\frac{1}{n_{L+1}} \tilde{X}^{(L+1)} \tilde{X}^{(L+1)^\top}$ and $s_{L+1}$ is the Stieltjes transform of limiting spectral distribution of $\widetilde{\Lambda}^{(L+1)}$. According to \cite{ledoit2011eigenvectors},
\begin{align*}
    (\mathsf{III})=\frac{p\sigma^2}{n_{L+1}} \frac{1}{p} \operatorname{tr}\Big(\Omega^{\frac{1}{2}}\Sigma^{(L+1)}\Omega^{\frac{1}{2}}\big(\Omega^{\frac{1}{2}}\hat{\Sigma}^{(L+1)}\Omega^{\frac{1}{2}}+\lambda I\big)^{-1}\Big) \rightarrow \gamma_{L+1}  \sigma^2  \Theta^{(1)}(-\lambda),
\end{align*}
where
\begin{align*}
\Theta^{(1)}(z)=\int_{-\infty}^{+\infty} \frac{t}{t(1-\gamma_{L+1}-\gamma_{L+1} z s_{L+1}(z))-z} d H_{\Lambda^{(L+1)}}(t),
\end{align*} 
and $H_{\Lambda^{(L+1)}}(t)$ is the limiting spectral distribution of $\Lambda^{(L+1)}=\Omega^{\frac{1}{2}}\Sigma^{(L+1)}\Omega^{\frac{1}{2}}$. Note that $\tilde{\Lambda}^{(L+1)}=\frac{1}{n_{L+1}}\tilde{X}^{(L+1)\top}\tilde{X}^{(L+1)}$ and  $s_{L+1}(z)$ is related to $v_{L+1}(z)$ by following Silverstein equation
\begin{align*}
\gamma_{L+1}\Big(s_{L+1}(z)+\frac{1}{z}\Big)=v_{L+1}(z)+\frac{1}{z}.   
\end{align*} 
According to \citet[Lemma 2]{ledoit2011eigenvectors}, we have that 
\begin{align*}
\Theta^{(1)}(z)=\frac{\gamma_{L+1}^{-2}}{\gamma_{L+1}^{-1}-1-z s_{L+1}(z)}-\gamma_{L+1}^{-1}.
\end{align*} 
Plugging in the Silverstein equation yields
\begin{align*}
\Theta^{(1)}(z)=\gamma_{L+1}^{-1}\Big(\frac{1}{-z v_{L+1}(z)}-1\Big),  
\end{align*} 
and 
\begin{align}
(\mathsf{III})\rightarrow \sigma^{2}\Big(\frac{1}{\lambda v_{L+1}(-\lambda)}-1\Big).\label{limit3}    
\end{align} 
Taking derivatives w.r.t. $z$ on both hand side of Silverstein equation gives
$$v_{L+1}^{\prime}(z)=\gamma_{L+1}(s_{L+1}^{\prime}(z)-z^{-2})+z^{-2}.$$
Now, following the the steps by \citet[Proof of Theorem 2.1]{dobriban2018high}, we have that 
\begin{align}
     (\mathsf{I})+(\mathsf{II}) \rightarrow \big(\lambda^2 -\lambda \gamma_{L+1}  \sigma^2\big)\frac{v_{L+1}(-\lambda)-\lambda v_{L+1}^{\prime}(-\lambda)}{\gamma_{L+1}\big(\lambda v_{L+1}(-\lambda)\big)^2}.\label{limit1+2}
\end{align}
Combining \eqref{limit1+2} and \eqref{limit3} together and replacing $v$ in terms of $s$, it holds that
\begin{align*}
    &\oraclerisk(\Omega\mid X^{(L+1)})\stackrel{a.s.}{\rightarrow}\\
    &\frac{1}{\lambda \gamma_{L+1} s_{L+1}(-\lambda)+(1-\gamma_{L+1})}\Big[\sigma^2+\big(\frac{\lambda}{\gamma_{L+1}}-\sigma^2\big) \frac{\lambda^2 \gamma_{L+1} s_{L+1}^{\prime}(-\lambda)+(1-\gamma_{L+1})}{\gamma_{L+1} \lambda s_{L+1}(-\lambda)+(1-\gamma_{L+1})}\Big].
\end{align*}
\end{proof}
\begin{proof}[Proof of Lemma \ref{lemma1:for:consistency:of:Omega:inverse}]
Since the condition number of $\Omega$ is upper bounded and naturally bounded below by $1$, under (ii), one has that $\|\Omega^{-1}\|$ is upper bounded. Besides, by triangle inequality, it holds that 
$$\|\hat{\Omega}-\Omega+\Omega\| \geq|\|\Omega\|-\|\hat{\Omega}-\Omega\||.$$ 
Since $\|\Omega-\hat{\Omega}\|\rightarrow 0$ in probability when $p,L\rightarrow\infty$ and $\|\Omega\|$ is bounded away from $0$ for any $p$, $\|\hat{\Omega}\|$ is also bounded away from $0$ for sufficient large $p$ and $L$ with high probability. Therefore, $\|\hat{\Omega}^{-1}\|$ is bounded for sufficient large $p$ and $L$ with high probability.

Note that as $\hat{\Omega}^{-1}-\Omega^{-1}=\hat{\Omega}^{-1}(\Omega-\hat{\Omega}) \Omega^{-1}$, it holds that
\begin{align*}
\big\|\hat{\Omega}^{-1}-\Omega^{-1}\big\| \leq\big\|\hat{\Omega}^{-1}\big\|\big\|\Omega-\hat{\Omega}\big\|\big\|\Omega^{-1}\big\|,    
\end{align*} 
and 
\begin{align*}
\big\|\hat{\Omega}^{-1} \Omega-I\big\|=\big\|\hat{\Omega}^{-1}(\Omega-\hat{\Omega})\big\| \leq\big\|\hat{\Omega}^{-1}\big\|\big\|\Omega-\hat{\Omega}\big\|.    
\end{align*} 
Hence, as long as $\|\Omega-\hat{\Omega}\|\rightarrow 0$ in probability when $p,L\rightarrow\infty$, and $\|\hat{\Omega}^{-1}\|$ and $\|\Omega^{-1}\|$ is bounded for sufficiently large $p$ and $L$, it holds that $\|\hat{\Omega}^{-1}-\Omega^{-1}\|\rightarrow 0$ and $\|\hat{\Omega}^{-1}\Omega-I\|\rightarrow 0$ in probability as $p,L\rightarrow\infty$.
\end{proof}
 
\begin{proof}[Proof of Theorem \ref{thm_consistency_L}]
To analyze the asymptotic behavior of $(\mathsf{I}')$, $(\mathsf{II}')$ and $(\mathsf{III}')$, we first investigate the behavior of $(\mathsf{III}')$ or equivalently the term $\frac{1}{p}\operatorname{tr}\big(\Sigma^{(L+1)}\big(\hat{\Sigma}^{(L+1)}-z \hat{\Omega}^{-1}\big)^{-1}\big)$. Applying resolvent identity $A^{-1}-B^{-1}=A^{-1}(B-A)B^{-1}$ with 
\begin{align*}
A=(\hat{\Sigma}^{(L+1)}-z\hat{\Omega}^{-1}),\quad
B=(\hat{\Sigma}^{(L+1)}-z\Omega^{-1}),
\end{align*}
yields
\begin{align*}
&\frac{1}{p}\operatorname{tr}\Big(\Sigma^{(L+1)}\big(\hat{\Sigma}^{(L+1)}-z \hat{\Omega}^{-1}\big)^{-1}\Big)\\
=&\frac{1}{p}\operatorname{tr}\Big(\Sigma^{(L+1)}\big(\hat{\Sigma}^{(L+1)}-z \Omega^{-1}\big)^{-1}\Big)\\
&+z\frac{1}{p} \operatorname{tr}\Big(\Sigma^{(L+1)}\big(\hat{\Sigma}^{(L+1)}-z \hat{\Omega}^{-1}\big)^{-1}\big(\hat{\Omega}^{-1} - \Omega^{-1} \big)\big(\hat{\Sigma}^{(L+1)}-z \Omega^{-1}\big)^{-1}\Big).
\end{align*}
Now using the fact that for $p \times p$ matrices $C, D$, $|\operatorname{tr} C D| \leq (\operatorname{tr} C C^\top \operatorname{tr} D D^\top)^{1 / 2} \leq p\|C\|\|D\|$, for $z\in\mathbb{C}$ and $\Re z<0$, the second term could be bounded as
\begin{align*}
    & \Bigg| \frac{z}{p} \operatorname{tr}\Big(\Sigma^{(L+1)}\big(\hat{\Sigma}^{(L+1)}-z \hat{\Omega}^{-1}\big)^{-1}\big(\hat{\Omega}^{-1} - \Omega^{-1} \big)\big(\hat{\Sigma}^{(L+1)}-z \Omega^{-1}\big)^{-1}\Big)\Bigg| \\
    \leq& \big|z\big| \big\|\Sigma^{(L+1)}\big\| \big\|(\hat{\Sigma}^{(L+1)}-z \hat{\Omega}^{-1})^{-1}\big\| \big\|\Omega^{-1}-  \hat{\Omega}^{-1}\big\| \big\|(\hat{\Sigma}^{(L+1)}-z \Omega^{-1})^{-1}\big\| \\
    \leq& \big|z\big|  \big\|\Sigma^{(L+1)}\big\|\big\|\hat{\Omega}^{\frac{1}{2}}\big\|\big\|(\hat{\Omega}^{\frac{1}{2}} \hat{\Sigma}^{(L+1)} \hat{\Omega}^{\frac{1}{2}}-z I)^{-1}\big\|\\
    &\quad\quad\quad\big\|\hat{\Omega}^{\frac{1}{2}}\big\|\big\|\Omega^{-1}-\hat{\Omega}^{-1}\big\|\big\|\Omega^{\frac{1}{2}}\big\|\big\|\big(\Omega^{\frac{1}{2}} \hat{\Sigma}^{(L+1)} \Omega^{\frac{1}{2}}-z I\big)^{-1}\big\|\big\|\Omega^{\frac{1}{2}}\big\|\\ 
    \leq&\big\|\Sigma^{(L+1)}\big\|\big\|\Omega^{-1}-\hat{\Omega}^{-1}\big\|\big\|\hat{\Omega}^{\frac{1}{2}}\big\|^2\big\|\Omega^{\frac{1}{2}}\big\|^2\frac{1}{|z|},
\end{align*}
where the second inequality follows from the fact  
\begin{align*}
\big(\hat{\Sigma}^{(L+1)}-z \hat{\Omega}^{-1}\big)^{-1} 
&=\hat{\Omega}^{\frac{1}{2}}\big(\hat{\Omega}^{\frac{1}{2}} \hat{\Sigma}^{(L+1)} \hat{\Omega}^{\frac{1}{2}}-z I\big)^{-1} \hat{\Omega}^{\frac{1}{2}} \\
\big(\hat{\Sigma}^{(L+1)}-z \Omega^{-1}\big)^{-1} &=\Omega^{\frac{1}{2}}\big(\Omega^{\frac{1}{2}} \hat{\Sigma}^{(L+1)} \Omega^{\frac{1}{2}}-z I\big)^{-1} \Omega^{\frac{1}{2}}
\end{align*}
and third inequality follows from the fact that for any Hermitian matrix $A$, the operator norm of its resolvent could be bounded by $\|(A-zI)^{-1}\|\leq 1/\operatorname{dist}(z,\operatorname{supp}(F^{A}))$ and if $z\in \mathbb{R}^{-}$ and $A$ has all non-negative eigenvalues, it could be further bounded by $1/|z|$. 

By Assumption \ref{asp_Omegahat} and Lemma \ref{lemma1:for:consistency:of:Omega:inverse}, we have that 
$$\|\hat{\Omega}^{-1}-\Omega^{-1}\|\stackrel{p}{\rightarrow} 0$$
and $\|\Sigma\|$, $\|\hat{\Omega}^{\frac{1}{2}}\|$ and $\|\Omega^{\frac{1}{2}}\|$ is bounded as $p,L\rightarrow \infty$. Therefore, 
$$\Bigg|\frac{z}{p} \operatorname{tr}\Big(\Sigma^{(L+1)}\big(\hat{\Sigma}^{(L+1)}-z \hat{\Omega}^{-1}\big)^{-1}\big(\hat{\Omega}^{-1} - \Omega^{-1} \big)\big(\hat{\Sigma}^{(L+1)}-z \Omega^{-1}\big)^{-1}\Big)\Bigg|\stackrel{p}{\rightarrow} 0,$$
as $p,L\rightarrow\infty$. On the other hand, for any fixed $L$ and for any $z \in \mathbb{C}^{+}$, $s_{L+1}(z)$ is the solution of following fixed point problem,
\begin{align*}
    s_{L+1}(z)=\int_{-\infty}^{+\infty}\left\{\tau\left[1-\gamma_{L+1}-\gamma_{L+1} z s_{L+1}(z)\right]-z\right\}^{-1} d H_{\Lambda^{(L+1)}}(\tau).
\end{align*}
For every fixed $z\in\mathbb{C}\setminus \operatorname{supp}(H_{\Lambda^{(L+1)}})$, the function $|s_{L+1}(z)|\leq \frac{1}{\mathrm{Im} z}$. As $L\rightarrow\infty$, $H_{\Lambda^{(L+1)}}\Rightarrow H_{\Lambda}$ whose support is contained in a compact interval. Also, as $\gamma_{L+1}\rightarrow\gamma^{*}$, by Arzela–Ascoli Theorem, for every subsequence $\{s_{L_{k}+1}\}$, there exists a sub-subsequence $\{s_{L_{k_{i}}+1}\}$ such that the limit of the subsequence exists and is uniform. By dominated convergence theorem, for each convergent subsequence of $\{s_{L+1}\}$, the limit must be the solution to the following fixed point problem
\begin{align}
    s(z)=\int_{-\infty}^{+\infty}\left\{\tau\left[1-\gamma_{*}-\gamma_{*} z s(z)\right]-z\right\}^{-1} d H_{\Lambda}(\tau).\label{limit_equation_of_sL+1}
\end{align}
And \eqref{limit_equation_of_sL+1} has unique solution by a similar argument as that in \cite[Chapter 6]{silverstein1995strong}. So under Assumption \ref{asp5}, $s_{L+1}(z)$ converges pointwisely to $s(z)$ which is uniquely defined by \eqref{limit_equation_of_sL+1}.

%By Assumption \ref{asp5} and Chapter 6 \cite{bai2010spectral}, as $L\rightarrow\infty$ and $\gamma_{L+1}\rightarrow \gamma_{*}$, $s_{L+1}(z)$ converges pointwisely to the 

Therefore, according to \citet[Lemma 2]{ledoit2011eigenvectors}, for any fixed $L$ one has
\begin{align*}
\sigma^2\frac{p}{n_{L+1}}\frac{1}{p}\operatorname{tr}\big(\Sigma^{(L+1)}(\hat{\Sigma}^{(L+1)}+\lambda \Omega^{-1})^{-1}\big)\rightarrow \sigma^{2}\Big(\frac{1}{\lambda v_{L+1}(-\lambda)}-1\Big)  
\end{align*} 
as $p,n_{L+1}\rightarrow\infty,\frac{p}{n_{L+1}}\rightarrow\gamma_{L+1}$. Now, when $L\rightarrow\infty$ and $\gamma_{L+1}\rightarrow \gamma_{*}$, we have that 
\begin{align*}
 \sigma^{2}\Big(\frac{1}{\lambda v_{L+1}(-\lambda)}-1\Big)\rightarrow \sigma^{2}\Big(\frac{1}{\lambda v(-\lambda)}-1\Big),
\end{align*}
where $v(z)$ is related to $s(z)$ by following equation for all $z \in \mathbb{C} \backslash \mathbb{R}^{+}$:
\begin{align*}
\gamma_{*}\Big(s(z)+\frac{1}{z}\Big)=v(z)+\frac{1}{z}.    
\end{align*} 
Also, as $p,n_{L+1}\rightarrow\infty$ such that $\frac{p}{n_{L+1}}\rightarrow\gamma_{L+1}$ and $L\rightarrow\infty$ such that $\gamma_{L+1}\rightarrow\gamma_{*}$, $|(\mathsf{III})-(\mathsf{III}')|\stackrel{p}{\rightarrow}0$. Therefore, as $p,n_{L+1}\rightarrow\infty$ such that $\frac{p}{n_{L+1}}\rightarrow\gamma_{L+1}$ and $L\rightarrow\infty$ such that $\gamma_{L+1}\rightarrow\gamma_{*}$,
\begin{align*}
    (\mathsf{III}')\stackrel{p}{\rightarrow}\sigma^{2}\Big(\frac{1}{\lambda v(-\lambda)}-1\Big). 
\end{align*}

Now for the second term $(\mathsf{II}')$, it holds that
\begin{align*}
(\mathsf{II}')=-\frac{\lambda p\sigma^2}{n_{L+1}}\frac{1}{p} \operatorname{tr}\big(\hat{\Omega}^{\frac{1}{2}} \Sigma^{(L+1)} \hat{\Omega}^{\frac{1}{2}}\big(\hat{\Omega}^{\frac{1}{2}} \hat{\Sigma}^{(L+1)} \hat{\Omega}^{\frac{1}{2}}+\lambda I\big)^{-2}\big).
\end{align*}
Consider the quantity $\frac{1}{p} \operatorname{tr}\big(\hat{\Omega}^{\frac{1}{2}} \Sigma^{(L+1)} \hat{\Omega}^{\frac{1}{2}}\big(\hat{\Omega}^{\frac{1}{2}} \hat{\Sigma}^{(L+1)} \hat{\Omega}^{\frac{1}{2}}+\lambda I\big)^{-2}\big)$. Note that as the eigenvalue of $\big(\hat{\Omega}^{\frac{1}{2}} \hat{\Sigma}^{(L+1)} \hat{\Omega}^{\frac{1}{2}}+\lambda I\big)^{-1}$ is upper bounded by $\frac{1}{\lambda}$,
\begin{align*}
    \Big|\frac{1}{p} \operatorname{tr}\big(\hat{\Omega}^{\frac{1}{2}} \Sigma^{(L+1)} \hat{\Omega}^{\frac{1}{2}}\big(\hat{\Omega}^{\frac{1}{2}} \hat{\Sigma}^{(L+1)} \hat{\Omega}^{\frac{1}{2}}+\lambda I\big)^{-1}\big)\Big|&\leq \frac{\|\hat{\Omega}^{\frac{1}{2}} \Sigma^{(L+1)} \hat{\Omega}^{\frac{1}{2}} \|}{\lambda}\\
    &\leq \frac{\|\Sigma^{(L+1)} \|\|\hat{\Omega}^{\frac{1}{2}}\|^2}{\lambda}.
\end{align*}
Now, $\|\Sigma^{(L+1)}\|$ is upper bounded for any $p$ and $L$, and $\|\hat{\Omega}^{\frac{1}{2}}\|$ is upper bounded for sufficiently large $p$ and $L$. Therefore, $\frac{1}{p} \operatorname{tr}\big(\hat{\Omega}^{\frac{1}{2}} \Sigma^{(L+1)} \hat{\Omega}^{\frac{1}{2}}\big(\hat{\Omega}^{\frac{1}{2}} \hat{\Sigma}^{(L+1)} \hat{\Omega}^{\frac{1}{2}}+\lambda I\big)^{-1}\big)$ is a bounded sequence. By Lemma \ref{lm2.14_bai_spectral}, it holds that
$$
(\mathsf{II}')\rightarrow -\lambda \gamma_{*}\sigma^2\frac{v(-\lambda)-\lambda v^{\prime}(-\lambda)}{\gamma_{*}\big(\lambda v(-\lambda)\big)^2},
$$
as $p,n_{L+1}\rightarrow\infty$ such that $\frac{p}{n_{L+1}}\rightarrow \gamma_{L+1}$ and $L\rightarrow\infty$ such that $\gamma_{L+1}\rightarrow\gamma_{*}$.

Finally, for the term $(\mathsf{I}')$, it holds that
\begin{align*}
    &\frac{1}{p} \operatorname{tr}\big(\Omega \hat{\Omega}^{-1}\big(\hat{\Sigma}^{(L+1)}+\lambda \hat{\Omega}^{-1}\big)^{-1} \Sigma^{(L+1)}\big(\hat{\Sigma}^{(L+1)}+\lambda \hat{\Omega}^{-1}\big)^{-1} \hat{\Omega}^{-1}\big)\\
    =&\frac{1}{p} \operatorname{tr}\big(\big(\hat{\Sigma}^{(L+1)}+\lambda \hat{\Omega}^{-1}\big)^{-1} \Sigma^{(L+1)}\big(\hat{\Sigma}^{(L+1)}+\lambda \hat{\Omega}^{-1}\big)^{-1} \hat{\Omega}^{-1}\big)\\
    &\quad\quad+\frac{1}{p} \operatorname{tr}\big(\big(\Omega \hat{\Omega}^{-1}-I\big)\big(\hat{\Sigma}^{(L+1)}+\lambda \hat{\Omega}^{-1}\big)^{-1} \Sigma^{(L+1)}\big(\hat{\Sigma}^{(L+1)}+\lambda \hat{\Omega}^{-1}\big)^{-1} \hat{\Omega}^{-1}\big).
\end{align*}
The second term $\frac{1}{p} \operatorname{tr}\big(\big(\Omega \hat{\Omega}^{-1}-I\big)\big(\hat{\Sigma}^{(L+1)}+\lambda \hat{\Omega}^{-1}\big)^{-1} \Sigma^{(L+1)}\big(\hat{\Sigma}^{(L+1)}+\lambda \hat{\Omega}^{-1}\big)^{-1} \hat{\Omega}^{-1}\big)$ could be bounded as
\begin{align*}
    &\frac{1}{p} \operatorname{tr}\big(\big(\Omega \hat{\Omega}^{-1}-I\big)\big(\hat{\Sigma}^{(L+1)}+\lambda \hat{\Omega}^{-1}\big)^{-1} \Sigma^{(L+1)}\big(\hat{\Sigma}^{(L+1)}+\lambda \hat{\Omega}^{-1}\big)^{-1} \hat{\Omega}^{-1}\big)\\
    = &\frac{1}{p} \operatorname{tr}\big(\big(\Omega \hat{\Omega}^{-1}-I\big) \hat{\Omega}^{\frac{1}{2}}\big(\hat{\Omega}^{\frac{1}{2}} \hat{\Sigma}^{(L+1)} \hat{\Omega}^{\frac{1}{2}}+\lambda I\big)^{-1} \hat{\Omega}^{\frac{1}{2}} \Sigma^{(L+1)} \hat{\Omega}^{\frac{1}{2}}\big(\hat{\Omega}^{\frac{1}{2}} \hat{\Sigma}^{(L+1)} \hat{\Omega}^{\frac{1}{2}}+\lambda I\big)^{-1} \hat{\Omega}^{-\frac{1}{2}}\big)\\
    \leq & \big\|\Omega \hat{\Omega}^{-1}-I\big\|\big\|\hat{\Omega}^{\frac{1}{2}}\big\|^3\|\Sigma^{(L+1)}\|\big\|\big(\hat{\Omega}^{\frac{1}{2}} \hat{\Sigma}^{(L+1)} \hat{\Omega}^{\frac{1}{2}}+\lambda I\big)^{-1}\big\|^2\big\|\hat{\Omega}^{-\frac{1}{2}}\big\|\\
    \leq &\frac{1}{\lambda^2}\big\|\Omega \hat{\Omega}^{-1}-I\big\|\big\|\hat{\Omega}^{\frac{1}{2}}\big\|^3\big\|\Sigma^{(L+1)}\big\|\big\|\hat{\Omega}^{-\frac{1}{2}}\big\|,
\end{align*}
which will converge to zero in probability as $p,L\rightarrow \infty$ since $\big\|\Omega \hat{\Omega}^{-1}-I\big\|$ converges to zero in probability and $\big\|\hat{\Omega}^{\frac{1}{2}}\big\|^3$, $\big\|\Sigma^{(L+1)}\big\|$ and $\big\|\hat{\Omega}^{-\frac{1}{2}}\big\|$ are bounded. Hence, under conditions mentioned above we have that
\begin{align*}
\prisk\big(\hat{\Omega} \mid X^{(L+1)}\big)\stackrel{p}{\rightarrow} \frac{1}{\lambda \gamma_{*} s(-\lambda)+(1-\gamma_{*})}\Big[\sigma^2+\Big(\frac{\lambda}{\gamma_{*}}-\sigma^2\Big) \frac{\lambda^2 \gamma_{*} s^{\prime}(-\lambda)+(1-\gamma_{*})}{\gamma_{*} \lambda s(-\lambda)+(1-\gamma_{*})}\Big],
\end{align*}
where $s(z)$ is the solution to the following equation 
\begin{align*}
    s(z)=\int_{-\infty}^{+\infty}\left\{\tau\left[1-\gamma_{*}-\gamma_{*} z s(z)\right]-z\right\}^{-1} d H_{\Lambda}(\tau).
\end{align*}
\end{proof}
\begin{proof}[Proof of Proposition \ref{Prop_Out-of-distribution_Prediction_Risk}]
We start by bounding the term $|\prisk(\hat{\Omega}\mid X^{(L+1)}) -\oraclerisktilde(\Upsilon\mid X^{(L+1)})|$ by triangle inequality:
\begin{align*}
    &|\prisk(\hat{\Omega}\mid X^{(L+1)}) -\oraclerisktilde(\Upsilon\mid X^{(L+1)})|\\
    &\leq  |\prisk(\hat{\Omega}\mid X^{(L+1)}) - \oraclerisk\big(\Omega\mid X^{(L+1)}\big) | + |\oraclerisk\big(\Omega\mid X^{(L+1)}\big) - \oraclerisktilde(\Upsilon\mid X^{(L+1)})|
\end{align*}
The first term in the right hand side above is already analyzed in Theorem \ref{thm_consistency_L}, and the second term $|\oraclerisk\big(\Omega\mid X^{(L+1)}\big) - \oraclerisktilde(\Upsilon\mid X^{(L+1)})|$ mainly depends on $\|\Upsilon^{-1}-\Omega^{-1}\|$ and is bounded next. Note that
\begin{align}
&|\oraclerisk\big(\Omega\mid X^{(L+1)}\big) - \oraclerisktilde(\Upsilon\mid X^{(L+1)})|\nonumber\\
\leq &\frac{\lambda \sigma^2}{n_{L+1}} \operatorname{tr}\Big(\Sigma^{(L+1)}\big(\hat{\Sigma}^{(L+1)}+\lambda \Omega^{-1}\big)^{-1}\big(\Upsilon^{-1}-\Omega^{-1}\big)\big(\hat{\Sigma}^{(L+1)}+\lambda \Omega^{-1}\big)^{-1}\Big)\nonumber\\
&+\Big(\frac{\lambda^2}{p}-\frac{\lambda\sigma^2}{n_{L+1}}\Big)\operatorname{tr}\Big[\Sigma^{(L+1)}\big(\big(\hat{\Sigma}^{(L+1)}+\lambda \Omega^{-1}\big)^{-1}-\big(\hat{\Sigma}^{(L+1)}+\lambda \Upsilon^{-1}\big)^{-1}\big)\nonumber\\
&\quad\quad\quad\quad\quad\quad\quad\quad\quad\quad\quad\quad \Omega^{-1}\big(\hat{\Sigma}^{(L+1)}+\lambda \Omega^{-1}\big)^{-1}\Big]\nonumber\\
&+\Big(\frac{\lambda^2}{p}-\frac{\lambda\sigma^2}{n_{L+1}}\Big)\operatorname{tr}\Big[\Sigma^{(L+1)}\big(\hat{\Sigma}^{(L+1)}+\lambda \Upsilon^{-1}\big)^{-1}\nonumber\\
&\quad\quad\quad\quad\quad\quad\quad\quad\quad\quad\quad\quad \big(\Omega^{-1}-\Upsilon^{-1}\big)\big(\hat{\Sigma}^{(L+1)}+\lambda \Omega^{-1}\big)^{-1}\Big]\nonumber\\
&+\Big(\frac{\lambda^2}{p}-\frac{\lambda\sigma^2}{n_{L+1}}\Big)\operatorname{tr}\Big(\Sigma^{(L+1)} \big(\hat{\Sigma}^{(L+1)}+\lambda \Upsilon^{-1}\big)^{-1}\nonumber\\
&\quad\quad\quad\quad\quad\quad\quad\quad\quad\quad\quad\quad \Upsilon^{-1}\big(\big(\hat{\Sigma}^{(L+1)}+\lambda \Omega^{-1}\big)^{-1}-\big(\hat{\Sigma}^{(L+1)}+\lambda \Upsilon^{-1}\big)^{-1}\big)\Big).\label{equ1_Prop_Out-of-distribution_Prediction_Risk}
\end{align}
By resolvent identity, it holds that 
\begin{align*}
\|\Omega^{-1}-\Upsilon^{-1}\|&=\|\Omega^{-1}(\Upsilon-\Omega)\Upsilon^{-1}\|\nonumber\\
&\leq \|\Omega^{-1}\| \|\Upsilon-\Omega\| \|\Upsilon^{-1}\|%\label{equ2_Prop_Out-of-distribution_Prediction_Risk}
\end{align*} 
Under our assumptions, $\|\Upsilon^{-1}\|$ is bounded by some universal constant $C_{\Upsilon}$, then $\|\Omega^{-1}-\Upsilon^{-1}\|\leq C'\vartheta$. Besides, we also have
\begin{align*}
    \big\|\big(\hat{\Sigma}^{(L+1)}+\lambda \Omega^{-1}\big)^{-1}\big\|&\leq \big\|\Omega^{\frac{1}{2}}\big\|^{2}\big\|\big(\Omega^{\frac{1}{2}}\hat{\Sigma}^{(L+1)}\Omega^{\frac{1}{2}}+\lambda I\big)^{-1}\big\|\leq \frac{1}{\lambda}\big\|\Omega^{\frac{1}{2}}\big\|^{2}\\ %\label{equ3_Prop_Out-of-distribution_Prediction_Risk}\\
    \big\|\big(\hat{\Sigma}^{(L+1)}+\lambda \Upsilon^{-1}\big)^{-1}\big\|&\leq \big\|\Upsilon^{\frac{1}{2}}\big\|^{2}\big\|\big(\Upsilon^{\frac{1}{2}}\hat{\Sigma}^{(L+1)}\Upsilon^{\frac{1}{2}}+\lambda I\big)^{-1}\big\|\leq \frac{1}{\lambda}\big\|\Upsilon^{\frac{1}{2}}\big\|^{2}%\label{equ4_Prop_Out-of-distribution_Prediction_Risk}
\end{align*}
Again, by resolvent identity, we obtain that
\begin{align}
&\big\|\big(\hat{\Sigma}^{(L+1)}+\lambda \Omega^{-1}\big)^{-1}-\big(\hat{\Sigma}^{(L+1)}+\lambda \Upsilon^{-1}\big)^{-1}\big\|\nonumber\\
=&\lambda\big\|\big(\hat{\Sigma}^{(L+1)}+\lambda \Omega^{-1}\big)^{-1}\big(\Upsilon^{-1}-\Omega^{-1}\big)\big(\hat{\Sigma}^{(L+1)}+\lambda \Upsilon^{-1}\big)^{-1}\big\|\nonumber\\
\leq &\frac{1}{\lambda}\big\|\Omega^{\frac{1}{2}}\big\|^2\big\|\Upsilon^{\frac{1}{2}}\big\|^2\big\|\Omega^{-1}-\Upsilon^{-1}\big\|\label{equ5_Prop_Out-of-distribution_Prediction_Risk}
\end{align}
Combining \eqref{equ1_Prop_Out-of-distribution_Prediction_Risk} to \eqref{equ5_Prop_Out-of-distribution_Prediction_Risk} together, it holds that
\begin{align*}
&|\oraclerisk\big(\Omega\mid X^{(L+1)}\big) - \oraclerisktilde(\Upsilon\mid X^{(L+1)})|\\
 \leq & ~~\frac{p\lambda\sigma^2}{n_{L+1}}\big\|\Sigma^{(L+1)}\big\|\big\|\Omega^{-1}\big\|\vartheta\big\|\Upsilon^{-1}\big\|\frac{1}{\lambda^2}\big\|\Omega^{\frac{1}{2}}\big\|^2\big\|\Upsilon^{\frac{1}{2}}\big\|^2\\
&\quad+\Big(\lambda^2-\frac{p\lambda\sigma^2}{n_{L+1}}\Big)\big\|\Sigma^{(L+1)}\big\|\frac{1}{\lambda^2}\big\|\Omega^{\frac{1}{2}}\big\|^4\big\|\Upsilon^{\frac{1}{2}}\big\|^2\big\|\Omega^{-1}\big\|^2\big\|\Upsilon^{-1}\big\|\vartheta \\
&\quad+\Big(\lambda^2-\frac{p\lambda\sigma^2}{n_{L+1}}\Big)\big\|\Sigma^{(L+1)}\big\|\frac{1}{\lambda^2}\big\|\Omega^{\frac{1}{2}}\big\|^2\big\|\Upsilon^{\frac{1}{2}}\big\|^2\big\|\Omega^{-1}\big\|\big\|\Upsilon^{-1}\big\|\vartheta\\
&\quad+\Big(\lambda^2-\frac{p\lambda\sigma^2}{n_{L+1}}\Big)\big\|\Sigma^{(L+1)}\big\|\frac{1}{\lambda^2}\big\|\Omega^{\frac{1}{2}}\big\|^2\big\|\Upsilon^{\frac{1}{2}}\big\|^4\big\|\Upsilon^{-1}\big\|^2\big\|\Omega^{-1}\big\| \vartheta,
\end{align*}
and
\begin{align*}
\big|\prisk(\hat{\Omega}\mid X^{(L+1)})-\oraclerisktilde(\Upsilon\mid X^{(L+1)})\big|
\leq &  |\prisk(\hat{\Omega}\mid X^{(L+1)}) - \oraclerisk\big(\Omega\mid X^{(L+1)}\big) |\\
    &\quad\quad+\frac{p\lambda\sigma^2}{n_{L+1}}\frac{\bar{c}^{(L+1)}}{\lambda^2}\varsigma(\Omega)\varsigma(\Upsilon)\vartheta\\
    &\quad\quad+\Big(\lambda^2-\frac{p\lambda\sigma^2}{n_{L+1}}\Big)\frac{\bar{c}^{(L+1)}}{\lambda^2}\varsigma(\Omega)\varsigma(\Upsilon)(1+\varsigma(\Omega)+\varsigma(\Upsilon))\vartheta.
\end{align*}
Therefore, as $L,n_{L+1},p\rightarrow\infty$ such that for each fixed $L$, $p/n_{L+1}\rightarrow\gamma_{L+1}$, while
$\lim_{L\to \infty} \gamma_{L+1} =\gamma_* \in (1,\infty)$, 
\begin{align*}
   \big|\prisk(\hat{\Omega}\mid X^{(L+1)})-\oraclerisktilde(\Upsilon\mid X^{(L+1)})\big|&\rightarrow M(\vartheta,\lambda),
\end{align*}
where the limit $M(\vartheta,\lambda)$ satisfies $$|M(\vartheta,\lambda)|\leq \frac{\gamma_{*}\sigma^2}{\lambda}\bar{c}_{\text{op}}c_{\Omega}c_{\Upsilon}\vartheta+\Big(1+\frac{\gamma_{*}\sigma^2}{\lambda}\Big)\bar{c}_{\text{op}}(1+c_{\Omega}+c_{\Upsilon})c_{\Omega}c_{\Upsilon}\vartheta.$$
\end{proof}
\subsection{Statistical Advantage of Using $\Omega$}
\begin{lemma}\label{technical_lemma_1}
For any $p\times p$ positive definite matrix $A\in\mathbb{S}_{p}^{+}$, it holds that
$$\operatorname{tr}\Big[\big(\exp(tA)-I\big)\frac{d}{dt}\exp(tA)\Big]\geq 0.$$
\end{lemma}
\begin{proof}
Note that for any $t>0$ and $A=\mathbb{S}_{p}^{+}$ with eigenvalue decomposition $A=UD U^{\top}$ and eigenvalues $\lambda_{11},\dots,\lambda_{pp}$
\begin{align*}
    \frac{d}{dt}\exp(tA) &=U\frac{d}{dt}\exp(tD) U^{\top}\\
    &=U\frac{d}{dt}\begin{bmatrix}
    \exp(t\lambda_{11}) & & \\
    & \ddots & \\
    & & \exp(t\lambda_{pp})
    \end{bmatrix}U^{\top}\\
    &=UD\exp(tD)U^{\top}\\
    \operatorname{tr}\Big[\big(\exp(tA)-I\big)\frac{d}{dt}\exp(tA)\Big]&=\operatorname{tr}\big[U(\exp(tD)-I)D\exp(tD)U^\top\big]\\
    &=\operatorname{tr}\big((\exp(tD)-I)D\exp(tD)\big)\\
    &=\sum_{i=1}^{p}\big(\exp(t\lambda_{ii})-1\big)\lambda_{ii}\exp(t\lambda_{ii})\geq 0
\end{align*}
for any $t\geq 0$ since $\big(\exp(t\lambda_{ii})-1\big)\lambda_{ii}\geq 0$ for any $t\geq 0$.
\end{proof}
\begin{lemma}\label{lm_Riemannian_gradient_of_risk}
By treating $\prisk(Q\mid X^{(L+1)})$ as a function of $Q^{-1}$, the  Riemannian gradient of $\prisk(Q\mid X^{(L+1)})$ w.r.t. $Q^{-1}$ is given by
\begin{align*}
    &\operatorname{grad} \prisk(Q^{-1}\mid X^{(L+1)})=\left.\operatorname{grad} \prisk(P_{Q^{-1}}(\Xi)\mid X^{(L+1)})\right|_{\Xi=0}\\
    =&2 \lambda\Big[B_P \Sigma^{(L+1)} B_P\Big(\frac{\lambda}{p} Q^{-1} \Omega-\frac{\sigma^2}{n_{L+1}} I\Big) \hat{\Sigma}^{(L+1)} B_P\\
    &\quad\quad+\Big(B_P \Sigma^{(L+1)} B_P\Big(\frac{\lambda}{p} Q^{-1} \Omega-\frac{\sigma^2}{n_{L+1}} I\Big) \hat{\Sigma}^{(L+1)} B_P\Big)^{\top} \\
    &\quad\quad -\operatorname{diag}\Big\{B_P \Sigma^{(L+1)} B_P\Big(\frac{\lambda}{p} Q^{-1} \Omega-\frac{\sigma^2}{n_{L+1}} I\Big) \hat{\Sigma}^{(L+1)} B_P\Big\}\Big].
\end{align*}
\end{lemma}
\begin{proof}
Let $\xi_{ij}$ be the $(i,j)$-th entry of symmetric matrices $\Xi$. By chain rule, it holds that
\begin{equation*}
    \frac{\partial \prisk\big(P_{Q^{-1}}(\Xi) \mid X^{(L+1)}\big)}{\partial \xi_{i j}}=\sum_{u,v} \frac{\partial \prisk\big(P \mid X^{(L+1)}\big)}{\partial P_{u v}} \frac{\partial P_{u v}}{\partial \xi_{i j}}=\operatorname{tr}\Big[\Big(\frac{\partial \prisk}{\partial P}\Big)^{\top} \frac{\partial P}{\partial \xi_{i j}}\Big],
\end{equation*}
where we slightly abuse the notation $P$ to be the $P_{Q^{-1}}(\Xi)$. Note that with 
$$P=P_{Q^{-1}}(\Xi)=Q^{-1}+\Xi+\frac{1}{2} \Xi Q \Xi,$$
it holds that for any $i<j$, 
\begin{align*}
    \frac{\partial P}{\partial \xi_{i j}}&=e_i e_j^{\top}+e_j e_i^{\top}+\frac{1}{2}\Big(\frac{\partial \Xi}{\partial \xi_{i j}} Q \Xi+\Xi Q \frac{\partial \Xi}{\partial \xi_{i j}}\Big)\\
    &=\frac{1}{2}(e_i e_j^{\top}+e_j e_i^{\top})(Q \Xi+I)+\frac{1}{2}(I+\Xi Q)(e_i e_j^{\top}+e_j e_i^{\top})\\
    \frac{\partial P}{\partial \xi_{i i}}&=e_i e_i^{\top}+\frac{1}{2} e_i e_i^{\top} Q \Xi+\frac{1}{2} \Xi Q e_i e_i^{\top}\\
    &=\frac{1}{2} e_i e_i^{\top}(I+Q \Xi)+\frac{1}{2}(\Xi Q+I) e_i e_i^{\top}.
\end{align*}
Define $B_{P}\coloneqq(\hat{\Sigma}^{(L+1)}+\lambda P)^{-1}$, then it holds that $\frac{\partial B_{P}}{\partial P_{u v}}=-\lambda B_P  e_u e_v^{\top} B_P$. In order to calculate the Riemannian gradient, we first calculate $\frac{\partial \prisk}{\partial P}$:
\begin{align}
    \frac{\partial \prisk}{\partial P_{u v}}=\lambda \frac{\partial}{\partial P_{u v}} \operatorname{tr}\Big(\Big(\frac{\lambda}{p} \Omega P-\frac{\sigma^2}{n_{L+1}} I\Big) B_P \Sigma^{(L+1)} B_P P\Big)+\frac{\sigma^2}{n_{L+1}} \operatorname{tr}\Big(\Sigma^{(L+1)} \frac{\partial B_P}{\partial P_{u v}}\Big).\label{advproof_equ1}
\end{align}
We calculate two terms in \eqref{advproof_equ1} now. The trace in second term is calculated as
\begin{align*}
\operatorname{tr}\Big(\Sigma^{(L+1)} \frac{\partial B_P}{\partial P_{u v}}\Big)=\operatorname{tr}\big(-\lambda \Sigma^{(L+1)} B_P e_u e_v^{\top} B_P\big)=-\lambda e_v^{\top} B_P \Sigma^{(L+1)} B_P e_u,  
\end{align*}
and the second term could be calculated in a similar way
\begin{align*}
    & \frac{\partial}{\partial P_{u v}} \operatorname{tr}\Big(\Big(\frac{\lambda}{p} \Omega P-\frac{\sigma^2}{n_{L+1}} I\Big) B_P \Sigma^{(L+1)} B_P P\Big)\\
    =&\operatorname{tr}\Big(\frac{\lambda}{p} \Omega \frac{\partial P}{\partial P_{u v}} B_P \Sigma^{(L+1)} B_P P\Big)\\
    &+\operatorname{tr}\Big(\Big(\frac{\lambda}{p} \Omega P-\frac{\sigma^2}{n_{L+1}} I\Big) \frac{\partial B_P}{\partial P_{u v}} \Sigma^{(L+1)} B_P P\Big)\\
    & +\operatorname{tr}\Big(\Big(\frac{\lambda}{p} \Omega P-\frac{\sigma^2}{n_{L+1}} I\Big) B_P \Sigma^{(L+1)} \frac{\partial B_P}{\partial P_{u v}} P\Big)\\
    & +\operatorname{tr}\Big(\Big(\frac{\lambda}{p} \Omega P-\frac{\sigma^2}{n_{L+1}} I\Big) B_P \Sigma^{(L+1)} B_P \frac{\partial P}{\partial P_{u v}}\Big)\\
    =&e_v^{\top}\Big[\frac{\lambda}{p}(B_P \Sigma^{(L+1)} B_P P \Omega+\Omega P B_P \Sigma^{(L+1)} B_P)-\frac{\sigma^2}{n_{L+1}} B_P \Sigma^{(L+1)} B_P\Big] e_u\\
    &\quad -e_v^{\top}\Big[\frac{\lambda^2}{p}(B_P \Sigma^{(L+1)} B_P P \Omega P B_P+B_P P \Omega P B_P \Sigma^{(L+1)} B_P)\\
    &\quad -\frac{\lambda \sigma^2}{n_{L+1}}(B_P \Sigma^{(L+1)} B_P P B_P+B_P P B_P \Sigma^{(L+1)} B_P)\Big] e_u. 
\end{align*}
Combining two terms $\frac{\sigma^2}{n_{L+1}}\operatorname{tr}\Big(\Sigma^{(L+1)} \frac{\partial B_P}{\partial P_{u v}}\Big)$ and $\lambda\frac{\partial}{\partial P_{u v}} \operatorname{tr}\Big(\Big(\frac{\lambda}{p} \Omega P-\frac{\sigma^2}{n_{L+1}} I\Big) B_P \Sigma^{(L+1)} B_P P\Big)$ together yields
\begin{align*}
    \Big(\frac{\partial R}{\partial P}\Big)^{\top}&=\lambda\Big[\frac{\lambda}{p}\big(B_P \Sigma^{(L+1)} B_P P \Omega+\Omega P B_P \Sigma^{(L+1)} B_P\big)-\frac{2 \sigma^2}{n_{L+1}} B_P \Sigma^{(L+1)} B_P \\
    &\quad -\frac{\lambda^2}{p}\big(B_P \Sigma^{(L+1)} B_P P \Omega P B_P+B_P P \Omega P B_P \Sigma^{(L+1)} B_P\big)\\
    &\quad +\frac{\lambda \sigma^2}{n_{L+1}}\big(B_P \Sigma^{(L+1)} B_P P B_P+B_P P B_P \Sigma^{(L+1)} B_P\big)\Big].
\end{align*}
Note that $\frac{\partial P}{\partial\xi_{ij}}$ is symmetric, plugging $\frac{\partial \prisk}{\partial P}$ in $\operatorname{tr}\Big[\Big(\frac{\partial \prisk}{\partial P}\Big)^{\top}\Big(\frac{\partial P}{\partial \xi_{ij}}\Big)\Big]$ yields
\begin{align*}
&\operatorname{tr}\Big[\Big(\frac{\partial \prisk}{\partial P}\Big)^{\top}\Big(\frac{\partial P}{\partial \xi_{ij}}\Big)\Big]\\
=&\lambda \operatorname{tr}\Big[(I-\lambda B_P P)\Big(\frac{\lambda}{p} \Omega P-\frac{\sigma^2}{n_{L+1}} I\Big) B_P \Sigma^{(L+1)} B_P \frac{\partial P}{\partial \xi_{i j}}\Big]\\
&\quad +\lambda \operatorname{tr}\Big[B_P \Sigma^{(L+1)} B_P\Big(\frac{\lambda}{p} P \Omega-\frac{\sigma^2}{n_{L+1}} I\Big)(I-\lambda P B_P) \frac{\partial P}{\partial \xi_{i j}}\Big]\\
=&2 \lambda \operatorname{tr}\Big[B_P \Sigma^{(L+1)} B_P\Big(\frac{\lambda}{p} P \Omega-\frac{\sigma^2}{n_{L+1}} I\Big)(I-\lambda P B_P) \frac{\partial P}{\partial \xi_{i j}}\Big]\\
=&\lambda e_j^{\top}(Q \Xi+I) B_P \Sigma^{(L+1)} B_P\Big(\frac{\lambda}{p} P \Omega-\frac{\sigma^2}{n_{L+1}} I\Big) \hat{\Sigma}^{(L+1)} B_P e_i\\
    &\quad+\lambda e_i^{\top}(Q \Xi+I) B_P \Sigma^{(L+1)} B_P\Big(\frac{\lambda}{p} P_{\Omega}-\frac{\sigma^2}{n_{L+1}} I\Big) \hat{\Sigma}^{(L+1)} B_P e_j\\
    &\quad+\lambda e_j^{\top} B_P \Sigma^{(L+1)} B_P\Big(\frac{\lambda}{p} P \Omega-\frac{\sigma^2}{n_{L+1}} I\Big) \hat{\Sigma}^{(L+1)} B_P(I+\Xi Q) e_i\\
    &\quad+\lambda e_i^{\top} B_P \Sigma^{(L+1)} B_P\Big(\frac{\lambda}{p} P \Omega-\frac{\sigma^2}{n_{L+1}} I\Big) \hat{\Sigma}^{(L+1)} B_P(I+\Xi Q) e_j,
\end{align*}
where we use the identity $I-\lambda P B_P=(\hat{\Sigma}^{(L+1)}+\lambda P-\lambda P)(\hat{\Sigma}^{(L+1)}+\lambda P)^{-1}=\hat{\Sigma}^{(L+1)} B_P$ for $B_{P}=(\hat{\Sigma}^{(L+1)}+\lambda P)^{-1}$. Therefore, by taking $\Xi=0$, it holds that
\begin{align*}
    \left.\operatorname{tr}\Big[\Big(\frac{\partial \prisk}{\partial P}\Big)^{\top}\Big(\frac{\partial P}{\partial \xi_{ij}}\Big)\Big]\right|_{\Xi=0}&=2 \lambda e_j^{\top} B_{P} \Sigma^{(L+1)} B_{P}\Big(\frac{\lambda}{p} Q^{-1} \Omega-\frac{\sigma^2}{n_{L+1}} I\Big) \hat{\Sigma}^{(L+1)} B_{P} e_i\\
    &\quad+2 \lambda e_i^{\top} B_{P} \Sigma^{(L+1)} B_{P}\big(\frac{\lambda}{p} Q^{-1} \Omega-\frac{\sigma^2}{n_{L+1}} I\big)\hat{\Sigma}^{(L+1)} B_{P} e_j.
\end{align*}
Similarly, since $\frac{\partial P}{\partial \xi_{i i}}=\frac{1}{2} e_i e_i^{\top}(I+Q \Xi)+\frac{1}{2}(\Xi \hat{\Omega}+I) e_i e_i^{\top}$, it holds that
\begin{align*}
    \operatorname{tr}\Big[\Big(\frac{\partial \prisk}{\partial P}\Big)^{\top}\Big(\frac{\partial P}{\partial \xi_{ii}}\Big)\Big]&=\lambda e_i^{\top}(I+Q \Xi) B_P \Sigma^{(L+1)} B_P\Big(\frac{\lambda}{p} P \Omega-\frac{\sigma^2}{n_{L+1}} I\Big) \hat{\Sigma}^{(L+1)} B_P e_i\\
    &\quad +\lambda e_i^{\top} B_P \Sigma^{(L+1)} B_P\Big(\frac{\lambda}{p} P \Omega-\frac{\sigma^2}{n_{L+1}} I\Big) \hat{\Sigma}^{(L+1)} B_P(\Xi Q+I) e_i,
\end{align*}
when $\Xi=0$, this becomes $2 \lambda e_i^{\top} B_P \Sigma^{(L+1)} B_P\Big(\frac{\lambda}{p} Q^{-1} \Omega-\frac{\sigma^2}{n_{L+1}} I\Big) \hat{\Sigma}^{(L+1)} B_P e_i$. Therefore, the Riemannian gradient is given by
\begin{align*}
    &\operatorname{grad} \prisk(Q^{-1}\mid X^{(L+1)})=\left.\operatorname{grad} \prisk(P_{Q^{-1}}(\Xi)\mid X^{(L+1)})\right|_{\Xi=0}\\
    =&2 \lambda\Big[B_P \Sigma^{(L+1)} B_P\Big(\frac{\lambda}{p} Q^{-1} \Omega-\frac{\sigma^2}{n_{L+1}} I\Big) \hat{\Sigma}^{(L+1)} B_P\\
    &\quad\quad+\Big(B_P \Sigma^{(L+1)} B_P\Big(\frac{\lambda}{p} Q^{-1} \Omega-\frac{\sigma^2}{n_{L+1}} I\Big) \hat{\Sigma}^{(L+1)} B_P\Big)^{\top} \\
    &\quad\quad -\operatorname{diag}\Big\{B_P \Sigma^{(L+1)} B_P\Big(\frac{\lambda}{p} Q^{-1} \Omega-\frac{\sigma^2}{n_{L+1}} I\Big) \hat{\Sigma}^{(L+1)} B_P\Big\}\Big].
\end{align*} 
\end{proof}
\begin{proof}[Proof of Proposition \ref{statistical_advantage_of_Omega}]
We treat the risk function $\prisk(Q\mid X^{(L+1)})$ as a function of $Q^{-1}$. By Lemma \ref{lm_Riemannian_gradient_of_risk} calculate Riemannian gradient of $\prisk(Q\mid X^{(L+1)})$ w.r.t. $Q^{-1}$ is given by 
\begin{align*}
    &\operatorname{grad} \prisk(Q^{-1}\mid X^{(L+1)})=\left.\operatorname{grad} \prisk(P_{Q^{-1}}(\Xi)\mid X^{(L+1)})\right|_{\Xi=0}\\
    =&2 \lambda\Big[B_P \Sigma^{(L+1)} B_P\Big(\frac{\lambda}{p} Q^{-1} \Omega-\frac{\sigma^2}{n_{L+1}} I\Big) \hat{\Sigma}^{(L+1)} B_P\\
    &\quad\quad+\Big(B_P \Sigma^{(L+1)} B_P\Big(\frac{\lambda}{p} Q^{-1} \Omega-\frac{\sigma^2}{n_{L+1}} I\Big) \hat{\Sigma}^{(L+1)} B_P\Big)^{\top} \\
    &\quad\quad -\operatorname{diag}\Big\{B_P \Sigma^{(L+1)} B_P\Big(\frac{\lambda}{p} Q^{-1} \Omega-\frac{\sigma^2}{n_{L+1}} I\Big) \hat{\Sigma}^{(L+1)} B_P\Big\}\Big].
\end{align*}
Therefore, according to \citet[Proposition 4.6]{boumal2020introduction}, $Q^{-1}\in\mathbb{S}_{p}^{+}$ is a critical point if and only if $\operatorname{grad} \prisk(Q^{-1}\mid X^{(L+1)})=0$. By setting the Riemannian gradient to zero, it is easy to see that $Q^{-1}=\frac{p \sigma^2}{\lambda n_{L+1}} \Omega^{-1}$ is the critical point.

Now we prove that $Q^{*}=\frac{\lambda n_{L+1}}{p\sigma^2}\Omega$ is actually the global minimizor of the predictive risk. We pick any $Q_{0}\in\mathrm{T}_{Q^{*}}\mathbb{S}_{p}^{+}$ and $Q_{0}\not=Q^{*}$. Consider the line segment between $Q^{*}$ and $Q_{0}$, i.e. $\alpha Q^{*}+(1-\alpha)Q_{0}\in \mathrm{T}_{Q^{*}}\mathbb{S}_{p}^{+}$. We project this line segment to $Q_{\alpha}\in\mathrm{T}_{Q^{*}}\mathbb{S}_{p}^{+}$ such that when $\alpha=1$, $Q_{\alpha}=Q^{*}$. Note that $\alpha Q^{*}+(1-\alpha) Q_{0}=Q^{*}+(1-\alpha) Q_{0}-(1-\alpha) Q^{*}=Q^{*}+(1-\alpha)(Q_{0}-Q^{*})$, one can define
$$Q_{\alpha}\coloneqq Q^{*\frac{1}{2}} \exp \Big\{(1-\alpha)  Q^{*-\frac{1}{2}}(Q_{0}-Q^{*})  Q^{*-\frac{1}{2}}\Big\}  Q^{*\frac{1}{2}},\quad B_{\alpha}\coloneqq\big(\hat{\Sigma}^{(L+1)}+\lambda Q_{\alpha}^{-1}\big)^{-1},$$
and the predictive risk is given by
\begin{align*}
    \prisk(Q_{\alpha}\mid X^{(L+1)})=\sigma^2+\lambda \operatorname{tr}\Big(\Big(\frac{\lambda}{p} \Omega Q_{\alpha}^{-1}-\frac{\sigma^2}{n_{L+1}} I\Big)B_{\alpha} \Sigma^{(L+1)} B_{\alpha}Q_{\alpha}^{-1}\Big)+\frac{\sigma^2}{n_{L+1}} \operatorname{tr}(\Sigma^{(L+1)} B_\alpha).
\end{align*}
We show that with $Q^{*}=\frac{\lambda n_{L+1}}{p\sigma^2}\Omega$, the predictive risk is the global minimizer along every geodesical line ending at $Q^{*}$, i.e. $\frac{\partial}{\partial \alpha} \prisk(Q_{\alpha}\mid X^{(L+1)})\leq 0$ for any $\alpha\in [0,1)$ and $\frac{\partial}{\partial \alpha} \prisk(Q_{\alpha}\mid X^{(L+1)})=0$ at $\alpha=1$ under arbitrary choice of $Q_{0}\not=Q^{*}$. Note that
\begin{align*}
    \frac{\partial Q_{\alpha}^{-1}}{\partial \alpha}&=-Q_{\alpha}^{-1} \frac{\partial Q_{\alpha}}{\partial \alpha} Q_{\alpha}^{-1}\\
    \frac{\partial B_\alpha}{\partial \alpha}&=\lambda B_\alpha Q_{\alpha}^{-1} \frac{\partial Q_{\alpha}}{\partial \alpha} Q_{\alpha}^{-1} B_\alpha.
\end{align*}
Taking derivative w.r.t. $\alpha$ yields 
\begin{align*}
    &\frac{\partial}{\partial \alpha} \prisk(Q_{\alpha}\mid X^{(L+1)})\\
    =&\lambda\Big[-\operatorname{tr}\Big(\frac{\lambda}{p} \Omega  Q_{\alpha}^{-1}\frac{\partial  Q_{\alpha}}{\partial \alpha}  Q_{\alpha}^{-1} B_\alpha \Sigma^{(L+1)} B_\alpha  Q_{\alpha}^{-1}\Big) \\
    &\quad\quad+\lambda \operatorname{tr}\Big(\Big(\frac{\lambda}{p} \Omega  Q_{\alpha}^{-1}-\frac{\sigma^2}{n_{L+1}} I\Big) B_\alpha \Sigma^{(L+1)} B_\alpha  Q_{\alpha}^{-1}\frac{\partial  Q_{\alpha}}{\partial \alpha}  Q_{\alpha}^{-1} B_\alpha  Q_{\alpha}^{-1}\Big)\\
    &\quad\quad +\lambda \operatorname{tr}\Big(\Big(\frac{1}{p} \Omega  Q_{\alpha}^{-1}-\frac{\sigma^2}{n_{L+1}} I\Big) B_\alpha  Q_{\alpha}^{-1}\frac{\partial  Q_{\alpha}}{\partial \alpha}  Q_{\alpha}^{-1} B_\alpha \Sigma^{(L+1)} B_\alpha  Q_{\alpha}^{-1}\Big)\\
    &\quad\quad -\operatorname{tr}\Big(\Big(\frac{\lambda}{p} \Omega  Q_{\alpha}^{-1}-\frac{\sigma^2}{n_{L+1}} I\Big) B_\alpha \Sigma^{(L+1)} B_\alpha  Q_{\alpha}^{-1}\frac{\partial  Q_{\alpha}}{\partial \alpha}  Q_{\alpha}^{-1}\Big)\Big]\\
    &\quad+\frac{\lambda \sigma^2}{n_{L+1}} \operatorname{tr}\Big(\Sigma^{(L+1)} B_\alpha  Q_{\alpha}^{-1}\frac{\partial  Q_{\alpha}}{\partial \alpha}  Q_{\alpha}^{-1} B_\alpha\Big)\\
    &=(i)+(ii),
\end{align*}
where the first term could be simplified to 
\begin{align*}
    (i)&=\Big[-\operatorname{tr}\Big(\frac{\lambda^2}{p} B_\alpha \Sigma^{(L+1)} B_\alpha  Q_{\alpha}^{-1}  Q_{\alpha}  Q_{\alpha}^{-1}\frac{\partial  Q_{\alpha}}{\partial \alpha}  Q_{\alpha}^{-1}\Big)\\
    &\quad\quad+\frac{\lambda \sigma^2}{n_{L+1}} \operatorname{tr}\Big(B_\alpha \Sigma^{(L+1)} B_\alpha  Q_{\alpha}^{-1}\frac{\partial  Q_{\alpha}}{\partial \alpha}  Q_{\alpha}^{-1}\Big)\\
    &\quad\quad -\operatorname{tr}\Big(\Big(\frac{\lambda^2}{p} \Omega  Q_{\alpha}^{-1}-\frac{\lambda \sigma^2}{n_{L+1}} I\Big) B_\alpha \Sigma^{(L+1)} B_\alpha  Q_{\alpha}^{-1}\frac{\partial  Q_{\alpha}}{\partial \alpha}  Q_{\alpha}^{-1}\Big)\Big]\\
    &=\operatorname{tr}\Big(B_\alpha \Sigma^{(L+1)} B_\alpha\Big(\frac{\lambda \sigma^2}{n_{L+1}} I-\frac{\lambda^2}{p}  Q_{\alpha}^{-1} \Omega\Big)  Q_{\alpha}^{-1}\frac{\partial  Q_{\alpha}}{\partial \alpha}  Q_{\alpha}^{-1}\Big)\\
    &\qquad\qquad+\operatorname{tr}\Big(B_\alpha \Sigma^{(L+1)} B_\alpha\Big(\frac{\lambda \sigma^2}{n_{L+1}} I-\frac{\lambda^2}{p} Q_{\alpha}^{-1} \Omega \Big)  Q_{\alpha}^{-1}\frac{\partial  Q_{\alpha}}{\partial \alpha}  Q_{\alpha}^{-1}\Big)\\
    &=2 \lambda \operatorname{tr}\Big(B_\alpha \Sigma^{(L+1)} B_\alpha  Q_{\alpha}^{-1}\Big(\frac{\sigma^2}{n_{L+1}}  Q_{\alpha}-\frac{\lambda}{p} \Omega\Big)  Q_{\alpha}^{-1}\frac{\partial  Q_{\alpha}}{\partial \alpha}  Q_{\alpha}^{-1}\Big),
\end{align*}
and 
\begin{align*}
    (ii)&=\lambda^2 \operatorname{tr}\Big(\Big(\frac{\lambda}{p} \Omega  Q_{\alpha}^{-1}-\frac{\sigma^2}{n_{L+1}} I\Big) B_\alpha  Q_{\alpha}^{-1}\frac{\partial  Q_{\alpha}}{\partial \alpha}  Q_{\alpha}^{-1} B_\alpha \Sigma^{(L+1)} B_\alpha  Q_{\alpha}^{-1}\Big)\\
    &\quad\quad+\lambda^2 \operatorname{tr}\Big(\Big(\frac{\lambda}{p} \Omega  Q_{\alpha}^{-1}-\frac{\sigma^2}{n_{L+1}} I\Big) B_\alpha \Sigma^{(L+1)} B_\alpha  Q_{\alpha}^{-1}\frac{\partial  Q_{\alpha}}{\partial \alpha}  Q_{\alpha}^{-1} B_\alpha  Q_{\alpha}^{-1}\Big)\\
    &=2 \lambda^2 \operatorname{tr}\Big(B_\alpha  Q_{\alpha}^{-1}\Big(\frac{\lambda}{p} \Omega-\frac{\sigma^2}{n_{L+1}}  Q_{\alpha}\Big)  Q_{\alpha}^{-1} B_\alpha  Q_{\alpha}^{-1}\frac{\partial  Q_{\alpha}}{\partial \alpha}  Q_{\alpha}^{-1} B_\alpha \Sigma^{(L+1)}\Big).
\end{align*}
Combining the two terms together, we obtain 
\begin{align*}
    &\frac{\partial \prisk( Q_{\alpha}\mid X^{(L+1)})}{\partial \alpha}\\
    =& 2 \lambda \operatorname{tr}\Big(B_\alpha \Sigma^{(L+1)} B_\alpha  Q_{\alpha}^{-1}\Big(\frac{\sigma^2}{n_{L+1}}  Q_{\alpha}-\frac{\lambda}{p} \Omega\Big)\Big(I-\lambda  Q_{\alpha}^{-1} B_\alpha\Big)  Q_{\alpha}^{-1}\frac{\partial  Q_{\alpha}}{\partial \alpha}  Q_{\alpha}^{-1}\Big)\\
    =&2 \lambda \operatorname{tr}\Big(B_\alpha \Sigma^{(L+1)} B_\alpha  Q_{\alpha}^{-1}\Big(\frac{\sigma^2}{n_{L+1}}  Q_{\alpha}-\frac{\lambda}{p} \Omega\Big) \hat{\Sigma}^{(L+1)} B_\alpha  Q_{\alpha}^{-1}\frac{\partial  Q_{\alpha}}{\partial \alpha}  Q_{\alpha}^{-1}\Big).
\end{align*}
We now show that $\frac{\partial \prisk(Q_{\alpha}\mid X^{(L+1)})}{\partial \alpha}\leq 0$. Note that
\begin{align*}
     Q_{\alpha}&=Q^{* \frac{1}{2}} \exp \big\{(1-\alpha) Q^{* -\frac{1}{2}}(Q_{0}-Q^{*}) Q^{*-\frac{1}{2}}\big\} Q^{* \frac{1}{2}}\\
    \frac{d  Q_{\alpha}}{d \alpha}&=-Q^{* \frac{1}{2}} \frac{\partial}{\partial (1-\alpha)}\exp \big\{(1-\alpha)Q^{* -\frac{1}{2}}(Q_{0}-Q^{*})Q^{* -\frac{1}{2}}\big\}Q^{* \frac{1}{2}},
\end{align*} 
and
\begin{align*}
\frac{\sigma^2}{n_{L+1}}  Q_{\alpha}-\frac{\lambda}{p} \Omega&=\frac{\sigma^2}{n_{L+1}} ( Q_{\alpha}-Q^{*} )\\
&=\frac{\sigma^2}{n_{L+1}} Q^{*\frac{1}{2}}\Big(\exp \big\{(1-\alpha) Q^{*-\frac{1}{2}} (Q_{0}-Q^{*} )Q^{*-\frac{1}{2}}\big\}-I\Big)Q^{*\frac{1}{2}}. 
\end{align*}
By plugging these expressions in the derivative, we get
\begin{align*}
   \frac{\partial \prisk( Q_{\alpha}\mid X^{(L+1)})}{\partial \alpha}&=-\frac{2 \lambda\sigma^2}{n_{L+1}} \operatorname{tr}\Big(A_{1}\big(\exp\{(1-\alpha)  Q^{*-\frac{1}{2}}(Q_{0}-Q^{*}) Q^{*-\frac{1}{2}}\}-I\big)\\
   &\quad\quad\quad\quad\quad\quad\quad\cdot A_{2} \frac{\partial \exp\{(1-\alpha) Q^{*-\frac{1}{2}} (Q_{0}-Q^{*} ) Q^{*-\frac{1}{2}}\}}{\partial (1-\alpha)}   \Big), 
\end{align*}
where $A_{1}\coloneqq Q^{*\frac{1}{2}}Q_{\alpha}^{-1}B_\alpha \Sigma^{(L+1)} B_\alpha  Q_{\alpha}^{-1} Q^{*\frac{1}{2}}$ and $A_{2}\coloneqq Q^{*\frac{1}{2}} \hat{\Sigma}^{(L+1)} B_\alpha  Q_{\alpha}^{-1} Q^{*\frac{1}{2}}$. Since $A_{1}$ and $A_{2}$ are two positive definite matrices. By Lemma \ref{technical_lemma_1}, it holds that
\begin{align*}
    \operatorname{tr}\Big(&\big(\exp \{(1-\alpha)  Q^{*-\frac{1}{2}}(Q_{0}-Q^{*}) Q^{*-\frac{1}{2}}\}-I\big)\\
    &\quad\quad\quad\frac{\partial}{\partial (1-\alpha)}\exp \{(1-\alpha) Q^{*-\frac{1}{2}}(Q_{0}-Q^{*}) Q^{*-\frac{1}{2}}\}\Big)\geq 0.
\end{align*}
Hence, $\frac{\partial \prisk(Q_{\alpha}\mid X^{(L+1)})}{\partial \alpha}\leq 0$ for any $\alpha\in [0,1]$. Besides, for any $Q_{0}\in\mathbb{S}_{p}^{+}$, 
\begin{align*}
    \left.\frac{\partial \prisk(Q_{\alpha}\mid X^{(L+1)})}{\partial\alpha}\right|_{\alpha=1}=0.
\end{align*}
Therefore, $Q^{*}=\frac{n_{L+1}\lambda}{p\sigma^{2}}\Omega$ is the global minimizer of the predictive risk. Besides, the risk in this case is given by
\begin{align*}
    \prisk(Q^{*}\mid X^{(L+1)})=\sigma^2+\frac{\sigma^2}{n_{L+1}} \operatorname{tr}\Big(\Sigma^{(L+1)}\big(\hat{\Sigma}^{(L+1)}+\frac{p \sigma^2}{n_{L+1}} \Omega^{-1}\big)^{-1}\Big).
\end{align*}
\end{proof}
\section{Proofs for Section \ref{MTL_cov_est_section}}
\begin{proof}[Proof of Proposition \ref{geod_convex_of_f}]
One needs to prove $$g(\Omega)=\Big\| y^{(\ell)} y^{(\ell) \top}-\frac{1}{p}X^{(\ell)} \Omega X^{(\ell) \top}-\sigma^{2}I\Big\|_F^2$$ is geodesically convex. Define $$g_{1}(\Omega)=\frac{1}{p}X^{(\ell)} \Omega X^{(\ell) \top}+\sigma^{2}I-y^{(\ell)} y^{(\ell) \top}\quad\quad g_{2}(\Omega)=\|\Omega\|_{F}^{2}$$
Then $g(\Omega)=g_{2}\circ g_{1}(\Omega)$. We note that $g_{2}$ are convex in usual sense. And if $\Omega_{1}\geq \Omega_{2}$ in Löwner order, then 
$$g_1(\Omega_{1})-g_1(\Omega_{2})=\frac{1}{p}X^{(\ell)} (\Omega_{1}-\Omega_{2}) X^{(\ell) \top}\geq 0$$
i.e. $g_{1}$ is monotone. Besides, $g_{2}$ monotone on the set of positve definite matrices $\mathbb{S}_{p}^{+}$, then by \citet[Proposition 3.5, Property (8)]{lim2013convex}, we have that $g=g_{2}\circ g_{1}$ is geodesically convex.
\end{proof}
\subsection{Technical Lemmas}

We start with a few preliminary results required to prove Theorem~\ref{consistency_of_Omegahat_Lpn}.

\begin{lemma}\label{concentration_lm_01}
Suppose that $\bar{\beta}^{(\ell)}$'s are independent zero mean and sub-Gaussian with parameter $\tau_{\beta}$; $\varepsilon^{(\ell)}$'s are independent zero mean and sub-Gaussian with parameter $\tau_{\varepsilon}$, then
\begin{align*}
    \mathbb{E}\Big\|\bar{\beta}^{(\ell)}\bar{\beta}^{(\ell)\top}-\frac{1}{p}\Omega\Big\|^{k}&\leq 1+4k\Big(\frac{2}{C_{\beta}p}\Big)^k\Gamma(k),\\
    \mathbb{E}\big\|\varepsilon^{(\ell)}\varepsilon^{(\ell)\top}-\sigma^{2}I\big\|^{k}&\leq n_{\ell}^{k}+4k\Big(\frac{2}{C_{\tau}}\Big)^k\Gamma(k),\\
    \mathbb{E} \|\bar{\beta}^{(\ell)}\|_{2}^{k}&\leq 1+ k\Big(\frac{4 \tau_{\beta}^2}{p}\Big)^{\frac{k}{2}} \Gamma(k/2),\\
    \mathbb{E}\|\varepsilon^{(\ell)}\|_{2}^{k}&\leq n_{\ell}^{\frac{k}{2}}+k(4 \tau_{\varepsilon}^2)^{\frac{k}{2}} \Gamma(k/2).
\end{align*}
In particular,
\begin{align*}
    \mathbb{E}\Big\|\bar{\beta}^{(\ell)}\bar{\beta}^{(\ell)\top}-\frac{1}{p}\Omega\Big\|^{2}&\leq 1+\frac{32}{C_{\beta}^{2}p^2}= \mathcal{O}(1),\\
    \mathbb{E}\|\varepsilon^{(\ell)}\varepsilon^{(\ell)\top}-\sigma^{2}I\|^{2}&\leq  n_{\ell}^{2}+\frac{32}{C_{\varepsilon}^{2}}=\mathcal{O}(n_{\ell}^{2}),\\
    \mathbb{E}\|\bar{\beta}^{(\ell)}\|_{2}^{2}&\leq 1+8 \tau_{\beta}^2p^{-1}=\mathcal{O}(1),\\
    \mathbb{E}\|\varepsilon^{(\ell)}\|_{2}^{2}&\leq n_{\ell}+8 \tau_{\varepsilon}^2 =\mathcal{O}(n_{\ell}).
\end{align*}
\end{lemma}
\begin{proof}[Proof of Lemma \ref{concentration_lm_01}]
Suppose that $\sqrt{p}\bar{\beta}^{(\ell)}$ is sub-Gaussian with parameter $\tau_{\beta}$, we bound the $\mathbb{E}\Big\|\bar{\beta}^{(\ell)}\bar{\beta}^{(\ell)\top}-\frac{1}{p}\Omega\Big\|^{k}$ using concentration results on $\bar{\beta}^{(\ell)}$:
\begin{align*}
    &\mathbb{E}\Big\|\bar{\beta}^{(\ell)} \bar{\beta}^{(\ell) \top}-\frac{1}{p} \Omega\Big\|^k\\
    =&\frac{1}{p^{k}}\mathbb{E} \|p\bar{\beta}^{(\ell)}\bar{\beta}^{(\ell)\top}-\Omega \|^{k}=\frac{1}{p^k}\int_0^{+\infty} \mathbb{P}\big( \|p\bar{\beta}^{(\ell)} \bar{\beta}^{(\ell)\top}- \Omega \|^k \geq u\big) d u\\
    =&\frac{1}{p^k}\int_0^{p^{k}} \mathbb{P}\big(\|p\bar{\beta}^{(\ell)}\bar{\beta}^{(\ell)\top}- \Omega\| \geq u^{\frac{1}{k}}\big) d u+\frac{1}{p^k}\int_{p^{k}}^{+\infty} \mathbb{P}\big(\|p\bar{\beta}^{(\ell)}\bar{\beta}^{(\ell)\top}- \Omega\| \geq u^{\frac{1}{k}}\big) d u\\
    \leq &\frac{1}{p^k}\Big( p^{k}+\int_{p^{k}}^{+\infty} 2 \cdot 9^p \exp \big\{-C_{\beta}u^{\frac{1}{k}}/2\big\} d u\Big) \quad\quad C_{\beta}=\min \Big\{\frac{1}{(32 \tau_{\beta}^2)^2}, \frac{1}{32 \tau_{\beta}^2}\Big\}\\
    =&\frac{1}{p^k}\Big(p^{k}+2 \cdot 9^p\int_{\frac{C_{\beta}}{2}p}^{+\infty} k\Big(\frac{2}{C_{\beta}}\Big)^k v^{k-1} e^{-v} d v\Big)  \quad\quad v=\frac{1}{2} C_{\beta} u^{\frac{1}{k}}, d u=k\Big(\frac{2}{C_{\beta}}\Big)^k v^{k-1}\\
    =&\frac{1}{p^k}\Big(p^{k}+2k\Big(\frac{2}{C_{\beta}}\Big)^k \int_{\frac{C_{\beta}}{2}p}^{+\infty}9^p v^{k-1} e^{-v} d v\Big) \\
    \leq &\frac{1}{p^k}\Big(p^{k}+4k\Big(\frac{2}{C_{\beta}}\Big)^{k}\int_{\frac{C_{\beta}}{2}p}^{+\infty}v^{k-1}e^{-\frac{v}{2}}dv\Big)  \\
    \leq& 1+4k\Big(\frac{4}{C_{\beta}p}\Big)^{k}\Gamma(k).
\end{align*}
Now we bound the term $\mathbb{E}\|\varepsilon^{(\ell)}\varepsilon^{(\ell)\top}-\sigma^{2}I\|^{k}$ in the same way
\begin{align*}
    &\mathbb{E}\|\varepsilon^{(\ell)}\varepsilon^{(\ell)\top}-\sigma^{2}I\|^{k}\\
    =&\int_0^{n_{\ell}^{k}} \mathbb{P}\Big(\|\varepsilon^{(\ell)}\varepsilon^{(\ell)\top}-\sigma^{2}I \| \geq u^{\frac{1}{k}}\Big) d u+\int_{n_{\ell}^{k}}^{+\infty} \mathbb{P}\Big( \|\varepsilon^{(\ell)}\varepsilon^{(\ell)\top}-\sigma^{2}I \| \geq u^{\frac{1}{k}}\Big) d u\\
    \leq &n_{\ell}^{k}+\int_{n_{\ell}^{k}}^{+\infty} 2 \cdot 9^{n_{\ell}} \exp \big\{- C_{\varepsilon}  u^{\frac{1}{k}}/2\big\} d u \quad\quad C_{\varepsilon}=\min \Big\{\frac{1}{(32 \tau_{\varepsilon}^2)^2}, \frac{1}{32 \tau_{\varepsilon}^2}\Big\}\\
    =&n_{\ell}^{k}+2 \cdot 9^{n_{\ell}}\int_{\frac{C_{\varepsilon}}{2}n_{\ell}}^{+\infty} k\Big(\frac{2}{C_{\varepsilon}}\Big)^k v^{k-1} e^{-v} d v  \quad\quad v=\frac{1}{2} C_{\varepsilon} u^{\frac{1}{k}}, d u=k\Big(\frac{2}{C_{\varepsilon}}\Big)^k v^{k-1}\\
    =&n_{\ell}^{k}+2k\Big(\frac{2}{C_{\varepsilon}}\Big)^k \int_{\frac{C_{\varepsilon}}{2}n_{\ell}}^{+\infty}9^{n_{\ell}} v^{k-1} e^{-v} d v \\
    \leq& n_{\ell}^{k}+4k\Big(\frac{2}{C_{\varepsilon}}\Big)^{k}\int_{\frac{C_{\varepsilon}}{2}n_{\ell}}^{+\infty}v^{k-1}e^{-\frac{v}{2}}dv  \\
    \leq& n_{\ell}^{k}+4k\Big(\frac{4}{C_{\varepsilon}}\Big)^k\Gamma(k).
\end{align*}
Now in particular, if $k=2$, we have 
\begin{align*}
    \mathbb{E}\Big\|\bar{\beta}^{(\ell)}\bar{\beta}^{(\ell)\top}-\frac{1}{p}\Omega\Big\|^{2}&\leq 1+\frac{32}{p^2C_{\beta}^{2}},\\
    \mathbb{E} \|\varepsilon^{(\ell)}\varepsilon^{(\ell)\top}-\sigma^{2}I \|^{2}&\leq n_{\ell}^{2}+\frac{32}{C_{\varepsilon}^{2}}.
\end{align*}
Similarly,
\begin{align*}
    &\mathbb{E}\|\varepsilon^{(\ell)}\|_2^k\\
    =&\int_0^{n_{\ell}^{\frac{k}{2}}} \mathbb{P}\big(\|\varepsilon^{(\ell)}\|_2^k \geq u\big) d u  +\int_{n_{\ell}^{\frac{k}{2}}}^{+\infty} \mathbb{P}\big(\|\varepsilon^{(\ell)}\|_2^k \geq u\big) d u\\
    \leq & n_{\ell}^{\frac{k}{2}}+\int_{n_{\ell}^{\frac{k}{2}}}^{+\infty} 5^{n_{\ell}} \exp \Big\{-\frac{u^{\frac{2}{k}}}{2 \tau_{\varepsilon}^2}\Big\} du\\
    =& n_{\ell}^{\frac{k}{2}}+ \int_{\frac{n_{\ell}}{2 \tau_{\varepsilon}^2}}^{+\infty} 5^{n_{\ell}} \exp\{-v\} (2\tau_{\varepsilon}^2)^{\frac{k}{2}}\frac{k}{2} v^{\frac{k}{2}-1} d v \quad v=\frac{1}{2 \tau_{\varepsilon}^2} u^{\frac{2}{k}}; d u=(2 \tau_{\varepsilon}^2)^{\frac{k}{2}} \frac{k}{2} v^{\frac{k}{2}-1} d v  \\
    =&n_{\ell}^{\frac{k}{2}}+(2 \tau_{\varepsilon}^2)^{\frac{k}{2}} \frac{k}{2} \int_{\frac{n_{\ell}}{2 \tau_{\varepsilon}^2}}^{+\infty} 5^{n_{\ell}} \exp\{-v\} v^{\frac{k}{2}-1} d v \\
    \leq& n_{\ell}^{\frac{k}{2}}+ k(4 \tau_{\varepsilon}^2)^{\frac{k}{2}} \Gamma\Big(\frac{k}{2}\Big),
\end{align*}
and
\begin{align*}
    &\mathbb{E}\|\bar{\beta}^{(\ell)}\|_2^k=\frac{1}{p^{\frac{k}{2}}}\mathbb{E}\|\sqrt{p}\bar{\beta}^{(\ell)}\|_{2}^{k}=\frac{1}{p^{\frac{k}{2}}}\int_0^{+\infty} \mathbb{P}\big(\|\sqrt{p}\bar{\beta}^{(\ell)}\|_2^k \geq u\big) d u\\
    =&\frac{1}{p^{\frac{k}{2}}}\int_0^{p^{\frac{k}{2}}} \mathbb{P}\big(\|\sqrt{p}\bar{\beta}^{(\ell)}\|_2^k \geq u\big)du+\frac{1}{p^{\frac{k}{2}}}\int_{p^{\frac{k}{2}}}^{+\infty} \mathbb{P}\big(\|\bar{\beta}^{(\ell)}\|_2^k \geq u\big) d u\\
    \leq & 1+\frac{1}{p^{\frac{k}{2}}}\int_{p^{\frac{k}{2}}}^{+\infty} 5^p \exp \Big\{-\frac{u^{\frac{2}{k}}}{2 \tau_{\beta}^2}\Big\} d u\\
    = & 1+ \frac{1}{p^{\frac{k}{2}}}\int_{\frac{p}{2 \tau_{\beta}^2}}^{+\infty} 5^{p} \exp\{-v\} (2 \tau_{\beta}^2)^{\frac{k}{2}} \frac{k}{2} v^{\frac{k}{2}-1} d v \quad\quad v=\frac{1}{2 \tau_{\beta}^2} u^{\frac{2}{k}}; d u=(2 \tau_{\beta}^2)^{\frac{k}{2}} \frac{k}{2} v^{\frac{k}{2}-1} d v  \\
    = & 1+\big(\frac{2 \tau_{\beta}^2}{p}\big)^{\frac{k}{2}} \frac{k}{2} \int_{\frac{p}{2 \tau_{\beta}^2}}^{+\infty} 5^{p} \exp\{-v\} v^{\frac{k}{2}-1} d v \\
    \leq & 1+ k\big(\frac{4 \tau_{\beta}^2}{p}\big)^{\frac{k}{2}} \Gamma\big(\frac{k}{2}\big).
\end{align*}
Hence, in particular,
\begin{align*}
    \mathbb{E}\|\bar{\beta}^{(\ell)}\|_{2}^{2}&\leq 1+8 \tau_{\beta}^2p^{-1}=\mathcal{O}(1),\\
    \mathbb{E}\|\varepsilon^{(\ell)}\|_{2}^{2}&\leq n_{\ell}+8 \tau_{\varepsilon}^2 =\mathcal{O}(n_{\ell}).
\end{align*}
\end{proof} 
%The main tool to prove this problem is Theorem 2.7 from \cite{koltchinskii2011oracle} (stated in Theorem \ref{concentration_thm_01} from Appendix), which establish the concentration inequality in terms of $\psi_{\alpha}$ norm ($\alpha\geq 1$) of the individual quantities appears in the summation from Riemannian gradient $\operatorname{grad}f(\Omega)$. 
In the next lemma, we bound the $\psi_1$ norm of related quantities appears in Riemannian gradient $\operatorname{grad}f(\Omega)$. According to \citet[Equation (1)]{maurer2021concentration}, we can define the usual sub-Gaussian and sub-exponential norms $\|\cdot\|_{\psi_2}$ and $\|\cdot\|_{\psi_1}$ for any real random variable $Z$ as
\begin{equation}\label{psi_norm_def}
\|Z\|_{\psi_2}=\sup _{k \geq 1} \frac{\|Z\|_k}{\sqrt{k}} \text { and }\|Z\|_{\psi_1}=\sup _{k \geq 1} \frac{\|Z\|_k}{k} 
\end{equation}
where the $L_k$-norms are defined as $\|Z\|_k=\big(\mathbb{E}\big[|Z|^k\big]\big)^{1 / k}$.

\begin{lemma}\label{concentration_lm_02}
Suppose that $\bar{\beta}^{(\ell)}$'s are independent zero mean and sub-Gaussian with parameter $\tau_{\beta}$; $\varepsilon^{(\ell)}$'s are independent zero mean and sub-Gaussian with parameter $\tau_{\varepsilon}$, then
\begin{align*}
    \Big\|\Big\|\frac{1}{\sqrt{p}}X^{(\ell)\top} X^{(\ell)}\Big(\bar{\beta}^{(\ell)} \bar{\beta}^{(\ell)\top}-\frac{1}{p} \Omega\Big) X^{(\ell)\top} X^{(\ell)}\Big\|\Big\|_{\psi_{1}}&\leq 8n_{\ell}^{2}\big(1+\sqrt{\gamma_{\ell}}\big)^{4}\lambda_{\max}^{2}\big(\Sigma^{(\ell)}\big)\sqrt{p},\\
    \Big\|\Big\|\frac{1}{\sqrt{p}} X^{(\ell)\top}\big(\varepsilon^{(\ell)} \varepsilon^{(\ell) \top}-\sigma^2 I\big) X^{(\ell)}\Big\|\Big\|_{\psi_{1}}&\leq 4 \frac{n_{\ell}^{\frac{5}{2}}}{\sqrt{p}} (1+\sqrt{\gamma_{\ell}})^{2}\lambda_{\max }\big(\Sigma^{(\ell)}\big),
\end{align*}
almost surely.
\end{lemma}  
\begin{proof}[Proof of Lemma \ref{concentration_lm_02}]
Note that
\begin{align*}
&\Big\|\Big\|\frac{1}{\sqrt{p}}X^{(\ell)\top} X^{(\ell)}\big(\bar{\beta}^{(\ell)} \bar{\beta}^{(\ell)\top}-\frac{1}{p} \Omega\big) X^{(\ell)\top} X^{(\ell)}\Big\|\Big\|_{\psi_{1}}\\
\leq& \frac{4n_{\ell}^{2}}{\sqrt{p}}(1+\sqrt{\gamma_{\ell}})^{4}\lambda_{\max}^{2}\big(\Sigma^{(\ell)}\big)\Big\|\Big\|\bar{\beta}^{(\ell)} \bar{\beta}^{(\ell)\top}-\frac{1}{p} \Omega\Big\|\Big\|_{\psi_{1}}    \\
\leq& \frac{4n_{\ell}^{2}}{\sqrt{p}}(1+\sqrt{\gamma_{\ell}})^{4}\lambda_{\max}^{2}(\Sigma^{(\ell)})\sup_{k\geq 1}\frac{\big[ \mathbb{E}\big\|\bar{\beta}^{(\ell)} \bar{\beta}^{(\ell) \top}-\frac{1}{p} \Omega\big\|^k\big]^{\frac{1}{k}}}{k}.
\end{align*}
Therefore, to bound this $\psi_1$ norm, it suffices to bound $\sup_{k\geq 1}\frac{\big[ \mathbb{E}\big\|\bar{\beta}^{(\ell)} \bar{\beta}^{(\ell) \top}-\frac{1}{p} \Omega\big\|^k\big]^{\frac{1}{k}}}{k}$.
\begin{align*}
    \sup_{k\geq 1}\frac{\Big[ \mathbb{E}\Big\|\bar{\beta}^{(\ell)} \bar{\beta}^{(\ell) \top}-\frac{1}{p} \Omega\Big\|^k\Big]^{\frac{1}{k}}}{k}&\leq \sup_{k\geq 1}\frac{1}{k}\Big(p^{k}+4k\Big(\frac{2}{C_{\beta}}\Big)^k\Gamma(k)\Big)^{\frac{1}{k}}\\
    &=1+\frac{8}{C_{\beta}p}\leq 2.
\end{align*}
Hence, it holds that for $p$ sufficiently large,
\begin{align*}
\Big\|\Big\|\frac{1}{\sqrt{p}}X^{(\ell)\top} X^{(\ell)}\Big(\bar{\beta}^{(\ell)} \bar{\beta}^{(\ell)\top}-\frac{1}{p} \Omega\Big) X^{(\ell)\top} X^{(\ell)}\Big\|\Big\|_{\psi_{1}}&\leq \frac{4n_{\ell}^{2}}{\sqrt{p}}(1+\sqrt{\gamma_{\ell}})^{4}\lambda_{\max}^{2}(\Sigma^{(\ell)})2\\
&=8n_{\ell}^{2}(1+\sqrt{\gamma_{\ell}})^{4}\lambda_{\max}^{2}(\Sigma^{(\ell)})/\sqrt{p}.
\end{align*}
Similarly,
\begin{align*}
&\Big\|\Big\|\frac{1}{\sqrt{p}} X^{(\ell)\top}(\varepsilon^{(\ell)} \varepsilon^{(\ell) \top}-\sigma^2 I) X^{(\ell)}\Big\|\Big\|_{\psi_{1}}\\ 
\leq & 2 \sqrt{\frac{n_{\ell}}{p}} n_{\ell} (1+\sqrt{\gamma_{\ell}})^{2}\lambda_{\max }(\Sigma^{(\ell)})\big\|\big\|\varepsilon^{(\ell)} \varepsilon^{(\ell)\top}-\sigma^2 I\big\|\big\|_{\psi_{1}}\\
\leq & 2 \sqrt{\frac{n_{\ell}}{p}} n_{\ell} (1+\sqrt{\gamma_{\ell}})^{2}\lambda_{\max }(\Sigma^{(\ell)}) \sup_{k\geq 1}\frac{\big[\mathbb{E}\big\|\varepsilon^{(\ell)} \varepsilon^{(\ell)\top}-\sigma^2 I\big\|^{k}\big]^{\frac{1}{k}}}{k} \\
\leq & 2 \sqrt{\frac{n_{\ell}}{p}} n_{\ell} (1+\sqrt{\gamma_{\ell}})^{2}\lambda_{\max }(\Sigma^{(\ell)}) \sup_{k\geq 1}\frac{\Big(n_{\ell}^{k}+4k\big(\frac{2}{C_{\tau}}\big)^k\Gamma(k)\Big)^{\frac{1}{k}}}{k} \\
\leq & 2 \sqrt{\frac{n_{\ell}}{p}} n_{\ell} (1+\sqrt{\gamma_{\ell}})^{2}\lambda_{\max }(\Sigma^{(\ell)})\big(n_{\ell}+8/C_{\varepsilon}\big)\\
\leq & 4  n_{\ell}^{\frac{5}{2}} (1+\sqrt{\gamma_{\ell}})^{2}\lambda_{\max }(\Sigma^{(\ell)})/\sqrt{p}.
\end{align*}
\end{proof}

Another tool we use to prove Theorem \ref{consistency_of_Omegahat_Lpn} is the concentration inequality stated in Theorem \ref{concentration_thm_01}.

\subsection{Consistency of $\hat{\Omega}$ as $L,n_{\ell},p$ go to infinity}
Now we are ready to prove Theorem \ref{consistency_of_Omegahat_Lpn}.
\begin{proof}[Proof of Theorem \ref{consistency_of_Omegahat_Lpn}]
Define the target function,  by the definition of Frobenius norm, to be
\begin{align*}
    f(\tilde{\Omega})=\frac{1}{L} \sum_{\ell=1}^L \operatorname{tr}\Big[\big(y^{(\ell)} y^{(\ell)\top}-\frac{1}{p} X^{(\ell)} \tilde{\Omega} X^{(\ell)^{\top}}-\sigma^2 I\big)^{\top}\big(y^{(\ell)} y^{(\ell)\top}-\frac{1}{p} X^{(\ell)} \tilde{\Omega} X^{(\ell)\top}-\sigma^{2} I\big)\Big].%\label{targetfun_general}
\end{align*}
The minimizer of $f(\tilde{\Omega})$ could be characterized by setting the Riemannian gradient to be zero. Using retraction $P_{\tilde{\Omega}}(\Xi)=\tilde{\Omega}+\Xi+\frac{1}{2} \Xi \tilde{\Omega}^{-1} \Xi$, the Riemannian gradient is given by
\begin{align*}
    \operatorname{grad} f(\tilde{\Omega})=-\frac{4}{p L} \sum_{\ell=1}^L X^{(\ell)^{\top}}\Big(y^{(\ell)} y^{(\ell)\top}-\frac{1}{p} X^{(\ell)} \tilde{\Omega} X^{(\ell)^{\top}}-\sigma^2 I\Big) X^{(\ell)}.
\end{align*}
Note that $y^{(\ell)}=X^{(\ell)}\bar{\beta}^{(\ell)}+\varepsilon^{(\ell)}$, hence, it holds that
\begin{align*}
    &\frac{1}{L}\sum_{\ell=1}^{L}X^{(\ell)^{\top}}y^{(\ell)}y^{(\ell)^\top}X^{(\ell)}\\
    =&\frac{1}{L}\sum_{\ell=1}^{L}X^{(\ell)^{\top}}\Big(X^{(\ell)}\bar{\beta}^{(\ell)}+\varepsilon^{(\ell)}\Big)\Big(X^{(\ell)}\bar{\beta}^{(\ell)}+\varepsilon^{(\ell)}\Big)^\top X^{(\ell)}\\
    =&\frac{1}{L}\sum_{\ell=1}^{L}X^{(\ell)^{\top}}X^{(\ell)}\bar{\beta}^{(\ell)}\bar{\beta}^{(\ell)\top}X^{(\ell)\top}X^{(\ell)}+\frac{1}{L}\sum_{\ell=1}^{L}X^{(\ell)^{\top}}\varepsilon^{(\ell)}\varepsilon^{(\ell)\top}X^{(\ell)}\\
    &\quad\quad+\frac{1}{L}\sum_{\ell=1}^{L}X^{(\ell)^{\top}}\varepsilon^{(\ell)}\bar{\beta}^{(\ell)\top}X^{(\ell)\top}X^{(\ell)}+\frac{1}{L}\sum_{\ell=1}^{L}X^{(\ell)^{\top}}X^{(\ell)}\bar{\beta}^{(\ell)}\varepsilon^{(\ell)\top}X^{(\ell)},
\end{align*}
which implies that
\begin{align*}
    \operatorname{grad}f(\Omega)&=-\frac{4}{pL} \sum_{\ell=1}^L X^{(\ell)^{\top}} X^{(\ell)}\Big(\bar{\beta}^{(\ell)} \bar{\beta}^{(\ell)^{\top}}-\frac{1}{p} \Omega\Big) X^{(\ell)^{\top}} X^{(\ell)}\\
    &\quad\quad-\frac{4}{pL} \sum_{\ell=1}^L X^{(\ell)^{\top}}\big(\varepsilon^{(\ell)} \varepsilon^{(\ell)^{\top}}-\sigma^2 I\big) X^{(\ell)}\\
    &\quad\quad-\frac{4}{pL}\sum_{\ell=1}^{L}X^{(\ell)^{\top}}\varepsilon^{(\ell)}\bar{\beta}^{(\ell)\top}X^{(\ell)\top}X^{(\ell)}-\frac{4}{pL}\sum_{\ell=1}^{L}X^{(\ell)^{\top}}X^{(\ell)}\bar{\beta}^{(\ell)}\varepsilon^{(\ell)\top}X^{(\ell)}.
\end{align*}
The main idea to prove $\|\hat{\Omega}-\Omega\|_{F}\stackrel{p}{\rightarrow}0$ is trying to bound the $\|\hat{\Omega}-\Omega\|_{F}$ by $\|\operatorname{grad} f(\Omega)\|_{F}$. We first prove that under Assumption \ref{nonsparse_asp},%\todo{which first condition?} 
 it holds that for sufficiently large $p, n_{\ell}$ and $L$,
\begin{align}
2cd^{2}\min_{1\leq \ell\leq L} \frac{n_{\ell}^2}{p^2} \|\Omega-\hat{\Omega}\|_{F} \leq \|\operatorname{grad}f(\Omega)\|_{F}. \label{Key_equation1_concentration}
\end{align} 
where $c$ is the limit of $L_0/L$ as $L\rightarrow\infty$ and $d=\inf_{\ell\leq L_0}\frac{1}{2}(1-\sqrt{\gamma_{\ell}})^{2}\lambda_{\min}(\Sigma^{(\ell)})$. To prove \eqref{Key_equation1_concentration}, note that for the minimizer $\hat{\Omega}$ of \eqref{optimization_for_Omega}, $\operatorname{grad}f(\hat{\Omega})=0$ so it holds that
\begin{align*}
|\langle \operatorname{grad}f(\Omega)-\operatorname{grad}f(\hat{\Omega}),\Omega-\hat{\Omega}\rangle|&=|\langle \operatorname{grad}f(\Omega),\Omega-\hat{\Omega}\rangle|\\
&\leq \|\operatorname{grad}f(\Omega)\|_{F}\|\hat{\Omega}-\Omega\|_{F},
\end{align*}
and
\begin{align*}
  \operatorname{grad}f(\Omega)-\operatorname{grad} f(\hat{\Omega})&=-\frac{4}{p L} \sum_{\ell=1}^L \frac{1}{p} X^{(\ell)^{\top}} X^{(\ell)}(\hat{\Omega}-\Omega) X^{(\ell)^{\top}} X^{(\ell)} \\
  &=\frac{4}{L} \sum_{\ell=1}^L \frac{n_{\ell}^{2}}{p^2}\frac{1}{n_{\ell}^2}X^{(\ell)^{\top}} X^{(\ell)}(\Omega-\hat{\Omega}) X^{(\ell)^{\top}} X^{(\ell)}.
\end{align*} 
Therefore, the inner product $\langle \operatorname{grad}f(\Omega)-\operatorname{grad}f(\hat{\Omega}),\Omega-\hat{\Omega}\rangle$ is given by
\begin{equation*}
    \langle\operatorname{grad}f(\Omega)-\operatorname{grad}f(\hat{\Omega}),\Omega-\hat{\Omega}\rangle=\frac{4}{L} \sum_{\ell=1}^L\operatorname{tr}\Big( \frac{n_{\ell}^{2}}{p^2}\frac{1}{n_{\ell}^2}X^{(\ell)^{\top}} X^{(\ell)}(\Omega-\hat{\Omega}) X^{(\ell)^{\top}} X^{(\ell)}(\Omega-\hat{\Omega})\Big).
\end{equation*}
Now we seek to  lower bound for this inner product in terms of $\|\hat{\Omega}-\Omega\|_{F}^{2}$ for any finite $L$. Define $\hat{\Sigma}^{(\ell)}=\frac{1}{n_{\ell}}X^{(\ell)^{\top}} X^{(\ell)}$, and note that
\begin{align*}
    &\langle \operatorname{grad}f(\Omega)-\operatorname{grad}f(\hat{\Omega}),\Omega-\hat{\Omega}\rangle\\
    =&\frac{4}{L} \sum_{\ell=1}^L \frac{n_{\ell}^{2}}{p^2}\operatorname{tr}\Big( \frac{1}{n_{\ell}^2}X^{(\ell)^{\top}} X^{(\ell)}(\Omega-\hat{\Omega}) X^{(\ell)^{\top}} X^{(\ell)}(\Omega-\hat{\Omega})\Big)\\
    \geq &4\min_{\ell=1,\dots,L}\frac{n_{\ell}^{2}}{p^2}\frac{1}{L}\sum_{\ell=1}^{L}\operatorname{tr}\big( \hat{\Sigma}^{(\ell)}(\Omega-\hat{\Omega}) \hat{\Sigma}^{(\ell)}(\Omega-\hat{\Omega})\big)\\
    \geq& \frac{4}{L} \min_{1\leq \ell\leq L} \frac{n_{\ell}^2}{p^2} \sum_{\ell=1}^{L_{0}} \lambda_{\min }^2(\hat{\Sigma}^{(\ell)})\|\Omega-\hat{\Omega}\|_F^2\\
    \geq& \frac{4L_{0}}{L} \min_{1\leq \ell\leq L} \frac{n_{\ell}^2}{p^2} \min_{1\leq \ell\leq L_{0}}\lambda_{\min }^2(\hat{\Sigma}^{(\ell)})\|\Omega-\hat{\Omega}\|_F^2.
\end{align*}
where $\lambda_{\min}(\hat{\Sigma}^{(\ell)})$ is the minimum eigenvalue of $\hat{\Sigma}^{(\ell)}=\frac{1}{n_{\ell}}X^{(\ell)\top}X^{(\ell)}$. Now note that for $\ell=1,\dots,L_{0}$,
\begin{align*}
    \lambda_{\min}(\hat{\Sigma}^{(\ell)})&=\lambda_{\min }\Big(\frac{1}{n_{\ell}} \Sigma^{(\ell)\frac{1}{2}} Z^{(\ell)^{\top}} Z^{(\ell)} \Sigma^{(\ell)\frac{1}{2}}\Big)\\
    &\geq \lambda_{\min }\big(\Sigma^{(\ell)\frac{1}{2}}\big)^2 \lambda_{\min }\big(Z^{(\ell) \top} Z^{(\ell)}/n_{\ell}\big).
\end{align*}
According to \citet[ Theorem 5.11]{bai2010spectral}, with probability $1$, 
$$
\lambda_{\min }\big(Z^{(\ell) \top} Z^{(\ell)}/n_{\ell}\big)\rightarrow (1-\sqrt{\gamma_{\ell}})^{2}.
$$ 
Thus, with probability $1$, for $p$ and $n_{\ell}$ being sufficiently large, 
\begin{align*}
\lambda_{\min}\big(\hat{\Sigma}^{(\ell)}\big)\geq \frac{1}{2}\big(1-\sqrt{\gamma_{\ell}}\big)^{2}\lambda_{\min}(\Sigma^{(\ell)})\geq d>0\quad\quad \ell=1,\dots,L_{0},  
\end{align*} 
where $d=\inf_{\ell\leq L_0}\frac{1}{2}(1-\sqrt{\gamma_{\ell}})^{2}\lambda_{\min}(\Sigma^{(\ell)})$. This implies $\min_{1\leq \ell\leq L_{0}}\lambda_{\min}(\hat{\Sigma}^{(\ell)})\geq d>0$. Therefore,
\begin{align*}
    \langle \operatorname{grad}f(\Omega)-\operatorname{grad}f(\hat{\Omega}),\Omega-\hat{\Omega}\rangle&\geq \frac{4L_{0}}{L} \min_{1\leq \ell\leq L} \frac{n_{\ell}^2}{p^2} \min_{1\leq \ell\leq L_{0}}\lambda_{\min }^2(\hat{\Sigma}^{(\ell)})\|\Omega-\hat{\Omega}\|_F^2\\
    &\geq \frac{4L_0}{L} \min_{1\leq \ell\leq L} \frac{n_{\ell}^2}{p^2} d^{2}\|\Omega-\hat{\Omega}\|_{F}^{2}.
\end{align*}
Hence, this implies
\begin{align}
\frac{4L_0}{L}\min_{1\leq \ell\leq L} \frac{n_{\ell}^2}{p^2} d^{2}\|\Omega-\hat{\Omega}\|_{F}^{2}&\leq \langle \operatorname{grad}f(\Omega)-\operatorname{grad}f(\hat{\Omega}),\Omega-\hat{\Omega}\rangle\nonumber\\
&\leq \|\operatorname{grad}f(\Omega)\|_{F}\|\hat{\Omega}-\Omega\|_{F}. \label{key_inequality_for_concentration}   
\end{align} 
Hence, for sufficiently large $p,n_{\ell}$ and $L$, \eqref{Key_equation1_concentration} holds. Now, if $\frac{L_0}{L}\rightarrow c>0$, it suffices to control $\|\operatorname{grad}f(\Omega)\|_{F}$. We prove $\|\operatorname{grad}f(\Omega)\|_{F}\stackrel{p}{\rightarrow}0$ by deriving the explicit concentration inequality in terms of $p,L,n_{\ell}$. In order to guarantee $\|\operatorname{grad} f(\Omega)\|_F \stackrel{p}{\rightarrow} 0$ as $p, L,n_{\ell} \rightarrow \infty$, we only need to control following three terms
\begin{align}
    &\frac{1}{pL} \sum_{\ell=1}^L X^{(\ell)^{\top}} X^{(\ell)}\Big(\bar{\beta}^{(\ell)} \bar{\beta}^{(\ell)\top}-\frac{1}{p} \Omega\Big) X^{(\ell)^{\top}} X^{(\ell)}\label{non_sparse_term_1}\\
    &\frac{1}{pL} \sum_{\ell=1}^L X^{(\ell)^{\top}}\big(\varepsilon^{(\ell)} \varepsilon^{(\ell)\top}-\sigma^2 I\big) X^{(\ell)}\label{non_sparse_term_2}\\
    &\frac{1}{pL}\sum_{\ell=1}^{L}X^{(\ell)\top}\varepsilon^{(\ell)}\bar{\beta}^{(\ell)\top}X^{(\ell)\top}X^{(\ell)}.\label{non_sparse_term_3}
\end{align}
We want to guarantee that the Frobenius norm of these three terms converge in probability to $0$ as $p,L,n_{\ell}\rightarrow \infty$. Here, we use the result by \citet[Theorem 2.7]{koltchinskii2011oracle}, recalled in Theorem \ref{concentration_thm_01}.\\

\noindent\textbf{Concentration inequality for the term \eqref{non_sparse_term_1} in the gradient.}

\noindent In order to prove the Frobenius norm of term \eqref{non_sparse_term_1} converges to zero in probability as $p,n_{\ell}$ and $L$ go to infinity, we use the fact that for $A\in\mathbb{R}^{p\times p}$, $\|A\|_{F}\leq \sqrt{p}\|A\|$. Hence, we only need to utilize Theorem \ref{concentration_thm_01} to guarantee 
\begin{align*}
\Big\|\frac{1}{\sqrt{p} L} \sum_{\ell=1}^L X^{(\ell)\top} X^{(\ell)}\big(\bar{\beta}^{(\ell)} \bar{\beta}^{(\ell)\top}-\frac{1}{p} \Omega\big) X^{(\ell)\top} X^{(\ell)}\Big\|\stackrel{p}{\rightarrow}0.    
\end{align*} 
To apply Theorem \ref{concentration_thm_01}, it suffices to bound all the required constants in the result and specifying $U^{(1)}$ in the Theorem. By Lemma \ref{concentration_lm_01} and Lemma \ref{concentration_lm_02}, it holds that 
\begin{align*}
    \mathbb{E}\Big\|\bar{\beta}^{(\ell)}\bar{\beta}^{(\ell)\top}-\frac{1}{p}\Omega\Big\|^{2}&\leq 1+\frac{32}{C_{\beta}^{2}p^2}=\mathcal{O}(1),\\
    \Big\|\Big\|\frac{1}{\sqrt{p}}X^{(\ell)\top} X^{(\ell)}\big(\bar{\beta}^{(\ell)} \bar{\beta}^{(\ell)\top}-\frac{1}{p} \Omega\big) X^{(\ell)\top} X^{(\ell)}\Big\|\Big\|_{\psi_{1}}&\leq 8n_{\ell}^{2}(1+\sqrt{\gamma_{\ell}})^{4}\lambda_{\max}^{2}(\Sigma^{(\ell)})/\sqrt{p}.
\end{align*}
Also, by \citet[ Theorem 5.11]{bai2010spectral}, with probability $1$, it holds that $\lambda_{\max}(\hat{\Sigma}^{(\ell)})\rightarrow (1+\sqrt{\gamma_{\ell}})^2\lambda_{\max}(\Sigma^{(\ell)})$. Then for sufficiently large $p$ and $n_{\ell}$, $\|\hat{\Sigma}^{(\ell)}\|\leq 2(1+\sqrt{\gamma_{\ell}})^2\lambda_{\max}(\Sigma^{(\ell)})$ with probability 1. Therefore,
\begin{align*}
&\mathbb{E}\Big[\big\|\frac{1}{\sqrt{p}}X^{(\ell)\top} X^{(\ell)}\big(\bar{\beta}^{(\ell)} \bar{\beta}^{(\ell)\top}-\frac{1}{p} \Omega\big) X^{(\ell)\top} X^{(\ell)}\big\|^{2}\Big]^{\frac{1}{2}}\\
\leq &\frac{4n_{\ell}^{2}}{\sqrt{p}}(1+\sqrt{\gamma_{\ell}})^{4}\lambda_{\max}^{2}(\Sigma^{(\ell)}) \sqrt{\mathbb{E}\big\|\bar{\beta}^{(\ell)} \bar{\beta}^{(\ell)\top}-\frac{1}{p} \Omega\big\|^2}\\
\leq &\frac{4n_{\ell}^{2}}{\sqrt{p}}(1+\sqrt{\gamma_{\ell}})^{4}\lambda_{\max}^{2}(\Sigma^{(\ell)}). 
\end{align*}
Finally we bound the term $\big\|\mathbb{E}\big[\frac{1}{\sqrt{p}} X^{(\ell)\top} X^{(\ell)}\big(\bar{\beta}^{(\ell)} \bar{\beta}^{(\ell)\top}-\frac{1}{p} \Omega\big) X^{(\ell)\top} X^{(\ell)}\big]^{2}\big\|$. Note that 
\begin{align*}
&\frac{1}{L}\Big\|\mathbb{E}\sum_{\ell=1}^{L}\Big[\frac{1}{\sqrt{p}}X^{(\ell)\top} X^{(\ell)}\big(\bar{\beta}^{(\ell)} \bar{\beta}^{(\ell)\top}-\frac{1}{p} \Omega\big) X^{(\ell)\top} X^{(\ell)}\Big]^2\Big\|\\
\leq &\frac{1}{L}\sum_{\ell=1}^{L}\mathbb{E}\Big\|\frac{1}{p}\Big[X^{(\ell)\top} X^{(\ell)}\big(\bar{\beta}^{(\ell)} \bar{\beta}^{(\ell)\top}-\frac{1}{p} \Omega\big) X^{(\ell)\top} X^{(\ell)}\Big]^{2}\Big\|\\
\leq &\frac{1}{L}\sum_{\ell=1}^{L}\frac{1}{p}\mathbb{E}\Big\| X^{(\ell)\top} X^{(\ell)}\big(\bar{\beta}^{(\ell)} \bar{\beta}^{(\ell)\top}-\frac{1}{p} \Omega\big) X^{(\ell)\top} X^{(\ell)} \Big\|^{2}\\
\leq &\frac{1}{L}\sum_{\ell=1}^{L}\frac{1}{p}\mathbb{E}\big\|X^{(\ell)\top} X^{(\ell)}\big\|^{4}\big\| \big(\bar{\beta}^{(\ell)} \bar{\beta}^{(\ell)\top}-\frac{1}{p} \Omega\big)\big\|^{2}\\
\leq &\frac{16}{pL} \sum_{\ell=1}^{L} n_{\ell}^4 \lambda_{\max }^4(\Sigma^{(\ell)})(1+\sqrt{\gamma_{\ell}})^{8} \mathbb{E}\big\|\bar{\beta}^{(\ell)}\bar{\beta}^{(\ell)\top}-\frac{1}{p} \Omega\big\|^2\\
\leq & \frac{16}{pL} \sum_{\ell=1}^{L} n_{\ell}^4 \lambda_{\max }^4(\Sigma^{(\ell)})(1+\sqrt{\gamma_{\ell}})^{8}\mathcal{O}(1)= \mathcal{O}(p^3).
\end{align*}  
Therefore, in Theorem \ref{concentration_thm_01}, one can set
\begin{align*}
U^{(1)}&=\max\Big\{\Big\|\Big\|\frac{1}{\sqrt{p}}X^{(\ell)\top} X^{(\ell)}\big(\bar{\beta}^{(\ell)} \bar{\beta}^{(\ell)\top}-\frac{1}{p} \Omega\big) X^{(\ell)\top} X^{(\ell)}\Big\|\Big\|_{\psi_{1}},\\
&\qquad\qquad\qquad\qquad\mathbb{E}\Big[\Big\|\frac{1}{\sqrt{p}}X^{(\ell)\top} X^{(\ell)}\big(\bar{\beta}^{(\ell)} \bar{\beta}^{(\ell)\top}-\frac{1}{p} \Omega\big) X^{(\ell)\top} X^{(\ell)}\Big\|^{2}\Big]^{\frac{1}{2}}\Big\}\\ &=\mathcal{O}\big(p^{\frac{3}{2}}\big).
\end{align*} 
Finally, we are ready to derive the concentration inequality using Theorem \ref{concentration_thm_01}. We have that
\begin{align}
&\mathbb{P}\Big(\Big\|\frac{1}{\sqrt{p} L} \sum_{\ell=1}^L X^{(\ell)\top} X^{(\ell)}\big(\bar{\beta}^{(\ell)} \bar{\beta}^{(\ell)\top}-\frac{1}{p} \Omega\big) X^{(\ell)\top} X^{(\ell)}\Big\|\geq t\Big)\nonumber\\
\leq & 2p\exp\Bigg\{-\frac{1}{K}\frac{L^2t^2}{L\mathcal{O}(p^3)+Lt\mathcal{O}\big(p^{\frac{3}{2}}\big)\log \big(\mathcal{O}\big(p^{\frac{3}{2}}\big)\big)}\Bigg\}\nonumber\\
=&2p\exp\Bigg\{-\frac{1}{K}\frac{t^2}{\mathcal{O}(p^{3}/L)+t\mathcal{O}\big( p^{\frac{3}{2}} \log\big(p^{\frac{3}{2}}\big)/L\big)}\Bigg\}.\label{non_sparse_first_concentration} 
\end{align}
Using \citet[Fact 1]{wang2019some}, the Frobenius norm of \eqref{non_sparse_term_1} is $\mathcal{O}_{P}(\frac{p^{\frac{3}{2}}}{L})$. Hence, as long as $\frac{p^{\frac{3}{2}}}{L}=o(1)$, the Frobenius norm of \eqref{non_sparse_term_1} will converge to zero in probability as $p,L\rightarrow\infty$.\\

\noindent\textbf{Concentration inequality for the term \eqref{non_sparse_term_2} in the gradient.} We follow a similar approach to derive the concentration inequality. Note that
$$\Big\|\frac{1}{pL} \sum_{\ell=1}^L X^{(\ell)^{\top}}(\varepsilon^{(\ell)} \varepsilon^{(\ell)^{\top}}-\sigma^2 I) X^{(\ell)}\Big\|_{F}\leq \sqrt{p}\Big\|\frac{1}{pL} \sum_{\ell=1}^L X^{(\ell)\top}(\varepsilon^{(\ell)} \varepsilon^{(\ell)\top}-\sigma^2 I) X^{(\ell)}\Big\|,$$
and
\begin{align*}
    &\Big\|\frac{1}{\sqrt{p}} X^{(\ell)\top}(\varepsilon^{(\ell)} \varepsilon^{(\ell) \top}-\sigma^2 I) X^{(\ell)}\Big\|\\
    \leq &\Big\|\frac{1}{\sqrt{p}} X^{(\ell)\top}(\varepsilon^{(\ell)} \varepsilon^{(\ell) \top}-\sigma^2 I) X^{(\ell)}\Big\|_{F}\\
    =&\sqrt{\frac{1}{p} \operatorname{tr}\big(X^{(\ell) \top}(\varepsilon^{(\ell)} \varepsilon^{(\ell)\top}-\sigma^2 I) X^{(\ell)} X^{(\ell)\top}(\varepsilon^{(\ell)} \varepsilon^{(\ell)\top}-\sigma^2 I) X^{(\ell)}\big)}\\
    \leq &\frac{1}{\sqrt{p}}\lambda_{\max}(X^{(\ell)}X^{(\ell)\top})\|\varepsilon^{(\ell)} \varepsilon^{(\ell)\top}-\sigma^2 I\|_{F}\\
    \leq &\sqrt{\frac{n_{\ell}}{p}} \lambda_{\max}(X^{(\ell)} X^{(\ell)^{\top}})\|\varepsilon^{(\ell)} \varepsilon^{(\ell)\top}-\sigma^2 I\|\\
    \leq &2 \sqrt{\frac{n_{\ell}}{p}} n_{\ell} (1+\sqrt{\gamma_{\ell}})^{2}\lambda_{\max}(\Sigma^{(\ell)})\|\varepsilon^{(\ell)} \varepsilon^{(\ell)\top}-\sigma^2 I\|.
\end{align*}
Now we want to use Theorem \ref{concentration_thm_01} to derive the concentration inequality for 
$$\Big\|\frac{1}{L} \sum_{\ell=1}^L \frac{1}{\sqrt{p}} X^{(\ell)^{\top}}\big(\varepsilon^{(\ell)} \varepsilon^{(\ell)\top}-\sigma^2 I\big) X^{(\ell)}\Big\|.$$ 
Again, by Lemma \ref{concentration_lm_01} and Lemma \ref{concentration_lm_02}, it holds that
\begin{align*} 
    \Big\|\Big\|\frac{1}{\sqrt{p}} X^{(\ell)\top}(\varepsilon^{(\ell)} \varepsilon^{(\ell) \top}-\sigma^2 I) X^{(\ell)}\Big\|\Big\|_{\psi_{1}}&\leq 4 \frac{n_{\ell}^{\frac{5}{2}}}{\sqrt{p}} (1+\sqrt{\gamma_{\ell}})^{2}\lambda_{\max }(\Sigma^{(\ell)})=\mathcal{O}\big(n_{\ell}^{\frac{5}{2}}/\sqrt{p}\big)\\
    \mathbb{E}\big\|\varepsilon^{(\ell)}\varepsilon^{(\ell)\top}-\sigma^{2}I\big\|^{2}&\leq \mathcal{O} (n_{\ell}^{2}).
\end{align*}
Therefore,
\begin{align*}
    &\Big(\mathbb{E}\Big\|\frac{1}{\sqrt{p}} X^{(\ell)\top}(\varepsilon^{(\ell)} \varepsilon^{(\ell) \top}-\sigma^2 I) X^{(\ell)}\Big\|^2\Big)^{\frac{1}{2}}\\ 
    \leq & 2 \sqrt{\frac{n_{\ell}}{p}} n_{\ell} (1+\sqrt{\gamma_{\ell}})^{2}\lambda_{\max }(\Sigma^{(\ell)})\big[\mathbb{E}\|\varepsilon^{(\ell)} \varepsilon^{(\ell)\top}-\sigma^2 I\|^{2}\big]^{\frac{1}{2}}\\
    \leq &\mathcal{O}(n_{\ell}^{\frac{5}{2}}/\sqrt{p}) =\mathcal{O}(p^{2}).
\end{align*} 
To apply Theorem \ref{concentration_thm_01}, we could set the required quantity $U^{(1)}$ to be
\begin{align*}
U^{(1)}&=\max\Big\{\Big\|\Big\|\frac{1}{\sqrt{p}} X^{(\ell)\top}(\varepsilon^{(\ell)} \varepsilon^{(\ell) \top}-\sigma^2 I) X^{(\ell)}\Big\|\Big\|_{\psi_{1}},\\
&\quad\quad\quad\quad\quad\quad\Big(\mathbb{E}\Big\|\frac{1}{\sqrt{p}} X^{(\ell)\top}(\varepsilon^{(\ell)} \varepsilon^{(\ell) \top}-\sigma^2 I) X^{(\ell)}\Big\|^2\Big)^{\frac{1}{2}}\Big\} \\
&\leq \mathcal{O}(p^{2}).
\end{align*} 
Now, we bound the other quantity required in the Theorem \ref{concentration_thm_01} as
\begin{align*}
&\frac{1}{L}\Big\|\mathbb{E}\sum_{\ell=1}^{L}\Big[\frac{1}{\sqrt{p}}X^{(\ell)\top} (\varepsilon^{(\ell)} \varepsilon^{(\ell)\top}-\sigma^{2}I) X^{(\ell)}\Big]^2\Big\|\\
\leq &\frac{1}{L}\sum_{\ell=1}^{L}\mathbb{E}\Big\|\frac{1}{p}\big[X^{(\ell)\top} (\varepsilon^{(\ell)} \varepsilon^{(\ell)\top}-\sigma^{2}I)  X^{(\ell)}\big]^{2}\Big\|\\
\leq & \frac{1}{L}\sum_{\ell=1}^{L}\frac{1}{p}\mathbb{E} \Big\|X^{(\ell)\top}(\varepsilon^{(\ell)} \varepsilon^{(\ell)\top}-\sigma^2 I) X^{(\ell)}\Big\|^2\\
\leq & \frac{1}{L}\sum_{\ell=1}^{L} 4\frac{n_{\ell}^{3}}{p} (1+\sqrt{\gamma_{\ell}})^{4}\lambda_{\max }^{2}(\Sigma^{(\ell)})\mathbb{E}\big\|\varepsilon^{(\ell)} \varepsilon^{(\ell)\top}-\sigma^2 I\big\|^{2}\\ 
\leq & \mathcal{O}(n_{\ell}^{5}/p)=\mathcal{O}(p^{4}).
\end{align*}  
Using Theorem \ref{concentration_thm_01}, we have that
\begin{align}
&\mathbb{P}\Big(\Big\|\frac{1}{\sqrt{p}L} \sum_{\ell=1}^L X^{(\ell)^{\top}}(\varepsilon^{(\ell)} \varepsilon^{(\ell)\top}-\sigma^2 I) X^{(\ell)}\Big\|\geq t\Big)\nonumber\\
\leq & 2p\exp\Bigg\{-\frac{1}{K}\frac{L^2t^2}{L\mathcal{O}(n_{\ell}^{5}/p)+Lt \mathcal{O}(n_{\ell}^{\frac{5}{2}}/\sqrt{p}) \log(\mathcal{O}(n_{\ell}^{\frac{5}{2}}/\sqrt{p}))}\Bigg\} \nonumber \\
=&2p\exp\Bigg\{-\frac{1}{K}\frac{t^2}{\mathcal{O}(p^{4}/L)+t \mathcal{O}(p^{2} \log(p^2)/L)}\Bigg\}.\label{non_sparse_second_concentration} 
\end{align}
Again, according to in \citet[Fact 1]{wang2019some}, the Frobenius norm of \eqref{non_sparse_term_2} is $\mathcal{O}_{P}(\frac{p^{2}}{L})$. As long as $\mathcal{O}(p^{2}/L)=o(1)$, the Frobenius norm of \eqref{non_sparse_term_2} will converge to zero in probability as $p,L,n_{\ell}\rightarrow\infty$. \\

\noindent\textbf{Concentration inequality for the term \eqref{non_sparse_term_3} in the gradient.} Finally, we deal with the cross term \eqref{non_sparse_term_3}: Note that
\begin{align*}
    \Big\|\frac{1}{p L} \sum_{l=1}^L X^{(\ell)\top} \varepsilon^{(\ell)} \bar{\beta}^{(\ell)\top} X^{(\ell)\top} X^{(\ell)}\Big\|_F& \leq \Big\|\frac{1}{L} \sum_{l=1}^L \frac{1}{\sqrt{p}} X^{(\ell)\top} \varepsilon^{(\ell)} \bar{\beta}^{(\ell)\top} X^{(\ell)\top} X^{(\ell)}\Big\|.
\end{align*}
Hence, it suffices to show that $\Big\|\frac{1}{\sqrt{p}} X^{(\ell)\top} \varepsilon^{(\ell)} \bar{\beta}^{(\ell)\top} X^{(\ell)\top} X^{(\ell)}\Big\|$ converges to 0 in probability using Theorem \ref{concentration_thm_01}. Note that
\begin{align*}
&\Big\|\frac{1}{\sqrt{p}} X^{(\ell)\top} \varepsilon^{(\ell)} \bar{\beta}^{(\ell)\top} X^{(\ell)\top} X^{(\ell)}\Big\|\\
\leq &\Big\|\frac{1}{\sqrt{p}} X^{(\ell)\top} \varepsilon^{(\ell)} \bar{\beta}^{(\ell)\top} X^{(\ell)\top} X^{(\ell)}\Big\|_{F}\\
=&\frac{1}{\sqrt{p}}\big[\operatorname{tr}(X^{(\ell)\top} X^{(\ell)} \bar{\beta}^{(\ell)} \varepsilon^{(\ell) \top} X^{(\ell)} X^{(\ell)\top} \varepsilon^{(\ell)}\bar{\beta}^{(\ell)\top} X^{(\ell)\top} X^{(\ell)})\Big]^{\frac{1}{2}}\\
\leq&\frac{1}{\sqrt{p}} 2 n_{\ell}(1+\sqrt{\gamma_{\ell}})^{2} \lambda_{\max} (\Sigma^{(\ell)})\big[\operatorname{tr}(\bar{\beta}^{(\ell)} \varepsilon^{(\ell)\top} X^{(\ell)} X^{(\ell)\top} \varepsilon^{(\ell)} \bar{\beta}^{(\ell)\top})\big]^{\frac{1}{2}}\\
\leq&\frac{1}{\sqrt{p}} 4 n_{\ell}^{\frac{3}{2}}(1+\sqrt{\gamma_{\ell}})^{3} \lambda_{\max}^{\frac{3}{2}} (\Sigma^{(\ell)})\|\bar{\beta}^{(\ell)} \varepsilon^{(\ell)\top}\|_{F}\\
=&\frac{1}{\sqrt{p}} 4 n_{\ell}^{\frac{3}{2}}(1+\sqrt{\gamma_{\ell}})^{3} \lambda_{\max}^{\frac{3}{2}} (\Sigma^{(\ell)})\|\bar{\beta}^{(\ell)}\|_{2}\| \varepsilon^{(\ell)\top}\|_{2}.
\end{align*}
Therefore,
\begin{align*}
    &\Big\|\Big\|\frac{1}{\sqrt{p}} X^{(\ell)\top} \varepsilon^{(\ell)} \bar{\beta}^{(\ell)\top} X^{(\ell)\top} X^{(\ell)}\Big\|\Big\|_{\psi_{1}}\\
    \leq&\frac{1}{\sqrt{p}} 4 n_{\ell}^{\frac{3}{2}}(1+\sqrt{\gamma_{\ell}})^{3} \lambda_{\max}^{\frac{3}{2}} (\Sigma^{(\ell)})\big\|\big\|\bar{\beta}^{(\ell)}\big\|_{2}\big\| \varepsilon^{(\ell)\top}\big\|_{2}\big\|_{\psi_{1}}\\
    \leq&\frac{1}{\sqrt{p}} 4 n_{\ell}^{\frac{3}{2}}(1+\sqrt{\gamma_{\ell}})^{3} \lambda_{\max}^{\frac{3}{2}} (\Sigma^{(\ell)}) \sup _{k \geq 1} \frac{\big[\mathbb{E}\|\varepsilon^{(\ell)}\|_2^k\|\bar{\beta}^{(\ell)}\|_2^k\big]^{\frac{1}{k}}}{k}.
\end{align*}
By Lemma \ref{concentration_lm_01}, it holds that
\begin{align*}
\mathbb{E}\|\bar{\beta}^{(\ell)}\|_{2}^{k}&\leq 1+ k(4 \tau_{\beta}^2/p)^{\frac{k}{2}} \Gamma(k/2),\\
\mathbb{E}\|\varepsilon^{(\ell)}\|_{2}^{k}&\leq n_{\ell}^{\frac{k}{2}}+k(4 \tau_{\varepsilon}^2)^{\frac{k}{2}} \Gamma(k/2).
\end{align*}
Therefore, the $\psi_{1}$ norm could be bounded as
\begin{align*}
    &\Big\|\Big\|\frac{1}{\sqrt{p}} X^{(\ell)\top} \varepsilon^{(\ell)} \bar{\beta}^{(\ell)\top} X^{(\ell)\top} X^{(\ell)}\Big\|\Big\|_{\psi_{1}}\\
    \leq &\frac{1}{\sqrt{p}} 4 n_{\ell}^{\frac{3}{2}}(1+\sqrt{\gamma_{\ell}})^{3} \lambda_{\max}^{\frac{3}{2}} (\Sigma^{(\ell)}) \sup _{k \geq 1} \frac{\big[\mathbb{E}\|\varepsilon^{(\ell)}\|_2^k\|\bar{\beta}^{(\ell)}\|_2^k\big]^{\frac{1}{k}}}{k} \\
    \leq &\frac{1}{\sqrt{p}} 4 n_{\ell}^{\frac{3}{2}}(1+\sqrt{\gamma_{\ell}})^{3} \lambda_{\max}^{\frac{3}{2}} (\Sigma^{(\ell)}) \sup _{k \geq 1} \frac{\big(1+ k(4 \tau_{\beta}^2/p)^{\frac{k}{2}} \Gamma(k/2)\big)^{\frac{1}{k}}\big(n_{\ell}^{\frac{k}{2}}+k(4 \tau_{\varepsilon}^2)^{\frac{k}{2}} \Gamma(k/2)\big)^{\frac{1}{k}}}{k}\\
    \leq &\frac{1}{\sqrt{p}} 8 n_{\ell}^{\frac{3}{2}}(1+\sqrt{\gamma_{\ell}})^{3} \lambda_{\max}^{\frac{3}{2}} (\Sigma^{(\ell)})\mathcal{O}(\sqrt{n_{\ell}}).
\end{align*}
Besides,
$$\mathbb{E}\|\bar{\beta}^{(\ell)}\|_2^2\leq \mathcal{O}(1), \quad\text{and}\quad \mathbb{E}\|\varepsilon^{(\ell)}\|_{2}^2 \leq \mathcal{O}(n_{\ell}).$$
Therefore,
\begin{align*}
&\Big(\mathbb{E}\Big\|\frac{1}{\sqrt{p}} X^{(\ell)\top} \varepsilon^{(\ell)} \bar{\beta}^{(\ell)\top} X^{(\ell)\top} X^{(\ell)}\Big\|^2\Big)^{\frac{1}{2}}\\
\leq&  8 \frac{1}{\sqrt{p}}(1+\sqrt{\gamma_{\ell}})^{3}\big(n_{\ell} \lambda_{\max }(\Sigma^{(\ell)})\big)^{\frac{3}{2}}\big(\mathbb{E}\|\bar{\beta}^{(\ell)}\|_2^2\|\varepsilon^{(\ell)}\|_2^2\big)^{\frac{1}{2}}\\
\leq& 64 \frac{1}{\sqrt{p}}(1+\sqrt{\gamma_{\ell}})^{3}\big(n_{\ell} \lambda_{\max }(\Sigma^{(\ell)})\big)^{\frac{3}{2}} \mathcal{O}(\sqrt{ n_{\ell} }).
\end{align*}
Therefore, one could set the quantity $U^{(1)}$ to be 
\begin{align*}
U^{(1)}&=\max_{\ell}\Big\{\Big\|\Big\|\frac{1}{\sqrt{p}} X^{(\ell)\top} \varepsilon^{(\ell)} \bar{\beta}^{(\ell)\top} X^{(\ell)\top} X^{(\ell)}\Big\|\Big\|_{\psi_{1}},\Big(\mathbb{E}\Big\|\frac{1}{\sqrt{p}} X^{(\ell)\top} \varepsilon^{(\ell)} \bar{\beta}^{(\ell)\top} X^{(\ell)\top} X^{(\ell)}\Big\|^2\Big)^{\frac{1}{2}}\Big\} \\
&\leq \mathcal{O}\big(p^{\frac{3}{2}}\big). 
\end{align*}
Besides, notice that for $n_{\ell},p$ sufficiently large
\begin{align*}
    &\Big\|\mathbb{E} \frac{1}{p} X^{(\ell)\top} X^{(\ell)} \bar{\beta}^{(\ell)} \varepsilon^{(\ell)\top} X^{(\ell)} X^{(\ell)\top} \varepsilon^{(\ell)} \bar{\beta}^{(\ell)\top} X^{(\ell)\top} X^{(\ell)}\Big\|\\
    \leq & \mathbb{E}\Big\|\frac{1}{p} X^{(\ell)\top} X^{(\ell)} \bar{\beta}^{(\ell)} \varepsilon^{(\ell)\top} X^{(\ell)} X^{(\ell)\top} \varepsilon^{(\ell)} \bar{\beta}^{(\ell)\top} X^{(\ell)\top} X^{(\ell)}\Big\|\\
    \leq & \frac{1}{p}\big\|X^{(\ell) \top} X^{(\ell)}\big\|^2\big\|X^{(\ell)} X^{(\ell) \top}\big\|\big\|\bar{\beta}^{(\ell)} \varepsilon^{(\ell) \top}\big\|\big\|\varepsilon^{(\ell)} \bar{\beta}^{(\ell) \top}\big\|\\
    \leq &\frac{8}{p} (1+\sqrt{\gamma_{\ell}})^{6} n_{\ell}^3 \lambda_{\max}^3(\Sigma^{(\ell)}) \mathbb{E}\|\bar{\beta}^{(\ell)}\|_2^2\|\varepsilon^{(\ell)}\|_2^2\\
    \leq &\frac{8}{p} (1+\sqrt{\overline{c}})^{6} n_{\ell}^3 \lambda_{\max}^3(\Sigma^{(\ell)}) \mathcal{O}(p)=\mathcal{O}(p^{3}),
\end{align*}
and
\begin{align*}
&\Big\|\mathbb{E} \frac{1}{p} X^{(\ell)\top} \varepsilon^{(\ell)} \bar{\beta}^{(\ell)\top} X^{(\ell)\top} X^{(\ell)} X^{(\ell)\top} X^{(\ell)} \bar{\beta}^{(\ell)} \varepsilon^{(\ell)\top} X^{(\ell)}\Big\|\\
\leq &\frac{1}{p} \mathbb{E}\Big\|X^{(\ell)\top} \varepsilon^{(\ell)} \bar{\beta}^{(\ell)\top} X^{(\ell)\top} X^{(\ell)} X^{(\ell)\top} X^{(\ell)} \bar{\beta}^{(\ell)} \varepsilon^{(\ell)\top} X^{(\ell)}\Big\|\\
\leq & \frac{1}{p} \mathbb{E}\Big[\lambda_{\max }\big(X^{(\ell)} X^{(\ell)\top}\big) \lambda_{\max }^2\big(X^{(\ell)\top} X^{(\ell)}\big)\big\|\bar{\beta}^{(\ell)} \varepsilon^{(\ell)\top} \varepsilon^{(\ell)} \bar{\beta}^{(\ell)\top}\big\|_F\Big]\\
\leq &\frac{1}{p} 8 n_{\ell}^3 \lambda_{\max}^{3}(\Sigma^{(\ell)})(1+\sqrt{\gamma_{\ell}})^{6} \mathbb{E}\|\bar{\beta}^{(\ell)}\|_2^2 \mathbb{E}\|\varepsilon^{(\ell)}\|_2^2\\
\leq &\frac{8}{p} (1+\sqrt{\gamma_{\ell}})^{6} n_{\ell}^3 \lambda_{\max}^3(\Sigma^{(\ell)})\mathcal{O}(p)\leq \frac{8}{p} (1+\sqrt{\overline{c}})^{6} p^3 \lambda_{\max}^3(\Sigma^{(\ell)}) \mathcal{O}(p)\\
=&\mathcal{O}(p^{3}).
\end{align*}
Now applying Theorem \ref{concentration_thm_01}, we have
\begin{align}
&\mathbb{P}\Big(\Big\|\frac{1}{L} \sum_{l=1}^L \frac{1}{\sqrt{p}} X^{(\ell)\top} \varepsilon^{(\ell)} \bar{\beta}^{(\ell)\top} X^{(\ell)\top} X^{(\ell)}\Big\|\geq t\Big)\nonumber\\
\leq& 2p\exp\Big\{-\frac{1}{K}\frac{t^2L^2}{L\mathcal{O}(p^{3})+tL\mathcal{O}(p^{\frac{3}{2}})\log\big(\mathcal{O}(p^{\frac{3}{2}})\big)}\Big\}\nonumber\\
= & 2p\exp\Big\{-\frac{1}{K}\frac{t^2}{\mathcal{O}\big(p^3/L \big)+t\mathcal{O}\big(p^{\frac{3}{2}}\log\big(p^{\frac{3}{2}}/L \big)\big)}\Big\}.\label{non_sparse_third_concentration} 
\end{align}
Once again, according to \citet[Fact 1]{wang2019some}, Frobenius norm of \eqref{non_sparse_term_3} is $\mathcal{O}_{P}(\frac{p^{\frac{3}{2}}}{L})$. As long as $p^{\frac{3}{2}}/L =o(1)$, i.e. $p^{\frac{3}{2}}/L\rightarrow 0$, the Frobenius norm of the cross term \eqref{non_sparse_term_3} will converge to zero in probability as $p,n_{\ell},L\rightarrow\infty$.\\

We now prove the claims in part (i) and part(ii).

\textbf{Proof of Part (i).}
By inequality \eqref{key_inequality_for_concentration}, it holds that
\begin{align*}
\min_{1\leq \ell\leq L} \frac{n_{\ell}^2}{p^2} d^{2}\|\Omega-\hat{\Omega}\|_{F}\leq \frac{L}{4L_0}\|\operatorname{grad}f(\Omega)\|_{F}   
\end{align*} 
    If $\frac{L_0}{L}\rightarrow c=0$, then we need to control $\frac{L}{L_0}\|\operatorname{grad}f(\Omega)\|_{F}$. We want $\frac{L}{L_0}\|\operatorname{grad}f(\Omega)\|_{F}\stackrel{p}{\rightarrow}0$ as $p,n_{\ell},L,L_0\rightarrow\infty$. In order to guarantee this, we only need to bound the operator norm of following terms
\begin{align}
    &\frac{1}{\sqrt{p}L}\frac{L}{L_0} \sum_{\ell=1}^L X^{(\ell)^{\top}} X^{(\ell)}\big(\bar{\beta}^{(\ell)} \bar{\beta}^{(\ell)^{\top}}-\frac{1}{p} \Omega\big) X^{(\ell)^{\top}} X^{(\ell)}\label{non_sparse_term_1'}\\
    &\frac{1}{\sqrt{p}L}\frac{L}{L_0} \sum_{\ell=1}^L X^{(\ell)^{\top}}(\varepsilon^{(\ell)} \varepsilon^{(\ell)^{\top}}-\sigma^2 I) X^{(\ell)}\label{non_sparse_term_2'}\\
    &\frac{1}{\sqrt{p}L}\frac{L}{L_0}\sum_{\ell=1}^{L}X^{(\ell)^{\top}}\varepsilon^{(\ell)}\bar{\beta}^{(\ell)\top}X^{(\ell)\top}X^{(\ell)}\label{non_sparse_term_3'}
\end{align}
By a similar calculation as above, the respective concentration inequalities for terms \eqref{non_sparse_term_1'}, \eqref{non_sparse_term_2'} and \eqref{non_sparse_term_3'} are given by
\begin{align*}
&\mathbb{P}\Big(\Big\|\frac{1}{\sqrt{p} L} \sum_{\ell=1}^L \frac{L}{L_0}X^{(\ell)\top} X^{(\ell)}\big(\bar{\beta}^{(\ell)} \bar{\beta}^{(\ell)\top}-\frac{1}{p} \Omega\big) X^{(\ell)\top} X^{(\ell)}\Big\|\geq t\Big)\\
\leq & 2p\exp\Big\{-\frac{1}{K}\frac{L_{0}^2t^2}{L\mathcal{O}(p^3)+L_{0}t\mathcal{O}(p^{\frac{3}{2}})\log \big(\mathcal{O}(p^{\frac{3}{2}})\big)}\Big\}\\
=&2p\exp\Bigg\{-\frac{1}{K}\frac{t^2}{\mathcal{O}\big(\big(\frac{L}{L_0}\big)^2\frac{p^{3}}{L}\big)+t\mathcal{O}\big(\frac{L}{L_0}\frac{p^{\frac{3}{2}}}{L}\log\big(p^{\frac{3}{2}}\big)\big)}\Bigg\},  
\end{align*}
and
\begin{align*}
&\mathbb{P}\Big(\Big\|\frac{1}{\sqrt{p}L} \sum_{\ell=1}^L X^{(\ell)^{\top}}(\varepsilon^{(\ell)} \varepsilon^{(\ell)^{\top}}-\sigma^2 I) X^{(\ell)}\Big\|\geq t\Big)\\
\leq & 2p\exp\Big\{-\frac{1}{K}\frac{L_{0}^2t^2}{L\mathcal{O}(p^{4})+L_{0}t \mathcal{O}(p^{2}) \log(\mathcal{O}(p^{2}) )}\Big\}\\
=&2p\exp\Bigg\{-\frac{1}{K}\frac{t^2}{\mathcal{O}\big(\big(\frac{L}{L_0}\big)^2\frac{p^{4}}{L}\big)+t \mathcal{O}\big(\big(\frac{L}{L_0}\big)\frac{p^{2}}{L} \log(p^{2})\big)}\Bigg\}
\end{align*}
and
\begin{align*}
&\mathbb{P}\Bigg(\Big\|\frac{1}{L} \sum_{\ell=1}^{L} \frac{1}{\sqrt{p}}\frac{L}{L_0} X^{(\ell)\top} \varepsilon^{(\ell)} \bar{\beta}^{(\ell)\top} X^{(\ell)\top} X^{(\ell)}\Big\|\geq t\Big)\\
\leq& 2p\exp\Big\{-\frac{1}{K}\frac{t^2L_{0}^2}{L\mathcal{O}(p^{3})+tL_{0}\mathcal{O}(p^{\frac{3}{2}})\log(\mathcal{O}(p^{\frac{3}{2}}))}\Big\}\\
= & 2p\exp\Bigg\{-\frac{1}{K}\frac{t^2}{\mathcal{O}\big(\big(\frac{L}{L_0}\big)^2\frac{p^3}{L}\big)+t\mathcal{O}\big(\big(\frac{L}{L_0}\big)\frac{ p^{\frac{3}{2}}}{L}\log\big(p^{\frac{3}{2}} \big)\big)}\Bigg\}.
\end{align*}
Therefore, as long as $\frac{L}{L_0}\frac{p^2}{L}=o(1)$, $\|\hat{\Omega}-\Omega\|_{F}\stackrel{p}{\rightarrow}0$. Now if $\frac{L_0}{L}\rightarrow 0$, then $L$ needs to be higher order in previous case.\\

\textbf{Proof of Part (ii).} When all $\gamma_{\ell}=\gamma$, we can track $\gamma$ in our theoretical results. The inequality \eqref{non_sparse_first_concentration} becomes
\begin{align}
&\mathbb{P}\Big(\Big\|\frac{1}{\sqrt{p} L} \sum_{\ell=1}^L X^{(\ell)\top} X^{(\ell)}\big(\bar{\beta}^{(\ell)} \bar{\beta}^{(\ell)\top}-\frac{1}{p} \Omega\big) X^{(\ell)\top} X^{(\ell)}\Big\|\geq t\Big)\nonumber\\
\leq & 2p\exp\Bigg\{-\frac{1}{K}\frac{L^2t^2}{L\mathcal{O}((1+\sqrt{\gamma})^8 p^3)+Lt\mathcal{O}\big((1+\sqrt{\gamma})^4 p^{\frac{3}{2}}\big)\log \big(\mathcal{O}\big(p^{\frac{3}{2}}\big)\big)}\Bigg\}\nonumber\\
=&2p\exp\Bigg\{-\frac{1}{K}\frac{t^2}{\mathcal{O}((1+\sqrt{\gamma})^8 p^{3}/L)+t\mathcal{O}\big((1+\sqrt{\gamma})^4 p^{\frac{3}{2}} \log\big(p^{\frac{3}{2}}\big)/L\big)}\Bigg\}.\label{non_sparse_first_concentration''}
\end{align}
So term \eqref{non_sparse_term_1} is $\mathcal{O}_{P}\big((1+\sqrt{\gamma})^4\frac{p^{3/2}}{L}\big)$. The second inequality \eqref{non_sparse_second_concentration} becomes
\begin{align}
&\mathbb{P}\Big(\Big\|\frac{1}{\sqrt{p}L} \sum_{\ell=1}^L X^{(\ell)^{\top}}(\varepsilon^{(\ell)} \varepsilon^{(\ell)\top}-\sigma^2 I) X^{(\ell)}\Big\|\geq t\Big)\nonumber\\
\leq & 2p\exp\Bigg\{-\frac{1}{K}\frac{L^2t^2}{L\mathcal{O}((1+\sqrt{\gamma})^4n_{\ell}^{5}/p)+Lt \mathcal{O}( (1+\sqrt{\gamma})^2n_{\ell}^{\frac{5}{2}}/\sqrt{p}) \log(\mathcal{O}(n_{\ell}^{\frac{5}{2}}/\sqrt{p}))}\Bigg\}\nonumber\\
=&2p\exp\Bigg\{-\frac{1}{K}\frac{t^2}{\mathcal{O}((1+\sqrt{\gamma})^4p^{4}/L)+t \mathcal{O}((1+\sqrt{\gamma})^2p^{2} \log(p^2)/L)}\Bigg\}. \label{non_sparse_second_concentration''}
\end{align}
Hence, the term \eqref{non_sparse_term_2} is $\mathcal{O}_P\big((1+\sqrt{\gamma})^2\frac{p^2}{L}\big)$. Finally, the inequality \eqref{non_sparse_third_concentration} is 
\begin{align}
&\mathbb{P}\Big(\Big\|\frac{1}{L} \sum_{l=1}^L \frac{1}{\sqrt{p}} X^{(\ell)\top} \varepsilon^{(\ell)} \bar{\beta}^{(\ell)\top} X^{(\ell)\top} X^{(\ell)}\Big\|\geq t\Big)\nonumber\\
\leq& 2p\exp\Big\{-\frac{1}{K}\frac{t^2L^2}{L\mathcal{O}(p^{3})+tL\mathcal{O}(p^{\frac{3}{2}})\log\big(\mathcal{O}(p^{\frac{3}{2}})\big)}\Big\}\nonumber\\
= & 2p\exp\Big\{-\frac{1}{K}\frac{t^2}{\mathcal{O}\big((1+\sqrt{\gamma})^6p^3/L \big)+t\mathcal{O}\big(p^{\frac{3}{2}}\log\big((1+\sqrt{\gamma})^3p^{\frac{3}{2}}/L \big)\big)}\Big\}.\label{non_sparse_third_concentration''}
\end{align}
Then the cross term \eqref{non_sparse_term_3} is $\mathcal{O}_{P}\big((1+\sqrt{\gamma})^3p^{\frac{3}{2}}/L\big)$. Therefore, $\|\operatorname{grad}f(\Omega)\|_{F}=\mathcal{O}_P\big((1+\sqrt{\gamma})^2\frac{p^2}{L}\big)$. On the other hand, the quantity $\gamma$ also appears in the inequality \eqref{key_inequality_for_concentration}. If we track $\gamma$ explicitly, inequality \eqref{key_inequality_for_concentration} becomes
\begin{align}
\frac{4L_0}{L}\mathcal{O}\big(\frac{1}{\gamma^2}(1-\sqrt{\gamma})^2\big)\|\Omega-\hat{\Omega}\|_F \leq \|\operatorname{grad}f(\Omega)\|_{F}   \label{key_inequality_for_concentration_gamma_case} 
\end{align} 
Finally, \eqref{key_inequality_for_concentration_gamma_case}, \eqref{non_sparse_first_concentration''}, \eqref{non_sparse_second_concentration''} and \eqref{non_sparse_third_concentration''} imply $\|\Omega-\hat{\Omega}\|_{F}=\mathcal{O}_{P}\Big(\frac{(1+\sqrt{\gamma})^2\gamma^2}{(1-\sqrt{\gamma})^2}\frac{p^2}{L}\Big)$.
\end{proof}

\subsection{Proof of Theorem \ref{thm1_sparse_cov_est}}
\begin{proof}[Proof of Theorem \ref{thm1_sparse_cov_est}]
Define functions $f(\tilde{\Omega})$ and $Q(\tilde{\Omega})$ to be 
\begin{align*}
f(\tilde{\Omega})\coloneqq\frac{1}{L}\sum_{\ell=1}^{L} \Big\| y^{(\ell)} y^{(\ell) \top}-\frac{1}{p}X^{(\ell)} \tilde{\Omega} X^{(\ell) \top}-\sigma^2 I\Big\|_F^2+\tilde{\lambda}\sum_{i\not=j}|\tilde{\Omega}_{ij}|,
\end{align*} 
and $Q(\tilde{\Omega})\coloneqq f(\tilde{\Omega})-f(\Omega)$. With immediate calculation, $Q(\tilde{\Omega})$ could be simplified to 
\begin{align*}
    Q(\tilde{\Omega})&=\frac{1}{L} \sum_{l=1}^L\Big(\Big\|y^{(\ell)} y^{(\ell) \top}-\frac{1}{p} X^{(\ell)} \tilde{\Omega} X^{(\ell)\top}-\sigma^2 I\Big\|_F^2\\
    &\quad\quad\quad\quad\quad-\Big\|y^{(\ell)} y^{(\ell)\top}-\frac{1}{p} X^{(\ell)} \Omega X^{(\ell)\top}-\sigma^2 I\Big\|_F^2\Big)\\
    &\quad\quad\quad\quad\quad+\tilde{\lambda}(|\tilde{\Omega}^{-}|_1-|\Omega^{-}|_1).
\end{align*}
Since the estimator $\hat{\Omega}$ minimize $f(\tilde{\Omega})$, then it holds that $\hat{\Omega}$ minimizes $Q(\Omega)$, or equivalently $\hat{\Delta}=\hat{\Omega}-\Omega$ minimizes $G(\Delta) \equiv$ $Q(\Omega+\Delta)$. Consider the set 
$$\Theta_p(M)=\big\{\Delta: \Delta=\Delta^T,\|\Delta\|_F=M r(p,L)\big\},$$
where 
\begin{itemize}
    \item $M$ is some absolute constant that does not depends on $L,p$ and $n_{\ell}$'s,
    \item $r(p,L)=\sqrt{\frac{(p+s)\log p}{L}}$ and goes to zero as $p,L$ and $n_{\ell}$ go to infinity such that $\frac{p}{n_{\ell}}\rightarrow \gamma_{\ell}$.
\end{itemize} 
Since $f(\tilde{\Omega})$ is geodesically convex, it follows that $G(\Delta)$ is also geodesically convex. Also, it holds that $G(\hat{\Delta})\leq 0$. If we could show that $\inf\{G(\Delta): \Delta \in \Theta_p(M)\}>0$, the minimizer $\hat{\Delta}$ must be inside the sphere defined by $\Theta_n(M)$, and hence
\begin{align*}
\|\hat{\Omega}-\Omega\|_{F}=\|\hat{\Delta}\|_F \leq M r(p,L).
\end{align*}
We now do a Taylor expansion of 
\begin{align*}
    f_{\ell}(t)&=\Big\|y^{(\ell)} y^{(\ell)}-\frac{1}{p} X^{(\ell)}(\Omega+t \Delta) X^{(\ell)\top}-\sigma^2 I\Big\|_F^2\\
    &=\operatorname{tr}\Big(\Big(y^{(\ell)} y^{(\ell)\top}-\frac{1}{p} X^{(\ell)}(\Omega+t \Delta) X^{(\ell)\top}-\sigma^2 I\Big)\Big(y^{(\ell)} y^{(\ell)\top}-\frac{1}{p} X^{(\ell)}(\Omega+t \Delta) X^{(\ell)}-\sigma^2 I\Big)\Big).
\end{align*}
Note that
\begin{align*}
    \frac{d f_{\ell}(t)}{d t}&=2 \operatorname{tr}\Big[\Big(y^{(\ell)} y^{(\ell)\top}-\frac{1}{p} X^{(\ell)}(\Omega+t \Delta) X^{(\ell)\top}-\sigma^2 I\Big)\Big(-\frac{1}{p} X^{(\ell)} \Delta X^{(\ell)\top}\Big)\Big],\\
    \frac{d^2 f_{\ell}(t)}{d t^2}&=2 \operatorname{tr}\Big(\frac{1}{p^2} X^{(\ell)} \Delta X^{(\ell)\top} X^{(\ell)} \Delta X^{(\ell)\top}\Big)=\frac{2}{p^2} \operatorname{tr}\Big(X^{(\ell)\top} X^{(\ell)} \Delta X^{(\ell)\top} X^{(\ell)} \Delta\Big).
\end{align*}
Therefore, using second order Taylor theorem, it holds that
\begin{align*}
    &\Big\|y^{(\ell)} y^{(\ell) \top}-\frac{1}{p} X^{(\ell)}(\Omega+\Delta) X^{(\ell)\top}-\sigma^2 I\Big\|_F^2-\Big\|y^{(\ell)} y^{(\ell)\top}-\frac{1}{p} X^{(\ell)} \Omega X^{(\ell)\top}-\sigma^2 I\Big\|_F^2\\
    =&2 \operatorname{tr}\Big(\Big(y^{(\ell)} y^{(\ell)^{\top}}-\frac{1}{p} X^{(\ell)} \Omega X^{(\ell)^{\top}}-\sigma^{2 }I\Big)\Big(-\frac{1}{p} X^{(\ell)} \Delta X^{(\ell) \top}\Big)\Big)\\
    &\qquad+\frac{2}{p^2} \int_0^1(1-v) \operatorname{tr}\big(X^{(\ell)\top} X^{(\ell)} \Delta X^{(\ell)\top} X^{(\ell)} \Delta\big) d v\\
    =&-\frac{2}{p} \operatorname{tr}\Big[\Big(y^{(\ell)} y^{(\ell)\top}-\frac{1}{p} X^{(\ell)} \Omega X^{(\ell)\top}-\sigma^2 I\Big) X^{(\ell)} \Delta X^{(\ell)\top}\Big]\\
    &\qquad +\frac{1}{p^2} \operatorname{tr}\big(X^{(\ell)\top} X^{(\ell)} \Delta X^{(\ell)\top} X^{(\ell)} \Delta\big).
\end{align*}
Hence, for $\Delta=\tilde{\Omega}-\Omega$, by taking the average on $\ell$ over $1$ to $L$, it holds that
\begin{align*}
    G(\Delta)=Q(\Omega+\Delta)&=-\frac{2}{pL} \sum_{\ell=1}^L \operatorname{tr}\Big(\Big(y^{(\ell)} y^{(\ell)\top}-\frac{1}{p} X^{(\ell)} \Omega X^{(\ell)\top}-\sigma^{2 }I\Big) X^{(\ell)} \Delta X^{(\ell)}\Big)\\
    &\quad\quad+\frac{1}{p^2 L} \sum_{\ell=1}^L \operatorname{tr}\Big(X^{(\ell)\top} X^{(\ell)} \Delta X^{(\ell)\top} X^{(\ell)} \Delta\Big)\\
    &\quad\quad+\tilde{\lambda}(|\Omega^{-}+\Delta^{-}|_1-|\Omega^{-}|_1).
\end{align*}
Note that $S=\{(i, j): \Omega_{i j} \neq 0, i \neq j\}$, it holds that $|\Omega_0^{-}+\Delta^{-}|_1=|\Omega_{S}^{-}+\Delta_S^{-}|_1+|\Delta_{S^{c}}^{-}|_1$ and $|\Omega^{-}|_1=|\Omega_{S}^{-}|_1$. Therefore, by triangle inequality, it holds that
$$
\tilde{\lambda}(|\Omega^{-}+\Delta^{-}|_1-|\Omega^{-}|_1) \geq \tilde{\lambda}(|\Delta_{S^c}^{-}|_1-|\Delta_S^{-}|_1).
$$

We first bound the term
    \begin{align*}
        &\frac{1}{pL} \sum_{\ell=1}^L \operatorname{tr}\Big(\Big(y^{(\ell)} y^{(\ell)\top}-\frac{1}{p} X^{(\ell)} \Omega X^{(\ell)\top}-\sigma^{2 }I\Big) X^{(\ell)} \Delta X^{(\ell)\top}\Big)\\
        =&\operatorname{tr}\Big(\frac{1}{pL} \sum_{\ell=1}^L\Big(y^{(\ell)} y^{(\ell)\top}-\frac{1}{p} X^{(\ell)} \Omega X^{(\ell)\top}-\sigma^{2 }I\Big) X^{(\ell)} \Delta X^{(\ell)\top}\Big)\\
        =&\operatorname{tr}\Big(\frac{1}{pL} \sum_{\ell=1}^LX^{(\ell)\top}\Big(y^{(\ell)} y^{(\ell)\top}-\frac{1}{p} X^{(\ell)} \Omega X^{(\ell)\top}-\sigma^{2 }I\Big) X^{(\ell)} \Delta \Big).
    \end{align*}
    Define the matrix 
    $$A=\frac{1}{pL} \sum_{\ell=1}^LX^{(\ell)\top}\Big(y^{(\ell)} y^{(\ell)\top}-\frac{1}{p} X^{(\ell)} \Omega X^{(\ell)\top}-\sigma^{2}I\Big) X^{(\ell)}.$$
    Then,
    $$|\operatorname{tr}(A\Delta)|\leq \Big|\sum_{i\not=j}A_{ij}\Delta_{ij}\Big|+\Big|\sum_{i=1}^{p}A_{ii}\Delta_{ii}\Big|=(i)+(ii).$$
    For term $(i)$, it holds that 
    $$(i)\leq \max_{i\not=j}|A_{ij}|\sum_{i\not=j}|\Delta_{ij}|=\Big(\max_{i\not=j}|A_{ij}|\Big)|\Delta^{-}|_{1}.$$
    For term $(ii)$, it holds that
    $$
    (ii)\leq \sqrt{\sum_{i=1}^{p}A_{ii}^2}\|\Delta^{+}\|_{F}\leq \sqrt{p}\Big(\max_{i=1,\dots,p}|A_{ii}|\Big)\|\Delta^{+}\|_{F}.
    $$
    We now bound $\max_{i=1,\dots,p}|A_{ii}|$ and $\max_{i\not=j}|A_{ij}|$ with probability tending to $1$ as $p,L\rightarrow\infty$. Recall that
    $$
    A=\frac{1}{pL} \sum_{\ell=1}^LX^{(\ell)\top}\Big(y^{(\ell)} y^{(\ell)\top}-\frac{1}{p} X^{(\ell)} \Omega X^{(\ell)\top}-\sigma^{2}I\Big) X^{(\ell)}.
    $$
    Denote 
    $$X^{(\ell)}=\begin{bmatrix}
    x_1^{(\ell) \top} \\ \vdots \\ x_{n_{\ell}}^{(\ell) \top}    
    \end{bmatrix},$$
    and the $(i,j)$-th entry of $A$ by $A_{ij}$. Therefore, 
    \begin{align*}
        &LA_{ij}\\
        =& \sum_{\ell=1}^{L}\Big[\frac{1}{p}X^{(\ell)\top} y^{(\ell)} y^{(\ell)\top} X^{(\ell)}-\frac{1}{p^2} X^{(\ell)\top} X^{(\ell)} \Omega X^{(\ell)\top} X^{(\ell)}-\frac{1}{p}\sigma^2 X^{(\ell)\top} X^{(\ell)}\Big]_{i j}\\
        =&\sum_{\ell=1}^{L}\Big[\Big[\frac{1}{\sqrt{p}}X^{(\ell)\top} y^{(\ell)}\Big]_i\Big[\frac{1}{\sqrt{p}}X^{\ell)\top} y^{(\ell)}\Big]_j-\Big(\frac{1}{p^2} X^{(\ell)\top} X^{(\ell)} \Omega X^{(\ell)\top} X^{(\ell)}+\frac{1}{p}\sigma^2 X^{(\ell)\top} X^{(\ell)}\Big)_{ij}\Big]\\
        =&\sum_{\ell=1}^{L}\Big[\Big[\frac{1}{\sqrt{p}}X^{(\ell)\top} y^{(\ell)}\Big]_i\Big[\frac{1}{\sqrt{p}}X^{\ell)\top} y^{(\ell)}\Big]_j-\xi_{ij}\Big],
    \end{align*}
    where $[X^{(\ell)\top} y^{(\ell)}]_i$ is the $i$-th entry of the vector $X^{(\ell)\top} y^{(\ell)}\in\mathbb{R}^p$ and 
    $$\xi_{ij}=\Big(\frac{1}{p^2} X^{(\ell)\top} X^{(\ell)} \Omega X^{(\ell)\top} X^{(\ell)}+\frac{1}{p}\sigma^2 X^{(\ell)\top} X^{(\ell)}\Big)_{ij}$$ is the $(i,j)$-th entry of matrix $\frac{1}{p^2} X^{(\ell)\top} X^{(\ell)} \Omega X^{(\ell)\top} X^{(\ell)}+\frac{1}{p}\sigma^2 X^{(\ell)\top} X^{(\ell)}$. 
    
    Note that
    \begin{align*}
        \frac{1}{\sqrt{p}}X^{(\ell)\top} y^{(\ell)}&=\frac{1}{\sqrt{p}}X^{(\ell)\top}X^{(\ell)}\bar{\beta}^{(\ell)}+\frac{1}{\sqrt{p}}X^{(\ell)\top}\varepsilon^{(\ell)}\\
        &=\frac{1}{p}X^{(\ell)\top}X^{(\ell)}(\sqrt{p}\bar{\beta}^{(\ell)})+\frac{1}{\sqrt{p}}X^{(\ell)\top}\varepsilon^{(\ell)}.
    \end{align*}
    Now the $i$-th entry of $\frac{1}{\sqrt{p}}X^{(\ell)\top} y^{(\ell)}$ is sub-Gaussian with parameter at most
    $$\Big\|\frac{1}{p}X^{(\ell)\top}X^{(\ell)}\Big\|\tau_{\beta}+\Big\|\frac{1}{\sqrt{p}}X^{(\ell)}\Big\|\tau_{\varepsilon}=\mathcal{O}(1).$$
    Finally, define the event
    $$A_{ij}(t)=\{|A_{ij}|>t\}=\Big\{\Big|\sum_{\ell=1}^{L}\Big[\frac{1}{\sqrt{p}}X^{(\ell)\top} y^{(\ell)}\Big]_i\Big[\frac{1}{\sqrt{p}}X^{\ell)\top} y^{(\ell)}\Big]_j-\xi_{ij}\Big|>Lt\Big\}.$$ 
    To derive the high probability bound for $A_{ij}(t)$, we proceed this first by decoupling the product 
    $$
    \Big[\frac{1}{\sqrt{p}}X^{(\ell)\top} y^{(\ell)}\Big]_i\Big[\frac{1}{\sqrt{p}}X^{\ell)\top} y^{(\ell)}\Big]_j.
    $$ 
    Define the random variables
    \begin{align*}
        U_{i j}^{(\ell)}&=\Big[\frac{1}{\sqrt{p}}X^{(\ell)\top} y^{(\ell)}\Big]_i+\Big[\frac{1}{\sqrt{p}}X^{\ell)\top} y^{(\ell)}\Big]_j,\\
        V_{i j}^{(\ell)}&=\Big[\frac{1}{\sqrt{p}}X^{(\ell)\top} y^{(\ell)}\Big]_i-\Big[\frac{1}{\sqrt{p}}X^{\ell)\top} y^{(\ell)}\Big]_j,
    \end{align*}
    whose the second moments are given by $\mathbb{E}\big[\big(U_{i j}^{(\ell)}\big)^2\big]=u_{i j}^{(\ell)}$ and $\mathbb{E}\big[\big(V_{i j}^{(\ell)}\big)^2\big]=v_{i j}^{(\ell)}$. Therefore,
    \begin{align*}
        &\sum_{\ell=1}^L\Big[\frac{1}{\sqrt{p}}X^{(\ell)\top} y^{(\ell)}\Big]_i\Big[\frac{1}{\sqrt{p}}X^{\ell)\top} y^{(\ell)}\Big]_j-\xi_{i j},\\ 
        =&\frac{1}{4} \sum_{\ell=1}^L\Big[U_{i j}^{(\ell)2}-u_{i j}^{(\ell)}\Big]-\frac{1}{4} \sum_{l=1}^L\Big[V_{i j}^{(\ell)2}-v_{i j}^{(\ell)}\Big].
    \end{align*}
    Hence, it holds that
    \begin{equation*}
        \mathbb{P}(A_{i j}(t)) \leq \mathbb{P}\Big(\sum_{\ell=1}^{L}\Big[U_{i j}^{(\ell)2}-u_{i j}^{(\ell)}\Big] \geq \frac{4 Lt}{2}\Big)+\mathbb{P}\Big(\sum_{\ell=1}^L\Big[V_{i j}^{(\ell)2}-v_{i j}^{(\ell)}\Big] \geq \frac{4 L t}{2}\Big).
    \end{equation*}
    Now, random variables $U_{ij}^{(\ell)}$ and $V_{ij}^{(\ell)}$ are sub-Gaussian with parameter $\sigma_{U}$ and $\sigma_{V}$ at most 
    $$\tilde{\sigma}=2\Big(\Big\|\frac{1}{p}X^{(\ell)\top}X^{(\ell)}\Big\|\tau_{\beta}+\Big\|\frac{1}{\sqrt{p}}X^{(\ell)}\Big\|\tau_{\varepsilon}\Big)=\mathcal{O}(1).$$
    Next we show that $U_{ij}^{(\ell)2}-u_{ij}^{(\ell)}$ and $V_{ij}^{(\ell)2}-v_{ij}^{(\ell)}$ are sub-exponential.
%    \begin{clm}
%        The random variables $U_{ij}^{(\ell)2}-u_{ij}^{(\ell)}$ and $V_{ij}^{(\ell)2}-v_{ij}^{(\ell)}$ are sub-exponential
%    \end{clm}
    To prove this, according to \citet[Theorem 2.2]{wainwright2019high} and \cite[Theorem 3.2]{buldygin2000metric}, if find a $B$ such that
    $$\sup _{m \geq 2}\Bigg[\frac{\mathbb{E}\big[\big(U_{i j}^{(\ell)2}-u_{i j}^{(\ell)}\big)^m\big]^{\frac{1}{m}}}{m !}\Bigg]\leq B,$$
    then $U_{ij}^{(\ell)2}-u_{ij}^{(\ell)}$ is sub-exponential with parameter $2 B$ in the interval $\big(-\frac{1}{2 B},+\frac{1}{2 B}\big)$. Note that
    \begin{align*}
        \frac{\mathbb{E}\big[\big(U_{i j}^{(\ell)2}-u_{i j}^{(\ell)}\big)^m\big]^{\frac{1}{m}}}{m !} & \leq\Bigg[\frac{2^m\big(\mathbb{E} (U_{i j}^{(\ell)})^{2 m}+(u_{i j}^{(\ell)})^m\big)}{m !}\Bigg]^{\frac{1}{m}}\\
        &\leq\Bigg[2^{2 m+1} \sigma_{U}^{2 m}+\frac{\big(2 u_{i j}^{(\ell)}\big)^m}{m !}\Bigg]^{\frac{1}{m}}\\
        &\leq 2^{\frac{1}{m}}\Bigg[\big(2^{2 m+1} \sigma_U^{2 m}\big)^{\frac{1}{m}}+\frac{2 u_{i j}^{(\ell)}}{(m !)^{\frac{1}{m}}}\Bigg]\\
        &=2^{\frac{1}{m}}\Big(2^{\frac{1}{m}} 4 \sigma_U^2+\frac{2 u_{i j}^{(\ell)}}{(m !)^{\frac{1}{m}}}\Big),
    \end{align*}
    where 
    \begin{itemize}
        \item the first inequality follows from the inequality $(a+b)^m \leq 2^m(a^m+b^m)$,
        \item the second inequality follows from the moment bound of sub-Gaussian random variable (e.g. Lemma 1.4 from \cite{buldygin2000metric}) $\mathbb{E}\big[\big(U_{i j}^{(\ell)}\big)^{2 m}\big] \leq 2\big(\frac{2 m}{e}\big)^m\big(\sigma_U^2\big)^{m}$ and inequality $m ! \geq(m / e)^m$,
        \item the third inequality follows from the inequality $(x+y)^{1 / m} \leq 2^{1 / m}\big(x^{1 / m}+y^{1 / m}\big)$, valid for any integer $m \in \mathbb{N}$ and positive $x,y$.
    \end{itemize}
    Therefore, at this point we bound $\frac{\mathbb{E}\big[\big(U_{i j}^{(\ell)2}-u_{i j}^{(\ell)}\big)^m\big]^{\frac{1}{m}}}{m !}$ by a decreasing function of $m$. It holds that
    \begin{align*}
    \sup _{m \geq 2}\Bigg[\frac{\mathbb{E}\big[\big(U_{i j}^{(\ell)2}-u_{i j}^{(\ell)}\big)^m\big]^{\frac{1}{m}}}{m !}\Bigg]\leq 2^{\frac{1}{2}}\Big(2^{\frac{1}{2}} 4 \sigma_U^2+\frac{u_{i j}^{(\ell)}}{2^{\frac{1}{2}}}\Big)=8 \sigma_U^2+u_{i j}^{(\ell)}\coloneqq B    
    \end{align*} 
    Therefore, $U_{i j}^{(\ell)2}-u_{i j}^{(l)}$ and $V_{i j}^{(\ell)2}-v_{i j}^{(\ell)}$ are sub-exponential with parameter at most $(16 \tilde{\sigma}^2+2u_{i j}^{(\ell)}, 16 \tilde{\sigma}^2+2u_{i j}^{(\ell)})$ and $(16 \tilde{\sigma}^2+2v_{i j}^{(\ell)}, 16 \tilde{\sigma}^2+2v_{i j}^{(\ell)})$

    Now we are ready to derive the bound for $\mathbb{P}(A_{ij}(t))$. We have that
    \begin{align*}
         \mathbb{P}(A_{i j}(t))\leq &\mathbb{P}\Big(\sum_{\ell=1}^{L}\big[U_{i j}^{(\ell)2}-u_{i j}^{(\ell)}\big] \geq 2 L t\Big)\\
         &+\mathbb{P}\Big(\sum_{\ell=1}^L\big[V_{i j}^{(\ell)2}-v_{i j}^{(\ell)}\big] \geq 2 L  t\Big)\\
         \leq& 2\exp\Bigg\{-\frac{4L^2t^2}{2\sum_{\ell=1}^{L}\big(16 \tilde{\sigma}^2+2u_{i j}^{(\ell)}\big)^2}\Bigg\}\\
         &+2\exp\Bigg\{-\frac{4L^2t^2}{2\sum_{\ell=1}^{L}\big(16 \tilde{\sigma}^2+2v_{i j}^{(\ell)}\big)^2}\Bigg\},
    \end{align*}
    where
    \begin{align*}
        u_{i j}^{(\ell)}&\leq \frac{4}{p} \Big\|\mathbb{E}\big[y^{(\ell)} y^{(\ell)\top}\big]\Big\|\Big\|X^{(\ell)}\Big\|^2=4\Big\|\frac{1}{\sqrt{p}} X^{(\ell)}\Big\|^2\Big\|\frac{1}{p} X^{(\ell)} \Omega X^{(\ell)\top}+\sigma^2 I\Big\|=\mathcal{O}(1).
    \end{align*}
    Similarly, $v_{ij}=\mathcal{O}(1)$. This implies 
    \begin{align*}
    \mathbb{P}(A_{i j}(t))\leq 4\exp\{-\mathcal{O}(Lt^2)\},\quad\quad\text{for }Lt\leq \delta.
    \end{align*} 
    Therefore, by taking $t=C_{1}\sqrt{\frac{\log p}{L}}$ with sufficient large absolute constant $C_{1}$, 
    $$\max _{i \neq j}|A_{ij}| \leq C_1 \sqrt{\frac{\log p}{L}},$$
    with probability tending to $1$. Similarly, 
    \begin{align*}
    (ii)\leq \sqrt{\sum_{i=1}^{p}A_{ii}^2}\|\Delta^{+}\|_{F}\leq \sqrt{p}\Big(\max_{i=1,\dots,p}|A_{ii}|\Big)\|\Delta^{+}\|_{F}\leq C_{2}\sqrt{\frac{p\log p}{L}}\|\Delta^{+}\|_{F}.
    \end{align*}
    Therefore, with probability tending to $1$, it holds that
    \begin{align*}
        (i)+(ii)&\leq C_1 \sqrt{\frac{\log p}{L}}|\Delta^{-}|_1+C_{2}\sqrt{\frac{p\log p}{L}}\|\Delta^{+}\|_F.
    \end{align*}

 Now, for the term $\frac{1}{p^2 L} \sum_{\ell=1}^L \operatorname{tr}\big(X^{(\ell)\top} X^{(\ell)} \Delta X^{(\ell)\top} X^{(\ell)} \Delta\big)$, by condition (e) in Assumption \ref{fixed_design_asp1} it holds that 
    \begin{align*}
    \frac{1}{p^2 L} \sum_{\ell=1}^L \operatorname{tr}\big(X^{(\ell)\top} X^{(\ell)} \Delta X^{(\ell)\top} X^{(\ell)} \Delta\big)&=\frac{1}{p^2L}\sum_{\ell=1}^{L}\big\|X^{(\ell)}\otimes X^{(\ell)}\operatorname{vec}(\Delta)\big\|_{2}^{2}\\
    &\geq   \min_{1\leq \ell\leq L} \kappa_{0}^{(\ell)}\|\Delta\|_F^2. 
\end{align*} 

Combining everything together, with choice of $\tilde{\lambda}=2C_{1}\sqrt{\frac{\log p}{L}}$, we have with probability tending to $1$,
\begin{align*}
    G(\Delta) &\geq \min_{1\leq \ell\leq L} \kappa_{0}^{(\ell)}\|\Delta\|_F^2-C_{1}\sqrt{\frac{\log p}{L}}|\Delta_{S}^{-} |_1-C_{1}\sqrt{\frac{\log p}{L}}| \Delta_{S^c}^{-}|_1\\
    &\quad\quad\quad\quad -C_{2}\sqrt{\frac{p\log p}{L}}\|\Delta^{+}\|_F+ \tilde{\lambda}(|\Delta_{S^c}^{-}|_1-|\Delta_S^{-}|_1)\\
    &\geq \min_{1\leq \ell\leq L} \kappa_{0}^{(\ell)}\|\Delta\|_F^2-\Big(C_{1}\sqrt{\frac{\log p}{L}}-\tilde{\lambda}\Big)|\Delta_{S^c}^{-} |_1\\
    &\quad\quad\quad\quad-\Big(C_{1}\sqrt{\frac{\log p}{L}}+\tilde{\lambda}\Big)| \Delta_{S}^{-}|_1-C_{2}\sqrt{\frac{p\log p}{L}}\|\Delta^{+}\|_F\\
    &\geq\|\Delta^{-}\|_F^2\Bigg( \min_{1\leq \ell\leq L} \kappa_{0}^{(\ell)}-\frac{\Big(C_{1}\sqrt{\frac{\log p}{L}}+\tilde{\lambda}\Big) \sqrt{s}}{Mr(p,L)}\Bigg)  \\
    &\quad\quad\quad\quad+\|\Delta^{+}\|_F^2\Bigg(\min_{1\leq \ell\leq L} \kappa_{0}^{(\ell)}-\frac{C_{2}\sqrt{\frac{p\log p}{L}}}{Mr(p,L)}\Bigg)\\
    &=\|\Delta^{-}\|_F^2\Bigg(\min_{1\leq \ell\leq L} \kappa_{0}^{(\ell)}-\frac{3C_{1}\sqrt{\frac{s\log p}{L}}}{Mr(p,L)}\Bigg)  +\|\Delta^{+}\|_F^2\Bigg(\min_{1\leq \ell\leq L} \kappa_{0}^{(\ell)}-\frac{C_{2}\sqrt{\frac{p\log p}{L}}}{Mr(p,L)}\Bigg).
\end{align*}
Now if $r(p,L)= \sqrt{\frac{(p+s)\log p}{L}}$ and $r(p,L)\rightarrow 0$, it holds that $\inf\big\{G(\Delta): \Delta \in \Theta_p(M)\big\}>0$ for sufficiently large $M$.  
\end{proof}
\subsection{Proof of Theorem \ref{thm2_sparse_cov_est}}
\begin{proof}[Proof of Theorem \ref{thm2_sparse_cov_est}]
Following the idea in the proof of Theorem \ref{thm1_sparse_cov_est}, define 
\begin{align*}
    G(\Delta) &\coloneqq \Big[\frac{1}{L}\sum_{\ell=1}^{L} \Big\|\hat{W}^{-\frac{1}{2}}z^{(\ell)}z^{(\ell)\top}\hat{W}^{-\frac{1}{2}}-\frac{1}{p}\tilde{\Theta}\Big\|_{F}^{2}+\tilde{\lambda}\sum_{i\not=j}|\tilde{\Theta}_{ij}|\Big]\\
    &\quad\quad\quad-\Big[\frac{1}{L}\sum_{\ell=1}^{L} \Big\|\hat{W}^{-\frac{1}{2}}z^{(\ell)}z^{(\ell)\top}\hat{W}^{-\frac{1}{2}}-\frac{1}{p}\Theta\Big\|_{F}^{2}+\tilde{\lambda}\sum_{i\not=j}|\Theta_{ij}|\Big].
\end{align*}
By a Taylor expansion, we obtain
\begin{align*}
    &\Big\|\hat{W}^{-\frac{1}{2}}z^{(\ell)} z^{(\ell) \top}\hat{W}^{-\frac{1}{2}}-\frac{1}{p} (\Theta+\Delta)\Big\|_F^2-\Big\|\hat{W}^{-\frac{1}{2}}z^{(\ell)} z^{(\ell)\top}\hat{W}^{-\frac{1}{2}}-\frac{1}{p} \Theta \Big\|_F^2\\
    =&2 \operatorname{tr}\Big(\Big(\hat{W}^{-\frac{1}{2}}z^{(\ell)} z^{(\ell)\top}\hat{W}^{-\frac{1}{2}}-\frac{1}{p}  \Theta \Big)\Big(-\frac{1}{p} \Delta \Big)\Big) +\frac{2}{p^2} \int_0^1(1-v) \operatorname{tr}(\Delta^2) d v\\
    =&-\frac{2}{p} \operatorname{tr}\Big[\Big(\hat{W}^{-\frac{1}{2}}z^{(\ell)} z^{(\ell)\top}\hat{W}^{-\frac{1}{2}}-\frac{1}{p} \Theta\Big) \Delta \Big] +\frac{1}{p^2} \operatorname{tr}(\Delta^2).
\end{align*}
Hence, for $\Delta=\tilde{\Theta}-\Theta$, by taking the average on $\ell$ over $1$ to $L$, it holds that
\begin{align*}
G(\Delta)=Q(\Theta+\Delta)&=-\frac{2}{pL} \sum_{\ell=1}^L \operatorname{tr}\Big[\Big(\hat{W}^{-\frac{1}{2}}z^{(\ell)} z^{(\ell)\top}\hat{W}^{-\frac{1}{2}}-\frac{1}{p} \Theta\Big) \Delta \Big]\\
&\quad\quad+\frac{1}{p^2} \operatorname{tr}(\Delta^2)+\tilde{\lambda}\big(|\Theta^{-}+\Delta^{-}|_1-|\Theta^{-}|_1\big)\\
&=-\frac{2}{p}  \operatorname{tr}\Big[\Big(\hat{W}^{-\frac{1}{2}}\Big(\frac{1}{L}\sum_{\ell=1}^Lz^{(\ell)} z^{(\ell)\top}\Big)\hat{W}^{-\frac{1}{2}}-\frac{1}{p} \Theta\Big) \Delta \Big]\\
&\quad\quad+\frac{1}{p^2} \operatorname{tr}(\Delta^2)+\tilde{\lambda}\big(|\Theta^{-}+\Delta^{-}|_1-|\Theta^{-}|_1\big).
\end{align*}
We first bound the linear term
\begin{align*}
\frac{1}{p}\operatorname{tr}\Big[\Big(\hat{W}^{-\frac{1}{2}}\Big(\frac{1}{L}\sum_{\ell=1}^Lz^{(\ell)} z^{(\ell)\top}\Big)\hat{W}^{-\frac{1}{2}}-\frac{1}{p} \Theta\Big) \Delta \Big].
\end{align*}
Denote $A=\frac{1}{p}\big(\hat{W}^{-\frac{1}{2}}\big(\frac{1}{L}\sum_{\ell=1}^Lz^{(\ell)} z^{(\ell)\top}\big)\hat{W}^{-\frac{1}{2}}-\frac{1}{p} \Theta\big)$ and note that 
\begin{align*}
&\Bigg|\frac{1}{p}\operatorname{tr}\Big[\Big(\hat{W}^{-\frac{1}{2}}\Big(\frac{1}{L}\sum_{\ell=1}^Lz^{(\ell)} z^{(\ell)\top}\Big)\hat{W}^{-\frac{1}{2}}-\frac{1}{p} \Theta\Big) \Delta \Big]\Bigg|\\
=&\Big|\sum_{i,j=1}^{p}A_{ij}\Delta_{ij}\Big|\leq\Big|\sum_{i \neq j} A_{i j} \Delta_{i j}\Big|+\Big|\sum_{i=1}^p A_{i i} \Delta_{i i}\Big|=(i)+(ii).
\end{align*}
For term (ii), the $(i,i)$-th entry of $A$ is exactly 0. For term (i), since $z^{(\ell)}=\big(X^{(\ell)}\big)^{-1}_{\mathsf{left}}y^{(\ell)}=\bar{\beta}^{(\ell)}$ is sub-Gaussian, according to \citet[Theorem 2.1]{shao2014necessary}), we have that
\begin{align*}
\max_{i\not=j}|A_{ij}|\leq \frac{C_{1}}{p^2}\sqrt{\frac{\log p}{L}}.
\end{align*}
Hence, it holds that $(i)\leq \frac{C_{1}}{p^2}\sqrt{\frac{\log p}{L}}|\Delta^{-}|_{1}$. Finally, with probability tending to $1$, with the choice of $\tilde{\lambda}=\frac{2C_{1}}{p^2}\sqrt{\frac{\log p}{L}}$, we obtain that 
\begin{align*}
    G(\Delta) &\geq \frac{1}{p^2} \|\Delta\|_F^2-C_{1}\sqrt{\frac{\log p}{p^4 L}}|\Delta_{S}^{-}|_1-C_{1}\sqrt{\frac{\log p}{p^4 L}}|\Delta_{S^c}^{-}|_1 + \tilde{\lambda}\big(|\Delta_{S^c}^{-}|_1-|\Delta_S^{-}|_1\big)\\
    &\geq \frac{1}{p^2}\|\Delta\|_F^2-\Big(C_{1}\sqrt{\frac{\log p}{p^4 L}}-\tilde{\lambda}\Big)|\Delta_{S^c}^{-}|_1-\Big(C_{1}\sqrt{\frac{\log p}{p^4 L}}+\tilde{\lambda}\Big)| \Delta_{S}^{-}|_1\\
    &\geq\|\Delta^{-}\|_F^2\Bigg(\frac{1}{p^2}-\frac{\big(C_{1}\sqrt{\frac{\log p}{p^4 L}}+\tilde{\lambda}\big) \sqrt{s}}{Mr(p,L)}\Bigg) +\frac{1}{p^2}\|\Delta^{+}\|_F^2 \\
    &=\|\Delta^{-}\|_F^2\Bigg( \frac{1}{p^2}-\frac{3C_{1}\sqrt{\frac{s\log p}{p^4 L}}}{Mr(p,L)}\Bigg)  +\frac{1}{p^2}\|\Delta^{+}\|_F^2.
\end{align*}
Now if $r(p,L)=\sqrt{\frac{s\log p}{L}}$ and $r(p,L)\rightarrow 0$, it holds that $\inf\{G(\Delta): \Delta \in \Theta_p(M)\}>0$ for sufficiently large $M$. Now, note that
\begin{align*}
    \hat{\Omega}_w-\Omega&=\hat{W}^{\frac{1}{2}} \hat{\Theta} \hat{W}^{\frac{1}{2}}-W^{\frac{1}{2}} \Theta W^{\frac{1}{2}}\\
    &=(\hat{W}^{\frac{1}{2}}-W^{\frac{1}{2}})(\hat{\Theta}-\Theta)(\hat{W}^{\frac{1}{2}}-W^{\frac{1}{2}})+(\hat{W}^{\frac{1}{2}}-W^{\frac{1}{2}}) \hat{\Theta} W^{\frac{1}{2}}\\
    &\quad\quad+\hat{W}^{\frac{1}{2}} \Theta\big(\hat{W}^{\frac{1}{2}}-W^{\frac{1}{2}}\big)+W^{\frac{1}{2}}(\hat{\Theta}-\Theta) \hat{W}^{\frac{1}{2}}.
\end{align*}
By the Lipshitz property of square root function when $x$ is bounded away from zero, it holds that $\|\hat{W}^{\frac{1}{2}}-W^{\frac{1}{2}}\|\leq \mathcal{O}_{P}\Big(\sqrt{\frac{\log p}{L}}\Big)$. Hence, we have that  
\begin{align*}
\|\hat{\Omega}_{w}-\Omega\|&\leq \|\hat{\Theta}-\Theta\|\|\hat{W}^{\frac{1}{2}}-W^{\frac{1}{2}}\|^2+\|\hat{W}^{\frac{1}{2}}-W^{\frac{1}{2}}\|\big(\|\hat{\Theta}\|\|W^{\frac{1}{2}}\|+\|\Theta\|\|\hat{W}^{\frac{1}{2}}\|\big)\\
&\quad\quad\quad+\|W^{\frac{1}{2}}\|\|\hat{\Theta}-\Theta\|\|\hat{W}^{\frac{1}{2}}\|\\
&\leq \mathcal{O}_{P}\Big(\sqrt{\frac{(s+1)\log p}{L}}\Big).
\end{align*}
\end{proof}
\subsection{Proof of Theorem \ref{thm3_sparse_cov_est}}
\begin{proof}[Proof of Theorem \ref{thm3_sparse_cov_est}]
To analyze the property of $\hat{\Theta}_{\lambda}$, we first consider the oracle estimator
\begin{align*}
&\tilde{\Theta}_{\lambda}^{\textsf{OR}}\\=&\arg\min_{\Theta\in\Gamma_{+}^{p}}\Big\{\frac{1}{L-L_0}\sum_{\ell=L_{0}+1}^{L}\Big\|y^{(\ell)}y^{(\ell)\top}-\frac{1}{p}X^{(\ell)}W^{\frac{1}{2}}\Theta W^{\frac{1}{2}}X^{(\ell)\top}\Big\|_{F}^{2}+\tilde{\lambda}\sum_{i\not=j}|\Theta_{ij}|\Big\}.
\end{align*}
Following similar steps as in \cite{rothman2008sparse}, we consider the Taylor expansion of 
\begin{align*}
    f_{\ell}(t)&=\Big\|y^{(\ell)} y^{(\ell)}-\frac{1}{p} X^{(\ell)}W^{\frac{1}{2}}(\Theta+t \Delta)W^{\frac{1}{2}} X^{(\ell)\top}\Big\|_F^2.
\end{align*}
The first and second order derivatives of $f_{\ell}(t)$ w.r.t. $t$ are given by
\begin{align*}
    \frac{d f_{\ell}(t)}{d t}&=2 \operatorname{tr}\Big[\Big(y^{(\ell)} y^{(\ell)\top}-\frac{1}{p} X^{(\ell)}W^{\frac{1}{2}}(\Omega+t \Delta)W^{\frac{1}{2}} X^{(\ell)\top}\Big)\Big(-\frac{1}{p} X^{(\ell)} W^{\frac{1}{2}}\Delta W^{\frac{1}{2}}X^{(\ell)\top}\Big)\Big],\\
    \frac{d^2 f_{\ell}(t)}{d t^2}&=2 \operatorname{tr}\Big[\frac{1}{p^2}\big( X^{(\ell)} W^{\frac{1}{2}}\Delta W^{\frac{1}{2}}X^{(\ell)\top}\big) \big(X^{(\ell)} W^{\frac{1}{2}}\Delta W^{\frac{1}{2}}X^{(\ell)\top}\big)\Big]\\
    &=\frac{2}{p^2} \operatorname{tr}\big(W^{\frac{1}{2}}X^{(\ell)\top} X^{(\ell)} W^{\frac{1}{2}}\Delta W^{\frac{1}{2}} X^{(\ell)\top} X^{(\ell)} W^{\frac{1}{2}}\Delta\big).
\end{align*}
Therefore, using second order Taylor theorem, it holds that
\begin{align*}
    &\Big\|y^{(\ell)} y^{(\ell) \top}-\frac{1}{p} X^{(\ell)}W^{\frac{1}{2}}(\Theta+\Delta)W^{\frac{1}{2}} X^{(\ell)\top}\Big\|_F^2-\Big\|y^{(\ell)} y^{(\ell)\top}-\frac{1}{p} X^{(\ell)}W^{\frac{1}{2}} \Theta W^{\frac{1}{2}}X^{(\ell)\top}\Big\|_F^2\\
    =&2 \operatorname{tr}\Big(\Big(y^{(\ell)} y^{(\ell)^{\top}}-\frac{1}{p} X^{(\ell)}W^{\frac{1}{2}} \Theta W^{\frac{1}{2}}X^{(\ell)^{\top}}\Big)\Big(-\frac{1}{p} X^{(\ell)} W^{\frac{1}{2}}\Delta W^{\frac{1}{2}}X^{(\ell) \top}\Big)\Big)\\
    &\quad\quad\quad+\frac{2}{p^2} \int_0^1(1-v) \operatorname{tr}\big(W^{\frac{1}{2}}X^{(\ell)\top} X^{(\ell)} W^{\frac{1}{2}}\Delta W^{\frac{1}{2}} X^{(\ell)\top} X^{(\ell)} W^{\frac{1}{2}}\Delta\big) d v\\
    =&-\frac{2}{p} \operatorname{tr}\Big[W^{\frac{1}{2}}X^{(\ell)\top}\Big(y^{(\ell)} y^{(\ell)\top}-\frac{1}{p} X^{(\ell)} W^{\frac{1}{2}}\Theta W^{\frac{1}{2}}X^{(\ell)\top}\Big) X^{(\ell)} W^{\frac{1}{2}}\Delta \Big]\\
    &\quad\quad\quad+\frac{1}{p^2} \operatorname{tr}\big(W^{\frac{1}{2}}X^{(\ell)\top} X^{(\ell)} W^{\frac{1}{2}}\Delta W^{\frac{1}{2}}X^{(\ell)\top} X^{(\ell)} W^{\frac{1}{2}}\Delta\big).
\end{align*}
By summing over $\ell$ from $L_0+1$ to $L$, it holds that
\begin{align*}
G(\Delta)&=-\frac{2}{p(L-L_{0})}\sum_{\ell=L_0+1}^{L}\operatorname{tr}\Big[W^{\frac{1}{2}}X^{(\ell)\top}\Big(y^{(\ell)} y^{(\ell)\top}-\frac{1}{p} X^{(\ell)} W^{\frac{1}{2}}\Theta W^{\frac{1}{2}}X^{(\ell)\top}\Big) X^{(\ell)} W^{\frac{1}{2}}\Delta \Big]\\
&\quad\quad+\frac{2}{p^2(L-L_0)}\sum_{\ell=L_0+1}^{L} \operatorname{tr}\big(W^{\frac{1}{2}}X^{(\ell)\top} X^{(\ell)} W^{\frac{1}{2}}\Delta W^{\frac{1}{2}} X^{(\ell)\top} X^{(\ell)} W^{\frac{1}{2}}\Delta\big)\\
&\quad\quad+\tilde{\lambda}\big(|\Theta^{-}+\Delta^{-}|_1-|\Theta^{-}|_1\big).
\end{align*}
Now we first consider the linear term
\begin{align*}
-\frac{2}{p(L-L_{0})}\sum_{\ell=1}^{L-L_0}\operatorname{tr}\Big[W^{\frac{1}{2}}X^{(\ell)\top}\Big(y^{(\ell)} y^{(\ell)\top}-\frac{1}{p} X^{(\ell)} W^{\frac{1}{2}}\Theta W^{\frac{1}{2}}X^{(\ell)\top}\Big) X^{(\ell)} W^{\frac{1}{2}}\Delta \Big].
\end{align*}
Here, the matrix $\Delta$ is the difference between some feasible matrix in $\Gamma_{p}^{+}$ and true correlation matrix. Note that in $\Gamma_{p}^{+}$, the diagonal elements are all equal to $1$. Hence, the diagonal elements of $\Delta$ are all zero. Now, define
\begin{align*}
\tilde{A}=&\frac{1}{p(L-L_0)}\sum_{\ell=L_{0}+1}^{L}W^{\frac{1}{2}}X^{(\ell)\top}\Big(y^{(\ell)} y^{(\ell)\top}-\frac{1}{p} X^{(\ell)} W^{\frac{1}{2}}\Theta W^{\frac{1}{2}}X^{(\ell)\top}\Big) X^{(\ell)} W^{\frac{1}{2}}\\
=&\frac{1}{p(L-L_0)}\sum_{\ell=L_{0}+1}^{L} W^{\frac{1}{2}}X^{(\ell)\top}X^{(\ell)}\Big(\bar{\beta}^{(\ell)}\bar{\beta}^{(\ell)\top}-\frac{1}{p}\Omega\Big)X^{(\ell)\top}X^{(\ell)}W^{\frac{1}{2}}. 
\end{align*} 
We first derive a high probability bound for $(i,j)$-th entry of $\tilde{A}$.  
Specifically, we show that there exists an absolute constant $C_{1}>0$ such that with probability tending to $1$,
\begin{align}
\max _{i \neq j}|\tilde{A}_{i j}| \leq C_1 \sqrt{\frac{\log p}{L-L_0}}.\label{bound_for_tildeA_ij}
\end{align}
To see that, note that by definition,
\begin{align*}
    (L-L_0)\tilde{A}_{ij} & =\frac{1}{p}\sum_{\ell=L_{0}+1}^{L} W^{\frac{1}{2}}X^{(\ell)\top}X^{(\ell)}\Big(\bar{\beta}^{(\ell)}\bar{\beta}^{(\ell)\top}-\frac{1}{p}\Omega\Big)X^{(\ell)\top}X^{(\ell)}W^{\frac{1}{2}} \\
    &=\sum_{\ell=L_0+1}^{L}\Big[\Big[\frac{1}{p}W^{\frac{1}{2}}X^{(\ell)\top}X^{(\ell)}\sqrt{p}\bar{\beta}^{(\ell)}\Big]_{i}\Big[\frac{1}{p}W^{\frac{1}{2}}X^{(\ell)\top}X^{(\ell)}\sqrt{p}\bar{\beta}^{(\ell)}\Big]_{j}-\tilde{\xi}_{ij}\Big]
\end{align*}
where $\xi_{ij}=\big(\frac{1}{p^2}W^{\frac{1}{2}}X^{(\ell)\top}X^{(\ell)}\Omega X^{(\ell)\top}X^{(\ell)}W^{\frac{1}{2}}\big)_{ij}$ is the $(i,j)$-th entry of matrix $$\frac{1}{p^2}W^{\frac{1}{2}}X^{(\ell)\top}X^{(\ell)}\Omega X^{(\ell)\top}X^{(\ell)}W^{\frac{1}{2}}.$$ 
Define the event
\begin{align*}
&\tilde{A}_{ij}(t)\coloneqq\{|\tilde{A}_{ij}|>t\}\\
=&\Bigg\{\Big|\sum_{\ell=L_0+1}^{L}\Big[\frac{1}{p}W^{\frac{1}{2}}X^{(\ell)\top}X^{(\ell)}\sqrt{p}\bar{\beta}^{(\ell)}\Big]_{i}\Big[\frac{1}{p}W^{\frac{1}{2}}X^{(\ell)\top}X^{(\ell)}\sqrt{p}\bar{\beta}^{(\ell)}\Big]_{j}-\tilde{\xi}_{ij}\Big|>(L-L_0)t\Bigg\}.
\end{align*}
Similarly to the previous proof, to derive the high probability bound for $\tilde{A}_{ij}(t)$, we proceed this first by decoupling the product $$\big[\frac{1}{p}W^{\frac{1}{2}}X^{(\ell)\top}X^{(\ell)}\sqrt{p}\bar{\beta}^{(\ell)}\big]_{i}\big[\frac{1}{p}W^{\frac{1}{2}}X^{(\ell)\top}X^{(\ell)}\sqrt{p}\bar{\beta}^{(\ell)}\big]_{j}.$$ 
Define the random variables
\begin{align*}
    \tilde{U}_{ij}^{(\ell)} &=\Big[\frac{1}{p}W^{\frac{1}{2}}X^{(\ell)\top}X^{(\ell)}\sqrt{p}\bar{\beta}^{(\ell)}\Big]_{i}+\Big[\frac{1}{p}W^{\frac{1}{2}}X^{(\ell)\top}X^{(\ell)}\sqrt{p}\bar{\beta}^{(\ell)}\Big]_{j}\\
    \tilde{V}_{ij}^{(\ell)} &=\Big[\frac{1}{p}W^{\frac{1}{2}}X^{(\ell)\top}X^{(\ell)}\sqrt{p}\bar{\beta}^{(\ell)}\Big]_{i}-\Big[\frac{1}{p}W^{\frac{1}{2}}X^{(\ell)\top}X^{(\ell)}\sqrt{p}\bar{\beta}^{(\ell)}\Big]_{j},
\end{align*}
whose second moments are given by
$\mathbb{E}\big[\big(\tilde{U}_{ij}^{(\ell)}\big)^{2}\big]=\tilde{u}_{ij}^{(\ell)}$ and $\mathbb{E}\big[\big(\tilde{V}_{ij}^{(\ell)}\big)^{2}\big]=\tilde{v}_{ij}^{(\ell)}$. Then, we have that  
\begin{align*}
&\sum_{\ell=L_0+1}^{L}\Big[\frac{1}{p}W^{\frac{1}{2}}X^{(\ell)\top}X^{(\ell)}\sqrt{p}\bar{\beta}^{(\ell)}\Big]_{i}\Big[\frac{1}{p}W^{\frac{1}{2}}X^{(\ell)\top}X^{(\ell)}\sqrt{p}\bar{\beta}^{(\ell)}\Big]_{j}-\tilde{\xi}_{ij}\\
=&\frac{1}{4}\sum_{\ell=L_0+1}^{L}\Big[\tilde{U}_{ij}^{(\ell)2}-\tilde{u}_{ij}^{(\ell)}\Big]-\frac{1}{4}\sum_{\ell=L_0+1}^{L}\Big[\tilde{V}_{ij}^{(\ell)2}-v_{ij}^{(\ell)}\Big].
\end{align*}
Therefore, it holds that
\begin{align*}
    \mathbb{P}(\tilde{A}_{ij}(t))\leq & \mathbb{P}\Big(\sum_{\ell=L_0+1}^{L}\big[\tilde{U}_{ij}^{(\ell)2}-\tilde{u}_{ij}^{(\ell)}\big]\geq 2(L-L_0)t\Big)\\
    &+\mathbb{P}\Big(\sum_{\ell=L_0+1}^{L}\big[\tilde{V}_{ij}^{(\ell)2}-v_{ij}^{(\ell)}\big]\geq 2(L-L_0)t\Big).
\end{align*}
The random variables $\tilde{U}_{ij}^{(\ell)}$ and $\tilde{V}_{ij}^{(\ell)}$ are sub-Gaussian with parameter at most
\begin{align*}
\tilde{\sigma}=2\Big(\Big\|\frac{1}{p}W^{\frac{1}{2}}X^{(\ell)\top}X^{(\ell)}\Big\|\tau_{\beta}+\Big\|\frac{1}{\sqrt{p}}W^{\frac{1}{2}}X^{(\ell)\top}\Big\|\tau_{\varepsilon}\Big)=\mathcal{O}(1).
\end{align*}
Next, with similar arguments as in the previous proof, $\tilde{U}_{i j}^{(\ell)2}-\tilde{u}_{i j}^{(\ell)}$ is sub-exponential with parameter $(16 \tilde{\sigma}^2+2\tilde{u}_{i j}^{(\ell)}, 16 \tilde{\sigma}^2+2\tilde{u}_{i j}^{(\ell)})$. Similarly, random variable $\tilde{V}_{i j}^{(\ell)2}-\tilde{v}_{i j}^{(\ell)}$ is sub-exponential with parameter $(16 \tilde{\sigma}^2+2\tilde{v}_{i j}^{(\ell)}, 16 \tilde{\sigma}^2+2\tilde{v}_{i j}^{(\ell)})$.

Now we are ready to derive the bound for $\mathbb{P}(\tilde{A}_{ij}(t))$. We have that
    \begin{align*}
         \mathbb{P}(\tilde{A}_{i j}(t)) \leq &\mathbb{P}\Big(\sum_{\ell=L_0+1}^{L}\big[\tilde{U}_{i j}^{(\ell)2}-\tilde{u}_{i j}^{(\ell)}\big] \geq 2 (L-L_0) t\Big)\\
         &+\mathbb{P}\Big(\sum_{\ell=L_0+1}^L\big[\tilde{V}_{i j}^{(\ell)2}-\tilde{v}_{i j}^{(\ell)}\big] \geq 2 (L-L_0) t\Big)\\
         \leq & 4\exp\big\{-\mathcal{O}((L-L_{0})t^2)\big\}.
    \end{align*}
Therefore, by taking $t=C_{1}\sqrt{\frac{\log p}{L-L_0}}$ with sufficient large absolute constant $C_{1}$, we obtain \eqref{bound_for_tildeA_ij} holds with probability tending to $1$.

Now we want to find a lower bound for the quadratic term. In particular, under Assumption \ref{fixed_design_asp1}, the quadratic term can be lower bounded in terms of $\|\Delta^{-}\|_{F}^{2}$ as  
\begin{align*}
    &\operatorname{tr}\big(W^{\frac{1}{2}}\Delta W^{\frac{1}{2}} X^{(\ell)\top} X^{(\ell)} W^{\frac{1}{2}}\Delta W^{\frac{1}{2}}X^{(\ell)\top} X^{(\ell)}\big)\\
    =& \operatorname{vec}^{\top}\big(W^{\frac{1}{2}}\Delta W^{\frac{1}{2}}\big) \big(X^{(\ell)\top} X^{(\ell)} \otimes X^{(\ell)\top} X^{(\ell)}\big) \operatorname{vec}\big(W^{\frac{1}{2}}\Delta W^{\frac{1}{2}}\big)\\ 
    =&\operatorname{vec}^{\top}\big(W^{\frac{1}{2}}\Delta W^{\frac{1}{2}}\big) \big(X^{(\ell)}\otimes X^{(\ell)} \big)^{\top}  \big( X^{(\ell)}\otimes X^{(\ell)}\big) \operatorname{vec}\big(W^{\frac{1}{2}}\Delta W^{\frac{1}{2}}\big)\\
    =&\big\|\big( X^{(\ell)}\otimes X^{(\ell)}\big) \operatorname{vec}\big(W^{\frac{1}{2}}\Delta W^{\frac{1}{2}}\big)\big\|_{2}^{2}.
\end{align*}
Taking average over $\ell=L_0+1,\dots,L$, we see that
\begin{align*}
    &\frac{1}{p^2(L-L_0)}\sum_{\ell=L_0+1}^{L}\operatorname{tr}\big(W^{\frac{1}{2}}\Delta W^{\frac{1}{2}} X^{(\ell)\top} X^{(\ell)} W^{\frac{1}{2}}\Delta W^{\frac{1}{2}}X^{(\ell)\top} X^{(\ell)}\big)\\
    =&\frac{1}{p^2(L-L_0)}\sum_{\ell=L_0+1}^{L}\big\|\big( X^{(\ell)}\otimes X^{(\ell)}\big) \operatorname{vec}\big(W^{\frac{1}{2}}\Delta W^{\frac{1}{2}}\big)\big\|_{2}^{2}\\
    \geq& \frac{1}{(L-L_0)}\sum_{\ell=L_{0}+1}^{L}\kappa_{0}^{(\ell)}\big\|W^{\frac{1}{2}}\Delta W^{\frac{1}{2}}\big\|_{F}^{2}\\
    \geq &\min_{\ell=L_0+1,\dots,L}\kappa_{0}^{(\ell)}\big\|W^{\frac{1}{2}}\Delta W^{\frac{1}{2}}\big\|_{F}^{2}.
\end{align*}
With a same computation as in the proof of Theorem \ref{thm1_sparse_cov_est}, it holds that
\begin{align*}
   \big\|\tilde{\Theta}_{\lambda}-\Theta\Big\|_{F}=\mathcal{O}_{p}\big(\sqrt{\frac{s\log p}{L-L_0}}\Big).
\end{align*} 
On the other hand
$$
\big\|\hat{W}-W\big\|=\mathcal{O}_{p}\Big(\sqrt{\frac{\log p}{L_0}}\Big).
$$ 

In the next step, we analyze what happens if we replace $W$ by $\hat{W}$, where our estimator is given by
\begin{equation*}
\hat{\Theta}_{\lambda}=\arg\min_{\tilde{\Theta}\in\Gamma_{+}^{p}}\Big\{\frac{1}{L-L_0}\sum_{\ell=L_{0}+1}^{L}\Big\|y^{(\ell)}y^{(\ell)\top}-\frac{1}{p}X^{(\ell)}\hat{W}^{\frac{1}{2}}\tilde{\Theta}\hat{W}^{\frac{1}{2}}X^{(\ell)\top}\Big\|_{F}^{2}+\tilde{\lambda}\sum_{i\not=j}|\tilde{\Theta}_{ij}|\Big\}.
\end{equation*}
In this case, the linear term would be 
\begin{align*}
&-\frac{2}{p(L-L_{0})}\sum_{\ell=L_0+1}^{L}\operatorname{tr}\Big[\hat{W}^{\frac{1}{2}}X^{(\ell)\top}\Big(y^{(\ell)} y^{(\ell)\top}-\frac{1}{p} X^{(\ell)} \hat{W}^{\frac{1}{2}}\Theta \hat{W}^{\frac{1}{2}}X^{(\ell)\top}\Big) X^{(\ell)} \hat{W}^{\frac{1}{2}}\Delta \Big]\\
=&-2\operatorname{tr}(\hat{A}\Delta),
\end{align*}
where 
\begin{align*}
\hat{A}&=\frac{1}{p(L-L_0)}\sum_{\ell=L_0+1}^{L}\hat{W}^{\frac{1}{2}}X^{(\ell)\top}\Big(y^{(\ell)} y^{(\ell)\top}-\frac{1}{p} X^{(\ell)} \hat{W}^{\frac{1}{2}}\Theta \hat{W}^{\frac{1}{2}}X^{(\ell)\top}\Big) X^{(\ell)} \hat{W}^{\frac{1}{2}}\\
&=\frac{1}{L-L_0}\sum_{\ell=L_0+1}^{L}\hat{W}^{\frac{1}{2}}X^{(\ell)\top}X^{(\ell)}\Big(\bar{\beta}^{(\ell)}\bar{\beta}^{(\ell)\top}-\frac{1}{p} \hat{W}^{\frac{1}{2}}\Theta \hat{W}^{\frac{1}{2}}\Big) X^{(\ell)\top}X^{(\ell)} \hat{W}^{\frac{1}{2}}.
\end{align*} 
In order to upper bound $\operatorname{tr}(\hat{A}\Delta)$,
we want to estimate the difference between $\tilde{A}_{ij}$ and $\hat{A}_{ij}$. We have that,
\begin{align*}
\big|\operatorname{tr}(\hat{A}\Delta)\big|&=\big|\operatorname{tr}(\tilde{A}\Delta)+\operatorname{tr}((\hat{A}-\tilde{A})\Delta)\big|\\
&\leq \big|\operatorname{tr}(\tilde{A}\Delta)\big|+\big|\operatorname{tr}((\hat{A}-\tilde{A})\Delta)\big|\\
&\leq \Big|\sum_{i\not=j}\tilde{A}_{ij}\Delta_{ij}\Big|+\Big|\sum_{i\not=j}(\hat{A}_{ij}-\tilde{A}_{ij})\Delta_{ij}\Big|.
\end{align*} 
For the first term, with probability tending to 1, $\max_{i\not=j}|\tilde{A}_{ij}|\leq C\sqrt{\frac{\log p}{L-L_0}}$. Note that
\begin{align*}
    &\hat{A}-\tilde{A}\\
    =&\frac{1}{p(L-L_0)}\sum_{\ell=L_0+1}^{L}\hat{W}^{\frac{1}{2}}X^{(\ell)\top}X^{(\ell)}\Big(\bar{\beta}^{(\ell)}\bar{\beta}^{(\ell)\top}-\frac{1}{p} \hat{W}^{\frac{1}{2}}\Theta \hat{W}^{\frac{1}{2}}\Big) X^{(\ell)\top}X^{(\ell)} \hat{W}^{\frac{1}{2}}\\
    &-\frac{1}{p(L-L_0)}\sum_{\ell=L_0+1}^{L}W^{\frac{1}{2}}X^{(\ell)\top}X^{(\ell)}\Big(\bar{\beta}^{(\ell)}\bar{\beta}^{(\ell)\top}-\frac{1}{p} W^{\frac{1}{2}}\Theta W^{\frac{1}{2}}\Big) X^{(\ell)\top}X^{(\ell)} W^{\frac{1}{2}}\\
    =&\frac{1}{p(L-L_0)}\sum_{\ell=L_0+1}^{L}(\hat{W}^{\frac{1}{2}}-W^{\frac{1}{2}})X^{(\ell)\top}X^{(\ell)}\Big(\bar{\beta}^{(\ell)}\bar{\beta}^{(\ell)\top}-\frac{1}{p} \hat{W}^{\frac{1}{2}}\Theta \hat{W}^{\frac{1}{2}}\Big) X^{(\ell)\top}X^{(\ell)} \hat{W}^{\frac{1}{2}}\\
    &+\frac{1}{p(L-L_0)}\sum_{\ell=L_0+1}^{L}W^{\frac{1}{2}}X^{(\ell)\top}X^{(\ell)}\Big(\frac{1}{p} W^{\frac{1}{2}}\Theta W^{\frac{1}{2}}-\frac{1}{p} \hat{W}^{\frac{1}{2}}\Theta \hat{W}^{\frac{1}{2}}\Big) X^{(\ell)\top}X^{(\ell)} \hat{W}^{\frac{1}{2}}\\
    &+\frac{1}{p(L-L_0)}\sum_{\ell=L_0+1}^{L}W^{\frac{1}{2}}X^{(\ell)\top}X^{(\ell)}\Big(\bar{\beta}^{(\ell)}\bar{\beta}^{(\ell)\top}-\frac{1}{p} W^{\frac{1}{2}}\Theta W^{\frac{1}{2}}\Big) X^{(\ell)\top}X^{(\ell)} (\hat{W}^{\frac{1}{2}}-W^{\frac{1}{2}})\\
    =&\frac{1}{p(L-L_0)}\sum_{\ell=L_0+1}^{L}(\hat{W}^{\frac{1}{2}}-W^{\frac{1}{2}})X^{(\ell)\top}X^{(\ell)}\Big(\bar{\beta}^{(\ell)}\bar{\beta}^{(\ell)\top}-\frac{1}{p} W^{\frac{1}{2}}\Theta W^{\frac{1}{2}}\Big) X^{(\ell)\top}X^{(\ell)} \hat{W}^{\frac{1}{2}}\\
    &+\frac{1}{p(L-L_0)}\sum_{\ell=L_0+1}^{L}(\hat{W}^{\frac{1}{2}}-W^{\frac{1}{2}})X^{(\ell)\top}X^{(\ell)}\Big(\frac{1}{p} W^{\frac{1}{2}}\Theta W^{\frac{1}{2}}-\frac{1}{p} \hat{W}^{\frac{1}{2}}\Theta \hat{W}^{\frac{1}{2}}\Big) X^{(\ell)\top}X^{(\ell)} \hat{W}^{\frac{1}{2}}\\
    &+\frac{1}{p(L-L_0)}\sum_{\ell=L_0+1}^{L}W^{\frac{1}{2}}X^{(\ell)\top}X^{(\ell)}\Big(\frac{1}{p} W^{\frac{1}{2}}\Theta W^{\frac{1}{2}}-\frac{1}{p} \hat{W}^{\frac{1}{2}}\Theta \hat{W}^{\frac{1}{2}}\Big) X^{(\ell)\top}X^{(\ell)} \hat{W}^{\frac{1}{2}}\\
    &+\frac{1}{p(L-L_0)}\sum_{\ell=L_0+1}^{L}W^{\frac{1}{2}}X^{(\ell)\top}X^{(\ell)}\Big(\bar{\beta}^{(\ell)}\bar{\beta}^{(\ell)\top}-\frac{1}{p} W^{\frac{1}{2}}\Theta W^{\frac{1}{2}}\Big) X^{(\ell)\top}X^{(\ell)} (\hat{W}^{\frac{1}{2}}-W^{\frac{1}{2}})\\
    =&\frac{1}{p(L-L_0)}\sum_{\ell=L_0+1}^{L}(\hat{W}^{\frac{1}{2}}-W^{\frac{1}{2}})X^{(\ell)\top}X^{(\ell)}\Big(\bar{\beta}^{(\ell)}\bar{\beta}^{(\ell)\top}-\frac{1}{p} W^{\frac{1}{2}}\Theta W^{\frac{1}{2}}\Big) X^{(\ell)\top}X^{(\ell)} \hat{W}^{\frac{1}{2}}\\
    &+\frac{1}{p(L-L_0)}\sum_{\ell=L_0+1}^{L}\hat{W}^{\frac{1}{2}}X^{(\ell)\top}X^{(\ell)}\Big(\frac{1}{p} W^{\frac{1}{2}}\Theta W^{\frac{1}{2}}-\frac{1}{p} \hat{W}^{\frac{1}{2}}\Theta \hat{W}^{\frac{1}{2}}\Big) X^{(\ell)\top}X^{(\ell)} \hat{W}^{\frac{1}{2}}\\
    &+\frac{1}{p(L-L_0)}\sum_{\ell=L_0+1}^{L}W^{\frac{1}{2}}X^{(\ell)\top}X^{(\ell)}\Big(\bar{\beta}^{(\ell)}\bar{\beta}^{(\ell)\top}-\frac{1}{p} W^{\frac{1}{2}}\Theta W^{\frac{1}{2}}\Big) X^{(\ell)\top}X^{(\ell)} (\hat{W}^{\frac{1}{2}}-W^{\frac{1}{2}})\\
    =&(i)+(ii)+(iii).
\end{align*}

For the term $(ii)$, note that
\begin{align*}
    \hat{W}^{\frac{1}{2}} \Theta \hat{W}^{\frac{1}{2}}-W^{\frac{1}{2}} \Theta W^{\frac{1}{2}}&=\hat{W}^{\frac{1}{2}} \Theta \hat{W}^{\frac{1}{2}}-W^{\frac{1}{2}} \Theta \hat{W}^{\frac{1}{2}}+W^{\frac{1}{2}}\Theta \hat{W}^{\frac{1}{2}}-W^{\frac{1}{2}} \Theta W^{\frac{1}{2}} \\
    &=(\hat{W}^{\frac{1}{2}}-W^{\frac{1}{2}})\Theta \hat{W}^{\frac{1}{2}}+W^{\frac{1}{2}}\Theta(\hat{W}^{\frac{1}{2}}-W^{\frac{1}{2}}).
\end{align*}
Therefore,
\begin{align*}
    (ii) &=\frac{1}{p^2(L-L_0)}\sum_{\ell=L_0+1}^{L}\hat{W}^{\frac{1}{2}}X^{(\ell)\top}X^{(\ell)}(\hat{W}^{\frac{1}{2}}-W^{\frac{1}{2}})\Theta \hat{W}^{\frac{1}{2}}X^{(\ell)\top}X^{(\ell)} \hat{W}^{\frac{1}{2}}\\
    &\quad\quad\quad+\frac{1}{p^2(L-L_0)}\sum_{\ell=L_0+1}^{L}\hat{W}^{\frac{1}{2}}X^{(\ell)\top}X^{(\ell)}W^{\frac{1}{2}}\Theta(\hat{W}^{\frac{1}{2}}-W^{\frac{1}{2}}) X^{(\ell)\top}X^{(\ell)} \hat{W}^{\frac{1}{2}}.
\end{align*}
Since the $(i,j)$-th entry of a matrix is upper bounded by its operator norm, we could bound these terms by following procedure
\begin{align*}
    [(i)]_{ij} &\leq \|(i)\| \leq \big\|\hat{W}^{\frac{1}{2}}\big\|\big\|\hat{W}^{\frac{1}{2}}-W^{\frac{1}{2}}\big\| \Big[\frac{1}{p(L-L_0)}\sum_{\ell=L_0+1}^{L}X^{(\ell)\top}X^{(\ell)}\\
    &\quad\quad\quad\quad\quad\quad\quad\quad\quad\quad\quad\quad\quad\quad\quad\Big(\bar{\beta}^{(\ell)}\bar{\beta}^{(\ell)\top}-\frac{1}{p}W^{\frac{1}{2}}\Theta W^{\frac{1}{2}}\Big)X^{(\ell)\top}X^{(\ell)}\Big]_{ij}\\
    [(ii)]_{ij}&\leq \|(ii)\|\leq \frac{1}{p^2(L-L_0)}\sum_{\ell=L_{0}+1}^{L}\big\|X^{(\ell)\top}X^{(\ell)}\big\|^{2}\big\|\hat{W}^{\frac{1}{2}}\big\|^2\\
    &\quad\quad\quad\quad\quad\quad\quad\quad\quad\quad\quad\quad\quad\quad\quad\big\|\Theta\big\|\big\|\hat{W}^{\frac{1}{2}}-W^{\frac{1}{2}}\big\|\big(\big\|W^{\frac{1}{2}}\big\|+\big\|\hat{W}^{\frac{1}{2}}\big\|\big)\\
    &=\frac{n_{\ell}^{2}}{p^2}\max_{\ell=L_{0}+1,\dots,L}\Big\|\frac{1}{n_{\ell}}X^{(\ell)\top}X^{(\ell)}\Big\|^{2}\Big\|\hat{W}^{\frac{1}{2}}\Big\|^2\Big\|\Theta\Big\|\Big\|\hat{W}^{\frac{1}{2}}-W^{\frac{1}{2}}\Big\|\Big(\Big\|W^{\frac{1}{2}}\Big\|+\Big\|\hat{W}^{\frac{1}{2}}\Big\|\Big),\\
    [(iii)]_{ij}&\leq \|(iii)\| \leq \big\|W^{\frac{1}{2}}\big\|\big\|\hat{W}^{\frac{1}{2}}-W^{\frac{1}{2}}\big\|\Big[\frac{1}{p(L-L_0)}\sum_{\ell=L_0+1}^{L}X^{(\ell)\top}X^{(\ell)}\\
    &\quad\quad\quad\quad\quad\quad\quad\quad\quad\quad\quad\quad\quad\quad\quad\Big(\bar{\beta}^{(\ell)}\bar{\beta}^{(\ell)\top}-\frac{1}{p}W^{\frac{1}{2}}\Theta W^{\frac{1}{2}}\Big)X^{(\ell)\top}X^{(\ell)}\Big]_{ij} 
\end{align*}
Since the square root function $\sqrt{x}$ is Lipshitz when $x$ is bounded away from zero, then it holds that
$$\big\|\hat{W}^{\frac{1}{2}}-W^{\frac{1}{2}}\big\|\leq C\big\|\hat{W}-W\big\|=\mathcal{O}_{P}\Bigg(\sqrt{\frac{\log p}{L_{0}}}\Bigg).$$
Besides, $\Big\|\frac{1}{n_{\ell}}X^{(\ell)\top}X^{(\ell)}\Big\|=\mathcal{O}(1)$. Also, by Assumption on $\Omega$, $\|W^{\frac{1}{2}}\|=\mathcal{O}(1)$ and $\|\Theta\|=\mathcal{O}(1)$. Therefore, as $n_{\ell},p\rightarrow\infty$
\begin{align*} 
[(ii)]_{ij}\leq\mathcal{O}_{P}\Bigg(\sqrt{\frac{\log p}{L_{0}}}\Bigg).
\end{align*}
Similar to previous case, it holds that
\begin{align*}
&\Big[\frac{1}{p(L-L_0)}\sum_{\ell=L_0+1}^{L}X^{(\ell)\top}X^{(\ell)}\Big(\bar{\beta}^{(\ell)}\bar{\beta}^{(\ell)\top}-\frac{1}{p}W^{\frac{1}{2}}\Theta W^{\frac{1}{2}}\Big)X^{(\ell)\top}X^{(\ell)}\Big]_{ij}\\
=&\mathcal{O}_{P}\Bigg(\sqrt{\frac{\log p}{L-L_0}}\Bigg),
\end{align*}
which implies 
\begin{align*}
[(i)]_{ij}+[(iii)]_{ij}=\mathcal{O}_{P}\Bigg(\sqrt{\frac{\log p}{L-L_0}}\sqrt{\frac{\log p}{L_0}}\Bigg).
\end{align*}
Therefore, it holds that
$$|\hat{A}_{ij}-\tilde{A}_{ij}|=\mathcal{O}_{P}\Bigg(\sqrt{\frac{\log p}{L_{0}}}+\sqrt{\frac{\log p}{L-L_0}}\sqrt{\frac{\log p}{L_0}}\Bigg),$$
and
\begin{align*}
|\operatorname{tr}(\hat{A}\Delta)|&\leq |\operatorname{tr}(\tilde{A}\Delta)|+|\operatorname{tr}((\hat{A}-\tilde{A})\Delta)|\\
&\leq \Big|\sum_{i\not=j}\tilde{A}_{ij}\Delta_{ij}\Big|+\Big|\sum_{i\not=j}(\hat{A}_{ij}-\tilde{A}_{ij})\Delta_{ij}\Big|\\
&\leq C_{1} \Bigg(\sqrt{\frac{\log p}{L-L_0}}+\sqrt{\frac{\log p}{L_{0}}}+\sqrt{\frac{\log p}{L-L_0}}\sqrt{\frac{\log p}{L_0}}\Bigg)|\Delta^{-}|_{1},
\end{align*} 
for some absolute constant $C_{1}$. 

Now by Assumption \ref{fixed_design_asp1}, we have that
\begin{align*}
G(\Delta)&=-\frac{2}{p(L-L_{0})}\sum_{\ell=L_0+1}^{L}\operatorname{tr}\Big[\hat{W}^{\frac{1}{2}}X^{(\ell)\top}\Big(y^{(\ell)} y^{(\ell)\top}-\frac{1}{p} X^{(\ell)} \hat{W}^{\frac{1}{2}}\Theta \hat{W}^{\frac{1}{2}}X^{(\ell)\top}\Big) X^{(\ell)} \hat{W}^{\frac{1}{2}}\Delta \Big]\\
&\quad\quad+\frac{2}{p^2(L-L_0)}\sum_{\ell=L_0+1}^{L} \operatorname{tr}\big(\hat{W}^{\frac{1}{2}}X^{(\ell)\top} X^{(\ell)} \hat{W}^{\frac{1}{2}}\Delta \hat{W}^{\frac{1}{2}} X^{(\ell)\top} X^{(\ell)} \hat{W}^{\frac{1}{2}}\Delta\big)\\
&\quad\quad+\tilde{\lambda}(|\Theta^{-}+\Delta^{-}|_1-|\Theta^{-}|_1)\\
&\geq C\min_{\ell=L_0+1,\dots,L}\kappa_{0}^{(\ell)}\|\Delta\|_{F}^{2}\\
&\quad\quad -C_{1}\Bigg(\sqrt{\frac{\log p}{L-L_0}}+\sqrt{\frac{\log p}{L_{0}}}+\sqrt{\frac{\log p}{L-L_0}}\sqrt{\frac{\log p}{L_0}}\Bigg)|\Delta^{-}|_{1}+\tilde{\lambda}(|\Delta_{S^c}^{-}|_1-|\Delta_S^{-}|_1)\\
&\geq C\min_{\ell=L_0+1,\dots,L}\kappa_{0}^{(\ell)}\|\Delta\|_{F}^{2}\\
&\quad\quad-\Bigg[C_{1}\Big(\sqrt{\frac{\log p}{L-L_0}}+\sqrt{\frac{\log p}{L_{0}}}+\sqrt{\frac{\log p}{L-L_0}}\sqrt{\frac{\log p}{L_0}}\Big)-\tilde{\lambda}\Bigg]|\Delta_{S^{c}}^{-}|_{1}\\
&\quad\quad-\Bigg[C_{1}\Bigg(\sqrt{\frac{\log p}{L-L_0}}+\sqrt{\frac{\log p}{L_{0}}}+\sqrt{\frac{\log p}{L-L_0}}\sqrt{\frac{\log p}{L_0}}\Bigg)+\tilde{\lambda}\Bigg]|\Delta_{S}^{-}|_{1}.
\end{align*}
Set $\tilde{\lambda}=2C_{1}\Big(\sqrt{\frac{\log p}{L-L_0}}+\sqrt{\frac{\log p}{L_{0}}}+\sqrt{\frac{\log p}{L-L_0}}\sqrt{\frac{\log p}{L_0}}\Big)$. Then, by inequality $$|\Delta_S^{-}|_1 \leq \sqrt{s}\|\Delta_S^{-}\|_F \leq \sqrt{s}\|\Delta^{-}\|_F,$$ we have that 
\begin{align*}
    G(\Delta) &\geq C\min_{\ell=L_0+1,\dots,L}\kappa_{0}^{(\ell)}\|\Delta^{-}\|_{F}^{2}+C\min_{\ell=L_0+1,\dots,L}\kappa_{0}^{(\ell)}\|\Delta^{+}\|_{F}^{2} \\
    &\qquad- 3C_{1}\Bigg(\sqrt{\frac{\log p}{L-L_0}}+\sqrt{\frac{\log p}{L_{0}}}+\sqrt{\frac{\log p}{L-L_0}}\sqrt{\frac{\log p}{L_0}}\Bigg) |\Delta_{S}^{-}|_{1}\\
    &\geq \|\Delta^{-}\|_{F}^{2}\Bigg[C\min_{\ell=L_0+1,\dots,L}\kappa_{0}^{(\ell)}\\
    &\qquad-3C_{1}\Bigg(\sqrt{\frac{s\log p}{L-L_0}}+\sqrt{\frac{s\log p}{L_{0}}}+\sqrt{\frac{s\log p}{L-L_0}}\sqrt{\frac{\log p}{L_0}}\Bigg)\|\Delta^{-}\|_{F}^{-1}\Bigg]\\
    &\qquad+C\min_{\ell=L_0+1,\dots,L}\kappa_{0}^{(\ell)}\|\Delta^{+}\|_{F}^{2}\\
    &\geq \|\Delta^{-}\|_{F}^{2}\Bigg[C\min_{\ell=L_0+1,\dots,L}\kappa_{0}^{(\ell)}-\frac{3C_{1}}{Mr(p,L)}\Bigg(\sqrt{\frac{s\log p}{L-L_0}}+\sqrt{\frac{s\log p}{L_{0}}}+\sqrt{\frac{s(\log p)^2}{L_0(L-L_0)}}\Bigg)\Bigg].
\end{align*}
Now as long as $\Big(\sqrt{\frac{s\log p}{L-L_0}}+\sqrt{\frac{s\log p}{L_0}}+\sqrt{\frac{s(\log p)^2}{L_0(L-L_0)}}\Big)\rightarrow 0$, it holds that
\begin{align*}
\|\hat{\Theta}_{\lambda}-\Theta\|_{F}=\mathcal{O}_{P}\Bigg(\sqrt{\frac{s\log p}{L-L_0}}+\sqrt{\frac{s\log p}{L_0}}+\sqrt{\frac{s(\log p)^2}{L_0(L-L_0)}}\Bigg).
\end{align*}
Furthermore, as
\begin{align*}
\|\hat{\Omega}_{w}-\Omega\|\leq &\|\hat{\Theta}-\Theta\|\|\hat{W}^{\frac{1}{2}}-W^{\frac{1}{2}}\|^2+\|\hat{W}^{\frac{1}{2}}-W^{\frac{1}{2}}\|(\|\hat{\Theta}\|\|W^{\frac{1}{2}}\|+\|\Theta\|\|\hat{W}^{\frac{1}{2}}\|)\\
&\qquad+\|W^{\frac{1}{2}}\|\|\hat{\Theta}-\Theta\|\|\hat{W}^{\frac{1}{2}}\|,
\end{align*}
the stochastic order of $\|\hat{\Omega}_{w}-\Omega\|$ is given by
\begin{align*}
    &\Bigg(\sqrt{\frac{s\log p}{L-L_0}}+\sqrt{\frac{s\log p}{L_0}}+\sqrt{\frac{s(\log p)^2}{L_0(L-L_0)}}\Bigg)\frac{\log p}{L_0}+\sqrt{\frac{\log p}{L_0}}\\
    &\qquad+\Bigg(\sqrt{\frac{s\log p}{L-L_0}}+\sqrt{\frac{s\log p}{L_0}}+\sqrt{\frac{s(\log p)^2}{L_0(L-L_0)}}\Bigg).
\end{align*}
\end{proof}

\subsection{Proof of Theorem \ref{sparsethm_lowerbound}}
The main tool to prove Theorem \ref{sparsethm_lowerbound} is the concentration inequality stated in Theorem \ref{thm3.4adamczak2015concentration}. Below, we start with some preliminary results.

\begin{lemma}\label{Lemma_Lq_norm}
Suppose that $\{x_{i}^{(\ell)}\}_{i=1}^{n_{\ell}}$ are i.i.d. sub-Gaussian random vectors with parameter $\tau_{x}$ and define $\hat{\Sigma}^{(\ell)}=\frac{1}{n_{\ell}}\sum_{\ell=1}^{n_{\ell}}x_{i}^{(\ell)}x_{i}^{(\ell)\top}$, then
\begin{align*}
    \mathbb{E} \lambda_{\max}^{3 q}(\hat{\Sigma}^{(\ell)}) &\leq 3q 2^{3 q-2}\Big[c_2 \frac{(\tau_{x}^{2})^{3 q}}{3 q}+c_2\Big(\frac{\tau_{x}^{2}}{c_3 n_\ell}\Big)^{3 q} \Gamma(3 q)\\
    &\qquad\qquad\qquad+b_{1}^{3 q-1}\Big(1+\frac{1}{c_3 n_\ell}\Big) c_2 \tau_{x}^{2}\Big]+b_{1}^{3q}, \\
    \big\|\lambda_{\max}^{3}(\hat{\Sigma}^{(\ell)})\big\|_{L_{q}}&\leq \Big[ 3q 2^{3 q-2}\Big[c_2 \frac{(\tau_{x}^{2})^{3 q}}{3 q}+c_2\big(\frac{\tau_{x}^{2}}{c_3 n_\ell}\big)^{3 q} \Gamma(3 q)\\
    &\qquad\qquad\qquad+b_{1}^{3 q-1}\big(1+\frac{1}{c_3 n_\ell}\big) c_2 \tau_{x}^{2}\Big]+b_{1}^{3q}\Big]^{\frac{1}{q}},\\
    \Big\|\lambda_{\max}^{2}\big(x_{k}^{(\ell)}x_{k}^{(\ell)\top}\big)\Big\|_{L_{2q}}&\leq \Big[ 4q 2^{4 q-2}\Big[c_2 \frac{(\tau_{x}^{2})^{4 q}}{4 q}+c_2\big(\frac{\tau_{x}^{2}}{c_3}\big)^{4 q} \Gamma(4 q)\\
    &\qquad\qquad\qquad+b_{2}^{4 q-1}\big(1+\frac{1}{c_3}\big) c_2 \tau_{x}^{2}\Big]+b_{2}^{4q}\Big]^{\frac{1}{2q}},  
\end{align*} 
where
\begin{align*}
  b_{1}&\coloneqq c_1\Big(\sqrt{\frac{p}{n_\ell}}+\frac{p}{n_\ell}\Big) \tau_{x}^{2}+\lambda_{\max }(\Sigma^{(\ell)}),\\
  b_{2}&\coloneqq c_1(\sqrt{p}+p) \tau_{x}^{2}+\lambda_{\max }(\Sigma^{(\ell)}).
\end{align*}
\end{lemma}
\begin{proof}
We first bound the quantity $\mathbb{E} \lambda_{\max }^{3 q}(\hat{\Sigma}^{(\ell)})$. With $$b_{1}\coloneqq c_1\Big(\sqrt{\frac{p}{n_\ell}}+\frac{p}{n_\ell}\Big) \tau_{x}^{2}+\lambda_{\max }(\Sigma^{(\ell)}),$$ it holds that
\begin{align*}
    \mathbb{E} \lambda_{\max}^{3 q}(\hat{\Sigma}^{(\ell)})
    &=\int_0^{+\infty} 3 q t^{3 q-1} \mathbb{P}\big(\lambda_{\max }(\hat{\Sigma}^{(\ell)}) \geq t\big) d t\\
    &=\Big(\int_0^{b_{1}}+\int_{b_1}^{+\infty}\Big) 3 q t^{3 q-1} \mathbb{P}\big(\lambda_{\max }(\hat{\Sigma}^{(\ell)}) \geq t\big) d t\quad\quad\quad \\
    &\leq \int_{0}^{b_1}3qt^{3q-1}dt+\int_{b_{1}}^{+\infty} 3 q t^{3 q-1} c_2 \exp  \{-c_3 n_\ell\min\{\delta(t), \delta^2(t)\} \}dt\\
    &=\Big(c_1\Big(\sqrt{\frac{p}{n_\ell}}+\frac{p}{n_\ell}\Big) \tau_{x}^{2}+\lambda_{\max }(\Sigma^{(\ell)})\Big)^{3q}\\
    &\quad\quad+\int_0^{+\infty} 3 q\Big\{\Big[c_1\Big(\sqrt{\frac{p}{n_\ell}}+\frac{p}{n_\ell}\Big)+\delta\Big] \tau_{x}^{2}+\lambda_{\max}(\Sigma^{(\ell)})\Big\}^{3 q-1}\\
    &\qquad\qquad\qquad\qquad \cdot c_2 \exp \{-c_3 n_\ell \min \{\delta, \delta^2\}\}\tau_{x}^{2} d \delta\\
    &\leq  3 q \int_0^{+\infty}\Big[\Big(\tau_{x}^{2} \delta\Big)^{3 q-1}+\Big(\lambda_{\max} (\Sigma^{(\ell)})+\tau_{x}^{2} c_1\Big(\sqrt{\frac{p}{n_\ell}}+\frac{p}{n_\ell}\Big)\Big)^{3 q-1}\Big] 2^{3 q-2}\\
    &\qquad\qquad\qquad\qquad \cdot c_2 \exp \{-c_3 n_\ell \min \{\delta, \delta^2\} \}\tau_{x}^{2} d \delta\\
    &=3 q \cdot 2^{3 q-2}\Big[\mathcal{I}_{1}+\Big(\lambda_{\max} (\Sigma^{(\ell)})+\tau_{x}^{2} c_1\Big(\sqrt{\frac{p}{n_\ell}}+\frac{p}{n_\ell}\Big)\Big)^{3 q-1}\mathcal{I}_{2}\Big],
\end{align*}
where
\begin{align*}
    \mathcal{I}_{1}&=\int_0^{+\infty}(\tau_{x}^{2} \delta)^{3 q-1} c_2 \exp \{-c_3 n_\ell \min \{\delta, \delta^2\} \}\tau_{x}^{2} d \delta\\
    \mathcal{I}_{2}&=\int_0^{+\infty} c_2 \exp \{-c_3 n_\ell \min \{\delta, \delta^2\}\} \tau_{x}^{2} d \delta.
\end{align*}
The term $\mathcal{I}_{2}$ could be bounded as
\begin{align*}
\mathcal{I}_{2}&=\int_0^1 c_2 \exp \{-c_3 n_\ell \delta^2\} \tau_{x}^{2} d \delta+\int_1^{+\infty} c_2 \exp \{-c_3 n_\ell \delta\} \delta^2 d \delta\\
&\leq c_2 \tau_{x}^{2}+\frac{c_2 \tau_{x}^{2}}{c_3 n_\ell} \int_0^{+\infty} \exp \{-c_3 n_\ell \delta\} d c_3 n_\ell \delta\\
&=\big(1+\frac{1}{c_3 n_\ell}\big) c_2 \tau_{x}^{2}.
\end{align*}
Now we deal with term $\mathcal{I}_{1}$. This term could be bounded as
\begin{align*}
\mathcal{I}_{1}&=\int_0^{+\infty}(\tau_{x}^{2} \delta)^{3 q-1} c_2 \exp \{-c_3 n_\ell \min \{\delta, \delta^2\}\}  \tau_{x}^{2} d \delta\\
&=\int_0^1(\tau_{x}^{2} \delta)^{3 q-1} c_2 \exp \{-c_3 n_\ell \delta^2\} \tau_{x}^{2} d \delta+\int_1^{+\infty}(\tau_{x}^{2} \delta)^{3 q-1} c_2 \exp \{-c_3 n_\ell \delta\} \tau_{x}^{2} d \delta\\
&\leq c_2 \int_0^1(\tau_{x}^{2} \delta)^{3 q-1} \tau_{x}^{2} d \delta+c_2(\tau_{x}^{2})^{3 q} \int_1^{+\infty} \delta^{3 q-1} \exp \{-c_3 n_\ell \delta\} d\delta\\
&=c_2 \frac{1}{3 q}(\tau_{x}^{2})^{3 q}+c_2 \frac{(\tau_{x}^{2})^{3 q}}{(c_3 n_\ell)^{3 q}} \int_1^{+\infty}(c_3 n_\ell \delta)^{3 q-1} \exp \{-c_3 n_\ell \delta\} d c_3 n_\ell \delta\\
&\leq c_2 \frac{(\tau_{x}^{2})^{3 q}}{3 q}+c_2(\frac{\tau_{x}^{2}}{c_3 n_\ell})^{3 q} \Gamma(3 q).
\end{align*}
Combining the above calculations together, it holds that
\begin{align*}
    \mathbb{E} \lambda_{\max}^{3 q}(\hat{\Sigma}^{(\ell)}) &\leq 3q 2^{3 q-2}\Big[c_2 \frac{(\tau_{x}^{2})^{3 q}}{3 q}+c_2\Big(\frac{\tau_{x}^{2}}{c_3 n_\ell}\Big)^{3 q} \Gamma(3 q)\\
    &\qquad\qquad\qquad\qquad+b_{1}^{3 q-1}\Big(1+\frac{1}{c_3 n_\ell}\Big) c_2 \tau_{x}^{2}\Big]+b_{1}^{3q},\\
    \big\|\lambda_{\max}^{3}(\hat{\Sigma}^{(\ell)})\big\|_{L_{q}}&\leq \Big[ 3q 2^{3 q-2}\Big[c_2 \frac{(\tau_{x}^{2})^{3 q}}{3 q}+c_2\big(\frac{\tau_{x}^{2}}{c_3 n_\ell}\big)^{3 q} \Gamma(3 q)\\
    &\qquad\qquad\qquad\qquad+b_{1}^{3 q-1}\big(1+\frac{1}{c_3 n_\ell}\big) c_2 \tau_{x}^{2}\Big]+b_{1}^{3q}\Big]^{\frac{1}{q}}.
\end{align*}    
Applying Theorem \ref{wainwright_thm6.5} with $n_{\ell}=1$, it holds that
\begin{align*}
\Big\|\lambda_{\max}^{2}\big(x_{k}^{(\ell)}x_{k}^{(\ell)\top}\big)\Big\|_{L_{2q}}&\leq \Big[ 4q 2^{4 q-2}\Big[c_2 \frac{(\tau_{x}^{2})^{4 q}}{4 q}+c_2\big(\frac{\tau_{x}^{2}}{c_3}\big)^{4 q} \Gamma(4 q)\\
&\qquad\qquad\qquad\qquad+b_{2}^{4 q-1}\big(1+\frac{1}{c_3}\big) c_2 \tau_{x}^{2}\Big]+b_{2}^{4q}\Big]^{\frac{1}{2q}},
\end{align*}
where $b_{2}=c_1(\sqrt{p}+p) \tau_{x}^{2}+\lambda_{\max }(\Sigma^{(\ell)})$.
\end{proof}
Now we are ready to prove Theorem \ref{sparsethm_lowerbound}.    
\begin{proof}[Proof of Theorem \ref{sparsethm_lowerbound}]
Consider $f(X^{(\ell)})=\frac{1}{p^2}\big\|X^{(\ell)} \otimes X^{(\ell)}v\big\|_{2}^{2}$ for $v=\operatorname{vec}(\Delta)$.

We first show that the expectation of $f(X^{(\ell)})$ is lower bounded. It holds that
    \begin{align*}
    \big\|\big( X^{(\ell)}\otimes X^{(\ell)}\big) \operatorname{vec}(\Delta )\big\|_{2}^{2}&=\operatorname{tr}(X^{(\ell) \top} X^{(\ell)}\Delta X^{(\ell) \top} X^{(\ell)} \Delta)\\
    &=\sum_{i=1}^{n_{\ell}}\big( x_{i}^{(\ell)\top}\Delta  x_{i}^{(\ell)}\big)^{2}+\sum_{i\not=j}^{n_{\ell}}\big(x_{i}^{(\ell)\top} \Delta   x_{j}^{(\ell)}\big)^2   
    \end{align*} 
For the first part, each term $\big( x_{i}^{(\ell)\top} \Delta   x_{i}^{(\ell)}\big)$ in the square is sub-exponential with 
\begin{align*}
\big\| x_{i}^{(\ell)\top} \Delta   x_{i}^{(\ell)}\big\|_{\psi_{1}}\leq \big\| \Delta  \big\|_{F}\big\|x_{i}^{(\ell)}\big\|_{\psi_{2}}^{2}.
\end{align*}
Besides,
\begin{align*}
    &\mathbb{E}\big( x_{i}^{(\ell)\top} \Delta   x_{i}^{(\ell)}\big)^2 \\
    =&\int_{0}^{\infty}2\mathbb{P}\big(\big| x_{i}^{(\ell)\top} \Delta   x_{i}^{(\ell)}\big|\geq t\big)tdt\\
    =& \int_{0}^{\infty}2\mathbb{P}\Bigg(\exp\Bigg\{\frac{| x_{i}^{(\ell)\top} \Delta   x_{i}^{(\ell)}|}{\|x_{i}^{(\ell)\top} \Delta   x_{i}^{(\ell)}\|_{\psi_{1}}}\Bigg\}\geq \exp\Bigg\{\frac{t}{\|x_{i}^{(\ell)\top} \Delta   x_{i}^{(\ell)}\|_{\psi_{1}}}\Bigg\}\Bigg)tdt\\
    \leq& 2\int_{0}^{\infty}\exp\Bigg\{-\frac{t}{\|x_{i}^{(\ell)\top} \Delta   x_{i}^{(\ell)}\|_{\psi_{1}}}\Bigg\}\mathbb{E}\exp\Bigg\{\frac{|x_{i}^{(\ell)\top} \Delta   x_{i}^{(\ell)}|}{\|x_{i}^{(\ell)\top} \Delta   x_{i}^{(\ell)}\|_{\psi_{1}}}\Bigg\}tdt\\
    \leq& 4\int_{0}^{\infty}\exp\Bigg\{-\frac{t}{\|x_{i}^{(\ell)\top} \Delta   x_{i}^{(\ell)}\|_{\psi_{1}}}\Bigg\}tdt\\
    \leq& 2\|x_{i}^{(\ell)\top} \Delta   x_{i}^{(\ell)}\|_{\psi_{1}}^{2}\Gamma(3)\\
    =&4\|x_{i}^{(\ell)\top} \Delta   x_{i}^{(\ell)}\|_{\psi_{1}}^{2}\leq 4 \| \Delta  \|_{F}^{2}\|x_{i}^{(\ell)}\|_{\psi_{2}}^{4}.
\end{align*}
Hence, we have similar results for the cross term $\big( x_{i}^{(\ell)\top} \Delta   x_{j}^{(\ell)}\big)^2$, i.e.
\begin{align*}
     \mathbb{E}(x_{i}^{(\ell)\top} \Delta   x_{i}^{(\ell)})^2\leq 4\|x_{i}^{(\ell)\top} \Delta   x_{j}^{(\ell)}\|_{\psi_{1}}^{2}.
\end{align*}
On the other hand, for the second part, it holds that
\begin{align*}
\mathbb{E}(x_{i}^{(\ell)\top} \Delta   x_{j}^{(\ell)})^2&= \mathbb{E}\big[x_{i}^{(\ell)\top} \Delta   x_{j}^{(\ell)}x_{j}^{(\ell)\top} \Delta   x_{i}^{(\ell)}\big] =\mathbb{E}\big[\operatorname{tr}\big( x_{i}^{(\ell)\top} \Delta   x_{j}^{(\ell)}x_{j}^{(\ell)\top} \Delta   x_{i}^{(\ell)}\big)\big]\\
&=\operatorname{tr}\big(\mathbb{E}x_{i}^{(\ell)} x_{i}^{(\ell)\top} \Delta   \mathbb{E}x_{j}^{(\ell)}x_{j}^{(\ell)\top} \Delta   \big) =\operatorname{tr}\big(\Sigma^{(\ell)} \Delta   \Sigma^{(\ell)}\Delta\big)\\
&\geq \lambda_{\min}^{2}\big(\Sigma^{(\ell)}\big)\big\| \Delta\big\|_{F}^{2}. 
\end{align*}
Also there are $n_{\ell}$ terms in the first part and $n_{\ell}(n_{\ell}-1)$ terms in the second part. The order of expectation of each term from the first part is not larger than that from the second part. Therefore, in order to find a lower bound for the expectation, the second part is the dominant term. There exists a absolute constant $C$ such that 
\begin{align*}
\frac{1}{p^2}\mathbb{E}\big\|\big( X^{(\ell)}\otimes X^{(\ell)}\big) \operatorname{vec}\big(\Delta\big)\big\|_{2}^{2}\geq C\frac{n_{\ell}(n_{\ell}-1)}{p^2}\lambda_{\min}^{2}\big(\Sigma^{(\ell)}\big)\big\| \Delta \big\|_{F}^{2}.
\end{align*}

Next, we use concentration inequalities to assert that for each fixed $r>0$, the random variable $f(X)$ is sharply concentrated around its expectation with high probability. For $v=\operatorname{vec}(\Delta)$, consider 
\begin{align*}
\frac{1}{p^2}\big\|\big(X^{(\ell)}\otimes X^{(\ell)}\big)v\big\|_{2}^{2}=\frac{1}{p^2}\Big(\sum_{i=1}^{n_{\ell}}( x_{i}^{(\ell)\top} \Delta   x_{i}^{(\ell)})^{2}+\sum_{i\not=j}^{n_{\ell}}(x_{i}^{(\ell)\top} \Delta   x_{j}^{(\ell)})^{2}\Big).
\end{align*}
In order to use Theorem \ref{thm3.4adamczak2015concentration},  we define
\begin{align*}
f(X^{(\ell)})=\frac{1}{p^2}\big\|\big(X^{(\ell)}\otimes X^{(\ell)}\big)v\big\|_{2}^{2}=\frac{1}{p^2}\Big(\sum_{i=1}^{n_{\ell}}\big( x_{i}^{(\ell)\top} \Delta   x_{i}^{(\ell)}\big)^{2}+\sum_{i\not=j}^{n_{\ell}}( x_{i}^{(\ell)\top} \Delta   x_{j}^{(\ell)})^{2}\Big).
\end{align*}
The gradient of $f(X^{(\ell)})$ w.r.t. $x_{k}^{(\ell)}$ is given below
\begin{align*}
    \nabla_k f(X^{(\ell)})&=\frac{1}{p^2}\nabla_{x_k^{(\ell)}}\Big[\sum_{i=1}^{n_{\ell}}(x_i^{(\ell) \top}   \Delta   x_i^{(\ell)})^2+\sum_{i \neq j}^{n_{\ell}}(x_i^{(\ell) \top}   \Delta   x_j^{(\ell)})^2\Big]\\
    &=\frac{1}{p^2}\nabla_{x_k^{(\ell)}}\Big[(x_k^{(\ell) \top}   \Delta   x_k^{(\ell)})^2+\sum_{j \neq k}^{n_{\ell}}(x_k^{(\ell)\top}   \Delta   x_j^{(\ell)})^2\Big]\\
    &=\frac{4}{p^2}(x_k^{(\ell) \top}   \Delta   x_k^{(\ell)})\Delta x_k^{(\ell)}+\frac{2}{p^2}\sum_{j \neq k}^{n_{\ell}} (x_k^{(\ell)^{\top}}\Delta x_j^{(\ell)})   \Delta x_j^{(\ell)}\\
    &=\frac{2}{p^2}(x_k^{(\ell) \top}   \Delta   x_k^{(\ell)})\Delta x_k^{(\ell)}+\frac{2}{p^2}\sum_{j,k=1}^{n_{\ell}}(x_k^{(\ell)\top}   \Delta x_j^{(\ell)})\Delta x_j^{(\ell)}.
\end{align*}
Define $c_{kj}=(x_k^{(\ell)^{\top}}   \Delta   x_j^{(\ell)})$. Consider eigenvalue decomposition $\Delta=\sum_{l=1}^p \lambda_l v_l v_l^{\top}.$ We have that
\begin{align*}
    \sum_{j=1}^{n_{\ell}} c_{k j}  \Delta x_{j}^{(\ell)}&=\sum_{j=1}^{n_{\ell}} \sum_{l=1}^p \lambda_l c_{k j}(v_l^\top x_j^{(\ell)}) v_l =\sum_{l=1}^p \lambda_l\Big(\sum_{j=1}^{n_{\ell}} c_{k j}\big(v_l^{\top} x_j^{(\ell)}\big)\Big) v_l,\\
    \Big\|\sum_{j=1}^{n_{\ell}} c_{k j}  \Delta x_j^{(\ell)}\Big\|_2^2&=\sum_{l=1}^p \lambda_l^2\Big(\sum_{j=1}^{n_{\ell}} c_{k j}\big(v_l^{\top} x_j^{(\ell)}\big)\Big)^2,\\
    \Big\|(x_k^{(\ell)\top}  \Delta x_k^{(\ell)})  \Delta x_k^{(\ell)}\Big\|_2^2&=\Big\|c_{k k} \sum_{l=1}^p \lambda_l(v_l^{\top} x_k^{(\ell)}) v_l\Big\|_2^2\\
    &=\sum_{l=1}^{p}\lambda_{l}^{2}\big(c_{kk}(v_l^{\top} x_k^{(\ell)})\big)^{2}.
\end{align*}
Finally, 
\begin{align*}
    |\nabla f(X^{(\ell)})|_{2}^2&\leq 2\sum_{k=1}^{n_{\ell}}\Big[\Big\|\frac{2}{p^2}\big(x_k^{(\ell) \top}  \Delta x_k^{(\ell)}\big)  \Delta x_k^{(\ell)}\Big\|_{2}^{2}+\Big\|\frac{2}{p^2}\sum_{j,k=1}^{n_{\ell}} \big(x_k^{(\ell)^{\top}}  \Delta x_j^{(\ell)}\big)  \Delta x_j^{(\ell)}\Big\|_{2}^{2}\Big]\\
    &=\frac{8}{p^{4}}\sum_{k=1}^{n_{\ell}} \sum_{l=1}^p \lambda_l^2 c_{k k}^2\big(v_l^{\top} x_{k}^{(\ell)}\big)^2+\frac{8}{p^{4}}\sum_{k=1}^{n_{\ell}} \sum_{l=1}^p \lambda_l^2\Big(\sum_{j=1}^{n_{\ell}} c_{k j}\big(v_l^{\top} x_{j}^{(\ell)}\big)\Big)^2.
\end{align*}
Next, we control the $L_q$ norm of two parts seperately and then using the Holder's inequality to control the $L_q$ norm of $|\nabla f(X)|_2$. First, we upper bound the second part as
\begin{align*}
    &\sum_{k=1}^{n_{\ell}} \sum_{l=1}^p \lambda_l^2\Big(\sum_{j=1}^{n_{\ell}} c_{k j}\big(v_l^{\top} x_{j}^{(\ell)}\big)\Big)^2\\
    \leq &\sum_{k=1}^{n_{\ell}} \sum_{l=1}^p \lambda_{l}^{2}\Big[\Big(\sum_{j=1}^{n_{\ell}}\big(x_k^{(\ell) \top}  \Delta x_{j}^{(\ell)}\big)^2\Big)\Big(\sum_{j=1}^{n_{\ell}}\big(v_j^{\top} x_{j}^{(\ell)}\big)^2\Big)\Big]\\
    =&\sum_{k=1}^{n_{\ell}} \sum_{l=1}^p \lambda_{l}^{2}\Big[\Big(\sum_{j=1}^{n_{\ell}} \operatorname{tr}\big( \Delta x_j^{(\ell)} x_j^{(\ell)\top}  \Delta x_{k}^{(\ell)} x_k^{(\ell)\top}\big)\Big) \Big(\sum_{j=1}^{n_{\ell}}\big(v_l^{\top} x_{j}^{(\ell)}\big)^2\Big)\Big]\\
    =&\sum_{k=1}^{n_{\ell}} \sum_{l=1}^p \lambda_{l}^{2} \Big[ n_{\ell} \operatorname{tr}\big( \Delta \hat{\Sigma}^{(\ell)}  \Delta x_{k}^{(\ell)} x_k^{(\ell)}\big)\Big(\sum_{j=1}^{n_{\ell}}\big(v_l^{\top} x_j^{(\ell)}\big)^2\Big)\Big]\\
    =&\sum_{l=1}^p \lambda_l^2 n_{\ell}^2 \operatorname{tr}( \Delta \hat{\Sigma}^{(\ell)}  \Delta \hat{\Sigma}^{(\ell)})\Big(\sum_{j=1}^{n_{\ell}}\big(v_l^{\top} x_j^{(\ell)}\big)^2\Big)\\
    =&n_{\ell}^2 \operatorname{tr}\big( \Delta \hat{\Sigma}^{(\ell)}  \Delta \hat{\Sigma}^{(\ell)}\big) \sum_{l=1}^{p} \lambda_l^2 \sum_{j=1}^{n_{\ell}}\big(v_l^{\top} x_j^{(\ell)}\big)^2\\
    =&n_{\ell}^2 \operatorname{tr}( \Delta \hat{\Sigma}^{(\ell)}  \Delta \hat{\Sigma}^{(\ell)}) \sum_{l=1}^{p} \sum_{j=1}^{n_{\ell}} \lambda_l^2 \operatorname{tr}\big(v_l^{\top} x_j^{(\ell)} x_j^{(\ell)\top} v_l\big)\\
    =&n_{\ell}^2 \operatorname{tr}( \Delta \hat{\Sigma}^{(\ell)}  \Delta \hat{\Sigma}^{(\ell)}) \sum_{l=1}^p \lambda_l^2 \operatorname{tr}\Big(\sum_{j=1}^{n_{\ell}} x_j^{(\ell)} x_j^{(\ell) \top} v_l v_l^{\top}\Big)\\
    =&n_{\ell}^3 \operatorname{tr}(\Delta \hat{\Sigma}^{(\ell)}  \Delta \hat{\Sigma}^{(\ell)})\sum_{l=1}^p \lambda_l^2 \operatorname{tr}\big(\hat{\Sigma}^{(\ell)} v_l v_l^{\top}\big)\\
    =&n_{\ell}^3 \operatorname{tr}(\Delta \hat{\Sigma}^{(\ell)} \Delta \hat{\Sigma}^{(\ell)}) \operatorname{tr}\Big[\hat{\Sigma}^{(\ell)}\Big(\sum_{l=1}^p \lambda_l^2 v_l v_l^{\top}\Big)\Big]\\
    =&n_{\ell}^3 \operatorname{tr}( \Delta \hat{\Sigma}^{(\ell)}  \Delta \hat{\Sigma}^{(\ell)}) \operatorname{tr}(\hat{\Sigma}^{(\ell)} \Delta^{2})\leq  n_{\ell}^{3} \lambda_{\max }^{3}(\hat{\Sigma}^{(\ell)})\| \Delta\|_F^4.
\end{align*}
Therefore, the main part in $|\nabla f(X)|_{2}^{2}$ could be bounded by
\begin{align*}
    \frac{8}{p^{4}}\sum_{k=1}^{n_{\ell}} \sum_{l=1}^p \lambda_l^2\Big(\sum_{j=1}^{n_{\ell}} c_{k j}\big(v_l^{\top} x_{j}^{(\ell)}\big)\Big)^2&\leq \frac{8}{p^{4}}n_{\ell}^{3} \lambda_{\max }^{3}(\hat{\Sigma}^{(\ell)})\| \Delta\|_F^4.
\end{align*}
To calculate the $L_{q}$ norm of $\lambda_{\max }^{3}(\hat{\Sigma}^{(\ell)})$, we need tail behavior of largest eigenvalue of the sub-Gaussian sample covariance matrix $\hat{\Sigma}^{(\ell)}=\frac{1}{n_{\ell}}\sum_{i=1}^{n_{\ell}}x_{i}^{(\ell)}x_{i}^{(\ell)\top}$. 
By Theorem \ref{wainwright_thm6.5}, for $\hat{\Sigma}^{(\ell)}=\frac{1}{n_{\ell}}\sum_{i=1}^{n_{\ell}}x_{i}^{(\ell)}x_{i}^{(\ell)\top}$, it holds that
\begin{equation*}
    \mathbb{P}\Bigg(\frac{\|\hat{\Sigma}^{(\ell)}-\Sigma^{(\ell)}\|_2}{\tau_{x}^{2}} \geq c_1\Big\{\sqrt{\frac{p}{n_{\ell}}}+\frac{p}{n_{\ell}}\Big\}+\delta\Bigg) \leq c_2 \exp \{-c_3 n_\ell \min \{\delta, \delta^2\}\}.
\end{equation*}
This means with probability at least $1-c_2 \exp \{-c_3 n_\ell \min \{\delta, \delta^2\}\}$, we have 
\begin{equation*}
    \frac{1}{\tau_{x}^{2}}\|\hat{\Sigma}^{(\ell)}-\Sigma^{(\ell)}\|_2 \leqslant c_1\Big\{\sqrt{\frac{p}{n_\ell}}+\frac{p}{n_\ell}\Big\}+\delta. 
\end{equation*}
Then by triangle inequality, with probability at least $1-c_2 \exp \{-c_3 n_\ell \min \{\delta, \delta^2\}\}$, we have that
\begin{align*}
    \lambda_{\max }(\hat{\Sigma}^{(\ell)})&=\|\hat{\Sigma}^{(\ell)}\|_2 \leqslant \|\hat{\Sigma}^{(\ell)}-\Sigma^{(\ell)}\|_2+\|\Sigma^{(\ell)}\|_2\\
    &\leq \Big[c_1\Big\{\sqrt{\frac{p}{n_\ell}}+\frac{p}{n_\ell}\Big\}+\delta\Big] \tau_{x}^{2}+\|\Sigma^{(\ell)}\|_2.
\end{align*}
Now we set $$t=\Big[c_1\Big\{\sqrt{\frac{p}{n_\ell}}+\frac{p}{n_\ell}\Big\}+\delta\Big] \tau_{x}^{2}+\lambda_{\max }(\Sigma^{(\ell)})$$
or equivalently,
$$\delta(t)\coloneqq\frac{t-\lambda_{\max }(\Sigma^{(\ell)})}{\tau_{x}^{2}}-c_1\Big\{\sqrt{\frac{p}{n_\ell}}+\frac{p}{n_\ell}\Big\}.$$
Theorem \ref{wainwright_thm6.5} indicates that
\begin{align*}
\mathbb{P}\big(\lambda_{\max }(\hat{\Sigma}^{(\ell)}) \geq t\big) \leq C_2 \exp \{-c_3 n_\ell \min \{\delta(t), \delta^2(t)\}\}    
\end{align*}
Also by Lemma \ref{Lemma_Lq_norm}, it holds that 
\begin{align*}
    \mathbb{E} \lambda_{\max}^{3 q}(\hat{\Sigma}^{(\ell)}) &\leq 3q 2^{3 q-2}\Big[c_2 \frac{(\tau_{x}^{2})^{3 q}}{3 q}+c_2\Big(\frac{\tau_{x}^{2}}{c_3 n_\ell}\Big)^{3 q} \Gamma(3 q)\\
    &\qquad\qquad\qquad\qquad+b_{1}^{3 q-1}\Big(1+\frac{1}{c_3 n_\ell}\Big) c_2 \tau_{x}^{2}\Big]+b_{1}^{3q}\\
    &=\mathcal{O}(1),\\
    \big\|\lambda_{\max}^{3}(\hat{\Sigma}^{(\ell)})\big\|_{L_{q}}&\leq \Big[ 3q 2^{3 q-2}\Big[c_2 \frac{(\tau_{x}^{2})^{3 q}}{3 q}+c_2\big(\frac{\tau_{x}^{2}}{c_3 n_\ell}\big)^{3 q} \Gamma(3 q)\\
    &\qquad\qquad\qquad\qquad+b_{1}^{3 q-1}\big(1+\frac{1}{c_3 n_\ell}\big) c_2 \tau_{x}^{2}\Big]+b_{1}^{3q}\Big]^{\frac{1}{q}}\\
    &=\mathcal{O}(1),
\end{align*}
as $p$ and $n_{\ell}$ goes to infinity such that $\frac{p}{n_{\ell}}\rightarrow\gamma_{\ell}$. The $L_q$ norm of second part could be bounded by
\begin{equation*}
\frac{8}{p^{4}}\Bigg\|\sum_{k=1}^{n_{\ell}} \sum_{l=1}^p \lambda_l^2\Big(\sum_{j=1}^{n_{\ell}} c_{k j}\big(v_l^{\top} x_{j}^{(\ell)}\big)\Big)^2\Bigg\|_{L_q}\leq \mathcal{O}(p^{-1})\| \Delta\|_{F}^{4}.
\end{equation*}
Next, we bound the $L_q$ norm of first part of gradient. 
\begin{align*}
    &\Big\|\frac{8}{p^{4}}\sum_{k=1}^{n_{\ell}} \sum_{l=1}^p \lambda_l^2 c_{k k}^2\big(v_l^{\top} x_{k}^{(\ell)}\big)^2\Big\|_{L_{q}}\\
    \leq& 
    \frac{8}{p^{4}}\sum_{k=1}^{n_{\ell}}\Big\|c_{k k}^2 \sum_{l=1}^p \lambda_l^2 \big(v_l^{\top} x_{k}^{(\ell)}\big)^2\Big\|_{L_{q}}\quad\quad\text{(triangle inequality)}\\
    \leq& \frac{8}{p^{4}}\sum_{k=1}^{n_{\ell}}\|c_{k k}^2 \|_{L_{2q}}\Big\|\sum_{l=1}^p \lambda_l^2 \big(v_l^{\top} x_{k}^{(\ell)}\big)^2\Big\|_{L_{2q}} \quad\quad\text{(Holder's inequality)}\\
    \leq& \frac{8}{p^{4}}\sum_{k=1}^{n_{\ell}}\|c_{k k}^2 \|_{L_{2q}}\Big[\sum_{l=1}^p \lambda_l^2\big\| (v_l^{\top} x_{k}^{(\ell)})^2\big\|_{L_{2q}}\Big] \quad\quad\text{(triangle inequality)}.
\end{align*}
We first bound the $L_{2q}$ norm of $c_{kk}^{2}$ by following upper bound on the quantity $c_{kk}^{2}$:
\begin{align*}
    c_{kk}^{2}&=\Big[\sum_{l=1}^{p}\lambda_{l}(x_{k}^{(\ell)\top}v_{l})^{2}\Big]^{2} \leq \Big(\sum_{l=1}^p \lambda_l^2\Big)\Big(\sum_{l=1}^p\big(x_k^{(\ell)\top} v_l\big)^4\Big)  \\
    &= \| \Delta\|_F^2 \sum_{l=1}^p \operatorname{tr}\big(x_k^{(\ell)} x_k^{(\ell)^{\top}} v_l v_l^{\top} x_k^{(\ell)} x_k^{(\ell)} v_l^{\top} v_l^{\top}\big) \\
    &\leq \|\Delta\|_F^2\sum_{l=1}^p \lambda_{\max }\big(x_k^{(\ell)} x_k^{(\ell) \top}\big) \operatorname{tr}\big(v_l v_l^{\top} x_k^{(\ell)} x_k^{(\ell) \top} v_l v_l^{\top}\big)\\
    &\leq \|\Delta\|_F^2\lambda_{\max}\big(x_k^{(\ell)} x_k^{(\ell) \top}\big) \operatorname{tr}\Big(\sum_{l=1}^p v_l v_l^{\top} x_k^{(\ell)} x_k^{(\ell) \top}\Big)\\
    & \leq  \| \Delta\|_F^2\lambda_{\max }^{2}\big(x_k^{(\ell)} x_k^{(\ell) \top}\big). 
\end{align*}
Now, for the first part in the gradient, by Lemma \ref{Lemma_Lq_norm}
\begin{align*}
    \Big\|\lambda_{\max}^{2}\big(x_{k}^{(\ell)}x_{k}^{(\ell)\top}\big)\Big\|_{L_{2q}}&\leq \Big[ 4q 2^{4 q-2}\Big[c_2 \frac{(\tau_{x}^{2})^{4 q}}{4 q}+c_2\big(\frac{\tau_{x}^{2}}{c_3}\big)^{4 q} \Gamma(4 q)\\
    &\qquad\qquad\qquad\qquad+b_{2}^{4 q-1}\big(1+\frac{1}{c_3}\big) c_2 \tau_{x}^{2}\Big]+b_{2}^{4q}\Big]^{\frac{1}{2q}}\\
    b_{2}&=c_1(\sqrt{p}+p) \tau_{x}^{2}+\lambda_{\max }(\Sigma^{(\ell)})=\mathcal{O}(p).
\end{align*}
Therefore, $\big\|\lambda_{\max}^{2}(x_{k}^{(\ell)}x_{k}^{(\ell)\top})\big\|_{L_{2q}}=\mathcal{O}(p^2)$. Next, for the $L_{2q}$ norm of $\big(v_l^{\top} x_{k}^{(\ell)}\big)^2$. Note that $x_{k}^{(\ell)}$ is sub-Gaussian random vector with parameter $\tau_{x}$ and $v_{l}$ is a unit vector. Therefore, $v_{l}^{\top}x_{k}^{(\ell)}$ is a sub-Gaussian random variable with parameter $\tau_{x}$. Therefore, for $k=1,\dots,n_{\ell}$,
\begin{equation*}
    \big\|\big(v_l^{\top} x_{k}^{(\ell)}\big)^2\big\|_{L_{2q}}=\big(\mathbb{E}\big[\big(v_l^{\top} x_{k}^{(\ell)}\big)^{4q}\big]\big)^{\frac{1}{2q}}\leq (c_4^{4 q} \sqrt{4 q}^{4 q})^{\frac{1}{2 q}}=c_4^2 4 q.
\end{equation*}
Finally, note that we have
\begin{align*}
    &\Big\|\frac{8}{p^{4}}\sum_{k=1}^{n_{\ell}} \sum_{l=1}^p \lambda_l^2 c_{k k}^2\big(v_l^{\top} x_{k}^{(\ell)}\big)^2\Big\|_{L_{q}}\\
    \leq& \frac{8}{p^{4}}\sum_{k=1}^{n_{\ell}}\|c_{k k}^2 \|_{L_{2q}}\sum_{l=1}^p \lambda_l^2\big\| \big(v_l^{\top} x_{k}^{(\ell)}\big)^2\big\|_{L_{2q}}\\
    \leq& \frac{8}{p^4}\sum_{k=1}^{n_{\ell}}\Big[\| \Delta\|_{F}^{2}\big\|\lambda_{\max}^{2}\big(x_{k}^{(\ell)}x_{k}^{(\ell)\top}\big)\big\|_{L_{2q}}\Big(\sum_{l=1}^{p}\lambda_{l}^{2}\big\| \big(v_l^{\top} x_{k}^{(\ell)}\big)^2\big\|_{L_{2q}} \Big)\Big]\\
    \leq& \frac{8}{p^4}\sum_{k=1}^{n_{\ell}}\Big[\| \Delta\|_{F}^{2}\big\|\lambda_{\max}^{2}\big(x_{k}^{(\ell)}x_{k}^{(\ell)\top}\big)\big\|_{L_{2q}}c_4^2 4 q\Big(\sum_{l=1}^{p}\lambda_{l}^{2} \Big)\Big]\\
    =& \frac{8}{p^4}\sum_{k=1}^{n_{\ell}}\Big[\| \Delta\|_{F}^{2}\big\|\lambda_{\max}^{2}\big(x_{k}^{(\ell)}x_{k}^{(\ell)\top}\big)\big\|_{L_{2q}}c_4^2 4 q\| \Delta\|_{F}^{2}\Big]\\
    =&\mathcal{O}(p^{-1})\| \Delta\|_{F}^{4},
\end{align*}
as $p,n_{\ell}\rightarrow\infty$ such that $\frac{p}{n_{\ell}}\rightarrow\gamma_{\ell}$. 

Combine the two parts together, the $L_{q}$ norm of $|\nabla f(X)|_{2}^{2}$ could be bounded by
\begin{align*}
   \||\nabla f(X^{(\ell)})|_{2}^{2}\|_{L_{2q}} &\leq \Big\|\frac{8}{p^{4}}\sum_{k=1}^{n_{\ell}} \sum_{l=1}^p \lambda_l^2 c_{k k}^2\big(v_l^{\top} x_{k}^{(\ell)}\big)^2\Big\|_{L_{q}}\\
   &\qquad\qquad\qquad+\Big\|\frac{8}{p^{4}}\sum_{k=1}^{n_{\ell}} \sum_{l=1}^p \lambda_l^2\Big(\sum_{j=1}^{n_{\ell}} c_{k j}\big(v_l^{\top} x_{j}^{(\ell)}\big)\Big)^2\Big\|_{L_{q}} \\
   &\leq \mathcal{O}(p^{-1})\|\Delta\|_{F}^{4}.
\end{align*}
Now we are able to control the $L_q$ norm of $|\nabla f(X)|_{2}$ by applying Holder's inequality ($1<r<s<\infty$)
$$\mathbb{E}\big[|Y|^r\big] \leq \big(\mathbb{E}\big[|Y|^s\big]\big)^{\frac{r}{s}}$$
with $r=\frac{q}{2}$, $s=q$ and random variable $Y=|\nabla f(X^{(\ell)})|_{2}^{2}$. We have that 
\begin{align*}
    \mathbb{E}\big[|\nabla f(X^{(\ell)})|_{2}^{q}\big]&\leq \big(\mathbb{E}|\nabla f(X^{(\ell)})|_{2}^{2q}\big)^{\frac{1}{2}}\\
    \||\nabla f(X^{(\ell)})|_{2}\|_{L_{q}}&=\mathbb{E}\big[|\nabla f(X^{(\ell)})|_{2}^{q}\big]\leq \big[\big(\mathbb{E}|\nabla f(X^{(\ell)})|_{2}^{2q}\big)^{\frac{1}{2}}\big]^{\frac{1}{q}}\\
    &=\big[\big(\mathbb{E}|\nabla f(X^{(\ell)})|_{2}^{2q}\big)^{\frac{1}{q}}\big]^{\frac{1}{2}}=\sqrt{\||\nabla f(X^{(\ell)})|_{2}^{2}\|_{L_{q}}}.
\end{align*}
Finally, the $L_q$ norm of $|\nabla f(X^{(\ell)})|_{2}$ could be bounded as
\begin{equation*}
    \||\nabla f(X^{(\ell)})|_{2}\|_{L_{q}}\leq \sqrt{\||\nabla f(X^{(\ell)})|_{2}^{2}\|_{L_{q}}}=\mathcal{O}(p^{-\frac{1}{2}})\|\Delta\|_{F}^{2}.
\end{equation*} 
By Theorem \ref{thm3.4adamczak2015concentration}, we then have that
\begin{align*}
&\|f(X^{(\ell)})-\mathbb{E} f(X^{(\ell)})\|_{L_{q}} \\
\leq & C_\beta D_{L S_\beta}^{1 / 2} q^{1 / 2}\||\nabla f(X^{(\ell)})|_2\|_{L_{q}}+D_{L S_\beta}^{1 / \beta} q^{1 / \alpha}\||\nabla f(X^{(\ell)})|_\beta\|_{L_{q}}\\
\leq &\mathcal{O}(p^{-\frac{1}{2}})\|\Delta\|_{F}^{2}.
\end{align*}
Finally, by Markov's inequality, 
\begin{align*}
    &\mathbb{P}(|f(X^{(\ell)})-\mathbb{E} f(X^{(\ell)})|>t)\\
    \leq &\frac{\mathbb{E}\big[|f(X^{(\ell)})-\mathbb{E} f(X^{(\ell)})|^q\big]}{t^q} =\frac{\|f(X^{(\ell)})-\mathbb{E} f(X^{(\ell)})\|_{L_q}^q}{t^q}\\
    =&\Bigg(\frac{\|f(X^{(\ell)})-\mathbb{E} f(X^{(\ell)})\|_{L_q}}{t}\Bigg)^q\\
    \leq&\mathcal{O}(t^{-q}p^{-\frac{q}{2}}\|\Delta\|_{F}^{2q}).
\end{align*} 
Setting $t=p^{-\frac{1}{4}}\|\Delta\|_{F}^{2}$, it holds that with probability at least $1-C_{q}p^{-\frac{q}{4}}$ tending to $1$ as $p,n_{\ell}\rightarrow\infty$ such that $\frac{p}{n_{\ell}}\rightarrow\gamma_{\ell}<\infty$,  that 
$$\Big|\frac{1}{p^2}\big\|X^{(\ell)}\otimes X^{(\ell)}\operatorname{vec}(\Delta)\big\|_{2}^{2}-\mathbb{E}\Big[\frac{1}{p^2}\big\|X^{(\ell)}\otimes X^{(\ell)}\operatorname{vec}(\Delta)\big\|_{2}^{2}\Big]\Big|\leq p^{-\frac{1}{4}}\|\Delta\|_{F}^{2},$$
where $C_{q}$ is some absolute constant only depends on the choice of $q$. 

Furthermore, the expectation $\mathbb{E}\big[\frac{1}{p^2}\big\|X^{(\ell)}\otimes X^{(\ell)}\big\|_{2}^{2}\big]$ is lower bounded by
\begin{align*}
\frac{1}{p^2}\mathbb{E}\big\|\big( X^{(\ell)}\otimes X^{(\ell)}\big) \operatorname{vec}(\Delta)\big\|_{2}^{2}&\geq C\frac{n_{\ell}(n_{\ell}-1)}{p^2}\lambda_{\min}^{2}(\Sigma^{(\ell)})\| \Delta \|_{F}^{2}\\
&\geq C\lambda_{\min}^{2}(\Sigma^{(\ell)})\| \Delta \|_{F}^{2}.
\end{align*}
Therefore, with probability at least $1-C_{q}p^{-\frac{q}{4}}$ for any $q\geq 2$
\begin{align*}
   \frac{1}{p^2}\mathbb{E}\big\|\big( X^{(\ell)}\otimes X^{(\ell)}\big) \operatorname{vec}(\Delta)\big\|_{2}^{2}&=\operatorname{tr}\big(X^{(\ell) \top} X^{(\ell)}   \Delta   X^{(\ell) \top} X^{(\ell)}   \Delta  \big) \\
   &\geq C\lambda_{\min}^{2}\big(\Sigma^{(\ell)}\big)\| \Delta \|_{F}^{2}.
\end{align*} 
\end{proof}
\begin{rmk}\label{rmk_track_gamma} 
For the approaches in Section \ref{MTL_sparse_section}, when all $\gamma_{\ell}=\gamma$ under random design, following constants in the proofs of of Theorem \ref{thm1_sparse_cov_est}, \ref{thm2_sparse_cov_est} and \ref{thm3_sparse_cov_est} depend on $\gamma$ as follows:
\begin{align*}
    \tilde{\sigma}^{2} & \leq \mathcal{O}\big((1+\gamma^{-1})^{2}(1+\sqrt{\gamma})^{2}\big),\\
    u_{ij}^{(\ell)},v_{ij}^{(\ell)}, \tilde{u}_{ij}^{(\ell)},\tilde{v}_{ij}^{(\ell)}&\leq \mathcal{O}\big((1+\sqrt{\gamma})^2(1+(1+\sqrt{\gamma})^2)\big),\\
    C_{1},C_{2} &=\mathcal{O}\Big((1+\sqrt{\gamma})^{2}\big((1+\sqrt{\gamma})^2+(1+\gamma^{-1})^2\big)\Big),\\
    \min_{\ell=1,\dots,L}\kappa_{0}^{(\ell)}&\geq \mathcal{O}\big(\gamma^{-2}\big). 
\end{align*}
under random design case.    
\end{rmk}

\section{Additional example for visualization of limiting risk}\label{sec:additionalriskeg}
Here, we provide another example for explicitly computing and visualizing the limiting risk. Suppose that $\Sigma^{(L+1)}=\Omega^{-\kappa}$, where $\Omega$ is as in~\eqref{eq:examplematrix} with $a=16$ and $b=5$, so that the eigenvalues of $\Omega$ are of the form $\lambda_k=16+10 \cos \frac{k \pi}{p+1}\in [6,26]$ according to \cite{elliott1953characteristic}. Under this model for $\Sigma^{(L+1)}$, the eigensubspaces for $\Omega$ coincide 
with those of $\Sigma^{(L+1)}$, so that the matrices commute, even though eigenvalues of $\Omega$ are smooth transforms of the eigenvalues the latter. In effect this model modifies the weights associated with eigensubspaces flexibly, through a single power index $\kappa$. A similar model has been considered in the context of two sample regularized tests in \cite{li2020generalizeHotelling}.

When $\kappa=0$, i.e. $\Sigma^{(L+1)}=I_{p}$ and $\Lambda^{(L+1)}=\Omega^{\frac{1}{2}}\Sigma^{(L+1)}\Omega^{\frac{1}{2}}=\Omega$. For any fixed $p$, solving $16+10 \cos \frac{k \pi}{p+1} \leq x$ gives $k \geq \Big\lceil \frac{p+1}{\pi} \cos^{-1} \frac{x-16}{10}\Big\rceil$. The empirical spectral distribution is given by
\begin{align*}
    F_{\Lambda^{(L+1)}}(x)=1-\Big\lceil \frac{p+1}{p\pi} \cos^{-1} \frac{x-16}{10}\Big\rceil.
\end{align*}
Hence, the limiting spectral distribution of $\Lambda^{(L+1)}$ is
\begin{align*}
    H_{\Lambda^{(L+1)}}(x)&=1-\frac{1}{\pi} \cos^{-1} \frac{x-16}{10}\\
    dH_{\Lambda^{(L+1)}}(x)&=\frac{1}{\pi} \frac{1}{\sqrt{100-(x-16)^2}}.
\end{align*}
In general, if $\Sigma^{(L+1)}=\Omega^{-\kappa}$, then $\Lambda^{(L+1)}=\Omega^{1-\kappa}$, and
\begin{align*}
    d H_{\Lambda^{(L+1)}}(x)=-\frac{1-\kappa}{\pi x^{\kappa}} \frac{1}{\sqrt{100-(x^{\frac{1}{1-\kappa}}-16)^2}}.
\end{align*}
Let $v$ be the Stieltjes transform of limiting spectral distribution of $\frac{1}{n_{L+1}} X^{(L+1)}\Omega X^{(L+1)^\top}$ and $s$ is the Stieltjes transform of limiting spectral distribution of $\widetilde{\Lambda}^{(L+1)}$
\begin{align*}
    v(z)&=\Big(-z+c \int \frac{t dH_{\Lambda^{(L+1)}}(t)}{1+v(z) t}\Big)^{-1},\\
    v^{\prime}(z)&=\Big(\frac{1}{v(z)^2}-\gamma_{L+1} \int \frac{t^2 d H_{\Lambda^{(L+1)}}(t)}{(1+t v(z))^2}\Big)^{-1},
\end{align*}
and $s(z)$ is related to $v(z)$ by following equation
\begin{align*}
    \gamma_{L+1}(s(z)+z^{-1})&=v(z)+z^{-1},\\
    \gamma_{L+1}(s^{\prime}(z)-z^{-2})&=v^{\prime}(z)-z^{-2}.
\end{align*}
Standard fixed point algorithm converges could be used to determining $s(z)$ for $z \in \mathbb{C} \backslash \mathbb{R}^{+}$, 
\begin{align*}
   \tilde{v}^{(0)}(z)&=0 \\
   \tilde{v}^{(t+1)}(z)&=\Big(-z+c \int \frac{t dH_{\Lambda^{(L+1)}}(t)}{1+\tilde{v}^{(t)}(z) t}\Big)^{-1}\\
   v^{\prime(t+1)}(z)&=\Big(\frac{1}{(\tilde{v}^{(t)}(z))^{2}}-\gamma_{L+1} \int \frac{t^2 d H_{\Lambda^{(L+1)}}(t)}{(1+t \tilde{v}^{(t)}(z))^2}\Big)^{-1}
\end{align*} 
After $T$ iterations 
\begin{align*}
    \tilde{s}^{(T)}(z)&=\gamma_{L+1}^{-1}(\tilde{v}^{(T)}(z)+z^{-1})-z^{-1} \\
   \tilde{s}^{\prime (T)}(z)&=\gamma_{L+1}^{-1}(\tilde{v}^{\prime (T)}(z)-z^{-2})+z^{-2}
\end{align*}
Then one could estimate the limiting risk at point $z=-\lambda$ by replacing quantities $s(-\lambda)$ and $s'(-\lambda)$ by $\tilde{s}^{(T)}(-\lambda)$ and $\tilde{s}^{\prime (T)}(-\lambda)$ in \eqref{oracle_limiting_risk}. One can generate the plots of limiting risk w.r.t. $\lambda$ by tabulating the limiting risk \eqref{oracle_limiting_risk} on a dense grid of $\lambda$.
\begin{figure}[t] 
\begin{subfigure}[b]{0.3\textwidth}
         \centering
         \includegraphics[width=\textwidth]{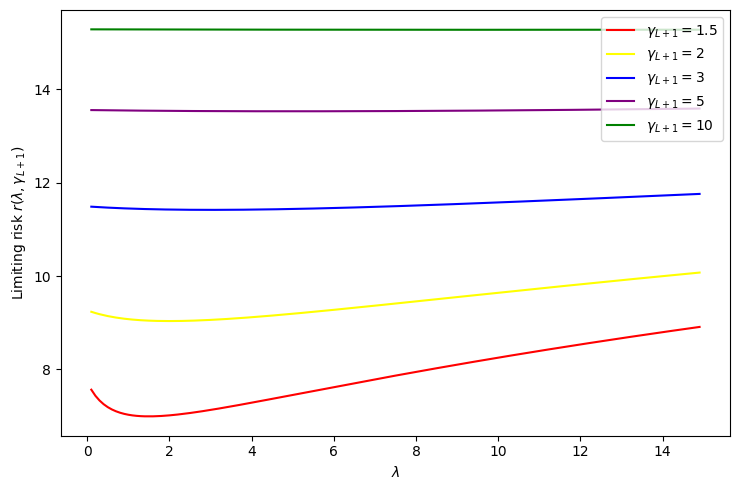}
	\caption{$\Sigma^{(L+1)}=I_p$}\label{limiting_risk_plt2}
     \end{subfigure}
 \begin{subfigure}[b]{0.3\textwidth}
         \centering
         \includegraphics[width=\textwidth]{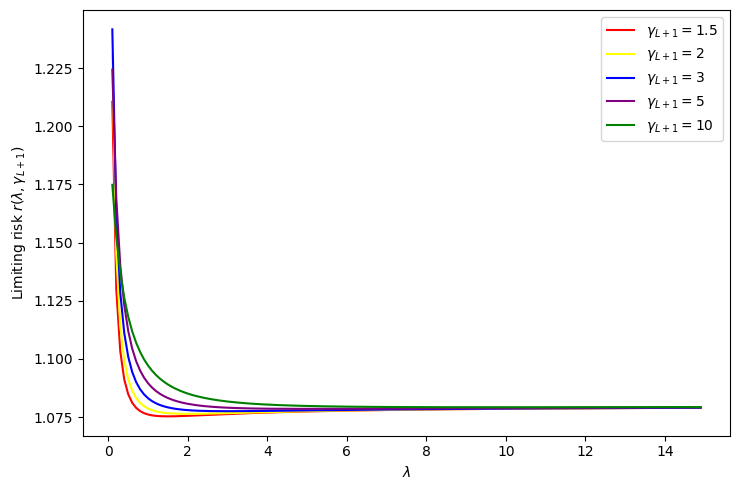}
	\caption{$\Sigma^{(L+1)}=\Omega^{-2}$}\label{limiting_risk_plt3}
     \end{subfigure}
     \begin{subfigure}[b]{0.3\textwidth}
         \centering
         \includegraphics[width=\textwidth]{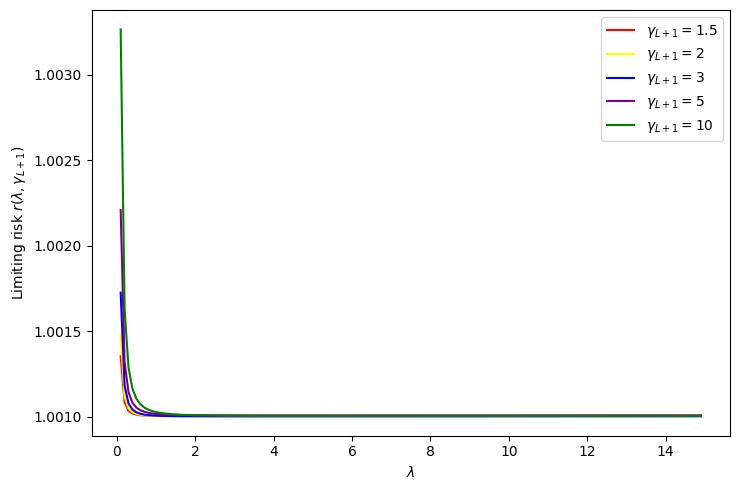}
	\caption{$\Sigma^{(L+1)}=\Omega^{-4}$}\label{limiting_risk_plt4}
     \end{subfigure}
     \caption{Plot of limiting risk when $\Sigma^{(L+1)}=\Omega^{-\kappa}$.}
\end{figure}   

\section{Comparison between Theorem \ref{asymptotic_behavior_of_predictive_risk} and \cite{wu2020optimal}}\label{sec:comparisontowuandxu}
The limiting behavior of oracle risk $\oraclerisk(\Omega\mid X^{(L+1)})$ could be derived from the approach by \citet[Theorem 1]{wu2020optimal} with $\mathbf{\Sigma}_{\omega\beta}=\Omega^{-\frac{1}{2}}\Omega \Omega^{-\frac{1}{2}}=I_{p}$ (in the notation of their paper). However, this approach is relatively complicated as we illustrate next. 

The bias part contained in oracle risk $\oraclerisk(\Omega\mid X^{(L+1)})$ is given by
\begin{align*}
    \frac{\lambda^2}{p} \operatorname{tr}\Big(\Omega^{\frac{1}{2}} \Sigma^{(L+1)} \Omega^{\frac{1}{2}}\big(\Omega^{\frac{1}{2}} \hat{\Sigma}^{(L+1)} \Omega^{\frac{1}{2}}+\lambda I\big)^{-2} \Big).
\end{align*}
Let $s(z)$ be the Stieltjes transform of $\frac{1}{n_{L+1}}X^{(L+1)\top}X^{(L+1)}$ and $m(z)$ be the Stieltjes transform of $\frac{1}{n_{L+1}}X^{(L+1)}X^{(L+1)\top}$. In Assumption 1 from \cite{wu2020optimal}, the random variable $g$ is degenerated to be a constant $1$. Define $S=\Omega^{\frac{1}{2}} \hat{\Sigma}^{(L+1)} \Omega^{\frac{1}{2}}+\lambda I$. Following analysis in proof of Theorem 1 in \cite{wu2020optimal}, this bias part is just $\frac{\lambda^2}{p}\operatorname{tr}(\Omega^{\frac{1}{2}}\Sigma^{(L+1)}\Omega^{\frac{1}{2}}S^{-2})$. By analyzing the quantity $\frac{1}{p}\operatorname{tr}(S^{-2}\Omega^{\frac{1}{2}}\hat{\Sigma}^{(L+1)}\Omega^{\frac{1}{2}})$ and similar to (A.3) in \cite{wu2020optimal}, it holds that
\begin{align*}
    \frac{1}{p} \operatorname{tr}\big(S^{-2} \Omega^{\frac{1}{2}} \hat{\Sigma}^{(L+1)} \Omega^{\frac{1}{2}}\big)=\frac{1}{p} \operatorname{tr}\big(\big(\Omega^{\frac{1}{2}} \hat{\Sigma}^{(L+1)}  \Omega^{\frac{1}{2}}-\lambda I\big)^{-1}-\lambda\big(\Omega^{\frac{1}{2}} \hat{\Sigma}^{(L+1)}  \Omega^{\frac{1}{2}}-\lambda I\big)^{-2}\big).
\end{align*}
Also, (A.5) in \cite{wu2020optimal} for analyzing $\oraclerisk(\Omega\mid X^{(L+1)})$ becomes 
\begin{align*}
    \frac{1}{p}\operatorname{tr}\big(S^{-2} \Omega^{\frac{1}{2}} \hat{\Sigma}^{(L+1)} \Omega^{\frac{1}{2}}\big) \stackrel{p}{\rightarrow} \frac{\frac{\lambda^2}{p} \operatorname{tr}\big(\Omega^{\frac{1}{2}}\Sigma^{(L+1)}\Omega^{\frac{1}{2}} S^{-2}\big)}{\big(\frac{1}{m(-\lambda)}\big)^2}.
\end{align*}
Then the bias part will converge in probability to 
\begin{align}
    \frac{1}{m^2(-\lambda)} \frac{1}{p} \operatorname{tr}\Big(\big(\Omega^{\frac{1}{2}} \hat{\Sigma}^{(L+1)} \Omega^{\frac{1}{2}}-\lambda I\big)^{-1}-\lambda\big(\Omega^{\frac{1}{2}} \hat{\Sigma}^{(L+1)} \Omega^{\frac{1}{2}}-\lambda I\big)^{-2}\Big).\label{Appendixcompequ0}
\end{align}

Hence, to analyze the limiting behavior of $\oraclerisk(\Omega\mid X^{(L+1)})$, we only need to analyze the limit of $\frac{1}{p}\operatorname{tr}\big(\big(\Omega^{\frac{1}{2}} \hat{\Sigma}^{(L+1)} \Omega^{\frac{1}{2}}-\lambda I\big)^{-1}\big)$ and $\frac{1}{p}\operatorname{tr}\big(\big(\Omega^{\frac{1}{2}} \hat{\Sigma}^{(L+1)} \Omega^{\frac{1}{2}}-\lambda I\big)^{-2}\big)$, while in \cite{wu2020optimal}, one needs to analyze $\frac{1}{n} \operatorname{tr}\big(\boldsymbol{\Sigma}_{w \beta}\big(\boldsymbol{X}_{/ w}^{\top} \boldsymbol{X}_{/ w}+\lambda \boldsymbol{I}\big)^{-1}\big)$ where $\boldsymbol{\Sigma}_{w \beta}=\boldsymbol{\Sigma}_{w}^{\frac{1}{2}}\boldsymbol{\Sigma}_{\beta}\boldsymbol{\Sigma}_{w}^{\frac{1}{2}}$ for arbitrary $\boldsymbol{\Sigma}_{w}$. 

Our approach differs from that in~\cite{wu2020optimal} at this point, as results in \cite{ledoit2011eigenvectors} can not be applied at this point following their approach. Indeed, it relies on the underlying condition that $\boldsymbol{\Sigma}_{w\beta}$ and $\boldsymbol{\Sigma}_{x / w}=\boldsymbol{\Sigma}_w^{-1 / 2} \boldsymbol{\Sigma}_x \boldsymbol{\Sigma}_w^{-1 / 2}$ shares the same eigenvectors. Only in this case, it is possible to write $\boldsymbol{\Sigma}_{w\beta}$ as a continuous function of $\boldsymbol{\Sigma}_{x / w}$. However, when $\boldsymbol{\Sigma}_{w}=\Omega^{-1}$ and $\boldsymbol{\Sigma}_{w\beta}=I$, analyzing the limit of bias part does not require the assumption 1 in \cite{wu2020optimal}. 

Now we show that the limit $\frac{m'(-\lambda)}{m^2(-\lambda)}\mathbb{E}\frac{h}{(hm(-\lambda)+1)^2}$ presented in \cite{wu2020optimal} for bias part is essentially the same as $\frac{m(-\lambda)-\lambda m'(-\lambda)}{\gamma_{L+1}m^2(-\lambda)}$ presented in Theorem \ref{asymptotic_behavior_of_predictive_risk}. To do so, it is required to use theorem 1 in \cite{rubio2011spectral} with $A=0$, $T=I$ and $R=\Omega^{\frac{1}{2}}\Sigma^{(L+1)}\Omega^{\frac{1}{2}}$. It holds that
\begin{align*}
    \Big|\frac{1}{p} \operatorname{tr}\big(\big(\Omega^{\frac{1}{2}} \hat{\Sigma}^{(L+1)} \Omega^{\frac{1}{2}}-z I\big)^{-1}\big)-\frac{1}{p} \operatorname{tr}\big(\big(c_{n_{L+1}}(z) \Omega^{\frac{1}{2}} \Sigma^{(L+1)} \Omega^{\frac{1}{2}}-z I\big)^{-1}\big)\Big| \stackrel{a.s.}{\rightarrow} 0,
\end{align*}
where
\begin{align}
    c_{n_{L+1}}(z)&=\frac{1}{n_{L+1}} \operatorname{tr}\big(\big(I_{n_{L+1}}+\frac{p}{n_{L+1}} e_{p}(z) I\big)^{-1}\big)\label{Appendixcompequ1}, \\
    e_p(z)&=\frac{1}{p} \operatorname{tr}\big(\Omega^{\frac{1}{2}} \Sigma^{(L+1)} \Omega^{\frac{1}{2}}\big(c_{n_{L+1}}(z) \Omega^{\frac{1}{2}} \Sigma^{(L+1)} \Omega^{\frac{1}{2}}-z I_p\big)^{-1}\big).\label{Appendixcompequ2}
\end{align}
By proof of Theorem 1 in \cite{rubio2011spectral}, we have that  $c_{n_{L+1}}(z)\rightarrow -zm(z)$. Combining \eqref{Appendixcompequ1} and \eqref{Appendixcompequ2} together, $c_{n_{L+1}}(z)$ should satisfy following equation
\begin{align*}
    c_{n_{L+1}}(z)&=1-\frac{1}{n_{L+1}} \operatorname{tr}\Big(c_{n_{L+1}}(z) \Omega^{\frac{1}{2}} \Sigma^{(L+1)} \Omega^{\frac{1}{2}}\big(c_{n_{L+1}}(z) \Omega^{\frac{1}{2}} \Sigma^{(L+1)} \Omega^{\frac{1}{2}}-z I\big)^{-1}\Big)\\
    &=1-\frac{p}{n_{L+1}}\frac{1}{p}\sum_{i=1}^{p}\frac{c_{n_{L+1}}(z)\lambda_{i}(\Omega^{\frac{1}{2}} \Sigma^{(L+1)} \Omega^{\frac{1}{2}})}{c_{n_{L+1}}(z)\lambda_{i}(\Omega^{\frac{1}{2}} \Sigma^{(L+1)} \Omega^{\frac{1}{2}})-z},
\end{align*}
where $\lambda_{i}(\Omega^{\frac{1}{2}} \Sigma^{(L+1)} \Omega^{\frac{1}{2}})$ is $i$-th eigenvalue of $\Omega^{\frac{1}{2}} \Sigma^{(L+1)} \Omega^{\frac{1}{2}}$. Taking $n_{L+1},p\rightarrow\infty$, then it holds that
\begin{align*}
    -zm(z)=1-\gamma\mathbb{E}\frac{-zm(z)h}{-zm(z)h-z}.
\end{align*}
Plugging in $z=-\lambda$, we see that 
\begin{align*}
    \lambda=\frac{1}{m(-\lambda)}-\gamma\mathbb{E}\frac{h}{hm(z)+1}.
\end{align*}
Besides, note that
\begin{align*}
   \frac{1}{p} \operatorname{tr}\big(\big(c_{n_{L+1}}(z) \Omega^{\frac{1}{2}} \Sigma^{(L+1)} \Omega^{\frac{1}{2}}-z I\big)^{-1}\big)&=\frac{1}{p}\sum_{i=1}^{p}\frac{1}{c_{n_{L+1}}(z)\lambda_{i}(\Omega^{\frac{1}{2}} \Sigma^{(L+1)} \Omega^{\frac{1}{2}})-z} \\
   &\stackrel{p}{\rightarrow}\frac{1}{-z}\mathbb{E}\frac{1}{hm(z)+1}.
\end{align*}
By taking $z=-\lambda$, it holds that
\begin{align*}
    \Big|\frac{1}{p} \operatorname{tr}\big(\big(\Omega^{\frac{1}{2}} \hat{\Sigma}^{(L+1)} \Omega^{\frac{1}{2}}+\lambda I\big)^{-1}\big)-\frac{1}{\lambda}\mathbb{E}\frac{1}{hm(-\lambda)+1}\Big|\stackrel{a.s.}{\rightarrow} 0.
\end{align*}
Also, it holds that
\begin{align*}
    \Big|\frac{1}{p} \operatorname{tr}\big(\big(\Omega^{\frac{1}{2}} \hat{\Sigma}^{(L+1)} \Omega^{\frac{1}{2}}-z I\big)^{-1}\big)-s(z)\Big|\stackrel{a.s.}{\rightarrow} 0\quad\quad s(z)=\gamma_{L+1}^{-1} m(z)+\gamma_{L+1}^{-1} z^{-1}.
\end{align*}
Hence, by uniqueness of limit,
\begin{align*}
\mathbb{E}\frac{1}{hm(-\lambda)+1}=\lambda(\gamma_{L+1}^{-1}m(-\lambda)-\gamma_{L+1}^{-1}\lambda^{-1}).   
\end{align*}
Now, by taking derivative w.r.t. $z$ and taking $z=-\lambda$, it holds that
\begin{align*}
    &\Big|\frac{1}{p}\operatorname{tr}\big(\big(\Omega^{\frac{1}{2}} \hat{\Sigma}^{(L+1)} \Omega^{\frac{1}{2}}+\lambda I\big)^{-2}\big)-\frac{1}{\lambda^2} \mathbb{E} \frac{1}{h m(-\lambda)+1}+\frac{1}{\lambda} \mathbb{E} \frac{ m^{\prime}(-\lambda)}{(h m(-\lambda)+1)^2}\Big|\stackrel{a.s.}{\rightarrow}0\\
    &\Big|\frac{1}{p}\operatorname{tr}\big(\big(\Omega^{\frac{1}{2}} \hat{\Sigma}^{(L+1)} \Omega^{\frac{1}{2}}+\lambda I_{p}\big)^{-2}\big)-\big(\gamma_{L+1}^{-1}\big(m^{\prime}(-\lambda)-\frac{1-\gamma_{L+1}}{\lambda^2})\big)\Big|\stackrel{a.s.}{\rightarrow}0,
\end{align*}
and again, 
\begin{align*}
    \frac{1}{\lambda^2} \mathbb{E} \frac{1}{h m(-\lambda)+1}-\frac{1}{\lambda} \mathbb{E} \frac{ m^{\prime}(-\lambda)}{(h m(-\lambda)+1)^2}=\big(\gamma_{L+1}^{-1}\big(m^{\prime}(-\lambda)-\frac{1-\gamma_{L+1}}{\lambda^2})\big).
\end{align*}
Now by \eqref{Appendixcompequ0}, the bias part will converge to 
\begin{align*}
    \frac{m'(-\lambda)}{m^2(-\lambda)}\mathbb{E}\frac{h}{(hm(-\lambda)+1)^2},
\end{align*}
and this is equal to 
\begin{align*}
    \frac{m(-\lambda)-\lambda m'(-\lambda)}{\gamma_{L+1}m^2(-\lambda)},
\end{align*}
which coincides with our results presented in Theorem \ref{asymptotic_behavior_of_predictive_risk}.

To summarize, our proof of Theorem~\ref{asymptotic_behavior_of_predictive_risk} is much simpler and specifically suited to the meta-learning problem that we focus on in this work.

\section{Additional experiments}\label{sec:addexp}

\subsection{Diminishing eigenvalue case on $\Omega$}
Instead of utilizing $\Omega$ as in~\eqref{eq:examplematrix}, we adopt an alternative approach where the eigenvalues of $\Omega$ decrease following a power-law decay, given by $\lambda_{j}(\Omega)=j^{-a}$ for some $a>0$, while keeping the matrices $\Sigma^{(\ell)}$ as identity matrices. Note in particular that this violates our assumption. The generation process for $\Omega$ is as follows: 
\begin{itemize}
\item We randomly generate an orthogonal matrix $P$. 
\item Let $D(\Omega)$ be a diagonal matrix with its $j$-th diagonal entry defined as $\lambda_{j}(\Omega)=j^{-a}$.
\item Set $\Omega=PD(\Omega)P^\top$.
\end{itemize} 
In this experiment, we set $p=128$, $n_{\ell}=50$ for all the tasks and $n_{L+1}$ varying from 25, 50, 75, 100, 125, 150. Then the estimator $\hat{\Omega}$ is calculated based on \eqref{optimization_for_Omega} and MLE. We set $a$ to be $\frac{1}{4},\frac{1}{10}$ and $\frac{1}{100}$ respectively. For $\lambda_{j}(\Omega)=j^{-1/4}$, this means the eigenvalue of $\Omega$ decays fast and it violates the assumption \ref{asp_Omegahat} most severely. On the contrary, for $\lambda_{j}(\Omega)=j^{-1/100}$, this means the eigenvalue of $\Omega$ decays very slowly and this is almost same to identity matrix. The results for $\lambda_{j}(\Omega)=j^{-1/4}$, $\lambda_{j}(\Omega)=j^{-1/10}$ and $\lambda_{j}(\Omega)=j^{-1/100}$ are given in Table \ref{decay_eigenvalue_Omega_Tab1}, Table \ref{decay_eigenvalue_Omega_Tab2} and Table \ref{decay_eigenvalue_Omega_Tab3} respectively. The results indicate that the behavior of predictive risk $\prisk(\hat{\Omega}\mid X^{(L+1)})$ is not very closed to the corresponding limiting risk $r(\lambda,\gamma_{L+1})$. The Difference Percentage is larger for the case when eigenvalue of $\Omega$ decays faster.

\begin{table}
\centering
    %\resizebox{\textwidth}{!}{
    \begin{subtable}{1\textwidth}
    \centering
    \begin{tabular}{|c|c|c|c|c|}
    \hline
       \specialcell{$n_{L+1}$} & \specialcell{$R(I\mid X)$} & \specialcell{$R(\hat{\Omega}\mid X)$} &\specialcell{$r(\lambda,\gamma_{L+1})$} &\specialcell{Difference Percentage} \\
        \hline 
        $25$     & 1.39  & 1.40  & 1.22 & 14.48\%\\\hline
        $50$     & 1.38  & 1.39  & 1.21 & 14.46\%\\\hline
        $75$     & 1.33  & 1.31  & 1.20  & 8.51\%\\\hline
        $100$    & 1.30  & 1.29  & 1.20  & 8.05\%\\\hline
        $125$    & 1.26  & 1.26  & 1.19  & 5.73\%\\\hline
        $150$    & 1.25  & 1.24  & 1.18  & 5.14\%\\\hline
        \end{tabular} 
        \caption{Results for \eqref{optimization_for_Omega} (Initialization: Random generated )}
    \end{subtable}\\
    \begin{subtable}{1\textwidth}
    \centering
    \begin{tabular}{|c|c|c|c|c|}\hline
       \specialcell{$n_{L+1}$} & \specialcell{$R(I\mid X)$} & \specialcell{$R(\hat{\Omega}\mid X)$} &\specialcell{$r(\lambda,\gamma_{L+1})$} &\specialcell{Difference Percentage} \\ 
        \hline 
        $25$     & 1.39  & 1.40  & 1.22 & 14.66\%\\\hline
        $50$     & 1.38  & 1.39  & 1.21 & 14.60\%\\\hline
        $75$     & 1.33  & 1.31  & 1.20  & 8.89\%\\\hline
        $100$    & 1.30  & 1.29  & 1.20  & 7.65\%\\\hline
        $125$    & 1.26  & 1.27  & 1.19  & 6.35\%\\\hline
        $150$    & 1.25  & 1.20  & 1.18  & 1.36\%\\\hline
        \end{tabular} 
        \caption{Results for MLE (Initialization: Output of \eqref{optimization_for_Omega})}
    \end{subtable}
    %}
\caption{Comparison of estimator $\hat{\Omega}$ based on \eqref{optimization_for_Omega} and MLE for the case $\lambda_{j}(\Omega)=\frac{1}{j^{\frac{1}{4}}}$}\label{decay_eigenvalue_Omega_Tab1}
\end{table}  

\begin{table}[t]
\centering
    %\resizebox{\textwidth}{!}{
    \begin{subtable}{1\textwidth}
    \centering
    \begin{tabular}{|c|c|c|c|c|}
            \hline 
       \specialcell{$n_{L+1}$} & \specialcell{$R(I\mid X)$} & \specialcell{$R(\hat{\Omega}\mid X)$} &\specialcell{$r(\lambda,\gamma_{L+1})$} &\specialcell{Difference Percentage} \\
        \hline \hline
        $25$     & 1.60  & 1.61   & 1.45 & 11.18\%\\        \hline 

        $50$     & 1.58  & 1.59 & 1.42 & 11.98\%\\        \hline 

        $75$     & 1.56  & 1.57 & 1.41 & 11.56\%\\        \hline 

        $100$    & 1.52 & 1.51 & 1.37 & 10.19\%\\        \hline 

        $125$    & 1.47  & 1.46   & 1.35 & 8.13\%\\        \hline 

        $150$    & 1.45  & 1.44   & 1.33 & 8.05\%\\        \hline 

        \end{tabular} 
        \caption{Results for \eqref{optimization_for_Omega} (Initialization: Random generated )}
    \end{subtable}\\
    \begin{subtable}{1\textwidth}
    \centering
    \begin{tabular}{|c|c|c|c|c|}        \hline 

       \specialcell{$n_{L+1}$} & \specialcell{$R(I\mid X)$} & \specialcell{$R(\hat{\Omega}\mid X)$} &\specialcell{$r(\lambda,\gamma_{L+1})$} &\specialcell{Difference Percentage} \\
        \hline \hline
        $25$     & 1.60  & 1.61   & 1.45 & 11.20\%\\        \hline 

        $50$     & 1.58  & 1.58 & 1.42 & 11.03\%\\        \hline 

        $75$     & 1.56  & 1.56 & 1.41 & 10.67\%\\        \hline 

        $100$    & 1.52 & 1.51 & 1.37 & 10.07\%\\        \hline 

        $125$    & 1.47  & 1.45   & 1.35 & 7.48\%\\        \hline 

        $150$    & 1.45  & 1.44   & 1.33 & 7.74\%\\        \hline 

        \end{tabular} 
        \caption{Results for MLE (Initialization: Output of  \eqref{optimization_for_Omega})}
    \end{subtable}
    %}
\caption{Comparison of estimator $\hat{\Omega}$ based on \eqref{optimization_for_Omega} and MLE for the case $\lambda_{j}(\Omega)=\frac{1}{j^{\frac{1}{10}}}$}\label{decay_eigenvalue_Omega_Tab2}
\end{table}  

\begin{table}
\centering
   % \resizebox{\textwidth}{!}{
    \begin{subtable}{1\textwidth}
    \centering
    \begin{tabular}{|c|c|c|c|c|}
    \hline
       \specialcell{$n_{L+1}$} & \specialcell{$R(I\mid X)$} & \specialcell{$R(\hat{\Omega}\mid X)$} &\specialcell{$r(\lambda,\gamma_{L+1})$} &\specialcell{Difference Percentage} \\
        \hline \hline
        $25$     & 1.85  & 1.86   & 1.84 & 1.24\%\\\hline
        $50$     & 1.81  & 1.83 & 1.77 & 3.44\%\\\hline
        $75$     & 1.74  & 1.72 & 1.71 & 0.60\%\\\hline
        $100$    & 1.65 & 1.66 & 1.65 & 0.75\%\\\hline
        $125$    & 1.60  & 1.61   & 1.59 & 0.73\%\\\hline
        $150$    & 1.55  & 1.57   & 1.55 & 1.38\% \\ \hline
        \end{tabular} 
        \caption{Results for \eqref{optimization_for_Omega} (Initialization: Random generated)}
    \end{subtable} \\
    
    \begin{subtable}{1\textwidth}
    \centering
    \begin{tabular}{|c|c|c|c|c|}
    \hline
       \specialcell{$n_{L+1}$} & \specialcell{$R(I\mid X)$} & \specialcell{$R(\hat{\Omega}\mid X)$} &\specialcell{$r(\lambda,\gamma_{L+1})$} &\specialcell{Difference Percentage} \\
        \hline \hline
        $25$     & 1.85  & 1.86   & 1.84 & 1.14\%\\\hline
        $50$     & 1.81  & 1.82 & 1.77 & 3.07\%\\\hline
        $75$     & 1.74  & 1.71 & 1.71 & -0.13\%\\\hline
        $100$    & 1.65 & 1.65 & 1.65 & 0.35\%\\\hline
        $125$    & 1.60  & 1.60   & 1.59 & 0.45\%\\\hline
        $150$    & 1.55  & 1.56   & 1.55 & 0.81\%\\\hline
        
        \end{tabular} 
        \caption{Results for MLE (Initialization: Output of \eqref{optimization_for_Omega})}
    \end{subtable}
    %}
    \caption{Comparison of estimator $\hat{\Omega}$ based on \eqref{optimization_for_Omega} and MLE for the case $\lambda_{j}(\Sigma^{(L+1)})=\frac{1}{j^{\frac{1}{100}}}$}\label{decay_eigenvalue_Omega_Tab3}
\end{table} 
 
\subsection{Diminishing eigenvalue case on $\Sigma^{(L+1)}$}
Instead of using identity matrix for $\Sigma^{(L+1)}$, we now let the eigenvalue of $\Sigma^{(L+1)}$ is decreasing as $\lambda_{j}=j^{-a}$ for some $a>0$. For this section, the choice of $\Omega$ is still given by \eqref{eq:examplematrix} with $a=16$ and $b=5$. Suppose that the eigenvalue decomposition of $\Omega$ is given by $\Omega=UD(\Omega) U^{\top}$. Then we set $\Sigma^{(L+1)}=UD(\Sigma^{(L+1)})U^{\top}$
where the eigenvalue of $\Sigma^{(L+1)}$ is given by
\begin{align*}
\lambda_{j}(\Sigma^{(L+1)})=j^{-a}    
\end{align*} 
Therefore, the eigenvectors of $\Sigma^{(L+1)}$ are the same as the eigenvectors of $\Omega$.

In this experiment, we set $p=128$, $n_{\ell}=50$ for all the tasks and $n_{L+1}$ varying from 25, 50, 75, 100, 125, 150. Then the estimator $\hat{\Omega}$ is calculated based on \eqref{optimization_for_Omega} and MLE. We set $a$ to be $\frac{1}{4},\frac{1}{10}$ and $\frac{1}{100}$ respectively. The results for $\lambda_{j}(\Sigma^{(L+1)})=j^{-1/4}$, $\lambda_{j}(\Sigma^{(L+1)})=j^{-1/10}$ and $\lambda_{j}(\Sigma^{(L+1)})=j^{-1/100}$ are given in Table \ref{decay_eigenvalue_Sigma_Tab1}, Table \ref{decay_eigenvalue_Sigma_Tab2} and \ref{decay_eigenvalue_Sigma_Tab3} respectively. The performance of predictive risk is still good in these three cases.

\begin{table}    \centering

    %\resizebox{\textwidth}{!}{
    \begin{subtable}{1\textwidth}
    \centering
    \begin{tabular}{|c|c|c|c|c|}  \hline 
       \specialcell{$n_{L+1}$} & \specialcell{$R(I\mid X)$} & \specialcell{$R(\hat{\Omega}\mid X)$} &\specialcell{$r(\lambda,\gamma_{L+1})$} &\specialcell{Difference Percentage} \\
        \hline \hline
        $25$     &  7.12  & 6.40   & 4.25   & 50.68\%\\\hline
        $50$     &  6.48  & 5.18   & 3.58   & 44.83\%\\\hline
        $75$     &  5.92  & 4.85   & 3.05   & 59.00\%\\\hline
        $100$    &  5.50  & 3.65   & 2.66   & 37.22\%\\\hline
        $125$    &  5.07  & 3.08   & 2.36   & 30.24\%\\\hline
        $150$    &  4.72  & 2.52   & 2.14   & 18.08\%\\\hline
        \end{tabular} 
        \caption{Results for \eqref{optimization_for_Omega} (Initialization: Random generated Initialization)}
    \end{subtable}
    \begin{subtable}{1\textwidth}
    \centering
    \begin{tabular}{|c|c|c|c|c|}\hline
       \specialcell{$n_{L+1}$} & \specialcell{$R(I\mid X)$} & \specialcell{$R(\hat{\Omega}\mid X)$} &\specialcell{$r(\lambda,\gamma_{L+1})$} &\specialcell{Difference Percentage} \\
        \hline \hline
        $25$     &  7.12  & 6.34   & 4.25   & 49.29\%\\\hline
        $50$     &  6.48  & 5.20   & 3.58   & 45.41\%\\\hline
        $75$     &  5.92  & 4.86   & 3.05   & 59.29\%\\\hline
        $100$    &  5.50  & 3.58   & 2.66   & 34.60\%\\\hline
        $125$    &  5.07  & 3.01   & 2.36   & 27.56\%\\\hline
        $150$    &  4.72  & 2.54   & 2.14   & 18.92\%\\\hline
        \end{tabular} 
        \caption{Results for MLE (Initialization: Output of \eqref{optimization_for_Omega})}
    \end{subtable}
    %}
    \caption{Comparison of estimator $\hat{\Omega}$ based on \eqref{optimization_for_Omega} and MLE for the case $\lambda_{j}(\Sigma^{(L+1)})=\frac{1}{j^{\frac{1}{4}}}$}\label{decay_eigenvalue_Sigma_Tab1}

\end{table}  

\begin{table}
\centering
    \begin{subtable}{1\textwidth}
    \centering
    \begin{tabular}{|c|c|c|c|c|}
    \hline
       \specialcell{$n_{L+1}$} & \specialcell{$R(I\mid X)$} & \specialcell{$R(\hat{\Omega}\mid X)$} &\specialcell{$r(\lambda,\gamma_{L+1})$} &\specialcell{Difference Percentage} \\
        \hline 
        $25$     &  10.84  & 9.62   & 8.22  & 17.04\%\\\hline
        $50$     &   9.65  & 8.35   & 6.59  & 26.54\%\\\hline
        $75$     &   8.69  & 6.84   & 5.18  & 31.98\%\\\hline
        $100$    &   7.72  & 5.65   & 4.11  & 37.41\%\\\hline
        $125$    &   6.90  & 4.69   & 3.33  & 40.61\%\\\hline
        $150$    &   6.18  & 3.95   & 2.79  & 41.70\%\\\hline
        \end{tabular} 
        \caption{Results for \eqref{optimization_for_Omega} (Initialization: Random generated )}
    \end{subtable}\\
%    \medskip
    \begin{subtable}{1\textwidth}
    \centering
    \begin{tabular}{|c|c|c|c|c|}
    \hline
       \specialcell{$n_{L+1}$} & \specialcell{$R(I\mid X)$} & \specialcell{$R(\hat{\Omega}\mid X)$} &\specialcell{$r(\lambda,\gamma_{L+1})$} &\specialcell{Difference Percentage} \\
        \hline \hline
        $25$     &  10.84  & 9.38   & 8.22  & 14.13\%\\\hline
        $50$     &   9.65  & 8.16   & 6.59  & 23.67\%\\\hline
        $75$     &   8.69  & 6.74   & 5.18  & 29.89\%\\\hline
        $100$    &   7.72  & 5.45   & 4.11  & 32.51\%\\\hline
        $125$    &   6.90  & 4.59   & 3.33  & 37.48\%\\\hline
        $150$    &   6.18  & 3.85   & 2.79  & 38.09\%\\ \hline
        \end{tabular} 
        \caption{Results for MLE (Initialization:Output of \eqref{optimization_for_Omega})}
    \end{subtable}
    \caption{Comparison of estimator $\hat{\Omega}$ based for \eqref{optimization_for_Omega} and MLE on the case $[\Lambda_{\Sigma^{(L+1)}}]_{jj}=\frac{1}{j^{\frac{1}{10}}}$}\label{decay_eigenvalue_Sigma_Tab2}
\end{table} 
\begin{table}
    \begin{subtable}{1\textwidth}
    \centering
    \begin{tabular}{|c|c|c|c|c|}\hline
       \specialcell{$n_{L+1}$} & \specialcell{$R(I\mid X)$} & \specialcell{$R(\hat{\Omega}\mid X)$} &\specialcell{$r(\lambda,\gamma_{L+1})$} &\specialcell{Difference Percentage} \\
        \hline \hline
        $25$     & 14.37   & 13.33   & 12.94  & 3.01\%\\\hline
        $50$     & 12.57   & 11.01   & 10.06  & 9.41\%\\\hline
        $75$     & 10.90   & 8.56   &  7.63  & 12.25\%\\\hline
        $100$    &  9.42   & 6.09   &  5.6368  & 8.0808\%\\\hline
        $125$    &  8.32   & 4.47   &  4.12  & 8.56\%\\\hline
        $150$    &  7.34   & 3.50   &  3.28  & 6.75\%\\\hline
        \end{tabular} 
        \caption{Results for estimator based on problem \eqref{optimization_for_Omega} (Random generated Initialization)}
    \end{subtable}
    \medskip
    \begin{subtable}{1\textwidth}
    \centering
    \begin{tabular}{|c|c|c|c|c|}\hline
       \specialcell{$n_{L+1}$} & \specialcell{$R(I\mid X)$} & \specialcell{$R(\hat{\Omega}\mid X)$} &\specialcell{$r(\lambda,\gamma_{L+1})$} &\specialcell{Difference Percentage} \\
        \hline \hline
        $25$     & 14.37   & 13.31   & 12.94  & 2.88\%\\\hline
        $50$     & 12.57   & 10.72   & 10.06  & 6.51\%\\\hline
        $75$     & 10.90   & 8.38   &  7.63  & 9.83\%\\\hline
        $100$    &  9.42   & 6.04   &  5.63  & 7.24\% \\\hline
        $125$    &  8.32   & 4.43   &  4.12  & 7.63\% \\\hline
        $150$    &  7.33   & 3.41   &  3.28  & 3.99\%\\\hline
        \end{tabular} 
        \caption{Results for MLE (Initialization: Result given by problem \eqref{optimization_for_Omega})}
    \end{subtable}
    \caption{Comparison of estimator $\hat{\Omega}$ based on \eqref{optimization_for_Omega} and MLE for the case $[\Lambda_{\Sigma^{(L+1)}}]_{jj}=\frac{1}{j^{\frac{1}{100}}}$}\label{decay_eigenvalue_Sigma_Tab3}

\end{table}   
\subsection{Influence of choosing different $\lambda$ in Sparse case}
\begin{figure}[t]
  \begin{center}
	\includegraphics[scale=0.5]{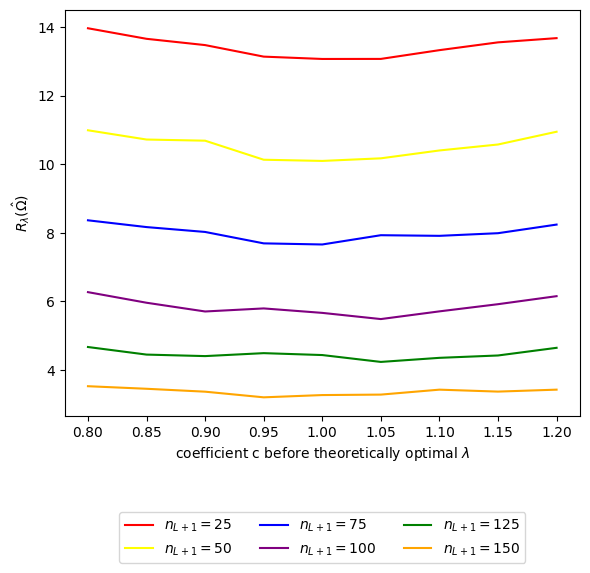}
  \end{center}
	\caption{Risk $R_{c\lambda^{*}}(\hat{\Omega})$ with different $c$}\label{figtrial1}
\end{figure}
Finally, we study the influence of the choice of different $\lambda$ in ridge regression. In this experiment, the number of task $L=1000$, the dimension $p=128$, the number of samples in each task $n_{\ell}=50$ ($\ell=1,\dots,L$) and the number of samples in the new task $n_{L+1}$ varies as $25,50,75,100,125,150$. The estimator $\hat{\Omega}$ is calculated based on \eqref{L1regularized_estimator_Omega} and in the new task the coefficient $\hat{\beta}^{(L+1)}$ is estimated by \eqref{estimator_generalized_ridge} with the choice of $\lambda$ given by $\lambda=c\frac{p\sigma^2}{n_{L+1}}$, where $c=0.8,0.85,0.9,0.95,1,1.05,1.1,1.15,1.2$. In particular, when $c=1$ and $\lambda=\frac{p\sigma^2}{n_{L+1}}$ is the theoretically optimal value for ridge regression that minimizes the predictive risk. 

The results of this simulation is given in Figure \ref{figtrial1}, where the x-axis is the choice of coefficient $c$ in front of the theoretical optimal $\lambda$ in ridge regression and y-axis is the predictive risk using different choice of $c$ in the $\lambda$. Figure \ref{figtrial1} indicate that the predictive risk is minimized near the theoretically optimal $\lambda$. 
\section{Auxiliary Results}

\begin{lemma}[Lemma 2.14 from \cite{bai2010spectral}]\label{lm2.14_bai_spectral} Let $f_1, f_2, \ldots$ be analytic on the domain $D$, satisfying $\left|f_n(z)\right| \leq M$ for every $n$ and $z$ in $D$. Suppose that there is an analytic function $f$ on $D$ such that $f_n(z) \rightarrow f(z)$ for all $z \in D$. Then it also holds that $f_n^{\prime}(z) \rightarrow f^{\prime}(z)$ for all $z \in D$.
\end{lemma}
The following standard concentration result is easy to obain.

\begin{lemma}\label{estimation_tech_lm2}
Let $X \in S G_p (\sigma),\|X\|_{2}=\sqrt{\sum_{i=1}^p X_i^2}$. Then
$$\mathbb{P}(\|X\|_{2} \geq t) \leq  5^p \exp \big\{-\frac{t^2}{8 \sigma^2}\big\}.$$
\end{lemma}

\begin{thm}[\cite{koltchinskii2011oracle}]\label{concentration_thm_01}
Given independent random $m_{1} \times m_{2}$ matrices $X_1, \ldots, X_n$ with $\mathbb{E} X_j=0$, denote
$$\sigma^2:=n^{-1}\max\Big\{\Big\|\mathbb{E}\sum_{i=1}^{n}X_{i}X_{i}^{\top}\Big\|,\Big\|\mathbb{E}\sum_{i=1}^{n}X_{i}^{\top}X_{i}\Big\|\Big\}.$$
Let $\alpha \geq 1$ and suppose that, for some $U^{(\alpha)}>0$ and for all $j=1, \ldots, n$,
we have that $\|\| X_j\|_{op}\|_{\psi_\alpha} \vee 2 \mathbb{E}^{1 / 2}\|X_{j}\|_{op}^2 \leq U^{(\alpha)}, \quad\quad\text{a.s.}$. Then, there exists a constant $K>0$ such that
$$\mathbb{P}\big\{\|X_1+\cdots+X_n\| \geq t\big\} \leq (m_{1}+m_{2}) \exp \Big\{-\frac{1}{K} \frac{t^2}{n \sigma^2+t U^{(\alpha)} \log ^{1 / \alpha}(U^{(\alpha)} / \sigma)}\Big\}.$$
\end{thm}

\begin{thm}[Theorem 3.4 from \cite{adamczak2015concentration}]\label{thm3.4adamczak2015concentration}
Let $\beta \in[2, \infty)$ and $Y$ be a random vector in $\mathbb{R}^k$, satisfying 
\begin{align*}
\text{Ent} f^2(Y)&=\mathbb{E} f^2(Y) \log f^2(Y)-\mathbb{E} f^2(Y) \log \mathbb{E} f^2(Y)\\
&\leq D_{L S_\beta}\Big(\mathbb{E}|\nabla f(Y)|^2+\mathbb{E} \frac{|\nabla f(Y)|^\beta}{f(Y)^{\beta-2}}\Big).
\end{align*}
Consider a random vector $X=\big(X_1, \ldots, X_m\big)$ in $\mathbb{R}^{m p}$, where $X_1, \ldots, X_m$ are independent copies of $Y$. Then for any locally Lipschitz $f: \mathbb{R}^{m p} \rightarrow \mathbb{R}$ such that $f(X)$ is integrable, and $q \geq 2$,
\begin{equation}\label{Lqbound_for_f(X)-Ef(X)}
\|f(X)-\mathbb{E} f(X)\|_{L_{q}} \leq C_\beta D_{L S_\beta}^{1 / 2} q^{1 / 2}\||\nabla f(X)|_2\|_{L_{q}}+D_{L S_\beta}^{1 / \beta} q^{1 / \alpha}\||\nabla f(X)|_\beta\|_{L_{q}}, 
\end{equation} 
where $\alpha=\frac{\beta}{\beta-1}$ is the Hölder conjugate of $\beta$. 
\end{thm}
\begin{rmk}
In particular, the second term $D_{L S_\beta}^{1 / \beta} q^{1 / \alpha}\||\nabla f(X)|_\beta\|_{L_{q}}$ in (\ref{Lqbound_for_f(X)-Ef(X)}) is upper bounded by first term. Hence, it suffices to bound the first term to prove concentration results needed in the proof of Theorem \ref{sparsethm_lowerbound}.
\end{rmk}  
 
\begin{thm}[Theorem 6.5 from \cite{wainwright2019high}]\label{wainwright_thm6.5}
There are universal constants $\{c_j\}_{j=0}^3$ such that, for any row-wise $\sigma$-sub-Gaussian random matrix $\mathbf{X} \in \mathbb{R}^{n \times d}$, the sample covariance $\widehat{\mathbf{\Sigma}}=\frac{1}{n} \sum_{i=1}^n x_i x_i^{\mathrm{T}}$ satisfies the bounds
\begin{align*}
\mathbb{E}\big[e^{\lambda\|\hat{\Sigma}-\Sigma\|_2}\big] \leq e^{c_0 \frac{\lambda^2 \sigma^4}{n}+4 d} \quad \text { for all }|\lambda|<\frac{n}{64 e^2 \sigma^2},
\end{align*}
and hence
\begin{align*}
\mathbb{P}\Bigg(\frac{\|\widehat{\boldsymbol{\Sigma}}-\boldsymbol{\Sigma}\|_2}{\sigma^2} \geq c_1\Big\{\sqrt{\frac{d}{n}}+\frac{d}{n}\Big\}+\delta\Bigg) \leq c_2 e^{-c_3 n \min\{\delta, \delta^2\}}, \quad \text { for all } \delta \geq 0.
\end{align*}
\end{thm}

\end{document}